\documentclass[12pt,reqno]{amsart}
\newcommand\ignore[1]{}
\input amssym.def
\usepackage[raiselinks=false,colorlinks=true,citecolor=blue,urlcolor=blue,linkcolor=blue,bookmarksopen=true,pdftex]{hyperref}
\usepackage{booktabs,multirow,lscape,longtable} % for tables 
\usepackage{mathtools} 
\usepackage{tikz} 
\numberwithin{equation}{section}
\numberwithin{equation}{subsection} 
\newtheorem{theorem}{Theorem}[section]
\newtheorem{proposition}[theorem]{Proposition}
\newtheorem{lemma}[theorem]{Lemma}
 
\newtheorem{remark}[theorem]{Remark}

\begin{document}
\baselineskip=15.5pt

\title
[Geometric cycles and automorphic representations]{Geometric cycles in compact Riemannian 
locally symmetric spaces of type IV and automorphic representations of 
complex simple Lie groups}  
\author{Pampa Paul}
\address{Department of Mathematics, Presidency University, 86/1 College Street, Kolkata 700073, India}
\email{pampa.maths@presiuniv.ac.in}

\thispagestyle{empty}
\date{}

\begin{abstract}
Let $G$ be a connected complex simple Lie group with maximal compact subgroup $U$. Let $\frak{g}$ be the Lie 
algebra of $G$, and $X = G/U$ be the associated Riemannian globally symmetric space of type IV. We have constructed 
three types of arithmetic uniform lattices in $G$, say of type $1$, type $2$, and type $3$ respectively. If 
$\frak{g} \neq \frak{b}_n\ (n \ge 1)$, then for each 
$1 \le i \le 3$, there is an arithmetic uniform torsion-free lattice $\Gamma$ in $G$ which is commensurable with a lattice 
of type $i$ such that the corresponding locally symmetric space 
$\Gamma \backslash X$ has some non-vanishing (in the cohomology level) geometric cycles, and the Poincar\'{e} duals of 
fundamental classes of such cycles are not represented by $G$-invariant differential forms on $X$. As a consequence, we 
are able to detect some automorphic representations of $G$, when $\frak{g} = \frak{\delta}_n\ (n >4),\ \frak{c}_n\ (n \ge 6)$,  
or $\frak{f}_4$. To prove these, we have 
simplified Ka\v{c}'s description of finite order automorphisms of $\frak{g}$ with respect to a Chevalley basis of $\frak{g}$. 
Also we have determined some orientation preserving group action on some subsymmetric spaces of $X$. 
\end{abstract}
\maketitle

\noindent 
\section{Introduction} 

  Let $G$ be a non-compact semisimple Lie group with finite centre and $K$ be a maximal compact subgroup of $G$ with $\theta$, the 
corresponding Cartan involution of $G$. Let 
$\Gamma$ be a torsion-free uniform lattice in $G$. Then $\Gamma$ acts freely on the 
Riemannian globally symmetric space $X := G \slash K$ and the canonical projection $\pi : X \longrightarrow \Gamma \backslash X$ 
is a covering projection. Let $B$ be a reductive subgroup of $G$ such that $K_B = B \cap K$ is a maximal compact subgroup of $B$. 
Set $X_B = B \slash K_B$ and $\Gamma_B = B \cap \Gamma$. Note that $X_B$ is a connected totally geodesic submanifold of $X$. 
Assume that the natural map $j : \Gamma_B \backslash X_B \longrightarrow  \Gamma \backslash X$ is an embedding. Then the image 
$C_B := j(\Gamma_B \backslash X_B)$ is called a geometric cycle. In literature, these are also known as modular symbols. 
Under certain conditions, the fundamental class $[C_B] \in 
H_d (\Gamma \backslash X ; \mathbb{C})\  (d = \textrm{dim}(\Gamma_B \backslash X_B))$ is non-trivial. So the Poincar\'{e} dual of 
$[C_B]$ contributes nontrivially to $H^* (\Gamma \backslash X ; \mathbb{C})$. 
See \cite[Th. 2.1]{mira}, \cite[Th. 4.11]{rs}. These theorems are restated here as Theorem \ref{mira}, Theorem \ref{rosch} 
respectively. 

  If $\Gamma \subset G$ be any lattice, 
the Hilbert space $L^2 (\Gamma \backslash G)$ of square integrable functions on $\Gamma \backslash G$ with respect to 
a $G$-invariant measure, is a unitary representation of $G$. Here the group action on 
$L^2 (\Gamma \backslash G)$ is given by the right translation of 
$G$ on $\Gamma \backslash G$. 
When $\Gamma$ is a uniform lattice, we have 
\[L^2 (\Gamma \backslash G) \cong \widehat{\bigoplus}_{\pi \in \hat{G}} m(\pi , \Gamma ) H_\pi, \] 
due to Gelfand and Pyatetskii-Shapiro \cite{ggp}, \cite{gp}; where $H_\pi$ is the representation space of $\pi \in \hat{G}$; 
$m(\pi , \Gamma ) \in \mathbb{N} \cup \{0\}$, the multiplicity of $\pi$ in $L^2 (\Gamma \backslash G)$. 
If $(\tau, \mathbb{C})$ is the trivial representation of $G$, then $m(\tau , \Gamma) = 1$. 
A unitary representation $\pi \in \hat{G}$ such that $m(\pi , \Gamma ) > 0$ for some uniform lattice $\Gamma$, is called an
automorphic representation of $G$ with respect to $\Gamma$.  

  The connection between the geometric cycles and automorphic representations has been made 
by the Matsushima's isomorphism. Assume now that $\Gamma$ is a torsion-free uniform lattice in $G$. Then the isomorphism 
$L^2 (\Gamma \backslash G) \cong \widehat{\bigoplus}_{\pi \in \hat{G}} m(\pi, \Gamma) H_\pi$ implies 
\[\bigoplus_{\pi \in \hat{G}} m_\pi H_{\pi,K} \xhookrightarrow { } C^\infty (\Gamma \backslash G)_K.\] 
Matsushima's formula \cite{matsushima} says that the above inclusion induces an isomorphism 
\begin{equation} \label{matsushima} 
\bigoplus_{\pi \in \hat{G}} m_\pi H^p (\frak{g}^\mathbb{C}, K; H_{\pi,K}) \cong 
H^p (\frak{g}^\mathbb{C}, K; C^\infty (\Gamma \backslash G)_K) \cong 
H^p (\Gamma \backslash X; \mathbb{C}), 
\end{equation} 
where $\frak{g}$ is the Lie algebra of $G$. \\
Hence a non-vanishing (in the cohomology level) geometric cycle will contribute to the RHS of \eqref{matsushima} 
and it may help to detect occurrence of 
some $\pi \in \hat{G}$ with non-zero $(\frak{g}^\mathbb{C}, K)$-cohomology. In fact, Theorem 2.1 in \cite{mira} states 
that under certain conditions, we have a pair of geometric cycles such that the corresponding cohomology classes are 
not only non-zero, but also these have non-zero $H^* (\frak{g}^\mathbb{C}, K; H_{\pi,K})$-components for some non-trivial 
$\pi \in \hat{G}$. 

  Based on Theorem 2.1 in \cite{mira}, this technique was used by Millson and Raghunathan \cite{mira} when $G = SU(p,q), 
SO_0(p,q), Sp (p,q)$. Based on Theorem 2.1 in \cite{mira} and Theorem 4.11 in \cite{rs}, 
Schwermer and Waldner \cite{sw} have done the case for $G = SU^*(2n)$, Waldner \cite{waldner} has done 
the case when $G$ is the non-compact real form of the exceptional complex Lie group $G_2$. 
The cases $G = SL(n,\mathbb{R}), SL(n,\mathbb{C})$ were considered by 
Schimpf \cite{schimpf}, the case $G = SO^*(2n)$ and more generally the case when 
$G/K$ is a Hermitian symmetric space were considered by Mondal and 
Sankaran \cite{mondal-sankaran1}, \cite{mondal-sankaran2}. Here we consider the case when $G$ is a connected 
complex simple Lie group. The main results are stated as Theorem \ref{th1} and Theorem \ref{th2}. In obtaining our results, 
we have first considered three types of arithmetic uniform lattices of a connected complex semisimple Lie group of adjoint type. 
The three types of lattices depend on how one views the Lie algebra $\frak{g}$. 
Type $1$ corresponds to viewing it as a real Lie algebra. Type $2$ views it 
as the complexification of the compact real form of $\frak{g}$. Type $3$ involves a choice of a non-compact real form of $\frak{g}$. 
Actually type $3$ is a union of a family of types, one for each non-compact real form of $\frak{g}$. Any lattice of type $i\ (i=1, 2, 
\textrm{or }3)$ is $\theta$-stable. See \S \ref{complexlattice} for details. 

    Let $\mathcal{L}_i(G)$ be the collection of 
$\theta$-stable torsion-free lattices of $G$ which are commensurable to $\textrm{Ad}^{-1}(\Gamma)$ for 
some $\Gamma$ of type $i$ ($i = 1, 2, 3$). 

\begin{theorem}\label{th1} 
Let $G$ be a connected complex simple Lie group with maximal compact subgroup $U$, and $X = G/U$. 
For each $i = 1, 2$; there exists  
$\Gamma \in \mathcal{L}_i (G)$ such that $H^k (\Gamma \backslash X ; \mathbb{C})$ contains a 
non-zero cohomology class which is not represented by $G$-invariant differential forms on $X$
for all $k$ of the form dim$(X(\bar{\sigma}))$, dim$(X(\bar{\sigma}\bar{\theta}))$ given in the Table \ref{resulttable}. 
Also depending on each pair $X(\bar{\sigma}),\ X(\bar{\sigma}\bar{\theta})$ in the Table \ref{resulttable}, there exists 
$\Gamma \in \mathcal{L}_3 (G)$ such that $H^k (\Gamma \backslash X ; \mathbb{C})$ contains a 
non-zero cohomology class which is not represented by $G$-invariant differential forms on $X$
for $k$ of the form dim$(X(\bar{\sigma}))$, dim$(X(\bar{\sigma}\bar{\theta}))$ given in the Table \ref{resulttable}.  
\end{theorem} 

  The proof of the above theorem is given in \S \ref{pfth1}. To prove the theorem we have used Ka\v{c}'s classification 
of finite order automorphisms of a complex simple Lie algebra $\frak{g}$ (\cite{kac}). Actually we have given a simple description of a finite 
order automorphism with respect to a Chevalley basis of $\frak{g}$ (see \S \ref{victor}), which is a new addition, as far as 
we know. Also we have described all finite order automorphisms of $\frak{g}^\mathbb{R}$ up to conjugation 
(see Remark \ref{autoremark}(iii)). Remark \ref{autoremark}(i), (ii) might be interesting from representation 
theoretic point of view. We also need to determine some orientation preserving group action on a subsymmetric space 
of a Riemannian globally symmetric space of type IV. The work has been done in \S \ref{op} and the result is summarised in 
Table \ref{ortable}. These are important in topology and other areas of mathematics also. 

  If $G$ is a connected complex simple Lie group, Theorem \ref {th1} gives us some non-vanishing (in the cohomology level) 
geometric cycles in the RHS of \eqref{matsushima}. To detect some automorphic representation of $G$, it is important to know 
the irreducible unitary representations of $G$ with non-zero relative Lie algebra cohomology, which appear in the LHS of 
\eqref{matsushima}. Let $G$ be a connected semisimple Lie group with finite centre, $\frak{g}=$Lie$(G)$, $K$ be a maximal compact subgroup of $G$ 
with Cartan involution $\theta$. The irreducible unitary representations of $G$ with non-zero $(\frak{g}^\mathbb{C}, K)$-cohomology 
are classified in terms of the $\theta$-stable parabolic 
subalgebras $\frak{q} \subset \frak{g}^\mathbb{C}$. A $\theta$-stable parabolic subalgebra of $\frak{g}$ is by definition, a parabolic subalgebra 
$\frak{q}$ of $\frak{g}^\mathbb{C}$ such that $\theta(\frak{q}) = \frak{q}$ and $\bar{\frak{q}} \cap \frak{q}$ is a Levi subalgebra of $\frak{q}$, 
where $\bar{ }$ denotes the conjugation of $\frak{g}^\mathbb{C}$ with respect to $\frak{g}$. Associated with a $\theta$-stable parabolic 
subalgebra $\frak{q}$, we have an irreducible unitary representation 
$A_\frak{q}$ of $G$ with trivial infinitesimal character and non-zero $(\frak{g}^\mathbb{C}, K)$-cohomology. 
If $\frak{q}$ is a $\theta$-stable parabolic subalgebra, then so is Ad$(k)(\frak{q})\ (k \in K)$; and $A_\frak{q}$, 
$A_{\textrm{Ad}(k)(\frak{q})}$ are unitarily equivalent. 
If $\frak{q} = \frak{g}$, then $A_\frak{q} = \mathbb{C}$, the trivial representation of $G$. 
If rank$(G)=$ rank$(K)$ and $\frak{q}$ is a $\theta$-stable Borel subalgebra, then $A_\frak{q}$ is a discrete series representation 
of $G$ with trivial infinitesimal character. See \S \ref{general} for more details. 

  Now let $G$ be complex, $\frak{u}$ be a compact real form 
of $\frak{g}$, and $\theta$ be the corresponding Cartan involution of $\frak{g}$. Let $\frak{h}$ be a $\theta$-stable Cartan subalgebra of 
$\frak{g}$. 
Choose a system of positive roots $\Delta^+$ in the set of all non-zero roots 
$\Delta = \Delta(\frak{g} , \frak{h})$. Let $\frak{b}$ be the corresponding Borel subalgebra of $\frak{g}$ and $\Phi$ be the set of 
all simple roots in $\Delta^+$. We can deduce that the $\theta$-stable parabolic subalgebras of $\frak{g}$ are of the form 
$\frak{q} \times \frak{q}$, where $\frak{q}$ is a parabolic subalgebra of $\frak{g}$ containing a $\theta$-stable Cartan subalgebra 
of $\frak{g}$ (see \S \ref{particular}). Also it is sufficient to consider the $\theta$-stable parabolic subalgebras of $\frak{g}$ of the 
form $\frak{q} \times \frak{q}$, where $\frak{q}$ is a parabolic subalgebra of $\frak{g}$ containing the Borel subalgebra $\frak{b}$ 
(see \S \ref{particular} again).
The parabolic subalgebras of $\frak{g}$ containing $\frak{b}$ are in one-one correspondence with the power set of $\Phi$. 
That is, $\frak{q}$ is of the form 
$\frak{q}_{\Phi'} = \frak{l}_{\Phi'} \oplus \frak{u}_{\Phi'}$, where 
$\frak{l}_{\Phi'} = \frak{h} \oplus \sum\limits_{\substack{n_\psi (\alpha) =0 \\ \forall \psi \in \Phi'}} \frak{g}_\alpha$,  
$\frak{u}_{\Phi'} = \sum\limits_{\substack{n_\psi (\alpha) >0 \\ \textrm{for some } \psi \in \Phi'}} \frak{g}_\alpha$; and 
$\alpha =\sum\limits_{\psi \in \Phi}  n_\psi (\alpha) \psi \in \Delta$, for some $\Phi' \subset \Phi$. 
Let $A_{\Phi'}$ be the irreducible 
unitary representation of $G$ with non-zero $(\frak{g} \times \frak{g}, U)$-cohomology corresponding to the 
$\theta$-stable parabolic subalgebra $\frak{q}_{\Phi'} \times \frak{q}_{\Phi'}$, where $U$ is the connected Lie subgroup of $G$ 
with Lie algebra $\frak{u}$. Then we have, \\ 
the Poincar\'{e} polynomial of $H^* (\frak{g} \times \frak{g}, U; A_{\Phi' , U})$ is given by 
\[ P( \Phi' , t ) = t^{\textrm{dim}(\frak{u}_{\Phi'})}(1 + t)^{|\Phi' |} P(\frak{l}_1 , t) P(\frak{l}_2, t) \cdots P(\frak{l}_k , t), \]
where $\frak{l}_1 , \frak{l}_2 , \ldots \frak{l}_k $ are the simple factors of the semisimple part $[\frak{l}_{\Phi'}, \frak{l}_{\Phi'}]$ of 
$\frak{l}_{\Phi'}$ and each $ P(\frak{l}_i , t)$ is given by the formula given below : 
 
  If $\frak{s}$ is a finite dimensional complex simple Lie algebra, the Poincar\'{e} polynomial $P(\frak{s}, t)$ is given by 
\[ P(\frak{s}, t) = (1 + t^{2d_1 + 1} )(1 + t ^{2d_2 +1 }) \cdots (1 + t^{2d_l + 1} ), \]
where $l = \textrm{rank}(\frak{s})$ and $d_1 , d_2 , \ldots d_l$ are the exponents of $\frak{s}$. 
We have deduced the formula for $P(\Phi' , t)$ from a more general result in \cite{knappvogan}. Also for $\Phi',\ \Phi'' \subset \Phi$, 
$A_{\Phi'}$ is unitarily equivalent to $A_{\Phi''}$ if and only if $\Phi' = \Phi''$. See \S \ref{particular} for details. 

Now combining these with Theorem \ref{th1}, we get 

\begin{theorem} \label{th2} 
Let $G$ be a connected complex simple Lie group. For each $i = 1, 2, 3$, there exists a uniform lattice 
$\Gamma \in \mathcal{L}_i (G)$ of $G$ such that $L^2 (\Gamma \backslash G)$ has an irreducible 
component $A_{\Phi'}$, where \\  
(i) $\Phi' = \{\psi_1\}$, or $\{ \psi_2 \}$, or $\{\psi_1 , \psi_2 \} \subset \Phi$, if $G$ is of $C_n$-type ($n \ge 6$). 
If $n=6,\ 8$, or $10$, we can discard $\{\psi_1\}$ among these. \\ 
(ii) $\Phi' = \{\psi_1\} \subset \Phi$, if $G$ is of $D_n$-type ($n >4$). \\
(iii) $\Phi' = \{\psi_1\}$ or $\{\psi_4\} \subset \Phi$, if $G$ is of $F_4$-type. 
\end{theorem} 

  The proof of the above theorem is given in \S \ref{pfth2}. In literature, there are non-vanishing results of the multiplicity 
of automorphic representations in $L^2(\Gamma \backslash G)$, for example see \cite{clozel}, \cite{anderson}, 
\cite{degeorge-wallach}, \cite{li}, \cite[Ch. VIII]{borel-wallach}, \cite[\S 6]{parthasarathy2}, \cite{mondal-sankaran2}. 
In all those cases, $G$ is an equi-rank group, that is rank of $G$ is equal to the rank of a maximal compact subgroup. 
But in our case, $G$ is complex, so can not be an equi-rank group. Schimpf \cite{schimpf} has identified 
some automorphic representation, when $G = SL(n, \mathbb{C})\ (n=2,\ 3)$. We also get Schimpf's result for $n=3$. 
See Remark \ref{conclusion}(i). The problem in identifying automorphic representation using this technique is that 
if a geometric cycle gives non-zero cohomology class in $H^k (\Gamma \backslash X ; \mathbb{C})$, then most of the times 
it happens that there are more than one $A_\frak{q}$ with $ H^k(\frak{g}^\mathbb{C}, K; A_{\frak{q}, K}) \neq 0$. 
Theorem 4.1 in \cite{mondal1}, or Theorem 1.2 in \cite{kobpreprint} might be a way to solve this problem. 
See Remark \ref{conclusion}(ii). 

\noindent
\section{Cartan involution of real semisimple Lie algebra with complex structure}\label{cartancomplex} 

  Let $\frak{g}$ be a complex semisimple Lie algebra and $\frak{u}$ be a compact real form of $\frak{g}$.
Let $\frak{g}^\mathbb{R}$ denote the Lie algebra $\frak{g}$ considered as a real Lie algebra and let $J$ denote the complex structure 
of $\frak{g}^\mathbb{R}$ corresponding to the multiplication by $i$ of $\frak{g}$. Then 
$\frak{g}^\mathbb{R} = \frak{u} \oplus J\frak{u}$ is a Cartan decomposition of $\frak{g}^\mathbb{R}$ with the corresponding 
Cartan involution $\theta$ (say). The complex linear extension of $\theta$ to the complexification $(\frak{g}^\mathbb{R})^\mathbb{C}$ 
is denoted by the same notation $\theta$. 
 
  Let $s$ denote the the involution $(X, Y) \mapsto (Y, X)$ of the product algebra $\frak{l} = \frak{u} \times \frak{u}$. Then $(\frak{l}, s)$ 
is an orthogonal symmetric algebra of the compact type and $\frak{l} = \frak{u}_* + \frak{e}_*$ is the decomposition of $\frak{l}$ 
into eigenspaces of $s$, where $\frak{u}_* = \{ (X, X) : X \in \frak{u} \}$ and $\frak{e}_* = \{ (X, -X) : X \in \frak{u} \}$. Let 
$(\frak{l}^*, s^*)$ denote the dual of $(\frak{l}, s)$, where $\frak{l}^*$ is the subset $\frak{u}_* + i\frak{e}_*$ of the complexification 
$\frak{l}^\mathbb{C}$ of $\frak{l}$ and $s^*$ is the map $T + iX \mapsto T - iX ( T \in \frak{u}_*, X \in \frak{e}_*)$. Now $\frak{g}$ is 
 isomorphic to $\frak{l}^*$ (as a real Lie algebra) via the map $\phi : X + JY \mapsto (X, X) + i(Y, -Y)$, where $X, Y \in \frak{u}$. 
Also we have $\phi \circ \theta = s^* \circ \phi$. Hence the complexification $(\frak{g}^\mathbb{R})^\mathbb{C}$ is isomorphic to 
$(\frak{l}^*)^\mathbb{C} = \frak{l}^\mathbb{C} \cong \frak{g} \times \frak{g}$ in such a way that $\theta$ corresponds to the 
complex linear extension of $s$, that is $\theta$ corresponds to the map $(Z_1, Z_2) \mapsto (Z_2, Z_1)$  of $\frak{g} \times \frak{g}$.

\noindent
\section{Arithmetic uniform lattices of connected complex semisimple Lie group}\label{complexlattice} 

  Let $G$ be a connected semisimple Lie group. The natural way to construct arithmetic uniform discrete subgroups 
of $G$ is Weil's restriction of scalars, which is described below:

  Let $F$ be an algebraic number field of degree $>1$ and $G'$ be a linear connected semisimple Lie group defined over 
$F$ such that $G$ is isogenous with $G'$. Then It is sufficient to consider arithmetic uniform discrete subgroups of $G'$. 
Let $S$ be the set of all infinite places of $F$. For each $s \in S$, 
define $F_s = \mathbb{R}$, if $s(F) \subset \mathbb{R}$; and $F_s = \mathbb{C}$, if $s(F) \not \subset \mathbb{R}$.
We can identify $G'$ with a subgroup of $SL(N, F_{\textrm{id}})$ defined over $F$ that is, there exists a finite subset 
$P$ of $F[x_{11}, \ldots , x_{NN}]$ such that 
$G'$ is the identity component of the group $\{ g \in SL(N, F_{\textrm{id}}) : p(g) = 0 \textrm{ for all } p \in P \}$. 
For each $s \in S$, let $G'^s$ be the identity component of the group $\{ g \in SL(N, F_s) : s(p)(g) = 0 \textrm{ for all } p \in P \}$. 
Let $\mathcal{O}$ be the ring of integers of $F$, and 
$G'_{\mathcal{O}} = G' \cap \textrm{GL}_N (\mathcal{O})$. Then 
$G'_{\mathcal{O}}$ is an arithmetic uniform lattice of $G'$ if $G'^s$ is compact for all $s \in S \setminus \{\textrm 
{id}\}$. 

  We shall follow the construction of Borel \cite{Borel} to construct some arithmetic uniform lattices in a connected 
complex semisimple Lie group. Let $G$ be a connected 
complex semisimple Lie group with Lie algebra $\frak{g}$. As before, let $\frak{g}^\mathbb{R}$ denote the Lie algebra $\frak{g}$ considered 
as a real Lie algebra and let $J$ denote the complex structure of $\frak{g}^\mathbb{R}$ corresponding to the multiplication 
by $i$ of $\frak{g}$. 
Note that $G /Z \cong \textrm{Ad}(G)$, where $Z$ denotes the centre of $G$. As $G$ is a connected complex semisimple Lie group, 
$Z$ is finite. So it is sufficient to determine arithmetic uniform lattices of $\textrm{Ad}(G)$, which is the identity component of 
$\textrm{Aut}(\frak{g}^\mathbb{R})$. As the Lie group $\textrm{Aut}(\frak{g}^\mathbb{R})$ has finitely many components, 
it is sufficient to determine uniform arithmetic lattices of $\textrm{Aut}(\frak{g}^\mathbb{R})$. We shall construct three types of 
arithmetic uniform lattices in $\textrm{Aut}(\frak{g}^\mathbb{R})$. But before 
proceeding further, we need some facts about algebraic number fields. 

\noindent 
\subsection{Algebraic number fields}\label{snf}  
Let $F$ be an algebraic number field and $S$ be the set of all real places of $F$. 
By the Theorem of the Primitive Element, we may write $F = \mathbb{Q}(u)$ for some $u \in F$. 

\begin{proposition}\label{nf}  
For any $k,l \in \mathbb{N} \cup \{0\}$ with $k + l = |S|$, We may choose a primitive element $u \in F$ such that 
the number of positive real conjugates of $u$ is $k$ and the number of negative real conjugates of $u$ is $l$. 
\end{proposition} 

\begin{proof} 
Let $S = \{s_1, s_2, \ldots , s_n \}$, where $n=|S|$. Let $u_i = s_i(u)$ for all $1 \le i \le n$. Assume that 
$u_1 < u_2 < \cdots < u_l < u_{l+1} < \cdots < u_{l+k}$ (here $k+l = n$). Choose $r \in \mathbb{Q}$ such that 
$u_l < r < u_{l+1}$. Then $u_i - r < 0$ for $1 \le i \le l$, and $u_{l+j} - r>0$ for $1 \le j \le k$. Clearly 
$F = \mathbb{Q}(u) = \mathbb{Q}(u-r)$, and $u_1 - r, u_2 - r, \ldots , u_l - r, u_{l+1} - r, \ldots , u_{l+k} - r$ are all real 
conjugates of $u$. So this $u-r$ is a primitive element with the required property. 
\end{proof} 

\begin{remark}\label{nfr}
{\em 
(i) If $F$ is a totally real number field, $F$ has a primitive element $u$ such that $u$ has exactly one positive conjugate 
(by Proposition \ref{nf}). Replacing $F$ by a conjugate of $F$ (if necessary), we may assume that $F = \mathbb{Q}(u)$ with 
$u > 0$ and $s(u) < 0$ for all $s \in S - \{\textrm{id} \}$, $S$ is the set of all infinite places of $F$.  

(ii) Let $F$ be an algebraic number field such that $F \not \subset \mathbb{R}$ and all other conjugates of $F$ are real. 
Then again by Proposition \ref{nf}, we may write $F = \mathbb{Q}(u)$, where $u \in \mathbb{C}$ with $s(u) < 0$ for all 
$s \in S - \{\textrm{id} \}$, $S$ is the set of all infinite places of $F$. 
} 
\end{remark} 

\noindent 
{\bf Examples :} 1. If $m$ is a positive square-free integer, the quadratic number field $\mathbb{Q}(\sqrt{m})$ is a totally real number field. 
More generally, if $f \in \mathbb{Q}[x]$ is irreducible and all roots of $f$ are real, then $\mathbb{Q}(\alpha)$ is a 
totally real number field, where $\alpha$ is a root of $f$. 

2. Let $h \in \mathbb{Q}[x]$ be an irreducible polynomial such that $h$ has exactly two non-real roots. 
For each $n \in \mathbb{N}$ with $n \ge 2$, 
there exists such a polynomial of degree $n$. For example, start with $f(x) = (x^2 + k) (x-k_1)\cdots (x-k_{n-2})$, where 
$k, k_1, \cdots , k_{n-2}$ are positive even integers and $k_1, k_2, \ldots , k_{n-2}$ are distinct. 
Let $x_1, x_2, \ldots , x_m (n-3 \le m \le n-1)$ be the real roots of $f'(x) = 0$. Since the real roots of $f$ are all distinct, $f(x_i) \neq 0$ 
for all $1 \le i \le m$. Let $\epsilon =$ min$\{|f(x_i)| : 1 \le i \le m \}$. For any $a \in \mathbb{R}$ with $|a| < \epsilon$, let 
$g_a(x) = f(x) + a$. Then $g_a(x) = 0$ has exactly $n-2$ real roots. For if $f$ has a local optimum value above (respectively, below) 
the $x$-axis, the corresponding local optimum value of $g_a$ is above (respectively, below) the $x$-axis; and vice versa. 
Let $q$ be an odd integer such that $\frac{2}{q} < \epsilon$. Then 
$f(x) + \frac{2}{q}=0$ has exactly $n-2$ real roots. Hence if $h(x) = qf(x) + 2$, then $h(x) = 0$ also has exactly $n-2$ real roots. 
If $f(x) = x^n + a_{n-1} x^{n-1} + \cdots + a_1x +a_0$, then $a_0, a_1, \cdots , a_{n-1}$ are all even integers. Also 
$h(x) = qx^n +(qa_{n-1}) x^{n-1} + \cdots + (qa_1)x + (qa_0 +2 )$. So $h \in \mathbb{Z}[x]$ is irreducible, by Eisenstein's Criterion
(see \cite[Ch. 4]{jacobson}). 

  The algebraic number field $\mathbb{Q}(\alpha)$ has exactly one complex place, where $\alpha$ is a root of $h$.

\noindent 
\subsection{Construction of some arithmetic uniform lattices in $\textrm{Aut}(\frak{g}^\mathbb{R})$}\label{lattice}
  Let $\frak{u}$ be a compact real form of $\frak{g}$ and $\frak{h}$ be a Cartan subalgebra of $\frak{g}$ with $\frak{h} =( \frak{u} \cap 
\frak{h}) \oplus (J\frak{u} \cap \frak{h})$. Let $\Delta = \Delta (\frak{g}, \frak{h})$ be the set of all non-zero roots of $\frak{g}$ 
with respect to the Cartan subalgebra $\frak{h}$, $\Delta ^+$ be the set of positive roots in $\Delta$ with respect to some chosen 
ordering and $\Phi$ the set of all simple roots in $\Delta ^+$. Let $B$ denote the Killing form of 
$\frak{g}$. 
For each $\alpha \in \Delta$, there exists unique $H_\alpha \in \frak{h}$ such that 
\[ \alpha (H) = B (H, H_\alpha ) \textrm{ for all } H \in \frak{h} . \]
  Let $H_\alpha ^* = 2 H_\alpha /\alpha (H_\alpha)$ for all $\alpha \in \Delta$. For each $\alpha \in \Delta$ there exists 
$E_\alpha \in \frak{g}$ such that 
\begin{equation}\label{chevalley}
\begin{aligned}
\ \ \ \ \ \ \ \ \ \ \ \ \ [H, E_\alpha] & =  \alpha (H) E_\alpha   \textrm{ for all } H \in \frak{h}, \\
 [E_\alpha , E_{-\alpha} ]  & =  H_\alpha ^* \textrm{ for all } \alpha \in \Delta , \\
 [E_\alpha , E_\beta ] & = 0 \textrm{ if } \alpha , \beta \in \Delta, \alpha + \beta \not\in \Delta , \alpha + \beta \neq 0 , \\
 [E_\alpha , E_\beta ] & = N_{\alpha , \beta} E_{\alpha + \beta} \textrm{ if } \alpha , \beta ,\alpha + \beta \in \Delta , \textrm{ where } \\
  N_{\alpha , \beta} & = -  N_{-\alpha , -\beta} = \pm(1-p), 
\end{aligned}
\end{equation}
and $\beta + n\alpha (p \le n \le q)$ is the $\alpha$-series containing $\beta$. Also we can choose $E_\alpha \ (\alpha \in \Delta)$ 
in such a way that 
\[ E_\alpha - E_{-\alpha}, \  i( E_\alpha + E_{-\alpha}) \in \frak{u} \textrm{ for all } \alpha \in \Delta .\]
Then $\{ H_\phi ^* \  , E_\alpha : \phi \in \Phi , \alpha \in \Delta \}$ is a Chevalley basis of $\frak{g}$ such that 
\begin{equation}\label{ubasis} 
\frak{u} =  \sum_{\phi \in \Phi} \mathbb{R} (i  H_\phi ^*) \oplus \sum_{\alpha \in \Delta^+} \mathbb{R} (E_\alpha - E_{-\alpha}) \oplus 
\sum_{\alpha \in \Delta^+} \mathbb{R} (i(E_\alpha + E_{-\alpha})). 
\end{equation} 
Let $X_\alpha = E_\alpha - E_{-\alpha} , Y_\alpha = i(E_\alpha + E_{-\alpha})$ for all $\alpha \in \Delta^+$. 

  Let $F$ be an algebraic number field of degree $>1$, $\mathcal{O}$ be the ring of integers of $F$ and $S$ be the set of all infinite places of $F$. 
Assume that $s(F) \subset \mathbb{R}$ for all $s \in S \setminus \{\textrm{id}\}$ (see Examples in \S \ref{nf}). If $G$ is real, we assume that 
$F \subset \mathbb{R}$. If $G$ is complex, we assume that $F \not \subset \mathbb{R}$. In any case, we may write $F = \mathbb{Q}(u)$, 
where $s(u) < 0$ for all $s \in S \setminus \{\textrm{id}\}$ (Remark \ref{nfr}). If $F \subset \mathbb{R}$, then we may choose $u >0$. Otherwise $u \in 
\mathbb{C}$. Let $v = \sqrt{u}$ and $v_s = \sqrt{-s(u)}$ for all $s \in S \setminus \{\textrm{id}\}$. Note that $v_s > 0$ for all 
 $s\in S \setminus \{\textrm{id}\}$. Also if $u >0$, then $v>0$. 
 
  Now we shall construct some arithmetic uniform lattices of $\textrm{Aut}(\frak{g}^\mathbb{R})$ as follows : 

$\bf{1.}$  First view $\frak{g}$ as a real Lie algebra $\frak{g}^\mathbb{R}$. Let $F$ be an algebraic number field as above with 
$F \subset \mathbb{R}$. 
Recall that $\frak{g}$ is isomorphic to the non-compact real form $\frak{l}^*$ of $\frak{g} \times \frak{g}$ in such a way that the 
Cartan decomposition $\frak{g}^\mathbb{R} = \frak{u} \oplus J \frak{u}$ corresponds to the Cartan decomposition $\frak{l}^* = 
\frak{u}_* \oplus \frak{e}$ of $\frak{l}^*$, where $\frak{u}_* = \{ (X, X) : X \in \frak{u} \}$ and $\frak{e} = \{ (iX, -iX) : X \in \frak{u} \}$ 
(see \S \ref{cartancomplex}). 
 Then $\{ (iH_\phi ^*, iH_\phi ^*) , (X_\alpha , X_\alpha ), (Y_\alpha , Y_\alpha) : \phi \in \Phi , \alpha \in \Delta^+ \} \cup 
\{ H_\phi ^*, -H_\phi ^*) , (iX_\alpha , -iX_\alpha ), (iY_\alpha , -iY_\alpha) : \phi \in \Phi , \alpha \in \Delta^+ \}$ is a basis 
of $\frak{g}^\mathbb{R}$ (via this identification) consisting of vectors belonging to either $\frak{u}_*$ or to $\frak{e}$, 
with respect to which the structural constants are integers.  \\ 
\indent 
  Let $\frak{m}$ be the vector space over $F$ spanned by the set $\{ (iH_\phi ^*, iH_\phi ^*) , (X_\alpha , X_\alpha ), 
(Y_\alpha , Y_\alpha) : \phi \in \Phi , \alpha \in \Delta^+ \} \cup \{v H_\phi ^*, -v H_\phi ^*) , (iv X_\alpha , -iv X_\alpha ), 
(ivY_\alpha , -ivY_\alpha) : \phi \in \Phi , \alpha \in \Delta^+ \}$ 
and $\frak{m}^s$ be the vector space over $F^s=s(F)$ spanned by 
the set $\{ (iH_\phi ^*, iH_\phi ^*) , (X_\alpha , X_\alpha ), (Y_\alpha , Y_\alpha) : \phi \in \Phi , \alpha \in \Delta^+ \} \cup 
\{iv_s H_\phi ^*,- iv_s H_\phi ^*) , (-v_s X_\alpha , v_s X_\alpha ), (-v_s Y_\alpha , v_s Y_\alpha) : \phi \in \Phi , \alpha \in \Delta^+ \}$ 
for all $s \in S-\{\textrm{id}\}$. Then $\frak{m}$ is a Lie algebra over $F$, $\frak{m}^s$ is a Lie algebra over $F^s$,  and the structural constants of 
$\frak{m}^s$ are the conjugates by $s$ of the structural constants of $\frak{m}$ with respect to the given bases 
for all $s \in S-\{\textrm{id}\}$. Thus $\frak{m}^s$ is the conjugate of $\frak{m}$ by $s$. We also have 
$\frak{m} \otimes \mathbb{R} = \frak{g}^\mathbb{R} ,\  \frak{m}^s \otimes \mathbb{R} = \frak{u} \times \frak{u}$ for all 
$s \in S-\{\textrm{id} \}$. \\
\indent
  Take a basis of $\frak{g}^\mathbb{R}$ contained in $\frak{m}$ and identify $\textrm{Aut}((\frak{g}^\mathbb{R})^\mathbb{C})$ with 
an algebraic subgroup $G'$ of $ GL (n, \mathbb{C}) \  (n = \textrm{dim} (\frak{g}^\mathbb{R}))$ defined over $F$, via this basis. Then 
$\textrm{Aut}(\frak{g}^\mathbb{R})$ is identified with $G'_\mathbb{R}$, the group of real matrices in $G'$. The group 
$(G'_\mathbb{R})^s$ is then $\textrm{Aut}(\frak{u} \times \frak{u})$, hence compact,  for all $s \in S-\{\textrm{id} \}$. Let $\mathcal{O}$ 
 be the ring of algebraic integers of $F$ and $\Gamma = G'_{\mathcal{O}} = G' \cap GL(n, \mathcal{O})$. As $(G'_\mathbb{R})^s$ is compact 
 for all $s \in S-\{\textrm{id} \}$, $\Gamma$ is a cocompact arithmetic lattice in $\textrm{Aut}(\frak{g}^\mathbb{R})$. 
An arithmetic uniform lattice of 
$\textrm{Aut}(\frak{g}^\mathbb{R})$ constructed in this way, is called a lattice of type $1$. 

$\bf{2.}$ Now view $\frak{g}$ as a complex Lie algebra, and $F$ be an algebraic number field $F \not\subset \mathbb{R}$ with 
$s(F) \subset \mathbb{R}$ for all $s \in S \setminus \{\textrm{id}\}$.  \\ 
(i) Let $B = \{iH_\phi ^*,  X_\alpha , Y_\alpha : \phi \in \Phi , \alpha \in \Delta^+ \}$.  Let $\frak{m}$ be the vector space over 
$F$ spanned by the set $B$ and $\frak{m}^s$ be the vector space over $F^s$ spanned by the set $B$ for all $s \in S-\{\textrm{id}\}$. 
Then $\frak{m}$ is a Lie algebra over $F$, $\frak{m}^s$ is a Lie algebra over $F^s$,  and the structural constants of $\frak{m}$ and 
$\frak{m}^s$ are integers with respect to the basis $B$ for all $s \in S-\{\textrm{id}\}$. Thus $\frak{m}^s$ is the conjugate of 
$\frak{m}$ by $s$. We also have $\frak{m} \otimes \mathbb{C} = \frak{g} ,\  \frak{m}^s \otimes \mathbb{R} = \frak{u}$ for all 
$s \in S-\{\textrm{id} \}$. Here note that the real span of $B$ is the compact real form of $\frak{g}$. \\
(ii) Let $\frak{g}_0$ be a non-compact real form of $\frak{g}$ and $\frak{g}_0 = \frak{k}_0 \oplus \frak{p}_0$ be a Cartan decomposition 
of $\frak{g}_0$ such that $\frak{u}= \frak{k}_0 \oplus i \frak{p}_0$. Let $\{ e_\lambda \}$ be a basis of 
$\frak{g}_0$ consisting of vectors belonging either to $\frak{k}_0$ or to $\frak{p}_0$, 
with respect to which the structural constants are all rational numbers \cite[Prop. 3.7]{Borel}. Let $k$ and $p$ stand for indices 
of the subbases for $\frak{k}_0$ and $\frak{p}_0$ respectively. \\
\indent
  Let $\frak{m}$ be the vector space over $F$ spanned by the elements 
$e_k$ and $v e_p$, and $\frak{m}^s$ be the vector space over $F^s$ spanned by the elements $e_k$ and $i v_s e_p$, 
for all $s \in S-\{\textrm{id}\}$. Then $\frak{m}$ is a Lie algebra over $F$, $\frak{m}^s$ is a Lie algebra over $F^s$, $\frak{m}^s$ is 
the conjugate of $\frak{m}$ by $s$ and $\frak{m} \otimes \mathbb{C} = \frak{g} ,\  \frak{m}^s \otimes \mathbb{R} = \frak{u}$ 
for all $s \in S-\{\textrm{id} \}$. 
Let $B'$ be the set consisting of vectors $e_k$ and $v e_p$.  \\  
\indent
Identify $\textrm{Aut}(\frak{g})$ with an algebraic subgroup $G'$  of $ GL (n, \mathbb{C}) \  
(n = \textrm{dim}_{\mathbb{C}}(\frak{g}))$ defined over $F$, 
via the basis $B$ (respectively, $B'$) in case (i) (respectively case (ii)). The group $(G')^s$ is then $\textrm{Aut}(\frak{u})$, 
hence compact,  for all $s \in S-\{\textrm{id} \}$. Let $\mathcal{O}$ be the ring of algebraic integers of $F$ and 
$\Gamma = G'_{\mathcal{O}} = G' \cap GL(n, \mathcal{O})$. As $(G')^s$ is compact for all $s \in S-\{\textrm{id} \}$, $\Gamma$ is a 
cocompact arithmetic lattice in $\textrm{Aut}(\frak{g})$. In case (i), $G'_\mathbb{R}$ is $\textrm{Aut}(\frak{u})$, which is compact. 
And in case (ii), $G'_\mathbb{R}$ is $\textrm{Aut}(\frak{g}_0)$, which is non-compact.  An arithmetic uniform lattice of 
$\textrm{Aut}(\frak{g})$ constructed as in 2.(i), is called a lattice of type $2$; and 
an arithmetic uniform lattice of $\textrm{Aut}(\frak{g})$ constructed in 2.(ii), is called a lattice of type $3$. 
 
   Note that $\frak{g}^\mathbb{R} = \frak{u} \oplus J\frak{u}$ is a Cartan decomposition of 
$\frak{g}^\mathbb{R}$. Let $\theta$ be the corresponding Cartan involution. Let $\Gamma$ be a 
cocompact arithmetic lattice of $\textrm{Aut}(\frak{g}^\mathbb{R})$ constructed as in 1 or 2. Note that, \\
(i) if $\Gamma$ is as in 1, then $\theta \in \Gamma$; and \\
(ii) if $\Gamma$ is as in 2, then $\theta \not \in \Gamma$, as $\theta \not \in 
\textrm{Aut}(\frak{g})$. \\
But $\theta \Gamma {\theta}^{-1} = \Gamma$, in both 1 and 2. \\ 
Also $\frak{g} \cong \textrm{ad}(\frak{g})$ and the real Lie algebra isomorphism of 
$\textrm{ad}(\frak{g})$ corresponding to the Cartan involution $\theta$ of 
$\frak{g}^\mathbb{R}$ is given by $\textrm{ad}(X) \mapsto \textrm{ad}(\theta X) = \theta \textrm{ad}(X)
{\theta}^{-1}$, which is denoted by the same notation $\theta$. Then $\theta$ is the differential at 
identity of the Lie group isomorphism $\tilde{\theta}$ of $\textrm{Aut}(\frak{g}^\mathbb{R})$ 
given by $\tilde{\theta}(\sigma )=\theta \sigma {\theta}^{-1}$. The Lie group isomorphism of $G$ whose differential 
at identity is $\theta$, is also denoted by the same notation $\theta$. Then we have $\textrm{Ad} 
\circ \theta = \tilde{\theta} \circ \textrm{Ad}$. So if $\Gamma$ is a cocompact arithmetic lattice of 
$\textrm{Aut}(\frak{g}^\mathbb{R})$ constructed as in 1 or 2, then $\theta (\textrm{Ad}^{-1}(\Gamma)) 
= \textrm{Ad}^{-1}(\Gamma)$.

\noindent
\section{Special cycles in Riemannian globally symmetric space of type IV} 

  Let $G$ be a real semisimple Lie group with finite centre and $K$ be a maximal compact subgroup of $G$. Let 
$\Gamma$ be a torsion-free uniform discrete subgroup of $G$. Then $\Gamma$ acts freely on the 
Riemannian globally symmetric space $X := G \slash K$ and the canonical projection $\pi : X \longrightarrow \Gamma \backslash X$ 
is a covering projection. One can identify the group cohomology $H^* (\Gamma ; \mathbb{C})$ with the cohomology 
$H^* (\Gamma \backslash X ; \mathbb{C})$ of the locally symmetric space $\Gamma \backslash X$. 

   Let $B$ be a reductive subgroup of $G$ such that $K_B = B \cap K$ is a maximal compact subgroup of $B$. 
 Set $X_B = B \slash K_B$ and $\Gamma_B = B \cap \Gamma$. Note that $X_B$ is a connected totally geodesic submanifold of $X$. 
Assume that 
the natural map $j : \Gamma_B \backslash X_B \longrightarrow  \Gamma \backslash X$ is an embedding. Then the image 
$C_B := j(\Gamma_B \backslash X_B)$ is called a geometric cycle. Under certain conditions, the fundamental class $[C_B] \in 
H_d (\Gamma \backslash X ; \mathbb{C})\  (d = \textrm{dim}(\Gamma_B \backslash X_B))$ is non-trivial. So the Poincar\'e dual of 
$[C_B]$ contributes nontrivially to $H^* (\Gamma \backslash X ; \mathbb{C})$. 

  If the reductive subgroup $B$ is the fixed point set of a finite order automorphism $\mu$ of $G$ such that $\mu (K) = K$ and 
$\mu (\Gamma) = \Gamma $, then we denote $B$ by $G(\mu)$, $K_B$ by $K(\mu)$, $X_B$ by $X(\mu)$ and $\Gamma_B$ by 
$\Gamma(\mu)$. In this case, the  natural map $j : \Gamma (\mu) \backslash X (\mu) \longrightarrow \Gamma \backslash X$ is 
an embedding and the image $C (\mu, \Gamma) := j(\Gamma (\mu) \backslash X (\mu))$ is called a special cycle. 

  Let $X_u$ denote the compact dual of $X$. We can identify the cohomology $H^* (X_u ; \mathbb{C})$ of $X_u$ with the cohomology 
$H^* (\Omega (X ; \mathbb{C})^G)$ of the complex $\Omega (X ; \mathbb{C})^G$ of $G$-invariant complex valued differential forms 
on $X$. Since $\Gamma$ is a cocompact discrete subgroup of $G$, the inclusion $j_\Gamma : \Omega (X ; \mathbb{C})^G \hookrightarrow 
\Omega (X ; \mathbb{C})^\Gamma $ induces an injective map $j_\Gamma ^ * : H^*(\Omega (X ; \mathbb{C})^G) \hookrightarrow 
H^*(\Omega (X ; \mathbb{C})^\Gamma)$ (the so called Matsushima map). Now we can identify the cohomology 
$H^* (\Gamma \backslash X ; \mathbb{C})$ of $\Gamma \backslash X$ with the cohomology $H^* (\Omega (X ; \mathbb{C})^\Gamma)$ 
of the complex $\Omega (X ; \mathbb{C})^\Gamma$. In this way we have an injective map $k_\Gamma : H^* (X_u ; \mathbb{C}) 
\longrightarrow H^* (\Gamma \backslash X ; \mathbb{C})$. So the elements in the image $k_\Gamma ( H^* (X_u ; \mathbb{C}))$ 
are represented by the $G$-invariant differential forms on $X$. 

  The following results state some conditions under which fundamental class of a special cycle is non-zero and the corresponding 
cohomology class does not lie in the image of the Matsushima map that is, it is not represented by a $G$-invariant differential form on $X$.

\begin{theorem}\label{mira}(Th. 2.1, \cite{mira})

Let $F$ be an algebraic number field of degree $>1$ with ring of integers $\mathcal{O}$. 
Let $G$ be a linear connected semisimple Lie group defined over $F$,  
$\theta$ be a Cartan involution of $G$ defined over $F$ and $K = \{g \in G : \theta (g) = g \}$. Let $\sigma$ be an 
involutive automorphism of $G$ defined over $F$ with $\sigma \theta = \theta \sigma$ and $\Gamma \subset G_{\mathcal{O}}$ be 
a torsion-free, $\langle \sigma , \theta \rangle$-stable, arithmetic uniform lattice of $G$ such that  
the Lie groups $G, G(\sigma), G(\sigma \theta)$ act orientation preservingly on $X, X(\sigma)$ and $X(\sigma \theta)$ respectively. \\
Then there exists a $\langle \sigma , \theta \rangle$-stable subgroup $\Gamma'$ of $\Gamma$ of finite index such that the cohomology 
classes defined by $[C(\sigma , \Gamma')], [C(\sigma \theta, \Gamma')]$ via Poincar\'e duality are non-zero and are not represented 
by $G$-invariant differential forms on $X$. 

\end{theorem}

\begin{theorem}\label{rosch}(Th. 4.11, \cite{rs}) 

Let $F, \mathcal{O}, G, \theta, K$ be as in the above theorem. Let $\sigma$ and $\tau$ be finite order automorphisms of $G$ 
defined over $F$ with $\sigma \theta = \theta \sigma , \tau \theta = \theta \tau$ and $\sigma \tau = \tau \sigma$. 
Let $\Gamma \subset G_{\mathcal{O}}$ be a torsion-free, 
$\langle \sigma , \tau \rangle$-stable, arithmetic uniform lattice of $G$ such that $\Gamma \backslash X , C(\sigma, \Gamma), 
C(\tau, \Gamma)$ and all connected components of their intersection are orientable. Assume that \\ 
(i) $\textrm{dim}(C(\sigma, \Gamma)) + \textrm{dim}(C(\tau, \Gamma)) = \textrm{dim}(\Gamma \backslash X)$, \\
(ii) the Lie groups $G, G(\sigma), G(\tau)$ act orientation preservingly on $X, X(\sigma)$ and $X(\tau)$ respectively, and \\
(iii) the group $G(\langle \sigma , \tau \rangle )$ is compact. \\
Then there exists a $\langle \sigma , \tau \rangle$-stable normal subgroup $\Gamma''$ of $\Gamma$ of finite index such that \\ 
$[C(\sigma , \Gamma'')][C(\tau, \Gamma'')] \neq 0$.  

\end{theorem} 

\begin{remark}\label{nonvanishing}
{\em 
  (i) If $\sigma$ is an involution with $\sigma\theta = \theta\sigma$, and 
$\tau = \sigma \theta$, then obviously  
$\textrm{dim}(C(\sigma, \Gamma)) + \textrm{dim}(C(\tau, \Gamma)) = \textrm{dim}(\Gamma \backslash X)$, 
and the group $G(\langle \sigma , \tau \rangle )$ is a closed subgroup of $K$, hence compact. Also in this case, the cycles 
$C(\sigma, \Gamma), C(\sigma \theta, \Gamma)$ intersect transversely, and so the connected components of their intersection are points. Hence if 
the Lie groups $G(\sigma), G(\sigma \theta)$ act orientation preservingly on $X(\sigma)$ and $X(\sigma \theta)$ respectively, then 
in particular, $C(\sigma, \Gamma), C(\sigma \theta, \Gamma)$ and all connected components of their intersection are orientable.

  (ii) Originally, Th. 2.1 in \cite{mira} has been stated under the assumption that $C(\sigma, \Gamma)$, $C(\sigma \theta, \Gamma)$ are orientable, 
and all intersections of $C(\sigma, \Gamma), C(\sigma\theta, \Gamma)$ 
are of positive multiplicity. Now the assumption in Th. \ref{mira} implies that there is a $\langle \sigma , \theta \rangle$-stable 
subgroup $\Gamma''$ of $\Gamma$ of finite index such that $[C(\sigma , \Gamma'')][C(\tau, \Gamma'')] \neq 0$, by Th. \ref{rosch}. Now  
the Th. \ref{mira} follows from the proof of Th. 2.1 in \cite{mira}. 

  (iii) If $G$ is a connected complex semisimple Lie group, then since the simply connected cover of $G$ 
is a linear Lie group, without loss of generality we may assume that 
$\Gamma \backslash X , C(\sigma, \Gamma), C(\sigma \theta, \Gamma)$ and 
all connected components of their intersection are orientable \cite{mira}[Prop. 2.3 and its Cor.]. In general, Rohlfs and Schwermer 
\cite{rs} proved that by passing to a suitable subgroup of finite index in $\Gamma$ if necessary, we may assume that 
$\Gamma \backslash X , C(\sigma, \Gamma), C(\sigma \theta, \Gamma)$ and all connected components of their intersection are orientable. 

  (iv) We say that the {\it condition Or} (as in \cite{rs}) is satisfied for $G,\ \sigma,\ \tau$ if the canonical action of 
$G(\mu)$ on $X(\mu)$ is orientation preserving for $\mu = \sigma , \tau$. 
}
\end{remark}

  Now the hypotheses of the above theorems have been checked in the following subsections for a connected complex simple Lie group $G$ 
so that we can apply the above theorems.

\noindent
\subsection{Automorphisms of finite order of a complex simple Lie algebra}\label{victor}
  
  Here we describe Victor Ka\v{c}'s Classification \cite{kac} of finite order automorphisms of a complex simple Lie algebra. We follow 
\cite[\S 5, Ch. X]{helgason} for this purpose. 
 
  Let $\frak{g}$ be a complex simple Lie algebra and $\frak{u}$ be a compact real form of $\frak{g}$.  As before, $\frak{g}^\mathbb{R} = 
\frak{u} \oplus J\frak{u}$ is a Cartan decomposition of $\frak{g}^\mathbb{R}$, where $\frak{g}^\mathbb{R}$ is the underlying real 
Lie algebra of $\frak{g}$ and $J$ is the complex structure of $\frak{g}^\mathbb{R}$ corresponding to the multiplication by $i$ of $\frak{g}$. 
Let $\theta$ be the corresponding Cartan involution. Let $\frak{h}$ be a $\theta$-stable Cartan subalgebra of $\frak{g}$. 
Choose a system of positive roots $\Delta^+$ in the set of all non-zero roots 
$\Delta = \Delta(\frak{g} , \frak{h})$. Let $\Phi$ be the set of all simple roots in $\Delta^+$. Let 
$\{ H_\phi ^* \  , E_\alpha : \phi \in \Phi , \alpha \in \Delta \}$ be a Chevalley basis for $\frak{g}$ as in  \eqref{chevalley}. Then the Lie algebra 
$\frak{g}$  is generated by the vectors $ H_\phi ^*,\  E_\phi ,\  E_{-\phi} \  (\phi \in \Phi)$. 

  For any $\sigma \in \textrm{Aut}(\frak{g})$ 
with $\sigma (\frak{h}) = \frak{h}$, define $\sigma (\alpha) (H) = \alpha (\sigma H)$ for all $H \in \frak{h}$, 
where $\alpha \in \frak{h}^*$. Then $\sigma (\Delta) = \Delta $. 
Now assume that $\sigma$ is a finite order automorphism of $\frak{g}$ with $\sigma \theta = \theta \sigma $, 
$\sigma (\frak{h}) = \frak{h}$ and $\sigma (\Delta^+ ) = 
\Delta^+$. Then $\sigma$ induces an automorphism of the Dynkin diagram of $\frak{g}$. As the order of any automorphism of a Dynkin 
diagram is $1, 2, $ or $3$; $\sigma |_{\frak{h}}$ has order $1,2,$ or $3$ respectively.  

  Conversely,  let $\bar{\nu}$ be an automorphism of the Dynkin diagram of $\frak{g}$ of order $k \ (k = 1, 2,\textrm{or } 3)$. As $\frak{g}$ 
is generated by $ H_\phi ^*,\  E_\phi ,\  E_{-\phi} \  (\phi \in \Phi)$, there exists a unique $\nu \in \textrm{Aut}(\frak{g})$ with 
\[ \nu ( H_\phi ^*) =  H_{\bar{\nu} (\phi)} ^* , \ \ \  \nu (E_\phi) = E_{\bar{\nu} (\phi)} , \ \ \  \nu (E_{-\phi}) = E_{-\bar{\nu} (\phi)} \ \ \ 
(\phi \in \Phi). \]
Note that $\nu$ is of order $k$, and $\nu\theta = \theta\nu$. We call $\nu$, 
an automorphism of $\frak{g}$ induced by an automorphism of the Dynkin diagram 
of $\frak{g}$. Let $\epsilon_0 = e^{\frac{2\pi i}{k}}$ be a primitive $k$-th root of unity. 
As $\nu$ has order $k$, any eigenvalue of $\nu$ has the form 
$\epsilon_0^i \ (i \in \mathbb{Z}_k )$  and $\frak{g} = \mathop{\oplus}\limits_{i \in \mathbb{Z}_k} \frak{g}_i ^\nu $ such that 
$[\frak{g}_i ^\nu,  \frak{g}_j ^\nu ] \subset \frak{g}_{i + j} ^\nu $, where 
$ \frak{g}_i ^\nu $ is the eigenspace of $\nu$ corresponding to the eigenvalue $\epsilon_0^i$. Since $k= 1,2,$ or $3$, 
$\frak{g}_{0}^\nu,\ \frak{g}_{\bar{1}}^\nu,\ \frak{g}_{\bar{2}}^\nu \neq 0$, 
where $\bar{a} = a + k \mathbb{Z} \in \mathbb{Z}_k$ for all $a \in \mathbb{Z}$.  
The Lie algebra $\frak{g}_0^\nu$ is 
reductive (in fact, it is simple \cite[the proof of Lemma 5.11, Ch. X]{helgason}) and $\frak{h}^\nu = \frak{h} \cap \frak{g}_0^\nu$ 
is a Cartan subalgebra of $\frak{g}_0^\nu$. Define a root of $\frak{g}$ with respect to $\frak{h}^\nu$ as a pair 
$(\alpha , i)\ (\alpha \in (\frak{h}^\nu)^* , i \in \mathbb{Z}_k )$, if the joint eigenspace $\frak{g}_{(\alpha , i)} = 
\{X \in \frak{g}_i ^\nu : [ H , X ] = \alpha (H) X \textrm{ for all } H \in \frak{h}^\nu \} \neq 0$. Note that a root of $\frak{g}$ 
with respect to $\frak{h}^\nu$ is just a weight of the $\frak{g}_0^\nu$-module $\frak{g}_i ^\nu$. 
We may add pairs by $(\alpha , i) + (\beta , j) = (\alpha +\beta , i + j )$. Let $\bar{\Delta}$ denote the set of all non-zero roots 
and $\bar{\Delta}_0$ the set of roots of the form $(0 , i) , i \in \mathbb{Z}_k$. Then we have 
\begin{equation}\label{1}
\frak{g} = \frak{h}^\nu \oplus \sum\limits_{(\alpha , i) \in \bar{\Delta}} \frak{g}_{(\alpha , i)}, \ \frak{h} = \sum\limits_{(\alpha , i) 
\in \bar{\Delta}_0} \frak{g}_{(\alpha , i)}, \  \frak{h}^\nu = \frak{g}_{(0,0)},
\end{equation}
\begin{equation}\label{2}
[ \frak{g}_{(\alpha , i)} , \frak{g}_{(\beta , j)}] \subset \frak{g}_{(\alpha , i) + (\beta , j)}, 
\end{equation}
\begin{equation}\label{3} 
\textrm{dim } \frak{g}_{(\alpha , i)} = 1 \textrm{ for all } (\alpha , i) \in \bar{\Delta} \setminus \bar{\Delta}_0, 
\end{equation} 
\begin{equation}\label{4}
[ \frak{g}_{(\alpha , i)} , \frak{g}_{(\beta , j)}] \neq 0 , \textrm{ if } (\alpha , i)\in  \bar{\Delta} \setminus \bar{\Delta}_0 \  ; (\beta , j), 
(\alpha , i ) + (\beta , j ) \in \bar{\Delta}, 
\end{equation}
\cite[\S 5, Ch. X]{helgason}. 

  Let $\Delta_0 = \Delta (\frak{g}_0^\nu , \frak{h}^\nu)$ be the set of all non-zero roots of the 
simple Lie algebra $\frak{g}_0^\nu$ with respect to the Cartan subalgebra $\frak{h}^\nu$. Then 
$\Delta_0 = \{ (\alpha, 0) \in \bar{\Delta} : \alpha \neq 0 \}$. Define 
\[ \bar{H}_\phi^* = \sum\limits_{i = 0}^{k-1}  H_{\bar{\nu}^i (\phi)}^* , \ 
\bar{E}_\phi = \sum\limits_{i =0} ^ {k-1} E_{\bar{\nu}^i (\phi)}, \  
\textrm{and} \ \bar{E}_{-\phi} = \sum\limits_{i = 0}^{k-1} E_{-\bar{\nu}^i (\phi)}\  (\phi \in \Phi).\]
Note that $\frak{h}^\nu = \sum\limits_{\phi \in \Phi } \mathbb{C} \bar{H}_\phi ^*$, and the vectors $\bar{E}_\phi, \ 
\bar{E}_{-\phi} \in \frak{g}_0^\nu$ for all $\phi \in \Phi$. Also the vectors $\sum\limits_{i=0}^{k-1} 
\epsilon_0 ^ i E_{\bar{\nu}^i (\phi)} , \ \sum\limits_{i=0}^{k-1} \epsilon_0^i E_{-\bar{\nu}^i (\phi)}
\in \frak{g}^\nu _{\overline{k-1}}$, and for $k=3$, the vectors $E_\phi + \epsilon_0 ^2 E_{\bar{\nu}(\phi)} + \epsilon_0 
E_{\bar{\nu}^2 (\phi)} , \ E_{-\phi} + \epsilon_0^2 E_{-\bar{\nu}(\phi)} + \epsilon_0 E_{-\bar{\nu}^2 (\phi)} \in 
\frak{g}^\nu _{\bar{1}}$ for all $\phi \in \Phi$ with $\phi \neq \bar{\nu}(\phi)$. 
Let $a_{\phi \psi} = \phi (H_\psi ^*) \textrm{ for all } \phi , \psi \in \Phi$. Then we have 
\[  [\bar{H}_\psi ^* , \bar{E}_\phi ] = \sum\limits_{i,j = 0}^{k-1} [ H_{\bar{\nu}^i (\psi)}^* ,  E_{\bar{\nu}^j (\phi)} ] 
= \sum\limits_{i,j = 0}^{k-1} a_{\bar{\nu}^j (\phi) \bar{\nu}^i (\psi)} E_{\bar{\nu}^j (\phi)} 
=\sum\limits_{j =0}^{k-1} (\sum\limits_{i=0}^{k-1} a_{\bar{\nu}^j (\phi) \bar{\nu}^i (\psi)})  E_{\bar{\nu}^j (\phi)} \]
\[=\sum\limits_{j =0}^{k-1} (\sum\limits_{i=0}^{k-1} a_{\phi \bar{\nu}^{i-j} (\psi)}) E_{\bar{\nu}^j (\phi)}\  (\textrm{as } 
a_{\phi \psi} = a_{\bar{\nu} (\phi) \bar{\nu} (\psi)}) 
=(\sum\limits_{i=0}^{k-1} a_{\phi \bar{\nu}^i(\psi)})\sum\limits_{j =0}^{k-1} E_{\bar{\nu}^j (\phi)}\ (\textrm{as } 
\bar{\nu}^k = \textrm{id})\]
\[=(\sum\limits_{i=0}^{k-1} a_{\phi \bar{\nu}^i(\psi)}) \bar{E}_\phi .\ \ \ \ \ \ \ \ \ \ \ \ \ \ \ \ \ \ \ \ \ \ \ \ \ \ \ \ \ \ \ \ \ \ \ \ \ \ \ \ \ \ \ \ \ \ 
\ \ \ \ \ \ \ \ \ \ \ \ \ \ \ \ \ \ \ \ \ \ \ \ \ \ \ \ \ \ \ \ \ \ \ \ \ \ \] 
Similarly 
\[  [\bar{H}_\psi ^* , \bar{E}_{-\phi} ] = -(\sum\limits_{i=0}^{k-1} a_{\phi \bar{\nu}^i(\psi)}) \bar{E}_{-\phi} ,\]
for all $\phi, \psi \in \Phi$. Thus $\bar{E}_\phi $ is a root vector corresponding to some root $\psi \in \Delta_0$, 
$\bar{E}_{-\phi }$ is a root vector corresponding to $-\psi \in \Delta_0$. Also note that  
\[ [\bar{H}_\psi ^* , \sum\limits_{j=0}^{k-1} \epsilon_0 ^j E_{\bar{\nu}^j (\phi)} ] 
 =(\sum\limits_{i=0}^{k-1} a_{\phi \bar{\nu}^i(\psi)}) \sum\limits_{j =0}^{k-1}\epsilon_0 ^j E_{\bar{\nu}^j (\phi)}, \textrm{and for } 
k =3 ,\]
\[[\bar{H}_\psi ^* ,  E_\phi + \epsilon_0 ^2 E_{\bar{\nu}(\phi)} + \epsilon_0 E_{\bar{\nu}^2 (\phi)}] 
=(\sum\limits_{i=0}^{k-1} a_{\phi \bar{\nu}^i(\psi)})( E_\phi + \epsilon_0 ^2 E_{\bar{\nu}(\phi)} + \epsilon_0 E_{\bar{\nu}^2 (\phi)})
\]
for all $\phi , \psi \in \Phi$. So if $\phi \in \Phi$ with $\phi \neq \bar{\nu}(\phi)$, and $\bar{E}_\phi$ is a root vector corresponding 
to the root $\psi \in \Delta_0$, then 
$\sum\limits_{j=0}^{k-1} \epsilon_0 ^j E_{\bar{\nu}^j (\phi)}$ is a weight vector corresponding to the weight  $\psi \in (\frak{h}^\nu)^*$ 
of the $\frak{g}_0^\nu$-module $\frak{g}^\nu _{\overline{k-1}}\ $, and for $k=3$, 
$E_\phi + \epsilon_0 ^2 E_{\bar{\nu}(\phi)} + \epsilon_0  E_{\bar{\nu}^2 (\phi)}$ is weight vector corresponding to the weight  
$\psi$ of $\frak{g}^\nu _{\bar{1}}$. Similarly  
$\sum\limits_{j=0}^{k-1} \epsilon_0 ^j E_{-\bar{\nu}^j (\phi)}$ is a weight vector corresponding to the weight  $-\psi \in (\frak{h}^\nu)^*$ 
of $\frak{g}^\nu _{\overline{k-1}}\ $, and for $k=3$, 
$E_{-\phi} + \epsilon_0 ^2 E_{-\bar{\nu}(\phi)} + \epsilon_0  E_{-\bar{\nu}^2 (\phi)}$ is weight vector corresponding to the weight  
$-\psi $ of $\frak{g}^\nu _{\bar{1}}$. 

  Actually there exists a basis $\Psi = \{\psi_1, \psi_2, \ldots , \psi_n \}$ of the root system $\Delta_0 = \Delta (\frak{g}_0^\nu , \frak{h}^\nu)$ 
such that $\bar{E}_\phi $ is a root vector corresponding to some root $\psi_i \in \Psi$,  $\bar{E}_{-\phi }$ is a root vector corresponding 
to $-\psi_i$, and $\{ \bar{H}_\phi^* ,\  \bar{E}_\phi ,\  \bar{E}_{-\phi} : \phi \in \Phi \}$ generates $\frak{g}_0^\nu$ 
\cite[the proof of Lemma 5.11, Ch. X ]{helgason}. Let $\Delta_0^+$ be the system of positive roots in $\Delta_0$ generated 
by the basis $\Psi$. Let $\alpha_0$ be the lowest weight (with respect to $\Delta_0^+$) of the 
$\frak{g}_0^\nu$-module $\frak{g}_{\bar{1}}^\nu $. Then $\alpha_0 \neq 0$, the set $B := \{ \alpha_0 , \psi_1, \ldots , \psi_n\}$ 
is linearly dependent and $B$ generates $\bar{\Delta}$ in the sense that each $(\alpha , i) \in \bar{\Delta}$ can be 
written in the form
\begin{equation}\label{5} 
(\alpha , i) =\pm ( n_0 (\alpha_0 , \bar{1}) + \sum\limits_{j=1}^{n} n_j (\psi_j, 0))  \ (n_j \in {\mathbb{N} \cup \{0 \}}
\textrm{ for all } 0\le j \le n), 
\end{equation} 
 \cite[Lemma 5.7, Ch. X]{helgason}. Note that if $\alpha \in (\frak{h}^\nu)^*$ is a weight of the 
 $\frak{g}_0^\nu$-module $\frak{g}_{\bar{a}} ^\nu \ ( 0 \le a \le k-1)$, then we may take $n_0 = a$ in the above decomposition. 
Also if $(\alpha , i) \in \bar{\Delta}$ with $ i \neq 0$ or $\alpha \in \Delta_0 ^+$ for $i = 0$, then $\alpha$ can be written as 
\begin{equation}\label{6}
\alpha = \beta_1 + \cdots+  \beta_k , 
\end{equation} 
where all $\beta_i \in B$, not necessarily distinct, such that each partial sum $\beta_1 + \cdots +\beta_j$ is the first 
component of some root in $\bar{\Delta}$ 
\cite[follows from (v) of Lemma 5.7, Ch. X]{helgason}. 
 For $X_1, X_2, \ldots , X_r \in \frak{g}$, let $[X_1, \ldots ,X_{r-1}, X_r ]$ denote the element $\textrm{ad}(X_1)\cdots  
\textrm{ad}(X_{r-1})(X_r) \in \frak{g}$. If $(\alpha , \bar{a}) \in \bar{\Delta} \setminus \bar{\Delta}_0 \ 
(a \in \mathbb{N})$ with $(\alpha , \bar{a}) = a (\alpha_0, \bar{1}) + \sum\limits_{i=1}^{n} n_i (\psi_i , 0) \ (n_i \in 
\mathbb{N} \cup \{0 \} \textrm{ for all }1 \le i \le n)$, then by \eqref{3}, \eqref{4} and \eqref{6} we have 
\begin{equation}\label{7} 
\frak{g}_{(\alpha , \bar{a})} = \mathbb{C} [ X_1 , \ldots , X_r ], 
\end{equation} 
for suitable vectors $X_1 , \ldots ,X_r $ lie in the eigenspaces of roots $(\alpha_0, \bar{1}) , (\psi_i , 0) \ (1 \le i \le n)$ 
such that the sum of the corresponding roots is $a (\alpha_0 , \bar{1}) + \sum\limits_{i =1}^{n} n_i (\psi_i , 0)$. 

  Choose $E_0 (\neq 0) \in \frak{g}_{(\alpha_0 , \bar{1})}$.  Then the vectors 
$E_0, \ \bar{E}_\phi \ (\phi \in \Phi)$ generate the Lie algebra $\frak{g}$ \cite[Th. 5.15(i), Ch. X]{helgason}. Let $\alpha_0 + 
\sum\limits_{i=1}^{n} a_i \psi_i =0 \ (a_i \in \mathbb{N} \textrm{ for all } 1 \le i \le n)$ \cite[Tables of Diagrams $S(A)$, \S 5, Ch. X]
{helgason}. Let $s_0, s_1, \ldots , s_n$ be non-negative integers without non-trivial common factor and put $m = k (s_0 + 
\sum\limits_{i = 1}^{n} a_i s_i )$. Let $\epsilon$ be a primitive $m$-th root of unity and $s_\phi : = s_i$,  if $\bar{E}_\phi $ is a 
root vector corresponding to the simple root $\psi_i \in \Psi $. Note that $s_{\bar{\nu}^j (\phi)} = s_\phi$ for all $\phi \in \Phi$. 
There exists a unique automorphism $\sigma$ of $\frak{g}$ of order $m$ with 
\begin{equation}\label{sigma}
\sigma (E_0 ) = \epsilon^{s_0} E_0 , \  \sigma ( \bar{E}_\phi ) = \epsilon^{s_\phi} \bar{E}_\phi \  (\phi \in \Phi) 
\end{equation}
\cite[Th. 5.15(i), Ch. X]{helgason}. The automorphism $\sigma$ is called an automorphism of type $(s_0, s_1, \ldots ,  s_n ; k)$. 
Note that the automorphism $\nu$ induced by the Dynkin diagram automorphism $\bar{\nu}$ is of type $(1, 0, \ldots , 0; k)$. 
The automorphism $\sigma$ is inner if and only if $k=1$ \cite[Th. 5.16(i), Ch. X]{helgason}. 

  For $1 \le i \le n$, if $\frak{g}_{(\psi_i, \bar{1})} \neq 0$, the decomposition \eqref{5} for $(\psi_i, \bar{1})$ is given 
by 
\[(\psi_i , \bar{1}) = (\alpha_0, \bar{1}) + \sum\limits_{\substack{j=1 \\ j\neq i}}^{n} a_j (\psi_j , 0) + (a_i + 1) (\psi_i, 0),
\textrm{ as } \alpha_0 + \sum\limits_{j=0}^{n} a_j \psi_j =0. \]
Similarly $\frak{g}_{(\psi_i, \bar{2})} \neq 0$ implies 
\[(\psi_i , \bar{2}) = 2(\alpha_0, \bar{1}) + \sum\limits_{\substack{j=1 \\ j\neq i}}^{n} 2a_j (\psi_j , 0) + (2a_i + 1) (\psi_i, 0),  \]
$\frak{g}_{(-\psi_i, \bar{1})} \neq 0$ implies 
\[(-\psi_i , \bar{1}) = (\alpha_0, \bar{1}) + \sum\limits_{\substack{j=1 \\ j\neq i}}^{n} a_j (\psi_j , 0) + (a_i - 1) (\psi_i, 0),  \textrm{ and}\]  
 $\frak{g}_{(-\psi_i, \bar{2})} \neq 0$ implies 
\[(-\psi_i , \bar{2}) = 2(\alpha_0, \bar{1}) + \sum\limits_{\substack{j=1 \\ j\neq i}}^{n} 2a_j (\psi_j , 0) + (2a_i - 1) (\psi_i, 0). \]
As $\frak{g}^\nu_{\bar{k}} = \frak{g}_0^\nu$,  $\frak{g}_{(-\psi_i, \bar{k})} \neq 0$ and 
\[(-\psi_i , \bar{k}) = k(\alpha_0, \bar{1}) + \sum\limits_{\substack{j=1 \\ j\neq i}}^{n} ka_j (\psi_j , 0) + (ka_i - 1) (\psi_i, 0). \]
By \eqref{7} for $\phi \in \Phi$ with $\phi \neq \bar{\nu}(\phi)$, and $k \neq 3$, we have 
\[ \sigma (\sum\limits_{j=0}^{k-1} \epsilon_0 ^j E_{\bar{\nu}^j (\phi)}) = \epsilon^{s_i +\frac{m}{k}}
\sum\limits_{j=0}^{k-1} \epsilon_0 ^j E_{\bar{\nu}^j (\phi)}= \epsilon_0 \epsilon^{s_i}
\sum\limits_{j=0}^{k-1} \epsilon_0 ^j E_{\bar{\nu}^j (\phi)} , \]
\[ \sigma (\sum\limits_{j=0}^{k-1} \epsilon_0 ^j E_{-\bar{\nu}^j (\phi)}) = \epsilon^{-s_i +\frac{m}{k}}
\sum\limits_{j=0}^{k-1} \epsilon_0 ^j E_{-\bar{\nu}^j (\phi)}= \epsilon_0 \epsilon^{-s_i}
\sum\limits_{j=0}^{k-1} \epsilon_0 ^j E_{-\bar{\nu}^j (\phi)} , \]
as $m = k(s_0 + \sum\limits_{j=1}^{n} a_j s_j)$ implies $s_0 + \sum\limits_{\substack{j=1 \\ j \neq i}}^{n} a_js_j + (a_i +1)s_i = 
s_i + \frac{m}{k}$, $s_0 + \sum\limits_{\substack{j=1 \\ j \neq i}}^{n} a_js_j + (a_i -1)s_i = - s_i + \frac{m}{k}$, and 
$\epsilon ^{\frac{m}{k}} = \epsilon_0$ for $k = 1,2$. 
Similarly for $\phi \in \Phi$ with $\phi \neq \bar{\nu}(\phi)$, and $k=3$, we have 
\[ \sigma (\sum\limits_{j=0}^{k-1} \epsilon_0 ^j E_{\bar{\nu}^j (\phi)}) = \epsilon^{s_i +\frac{2m}{3}}
\sum\limits_{j=0}^{k-1} \epsilon_0 ^j E_{\bar{\nu}^j (\phi)}= 
\begin{cases}
\epsilon_0^2 \epsilon^{s_i}\sum\limits_{j=0}^{k-1} \epsilon_0 ^j E_{\bar{\nu}^j (\phi)},  & \textrm{if } 
\epsilon^{\frac{m}{3}}= \epsilon_0 \\
\epsilon_0 \epsilon^{s_i}\sum\limits_{j=0}^{k-1} \epsilon_0 ^j E_{\bar{\nu}^j (\phi)},  & \textrm{if } 
\epsilon^{\frac{m}{3}}= \epsilon_0^2.
\end{cases}
\]
\[ \sigma (E_\phi + \epsilon_0 ^2 E_{\bar{\nu}(\phi)} + \epsilon_0 E_{\bar{\nu}^2 (\phi)}) = \epsilon^{s_i +\frac{m}{3}}
(E_\phi + \epsilon_0 ^2 E_{\bar{\nu}(\phi)} + \epsilon_0 E_{\bar{\nu}^2 (\phi)}) \ \ \ \ \ \ \ \ \ \ \ \ \ \ \ \ \ \ \ \ \ \ \ \ \ \ \]
\[\ \ \ \ \ \ \ \ \ \ \ \ \ \ \ \ \ \ \ \ \ \ \ \ \ \ \ \ \ =
\begin{cases}
\epsilon_0 \epsilon^{s_i} (E_\phi + \epsilon_0 ^2 E_{\bar{\nu}(\phi)} + \epsilon_0 E_{\bar{\nu}^2 (\phi)}),  & \textrm{if } 
\epsilon^{\frac{m}{3}}= \epsilon_0 \\
 \epsilon_0^2 \epsilon^{s_i} (E_\phi + \epsilon_0 ^2 E_{\bar{\nu}(\phi)} + \epsilon_0 E_{\bar{\nu}^2 (\phi)}),  & \textrm{if } 
\epsilon^{\frac{m}{3}}= \epsilon_0^2.
\end{cases}
\]
\[ \sigma (\sum\limits_{j=0}^{k-1} \epsilon_0 ^j E_{-\bar{\nu}^j (\phi)}) = \epsilon^{-s_i +\frac{2m}{3}}
\sum\limits_{j=0}^{k-1} \epsilon_0 ^j E_{-\bar{\nu}^j (\phi)}= 
\begin{cases}
\epsilon_0^2 \epsilon^{-s_i}\sum\limits_{j=0}^{k-1} \epsilon_0 ^j E_{-\bar{\nu}^j (\phi)},  & \textrm{if } 
\epsilon^{\frac{m}{3}}= \epsilon_0 \\
 \epsilon_0 \epsilon^{-s_i}\sum\limits_{j=0}^{k-1} \epsilon_0 ^j E_{-\bar{\nu}^j (\phi)},  & \textrm{if } 
\epsilon^{\frac{m}{3}}= \epsilon_0^2.
\end{cases}
\]
\[ \sigma (E_{-\phi} + \epsilon_0 ^2 E_{-\bar{\nu}(\phi)} + \epsilon_0 E_{-\bar{\nu}^2 (\phi)}) = \epsilon^{-s_i +\frac{m}{3}}
(E_{-\phi} + \epsilon_0 ^2 E_{-\bar{\nu}(\phi)} + \epsilon_0 E_{-\bar{\nu}^2 (\phi)}) \ \ \ \ \ \ \ \ \ \ \ \ \ \ \ \ \ \ \ \ \ \ \ \ \ \ \]
\[\ \ \ \ \ \ \ \ \ \ \ \ \ \ \ \ \ \ \ \ \ \ \ \ \ \ \ \ \ =
\begin{cases}
\epsilon_0 \epsilon^{-s_i} (E_{-\phi }+ \epsilon_0 ^2 E_{-\bar{\nu}(\phi)} + \epsilon_0 E_{-\bar{\nu}^2 (\phi)}),  & \textrm{if } 
\epsilon^{\frac{m}{3}}= \epsilon_0 \\
 \epsilon_0^2 \epsilon^{-s_i} (E_{-\phi} + \epsilon_0 ^2 E_{-\bar{\nu}(\phi)} + \epsilon_0 E_{-\bar{\nu}^2 (\phi)}),  & \textrm{if } 
\epsilon^{\frac{m}{3}}= \epsilon_0^2.
\end{cases}
\]
Also for any $k$, 
\[ \sigma (\bar{E}_{-\phi}) = \epsilon^{-s_i +m} \bar{E}_{-\phi}= \epsilon^{-s_i}\bar{E}_{-\phi} \ (\textrm{for all } \phi \in \Phi), \]
as $\bar{E}_{-\phi}$ is a root vector of  $\frak{g}_0^\nu$ corresponding to the root $-\psi_i$ and 
$-\psi_i = k \alpha_0 + \sum\limits_{\substack{j = 1 \\ j \neq i}}^{n} ka_j \psi_j + (ka_i - 1) \psi_i$ via the identification 
$\frak{g}_0^\nu = \frak{g}^\nu_{\bar{k}}$. 
Note that $s_i = s_\phi$. 
Obviously
\[\boxed{\sigma (E_\phi) = \epsilon^{s_\phi} E_\phi,\ \sigma (E_{-\phi}) = \epsilon^{-s_\phi} E_{-\phi}, \textrm{ and}} \]
\[\boxed{\sigma (H_\phi ^*) = \sigma ([E_\phi , E_{-\phi}])=[\sigma(E_\phi) , \sigma(E_{-\phi})]= 
[\epsilon^{s_\phi}E_\phi , \epsilon^{-s_\phi}E_{-\phi}]= H_\phi ^* \textrm{ for all } 
\phi \in \Phi \textrm{ with } \phi = \bar{\nu}(\phi). }\]
For $\phi \in \Phi \textrm{ with } \phi \neq \bar{\nu}(\phi)$, and $k=2$, 
\[ \sigma (E_\phi ) = \sigma \bigg{(}\frac{E_\phi + E_{\bar{\nu}(\phi)}}{2}\bigg{)} + \sigma \bigg{(}\frac{E_\phi - E_{\bar{\nu}(\phi)}}{2}
\bigg{)} = \epsilon^{s_\phi}\frac{E_\phi + E_{\bar{\nu}(\phi)}}{2} - \epsilon^{s_\phi} \frac{E_\phi - E_{\bar{\nu}(\phi)}}{2}= 
\epsilon^{s_\phi} E_{\bar{\nu}(\phi)},\]
\[ \sigma (E_{-\phi} ) = \sigma \bigg{(}\frac{E_{-\phi} + E_{-\bar{\nu}(\phi)}}{2}\bigg{)} + \sigma \bigg{(}\frac{E_{-\phi} - E_{-\bar{\nu}(\phi)}}{2}
\bigg{)} = \epsilon^{-s_\phi}\frac{E_{-\phi} + E_{-\bar{\nu}(\phi)}}{2} - \epsilon^{-s_\phi} \frac{E_{-\phi} - E_{-\bar{\nu}(\phi)}}{2}\]
\[= \epsilon^{-s_\phi} E_{-\bar{\nu}(\phi)}, \ \ \ \ \ \ \ \ \ \ \ \ \ \ \ \ \ \ \ \ \ \ \ \ \ \ \ \ \ \ \ \ \ \ \ \ \ \ \ \ \ \ \ \ \ \ \ \ \ \ \ \ \ \ \ \ \ \ \ \ \ \ \ \ 
\ \ \ \ \ \ \ \ \ \ \ \ \ \ \]
as $\epsilon_0 = -1$ here. That is 
\[\boxed{\sigma (E_\phi) = \epsilon^{s_\phi} E_{\bar{\nu}(\phi)},\ \sigma (E_{-\phi}) = \epsilon^{-s_\phi} E_{-\bar{\nu}(\phi)},
\textrm{ and }\sigma (H_\phi ^*)=  H_{\bar{\nu}(\phi)} ^*
\textrm{ for all } \phi \in \Phi \textrm{ with } \phi \neq \bar{\nu}(\phi), \textrm{ for } k=2. }\]
For $\phi \in \Phi \textrm{ with } \phi \neq \bar{\nu}(\phi)$, and $k=3$, 
\[ \sigma (E_\phi ) = \sigma \bigg{(}\frac{E_\phi + E_{\bar{\nu}(\phi)}+ E_{\bar{\nu}^2(\phi)}}{3}\bigg{)} + 
\sigma \bigg{(}\frac{E_\phi + \epsilon_0 E_{\bar{\nu}(\phi)}+\epsilon_0^2 E_{\bar{\nu}^2(\phi)}}{3}\bigg{)} 
+\sigma \bigg{(}\frac{E_\phi + \epsilon_0^2 E_{\bar{\nu}(\phi)}+\epsilon_0 E_{\bar{\nu}^2(\phi)}}{3}\bigg{)}
\]
\[\ \ \ \ \ \ \ \ \ \ \ \ =
\begin{cases}
\epsilon^{s_\phi} \frac{E_\phi + E_{\bar{\nu}(\phi)}+ E_{\bar{\nu}^2(\phi)}}{3} + 
\epsilon_0^2 \epsilon^{s_\phi} \frac{E_\phi + \epsilon_0 E_{\bar{\nu}(\phi)}+\epsilon_0^2 E_{\bar{\nu}^2(\phi)}}{3} + 
\epsilon_0 \epsilon^{s_\phi} \frac{E_\phi + \epsilon_0^2 E_{\bar{\nu}(\phi)}+\epsilon_0 E_{\bar{\nu}^2(\phi)}}{3}, & \textrm{if } 
 \epsilon^{\frac{m}{3}} = \epsilon_0 \\
\epsilon^{s_\phi} \frac{E_\phi + E_{\bar{\nu}(\phi)}+ E_{\bar{\nu}^2(\phi)}}{3} + 
\epsilon_0 \epsilon^{s_\phi} \frac{E_\phi + \epsilon_0 E_{\bar{\nu}(\phi)}+\epsilon_0^2 E_{\bar{\nu}^2(\phi)}}{3} + 
\epsilon_0^2 \epsilon^{s_\phi} \frac{E_\phi + \epsilon_0^2 E_{\bar{\nu}(\phi)}+\epsilon_0 E_{\bar{\nu}^2(\phi)}}{3}, & 
\textrm{if } \epsilon^{\frac{m}{3}} = \epsilon_0^2 
\end{cases}
\]
\[=
\begin{cases}
\epsilon^{s_\phi} E_{\bar{\nu}(\phi)}, & \textrm{if } \epsilon^{\frac{m}{3}} = \epsilon_0 \\
\epsilon^{s_\phi} E_{\bar{\nu}^2(\phi)}, & \textrm{if } \epsilon^{\frac{m}{3}} = \epsilon_0^2, 
\end{cases}
\ \ \ \ \ \ \ \ \ \ \ \ \ \ \ \ \ \ \ \ \ \ \ \ \ \ \ \ \ \ \ \ \ \ \ \ \ \ \ \ \ \ \ \ \ \ \ \ \ \ \ \ \ \ \ \ \ \ \ \ \ \ \ \ \ \ \]
\[ \sigma (E_{-\phi} ) = \sigma \bigg{(}\frac{E_{-\phi} + E_{-\bar{\nu}(\phi)}+ E_{-\bar{\nu}^2(\phi)}}{3}\bigg{)} + 
\sigma \bigg{(}\frac{E_{-\phi} + \epsilon_0 E_{-\bar{\nu}(\phi)}+\epsilon_0^2 E_{-\bar{\nu}^2(\phi)}}{3}\bigg{)}\ \ \ \ \ \ \ \  \ \ \ \ \ \ \ \ 
\ \ \ \ \ \ \ \ \ \ \ \ \ \ \ \ \ \ \ \ \ \ \ \ \ \]
\[+\sigma \bigg{(}\frac{E_{-\phi} + \epsilon_0^2 E_{-\bar{\nu}(\phi)}+\epsilon_0 E_{-\bar{\nu}^2(\phi)}}{3}\bigg{)}
\ \ \ \ \ \ \ \ \ \ \ \ \ \ \ \ \ \ \ \ \ \ \ \ \ \ \ \ \ \ \ \ \ \ \ \ \ \ \ \ \ \ \ \ \ \ \ \ \ \ \ \]
\[=
\begin{cases}
\epsilon^{-s_\phi} \frac{E_{-\phi} + E_{-\bar{\nu}(\phi)}+ E_{-\bar{\nu}^2(\phi)}}{3} + 
\epsilon_0^2 \epsilon^{-s_\phi} \frac{E_{-\phi} + \epsilon_0 E_{-\bar{\nu}(\phi)}+\epsilon_0^2 E_{-\bar{\nu}^2(\phi)}}{3} + 
\epsilon_0 \epsilon^{-s_\phi} \frac{E_{-\phi} + \epsilon_0^2 E_{-\bar{\nu}(\phi)}+\epsilon_0 E_{-\bar{\nu}^2(\phi)}}{3}, & \textrm{if } 
 \epsilon^{\frac{m}{3}} = \epsilon_0 \\
\epsilon^{-s_\phi} \frac{E_{-\phi} + E_{-\bar{\nu}(\phi)}+ E_{-\bar{\nu}^2(\phi)}}{3} + 
\epsilon_0 \epsilon^{-s_\phi} \frac{E_{-\phi} + \epsilon_0 E_{-\bar{\nu}(\phi)}+\epsilon_0^2 E_{-\bar{\nu}^2(\phi)}}{3} + 
\epsilon_0^2 \epsilon^{-s_\phi} \frac{E_{-\phi} + \epsilon_0^2 E_{-\bar{\nu}(\phi)}+\epsilon_0 E_{-\bar{\nu}^2(\phi)}}{3}, & 
\textrm{if } \epsilon^{\frac{m}{3}} = \epsilon_0^2 
\end{cases}
\]
\[=
\begin{cases}
\epsilon^{-s_\phi} E_{-\bar{\nu}(\phi)}, & \textrm{if } \epsilon^{\frac{m}{3}} = \epsilon_0 \\
\epsilon^{-s_\phi} E_{-\bar{\nu}^2(\phi)}, & \textrm{if } \epsilon^{\frac{m}{3}} = \epsilon_0^2, 
\end{cases}
\ \ \ \ \ \ \ \ \ \ \ \ \ \ \ \ \ \ \ \ \ \ \ \ \ \ \ \ \ \ \ \ \ \ \ \ \ \ \ \ \ \ \ \ \ \ \ \ \ \ \ \ \ \ \ \ \ \ \ \ \ \ \ \ \ \ \ \ \ \ \ \ \ \ \ \ \ \ \]
as $\epsilon_0 = \omega$ here. Hence $\textrm{for all } \phi \in \Phi \textrm{ with } \phi \neq \bar{\nu}(\phi)$,
\[\boxed{\sigma (E_\phi) = \epsilon^{s_\phi} E_{\bar{\nu}(\phi)},\ \sigma (E_{-\phi}) = \epsilon^{-s_\phi} E_{-\bar{\nu}(\phi)}, 
\textrm{ and }\sigma (H_\phi ^*)=  H_{\bar{\nu}(\phi)} ^*
 \textrm{ for } k=3 , \textrm{ if }\epsilon^{\frac{m}{3}} = \epsilon_0,}\]
or
\[\boxed{\sigma (E_\phi) = \epsilon^{s_\phi} E_{\bar{\nu}^2(\phi)},\ \sigma (E_{-\phi}) = \epsilon^{-s_\phi} E_{-\bar{\nu}^2(\phi)}, 
 \textrm{ and }\sigma (H_\phi ^*)=  H_{\bar{\nu}^2(\phi)} ^*
\textrm{ for } k=3 , \textrm{ if }\epsilon^{\frac{m}{3}} = \epsilon_0^2.}\]
So if we want the Dynkin diagram automorphism induced by $\sigma$ to be $\bar{\nu}$, we must need to take $\epsilon$ to be a primitive 
$m$-th root of unity with $\epsilon^{\frac{m}{k}} = \epsilon_0$. In this case, we have 
\[ \sigma (E_\alpha) = q_\alpha \epsilon^{n_\alpha} E_{\nu(\alpha)},\  \sigma (E_{-\alpha}) = q_\alpha \epsilon^{-n_\alpha} 
 E_{-\nu(\alpha)}, \textrm{ and }\sigma (H_\phi ^*)=  H_{\bar{\nu}(\phi)} ^*  \textrm{ for all } \alpha \in \Delta^+, \phi \in \Phi ;\] 
where $\{ H_\phi ^* \  , E_\alpha : \phi \in \Phi , \alpha \in \Delta \}$ is a Chevalley basis for $\frak{g}$ as in \eqref{chevalley}, 
$q_\alpha = \pm 1$, and $n_\alpha = \sum\limits_{\phi \in \Phi} n_\phi (\alpha ) s_\phi$ (if 
$\alpha = \sum\limits_{\phi \in \Phi} n_\phi (\alpha ) \phi,\  n_\phi (\alpha ) \in \mathbb{N} \cup \{0\}$) for all $\alpha \in \Delta^+$. 
Recall that $\nu$ is the unique automorphism of $\frak{g}$ with 
\[ \nu ( H_\phi ^*) =  H_{\bar{\nu} (\phi)} ^* , \ \ \  \nu (E_\phi) = E_{\bar{\nu} (\phi)} , \ \ \  \nu (E_{-\phi}) = E_{-\bar{\nu} (\phi)} \ \ \ 
(\phi \in \Phi). \]
So $\nu(\frak{h}) = \frak{h}$ and hence $\nu(\alpha) \ (\alpha \in \frak{h}^*)$ makes sense, where 
$\nu(\alpha)(H) = \alpha (\nu H)$ for all $H \in \frak{h}$. Note that $\nu(\phi) =\bar{\nu}(\phi)$ for all $\phi \in \Phi$. Also $q_\alpha = 
\pm 1$ by \eqref{chevalley}, for if $\beta + n\alpha \ (p \le n \le q)$ is the $\alpha$-string containing $\beta$, then 
$\nu(\beta) + n\nu(\alpha) \ (p \le n \le q)$ is the $\nu(\alpha)$-string containing $\nu(\beta)$, where $\alpha, \beta \in \Delta$. 

  If  $X_\alpha = E_\alpha - E_{-\alpha} , Y_\alpha = i(E_\alpha + E_{-\alpha})\ (\alpha \in \Delta^+)$, then the compact real form 
 $\frak{u}$ is given by 
\[ \frak{u} =  \sum_{\phi \in \Phi} \mathbb{R} (i  H_\phi ^*) \oplus \sum_{\alpha \in \Delta^+} \mathbb{R} X_\alpha \oplus 
\sum_{\alpha \in \Delta^+} \mathbb{R} Y_\alpha. \]
Now 
\[\sigma (X_\alpha ) = q_\alpha (\epsilon^{n_\alpha}  E_{\nu(\alpha)} - \epsilon^{-n_\alpha}  E_{-\nu(\alpha)})\ \ \ \ \ \ \ \ \ \ \ \ \ \ \ \ \ \ \ 
\ \ \ \ \ \ \ \ \ \ \ \ \ \ \ \ \ \ \ \ \ \ \ \ \ \ \ \ \ \ \ \ \ \ \ \ \ \ \ \ \ \ \ \ \ \ \ \ \ \ \ \ \ \ \ \ \ \ \ \ \ \ \ \ \ \ \ \ \ \ \ \ \ \ \ \ \ \ \ \ \ \]
\[=q_\alpha \bigg{(} \cos \frac{2bn_\alpha \pi}{m} E_{\nu(\alpha)} + i \sin \frac{2bn_\alpha \pi}{m} E_{\nu(\alpha)} - \cos \frac{2bn_\alpha \pi}{m} 
 E_{-\nu(\alpha)} + i \sin \frac{2bn_\alpha \pi}{m} E_{-\nu(\alpha)}\bigg{)} \]
\[= q_\alpha \cos \frac{2bn_\alpha \pi}{m}(E_{\nu(\alpha)} -  E_{-\nu(\alpha)}) + i q_\alpha \sin \frac{2bn_\alpha \pi}{m}(E_{\nu(\alpha)} + E_{-\nu(\alpha)}) 
\ \ \ \ \ \ \ \ \ \ \ \ \ \ \ \ \ \ \ \ \ \ \]
\[= q_\alpha \cos \frac{2bn_\alpha \pi}{m}  X_{\nu(\alpha)} + q_\alpha \sin \frac{2bn_\alpha \pi}{m}  Y_{\nu (\alpha)}, \ \ \ \ \ \ \ \ \ \ \ \ \ \ \ \ \ \ \ \ \ \ \ \ \ 
\ \ \ \ \ \ \ \ \ \ \ \ \ \ \ \ \ \ \ \ \ \ \ \ \ \ \ \]
\[\sigma (Y_\alpha ) = i q_\alpha (\epsilon^{n_\alpha}  E_{\nu(\alpha)} + \epsilon^{-n_\alpha}  E_{-\nu(\alpha)})\ \ \ \ \ \ \ \ \ \ \ \ \ \ \ \ \ \ 
\ \ \ \ \ \ \ \ \ \ \ \ \ \ \ \ \ \ \ \ \ \ \ \ \ \ \ \ \ \ \ \ \ \ \ \ \ \ \ \ \ \ \ \ \ \ \ \ \ \ \ \ \ \ \ \ \ \ \ \]
\[= q_\alpha \bigg{(} i \cos \frac{2bn_\alpha \pi}{m} E_{\nu(\alpha)} - \sin \frac{2bn_\alpha \pi}{m} E_{\nu(\alpha)} + i \cos \frac{2bn_\alpha \pi}{m} 
 E_{-\nu(\alpha)} + \sin \frac{2bn_\alpha \pi}{m} E_{-\nu(\alpha)}\bigg{)} \]
\[= i  q_\alpha \cos \frac{2bn_\alpha \pi}{m}(E_{\nu(\alpha)} +  E_{-\nu(\alpha)}) - q_\alpha \sin \frac{2bn_\alpha \pi}{m}(E_{\nu(\alpha)} - E_{-\nu(\alpha)}) 
\ \ \ \ \ \ \ \ \ \ \ \ \ \ \ \ \ \ \ \ \ \ \ \ \]
\[= q_\alpha \cos \frac{2bn_\alpha \pi}{m}  Y_{\nu(\alpha)} - q_\alpha \sin \frac{2bn_\alpha \pi}{m}  X_{\nu (\alpha)}, \ \ \ \ \ \ \ \ \ \ \ \ \ \ \ \ \ \ \ \ \ \ \ \ \ \ 
\ \ \ \ \ \ \ \ \ \ \ \ \ \ \ \ \ \ \ \ \ \ \ \ \ \ \ \]
for all $\alpha \in \Delta^+$, where $\epsilon = e^{\frac{2b\pi i}{m}}$ with gcd$(b,m)= 1$, is a primitive $m$-th root of unity. Obviously 
 $\sigma (i H_\phi ^*) = i H_{\bar{\nu}(\phi)}^*$ for all $\phi \in \Phi$. Hence $\sigma$ is an  automorphism of $\frak{g}$ of order $m$ such that 
$\sigma \theta  = \theta \sigma$, 
$\sigma(\frak{h}) = \frak{h}$, $\sigma (\Delta^+ )=\Delta^+$, and the Dynkin diagram automorphism induced by $\sigma$ is 
$\bar{\nu}$.  
Let $i_1 , \ldots, i_t$ be all the indices with $s_{i_1} = \cdots = s_{i_t}=0$. Then the Lie algebra $\frak{g}_0 ^\sigma = \{ X \in \frak{g} : 
\sigma (X) = X \}$ is the direct sum of an $(n-t)$-dimensional centre and a semisimple Lie algebra whose Dynkin diagram is the subdiagram 
of the follwing diagram $\frak{g}^{(k)}$ consisting of the vertices $\psi_{i_1} , \ldots , \psi_{i_t}$ \cite[Th. 5.15(ii), Ch. X]{helgason}.

\newpage 
\begin{tikzpicture}

%1
\draw (0,0) circle [radius = 0.1];
\draw (1,0) circle [radius = 0.1]; 
\draw (2.5,0) circle [radius = 0.1]; 
\draw (3.5,0) circle [radius = 0.1]; 
\draw (1.75,0.75) circle [radius = 0.1]; 
\node [below] at (0.05,-0.05) {$\psi_1$}; 
\node [below] at (1.05,-0.05) {$\psi_2$}; 
\node [below] at (2.75,-0.05) {$\psi_{n-1}$}; 
\node [below] at (3.55,-0.05) {$\psi_n$}; 
\node [above] at (1.80, 0.80) {$\alpha_0$}; 
\draw (0.1,0) -- (0.9,0); 
\draw (1.1,0) -- (1.5,0); 
\draw [dotted] (1.5,0) -- (2,0); 
\draw (2,0) -- (2.4,0); 
\draw (2.6,0) -- (3.4,0); 
\draw (0.1,0.05) -- (1.65,0.75); 
\draw (1.85,0.75) -- (3.4,0.05); 
\node [left] at (-0.5,0) {$\frak{a}_n^{(1)} :$};  
\node [left] at (-0.5,-0.6) {$(n >1)$};  

%2
\node [left] at (8.5,0) {$\frak{a}_{2n}^{(2)} :$}; 
\node [left] at (8.5, -0.6) {$(n >1)$}; 
\draw (9,0) circle [radius = 0.1]; 
\draw (10,0) circle [radius = 0.1]; 
\draw (11,0) circle [radius = 0.1]; 
\draw (12.5,0) circle [radius = 0.1]; 
\draw (13.5,0) circle [radius = 0.1]; 
\node [below] at (9.05,-0.15) {$\alpha_0$}; 
\node [below] at (10.05,-0.05) {$\psi_1$}; 
\node [below] at (11.05,-0.05) {$\psi_2$}; 
\node [below] at (12.75,-0.05) {$\psi_{n-1}$}; 
\node [below] at (13.55,-0.05) {$\psi_n$}; 
\draw (9.9,0) -- (9.8,0.1); 
\draw (9.9,0) -- (9.8,-0.1); 
\draw (9.1,0.025) -- (9.85,0.025); 
\draw (9.1,-0.025) -- (9.85,-0.025); 
\draw (10.1,0) -- (10.9,0); 
\draw (11.1,0) -- (11.5,0); 
\draw [dotted] (11.5,0) -- (12,0); 
\draw (12,0) -- (12.4,0); 
\draw (13.4,0) -- (13.3,0.1); 
\draw (13.4,0) -- (13.3,-0.1); 
\draw (12.6,0.025) -- (13.35, 0.025); 
\draw (12.6,-0.025) -- (13.35,-0.025); 

%3
\draw (0,-2) circle [radius = 0.1]; 
\draw (1,-2) circle [radius = 0.1]; 
\node [below] at (0.05,-2.15) {$\alpha_0$}; 
\node [below] at (1.05,-2.05) {$\psi_1$}; 
\node [left] at (-0.5, -2) {$\frak{a}_1^{(1)} :$}; 
\draw (0.08,-1.984) -- (0.92,-1.984); 
\draw (0.04,-1.93) -- (0.96,-1.93); 
\draw (0.08,-2.035) -- (0.92,-2.035); 
\draw (0.04,-2.09) -- (0.96,-2.09); 

%4
\node [left] at (8.5,-2) {$\frak{a}_2^{(2)} :$}; 
\draw (9,-2) circle [radius = 0.1]; 
\draw (10,-2) circle [radius = 0.1]; 
\node [below] at (9.05,-2.05) {$\psi_1$}; 
\node [below] at (10.05,-2.15) {$\alpha_0$}; 
\draw (9.1,-2) -- (9.2,-1.9); 
\draw (9.1,-2) -- (9.2,-2.1); 
\draw (9.08,-1.984) -- (9.92,-1.984); 
\draw (9.18,-1.93) -- (9.96,-1.93); 
\draw (9.08,-2.035) -- (9.92,-2.035); 
\draw (9.18,-2.09) -- (9.96,-2.09); 

%5
\draw (0,-4) circle [radius = 0.1]; 
\draw (1,-4) circle [radius = 0.1]; 
\draw (2,-4) circle [radius = 0.1]; 
\draw (1,-5) circle [radius = 0.1]; 
\draw (3.5,-4) circle [radius = 0.1]; 
\draw (4.5,-4) circle [radius = 0.1]; 
\node [above] at (0.05,-3.95) {$\psi_1$}; 
\node [above] at (1.05,-3.95) {$\psi_2$}; 
\node [above] at (2.05,-3.95) {$\psi_2$}; 
\node [below] at (1.05,-5.15) {$\alpha_0$}; 
\node [above] at (3.75,-3.95) {$\psi_{n-1}$}; 
\node [above] at (4.55,-3.95) {$\psi_n$}; 
\node [left] at (-0.5,-4) {$\frak{b}_n^{(1)} :$}; 
\node[ left] at (-0.5,-4.6) {$(n > 2)$}; 
\draw (0.1,-4) -- (0.9,-4); 
\draw (1.1,-4) -- (1.9,-4);
\draw (1,-4.1) -- (1,-4.9); 
\draw (2.1,-4) -- (2.5,-4); 
\draw [dotted] (2.5,-4) -- (3,-4); 
\draw (3,-4) -- (3.4,-4); 
\draw (4.4,-4) -- (4.3,-3.9); 
\draw (4.4,-4) -- (4.3,-4.1); 
\draw (3.6,-3.975) -- (4.35,-3.975); 
\draw (3.6,-4.025) -- (4.35,-4.025); 

%6 
\node [left] at (8.5,-4) {$\frak{\delta}_{n+1}^{(2)} :$}; 
\node [left] at (8.6,-4.6) {$(n > 1)$}; 
\draw (9,-4) circle [radius = 0.1]; 
\draw (10,-4) circle [radius = 0.1]; 
\draw (11.5,-4) circle [radius = 0.1]; 
\draw (12.5,-4) circle [radius = 0.1]; 
\node [above] at (9.05, -3.95) {$\alpha_0$}; 
\node [above] at (10.05,-3.95) {$\psi_1$}; 
\node [above] at (11.75,-3.95) {$\psi_{n-1}$}; 
\node [above] at (12.55,-3.95) {$\psi_n$}; 
\draw (9.1,-4) -- (9.2,-3.9); 
\draw (9.1,-4) -- (9.2,-4.1); 
\draw (9.15,-3.975) -- (9.9,-3.975); 
\draw (9.15,-4.025) -- (9.9,-4.025); 
\draw (10.1,-4) -- (10.5,-4); 
\draw [dotted] (10.5,-4) -- (11,-4); 
\draw (11,-4) -- (11.4,-4); 
\draw (12.4,-4) -- (12.3,-3.9); 
\draw (12.4,-4) -- (12.3,-4.1); 
\draw (11.6,-3.975) -- (12.35,-3.975); 
\draw (11.6,-4.025) -- (12.35,-4.025); 

%7
\draw (0,-7) circle [radius = 0.1]; 
\draw (1,-7) circle [radius = 0.1]; 
\draw (2.5,-7) circle [radius = 0.1]; 
\draw (3.5,-7) circle [radius = 0.1]; 
\node [above] at (0.05,-6.95) {$\alpha_0$}; 
\node [above] at (1.05,-6.95) {$\psi_1$}; 
\node [above] at (2.75,-6.95) {$\psi_{n-1}$}; 
\node [above] at (3.55,-6.95) {$\psi_n$}; 
\node [left] at (-0.5,-7) {$\frak{c}_n^{(1)} :$}; 
\node [left] at (-0.5,-7.6) {$(n > 1)$}; 
\draw (0.9,-7) -- (0.8,-6.9); 
\draw (0.9,-7) -- (0.8,-7.1); 
\draw (0.1,-6.975) -- (0.85,-6.975); 
\draw (0.1,-7.025) -- (0.85,-7.025); 
\draw (1.1,-7) -- (1.5,-7); 
\draw [dotted] (1.5,-7) -- (2,-7); 
\draw (2,-7) -- (2.4,-7); 
\draw (2.6,-7) -- (2.7,-6.9); 
\draw (2.6,-7) -- (2.7,-7.1); 
\draw (2.65,-6.975) -- (3.4,-6.975); 
\draw (2.65,-7.025) -- (3.4,-7.025); 

%8 
\node [left] at (8.5,-7) {$\frak{a}_{2n-1}^{(2)} :$}; 
\node [left] at (8.5,-7.6) {$(n > 2)$}; 
\draw (9,-7) circle [radius = 0.1]; 
\draw (10,-7) circle [radius = 0.1]; 
\draw (11,-7) circle [radius = 0.1]; 
\draw (10,-8) circle [radius = 0.1]; 
\draw (12.5,-7) circle [radius = 0.1]; 
\draw (13.5,-7) circle [radius = 0.1]; 
\node [above] at (9.05,-6.95) {$\psi_1$}; 
\node [above] at (10.05,-6.95) {$\psi_2$}; 
\node [above] at (11.05,-6.95) {$\psi_3$}; 
\node [below] at (10.05,-8.15) {$\alpha_0$}; 
\node [above] at (12.75,-6.95) {$\psi_{n-1}$}; 
\node [above] at (13.55,-6.95) {$\psi_n$}; 
\draw (9.1,-7) -- (9.9,-7); 
\draw (10.1,-7) -- (10.9,-7); 
\draw (10,-7.1) -- (10,-7.9); 
\draw (11.1,-7) -- (11.5,-7); 
\draw [dotted] (11.5,-7) -- (12,-7); 
\draw (12,-7) -- (12.4,-7); 
\draw (12.6,-7) -- (12.7,-6.9); 
\draw (12.6,-7) -- (12.7,-7.1); 
\draw (12.65,-6.975) -- (13.4,-6.975); 
\draw (12.65,-7.025) -- (13.4,-7.025); 

%9
\draw (0,-10) circle [radius = 0.1]; 
\draw (1,-10) circle [radius = 0.1]; 
\draw (2,-10) circle [radius = 0.1]; 
\draw (1,-11) circle [radius = 0.1]; 
\draw (3.5,-10) circle [radius = 0.1]; 
\draw (4.5,-10) circle [radius = 0.1]; 
\draw (5.5,-9) circle [radius = 0.1]; 
\draw (5.5,-11) circle [radius = 0.1]; 
\node [above] at (0.05,-9.95) {$\psi_1$}; 
\node [above] at (1.05,-9.95) {$\psi_2$}; 
\node [above] at (2.05,-9.95) {$\psi_3$}; 
\node [below] at (1.05,-11.15) {$\alpha_0$}; 
\node [above] at (3.35,-9.95) {$\psi_{n-3}$}; 
\node [above] at (4.35,-9.95) {$\psi_{n-2}$}; 
\node [above] at (5.55,-8.95) {$\psi_{n-1}$}; 
\node [below] at (5.55,-11.05) {$\psi_n$}; 
\node [left] at (-0.5,-10) {$\frak{\delta}_n^{(1)} :$}; 
\node [left] at (-0.5,-10.6) {$(n > 3)$}; 
\draw (0.1,-10) -- (0.9,-10); 
\draw (1.1,-10) -- (1.9,-10); 
\draw (1,-10.1) -- (1,-10.9); 
\draw (2.1,-10) -- (2.5,-10); 
\draw [dotted] (2.5,-10) -- (3,-10); 
\draw (3,-10) -- (3.4,-10); 
\draw (3.6,-10) -- (4.4,-10); 
\draw (4.6,-10) -- (5.45,-9.05); 
\draw (4.6,-10) -- (5.45,-10.95); 

%10
\node [left] at (8.5,-10) {$\frak{e}_6^{(2)} :$}; 
\draw (9,-10) circle [radius = 0.1]; 
\draw (10,-10) circle [radius = 0.1]; 
\draw (11,-10) circle [radius = 0.1]; 
\draw (12,-10) circle [radius = 0.1]; 
\draw (13,-10) circle [radius = 0.1]; 
\node [above] at (9.05,-9.95) {$\psi_1$}; 
\node [above] at (10.05,-9.95) {$\psi_2$}; 
\node [above] at (11.05,-9.95) {$\psi_3$}; 
\node [above] at (12.05,-9.95) {$\psi_4$}; 
\node [above] at (13.05,-9.95) {$\alpha_0$}; 
\draw (9.1,-10) -- (9.9,-10); 
\draw (10.9,-10) -- (10.8,-9.9); 
\draw (10.9,-10) -- (10.8,-10.1); 
\draw (10.1,-9.975) -- (10.85,-9.975); 
\draw (10.1,-10.025) -- (10.85,-10.025); 
\draw (11.1,-10) -- (11.9,-10); 
\draw (12.1,-10) -- (12.9,-10); 

%11
\draw (0,-14) circle [radius = 0.1]; 
\draw (1,-14) circle [radius = 0.1]; 
\draw (2,-14) circle [radius = 0.1]; 
\draw (2,-13) circle [radius = 0.1]; 
\draw (2,-12) circle [radius = 0.1]; 
\draw (3,-14) circle [radius = 0.1]; 
\draw (4,-14) circle [radius = 0.1]; 
\node [below] at (0.05,-14.05) {$\psi_6$}; 
\node [below] at (1.05,-14.05) {$\psi_5$}; 
\node [below] at (2.05,-14.05) {$\psi_4$}; 
\node [right] at (2.05,-13) {$\psi_2$}; 
\node [right] at (2.05,-12) {$\alpha_0$}; 
\node [below] at (3.05,-14.05) {$\psi_3$}; 
\node [below] at (4.05,-14.05) {$\psi_1$}; 
\node [left] at (-0.5,-14) {$\frak{e}_6^{(1)} :$}; 
\draw (0.1,-14) -- (0.9,-14); 
\draw (1.1,-14) -- (1.9,-14); 
\draw (2,-13.9) -- (2,-13.1); 
\draw (2,-12.9) -- (2,-12.1); 
\draw (2.1,-14) -- (2.9,-14); 
\draw (3.1,-14) -- (3.9,-14); 

%12
\node [left] at (8.5,-14) {$\frak{\delta}_4^{(3)} :$}; 
\draw (9,-14) circle [radius = 0.1]; 
\draw (10,-14) circle [radius = 0.1]; 
\draw (11,-14) circle [radius = 0.1]; 
\node [below] at (9.05,-14.05) {$\psi_2$}; 
\node [below] at (10.05,-14.05) {$\psi_1$}; 
\node [below] at (11.05,-14.15) {$\alpha_0$}; 
\draw (9.9,-14) -- (9.8,-13.9);
\draw (9.9,-14) -- (9.8,-14.1); 
\draw (9.1,-14) -- (9.9,-14); 
\draw (9.1,-13.925) -- (9.8,-13.925); 
\draw (9.1,-14.075) -- (9.8,-14.075); 
\draw (10.1,-14) -- (10.9,-14); 

%13 
\draw (0,-16) circle [radius = 0.1]; 
\draw (1,-16) circle [radius = 0.1]; 
\draw (2,-16) circle [radius = 0.1]; 
\draw (3,-16) circle [radius = 0.1]; 
\draw (3,-15) circle [radius = 0.1]; 
\draw (4,-16) circle [radius = 0.1]; 
\draw (5,-16) circle [radius = 0.1]; 
\draw (6,-16) circle [radius = 0.1]; 
\node [below] at (0.05,-16.05) {$\psi_7$}; 
\node [below] at (1.05,-16.05) {$\psi_6$}; 
\node [below] at (2.05,-16.05) {$\psi_5$}; 
\node [below] at (3.05,-16.05) {$\psi_4$}; 
\node [right] at (3.05,-15) {$\psi_2$}; 
\node [below] at (4.05,-16.05) {$\psi_3$}; 
\node [below] at (5.05,-16.05) {$\psi_1$}; 
\node [below] at (6.05, -16.15) {$\alpha_0$}; 
\node [left] at (-0.5,-16) {$\frak{e}_7^{(1)} :$}; 
\draw (0.1,-16) -- (0.9,-16); 
\draw (1.1,-16) -- (1.9,-16); 
\draw (2.1,-16) -- (2.9,-16); 
\draw (3,-15.9) -- (3,-15.1); 
\draw (3.1,-16) -- (3.9,-16); 
\draw (4.1,-16) -- (4.9,-16); 
\draw (5.1,-16) -- (5.9,-16); 

%14
\draw (0,-18) circle [radius = 0.1]; 
\draw (1,-18) circle [radius = 0.1]; 
\draw (2,-18) circle [radius = 0.1]; 
\draw (3,-18) circle [radius = 0.1]; 
\draw (4,-18) circle [radius = 0.1]; 
\draw (5,-18) circle [radius = 0.1]; 
\draw (5,-17) circle [radius = 0.1]; 
\draw (6,-18) circle [radius = 0.1]; 
\draw (7,-18) circle [radius = 0.1]; 
\node [below] at (0.05,-18.15) {$\alpha_0$}; 
\node [below] at (1.05,-18.05) {$\psi_8$}; 
\node [below] at (2.05,-18.05) {$\psi_7$}; 
\node [below] at (3.05,-18.05) {$\psi_6$}; 
\node [below] at (4.05,-18.05) {$\psi_5$}; 
\node [below] at (5.05,-18.05) {$\psi_4$}; 
\node [right] at (5.05,-17) {$\psi_2$}; 
\node [below] at (6.05,-18.05) {$\psi_3$}; 
\node [below] at (7.05,-18.05) {$\psi_1$}; 
\node [left] at (-0.5,-18) {$\frak{e}_8^{(1)} :$}; 
\draw (0.1,-18) -- (0.9,-18); 
\draw (1.1,-18) -- (1.9,-18); 
\draw (2.1,-18) -- (2.9,-18); 
\draw (3.1,-18) -- (3.9,-18); 
\draw (4.1,-18) -- (4.9,-18); 
\draw (5,-17.9) -- (5,-17.1); 
\draw (5.1,-18) -- (5.9,-18); 
\draw (6.1,-18) -- (6.9,-18); 

%15
\draw (0,-20) circle [radius = 0.1]; 
\draw (1,-20) circle [radius = 0.1]; 
\draw (2,-20) circle [radius = 0.1]; 
\draw (3,-20) circle [radius = 0.1]; 
\draw (4,-20) circle [radius = 0.1]; 
\node [below] at (0.05,-20.15) {$\alpha_0$}; 
\node [below] at (1.05,-20.05) {$\psi_1$}; 
\node [below] at (2.05,-20.05) {$\psi_2$}; 
\node [below] at (3.05,-20.05) {$\psi_3$}; 
\node [below] at (4.05,-20.05) {$\psi_4$}; 
\node [left] at (-0.5,-20) {$\frak{f}_4^{(1)} :$}; 
\draw (0.1,-20) -- (0.9,-20); 
\draw (1.1,-20) -- (1.9,-20); 
\draw (2.9,-20) -- (2.8,-19.9); 
\draw (2.9,-20) -- (2.8,-20.1); 
\draw (2.1,-19.975) -- (2.85,-19.975); 
\draw (2.1,-20.025) -- (2.85,-20.025); 
\draw (3.1,-20) -- (3.9,-20); 

%16
\draw (0,-21) circle [radius = 0.1]; 
\draw (1,-21) circle [radius = 0.1]; 
\draw (2,-21) circle [radius = 0.1]; 
\node [below] at (0.05,-21.15) {$\alpha_0$}; 
\node [below] at (1.05,-21.05) {$\psi_2$}; 
\node [below] at (2.05,-21.05) {$\psi_1$}; 
\node [left] at (-0.5,-21) {$\frak{g}_2^{(1)} :$}; 
\draw (0.1,-21) -- (0.9,-21); 
\draw (1.9,-21) -- (1.8,-20.9); 
\draw (1.9,-21) -- (1.8,-21.1); 
\draw (1.1,-21) -- (1.9,-21); 
\draw (1.1,-20.925) -- (1.8,-20.925); 
\draw (1.1,-21.075) -- (1.8,-21.075); 

\end{tikzpicture}

  Except for conjugation, these are all automorphisms of $\frak{g}$ of order $m$ \cite[Th. 5.15(iii), Ch. X]{helgason}. 

\begin{remark}\label{autoremark} 
{\em 
(i) {\it Let $\delta \in \Delta^+$ be the highest root of $\frak{g}$. Then $\frak{g}_\delta ,\  \frak{g}_{-\delta} \subset \frak{g}_{0}^\nu$, 
 except for $\frak{g} = \frak{a}_{2n} \ (n \ge 1)\ (k=2)$. 
For $\frak{g} = \frak{a}_{2n} \ (n \ge 1)$ with $k=2$,  
$\frak{g}_\delta, \frak{g}_{-\delta} \subset \frak{g}_{\bar{1}}^\nu$. Consequently $\alpha_0 = - \delta \big{|}_{\frak{h}^\nu}$ for 
$\frak{g} = \frak{a}_{2n} \ (n \ge 1)$ with $k=2$} : For $k=1,\  \frak{g}_0^\nu = 
\frak{g}_{\bar{1}}^\nu = \frak{g}$. Then obviously, 
$\frak{g}_\delta , \frak{g}_{-\delta} \subset \frak{g}_{0}^\nu$. For $k=2, \textrm{ or }3$, we prove it via case by case consideration. 
\\
Note that $\nu (\delta) = \delta$ and hence for any $E (\neq 0) \in \frak{g}_\delta $, $\nu (E) = E$ or $\nu (E) = -E$, by the 
definition of $\nu$. 
Thus if $\nu$ is an automorphism of order $3$, then $\frak{g}_\delta \subset \frak{g}_0^\nu$. Similarly 
$\frak{g}_{-\delta} \subset \frak{g}_0^\nu$. \\
Now assume that $k=2$. \\
Let $\frak{g} = \frak{a}_{2n}\ (n \ge 1)$. 

\begin{center}
\begin{tikzpicture} 

\draw (0,0) circle [radius = 0.1]; 
\draw (1,0) circle [radius = 0.1]; 
\draw (2.5,0) circle [radius = 0.1]; 
\draw (3.5,0) circle [radius = 0.1]; 
\draw (5,0) circle [radius = 0.1]; 
\draw (6,0) circle [radius = 0.1]; 
\node [left] at (-0.5,0) {$\frak{a}_{2n} :$}; 
\node [left] at (-0.5,-0.5) {$(n \ge 1)$}; 
\node [below] at (0.025,-0.025) {$\phi_1$}; 
\node [below] at (1.025,-0.025) {$\phi_2$}; 
\node [below] at (2.525,-0.025) {$\phi_n$}; 
\node [below] at (3.575,-0.025) {$\phi_{n+1}$}; 
\node [below] at (5.075,-0.025) {$\phi_{2n-1}$}; 
\node [below] at (6.1,-0.025) {$\phi_{2n}$}; 
\draw (0.1,0) -- (0.9,0); 
\draw (1.1,0) -- (1.5,0); 
\draw[dotted] (1.5,0) -- (2,0); 
\draw (2,0) -- (2.4,0); 
\draw (2.6,0) -- (3.4,0); 
\draw (3.6,0) -- (4,0); 
\draw[dotted] (4,0) -- (4.5,0); 
\draw (4.5,0) -- (4.9,0); 
\draw (5.1,0) -- (5.9,0); 

\end{tikzpicture}
\end{center} 

The highest root $\delta = \phi_1 + \cdots +\phi_{2n}$. Note that $\bar{\nu}$ is given by 
$\bar{\nu}(\phi_j) = \phi_{2n - j +1}$ for all $1 \le j \le 2n$. Let $E_j,\ E_{-j}$ be non-zero 
root vectors corresponding to the roots $\phi_j , \ -\phi_j$ respectively, for all $1 \le j \le 2n$. Then $[E_ n, E_{n-1}, \ldots , E_1], \ 
[E_{n+1}, E_{n+2}, \ldots , E_{2n}] \neq 0$, as $\phi_i + \cdots + \phi_j$ is a root for all $1 \le i < j \le 2n$. Let 
\[E = \big{[}[E_ n, E_{n-1}, \ldots , E_1], [E_{n+1}, E_{n+2}, \ldots , E_{2n}]\big{]}. \]  
Then $E \neq 0, \ E \in \frak{g}_\delta , \ \textrm{and } \nu (E) = - E$. Hence $\frak{g}_\delta  \subset \frak{g}_{\bar{1}}^\nu$. Similarly 
$\frak{g}_{-\delta} \subset \frak{g}_{\bar{1}}^\nu$.  \\
Let $\frak{g} = \frak{a}_{2n-1}\ (n \ge 2)$.  

\begin{center} 
\begin{tikzpicture} 

\draw (0,0) circle [radius = 0.1]; 
\draw (1,0) circle [radius = 0.1]; 
\draw (2.5,0) circle [radius = 0.1]; 
\draw (3.5,0) circle [radius = 0.1]; 
\draw (4.5,0) circle [radius = 0.1]; 
\draw (6,0) circle [radius = 0.1]; 
\draw (7,0) circle [radius = 0.1]; 
\node [left] at (-0.5,0) {$\frak{a}_{2n-1} :$}; 
\node [left] at (-0.5,-0.5) {$(n \ge 2)$}; 
\node [below] at (0.025,-0.025) {$\phi_1$}; 
\node [below] at (1.025,-0.025) {$\phi_2$}; 
\node [below] at (2.575,-0.025) {$\phi_{n-1}$}; 
\node [below] at (3.525,-0.025) {$\phi_n$}; 
\node [below] at (4.575,-0.025) {$\phi_{n+1}$}; 
\node [below] at (6.075,-0.025) {$\phi_{2n-2}$}; 
\node [below] at (7.3,-0.025) {$\phi_{2n-1}$}; 
\draw (0.1,0) -- (0.9,0); 
\draw (1.1,0) -- (1.5,0); 
\draw [dotted] (1.5,0) -- (2,0); 
\draw (2,0) -- (2.4,0); 
\draw (2.6,0) -- (3.4,0); 
\draw (3.6,0) -- (4.4,0); 
\draw (4.6,0) -- (5,0); 
\draw [dotted] (5,0) -- (5.5,0); 
\draw (5.5,0) -- (5.9,0); 
\draw (6.1,0) -- (6.9,0); 

\end{tikzpicture} 
\end{center} 

The highest root $\delta = \phi_1 + \cdots +\phi_{2n-1}$. Note that $\bar{\nu}$ is given by 
$\bar{\nu}(\phi_j) = \phi_{2n - j}$ for all $1 \le j \le 2n-1$. Let $E_j,\ E_{-j}$ be non-zero 
root vectors corresponding to the roots $\phi_j , \ -\phi_j$ respectively, for all $1 \le j \le 2n-1$. Then $[E_{n-1}, \ldots , E_1], \ 
[E_{n+1}, \ldots , E_{2n-1}] \neq 0$. Let 
\[E = \big{[}[E_{n-1}, \ldots , E_1], [E_n, [E_{n+1}, \ldots , E_{2n-1}]]\big{]}. \]  
Then $E \neq 0, \ E \in \frak{g}_\delta , \ \textrm{and } \nu (E) = E$. Hence $\frak{g}_\delta  \subset \frak{g}_0^\nu$. Similarly 
$\frak{g}_{-\delta} \subset \frak{g}_0^\nu$. \\
Let $\frak{g} = \frak{\delta}_{n+1}\ (n \ge 3)$. 

\begin{center} 
\begin{tikzpicture} 

\draw (0,0) circle [radius = 0.1]; 
\draw (1,0) circle [radius = 0.1]; 
\draw (2.5,0) circle [radius = 0.1]; 
\draw (3.5,1) circle [radius = 0.1]; 
\draw (3.5,-1) circle [radius = 0.1]; 
\node [left] at (-0.5,0) {$\frak{\delta}_{n+1} :$}; 
\node [left] at (-0.5,-0.5) {$(n \ge 3)$}; 
\node [below] at (0.025,-0.025) {$\phi_1$}; 
\node [below] at (1.025,-0.025) {$\phi_2$}; 
\node [below] at (2.475,-0.025) {$\phi_{n-1}$}; 
\node [right] at (3.525,1) {$\phi_{n}$}; 
\node [right] at (3.525,-1) {$\phi_{n+1}$}; 
\draw (0.1,0) -- (0.9,0); 
\draw (1.1,0) -- (1.5,0); 
\draw [dotted] (1.5,0) -- (2,0); 
\draw (2,0) -- (2.4,0);
\draw (2.6,0) -- (3.425,0.925); 
\draw (2.6,0) -- (3.425,-0.925); 

\end{tikzpicture} 
\end{center}

The highest root $\delta = \phi_1 + 2\phi_2 + \cdots + 2\phi_{n-1} + \phi_{n} 
+ \phi_{n+1}$. Note that $\bar{\nu}$ is given by 
$\bar{\nu}(\phi_j) = \phi_j$ for all $1 \le j \le n-1$, $\bar{\nu}(\phi_{n}) = \phi_{n+1}$. Let $E_j,\ E_{-j}$ be non-zero 
root vectors corresponding to the roots $\phi_j , \ -\phi_j$ respectively, for all $1 \le j \le n+1$. Then $[E_ 2,\ldots , E_{n-1}, E_{n}], \ 
[E_2, \ldots , E_{n-1}, E_{n+1}] \neq 0$. Let 
\[E = \big{[}[E_ 2,\ldots , E_{n-1}, E_n] , [E_1, [E_2, \ldots , E_{n-1}, E_{n+1}] ]\big{]}. \]  
Then $E \neq 0, \ E \in \frak{g}_\delta , \ \textrm{and } \nu (E) = E$. Hence $\frak{g}_\delta  \subset \frak{g}_{0}^\nu$. Similarly 
$\frak{g}_{-\delta} \subset \frak{g}_{0}^\nu$.  \\
Let $\frak{g} = \frak{e}_6$. 

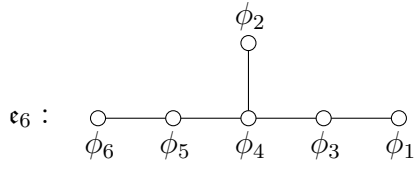
\begin{center} 
\begin{tikzpicture} 

\draw (0,0) circle [radius = 0.1]; 
\draw (1,0) circle [radius = 0.1]; 
\draw (2,0) circle [radius = 0.1]; 
\draw (2,1) circle [radius = 0.1]; 
\draw (3,0) circle [radius = 0.1]; 
\draw (4,0) circle [radius = 0.1]; 
\node [left] at (-0.5,0) {$\frak{e}_6 :$}; 
\node [below] at (0.025,-0.025) {$\phi_6$}; 
\node [below] at (1.025,-0.025) {$\phi_5$}; 
\node [below] at (2.025,-0.025) {$\phi_4$}; 
\node [above] at (2.025,1.025) {$\phi_2$}; 
\node [below] at (3.025,-0.025) {$\phi_3$}; 
\node [below] at (4.025,-0.025) {$\phi_1$}; 
\draw (0.1,0) -- (0.9,0); 
\draw (1.1,0) -- (1.9,0); 
\draw (2,0.1) -- (2,0.9); 
\draw (2.1,0) -- (2.9,0); 
\draw (3.1,0) -- (3.9,0); 

\end{tikzpicture} 
\end{center} 

The highest root $\delta = \phi_1 + 2\phi_2 + 2\phi_3 + 3\phi_4 + + 2\phi_5 + \phi_6$. Note that 
$\bar{\nu}$ is given by 
$\bar{\nu}(\phi_1) = \phi_6, \ \bar{\nu}(\phi_3) = \phi_5, \ \bar{\nu}(\phi_2) = \phi_2, \ \bar{\nu}(\phi_4) = \phi_4$. 
Let $E_j,\ E_{-j}$ be non-zero 
root vectors corresponding to the roots $\phi_j , \ -\phi_j$ respectively, for all $1 \le j \le 6$. Let 
$E'_1 = \big{[}[E_1, E_3, E_4], [E_2, [E_6, E_5, E_4]]\big{]}$. Then $E'_1 \neq 0$ and $\nu (E'_1) = E'_1$. 
Let $E'_2 = [E_5, [E_3, E'_1]]$. Then $E'_2 \neq 0$ and $\nu (E'_2) = E'_2$. Let 
\[E = [E_ 2, E_4, E'_2]. \]  
Then $E \neq 0, \ E \in \frak{g}_\delta , \ \textrm{and } \nu (E) =  E$. Hence $\frak{g}_\delta  \subset \frak{g}_{0}^\nu$. Similarly 
$\frak{g}_{-\delta} \subset \frak{g}_{0}^\nu$.  

(ii) {\it The module $\frak{g}_{\bar{a}}^\nu$ is an irreducible $\frak{g}_0^{\nu}$-module for all $0 \le a \le k-1$} :  It remain to prove that 
 the $\frak{g}_0^\nu$-modules $\frak{g}_{\bar{1}}^\nu$ (for $k=2, 3$) and $\frak{g}_{\bar{2}}^\nu$ (for $k=3$) are irreducible. 
Let $\{ H_\phi ^* \  , E_\alpha : \phi \in \Phi , \alpha \in \Delta \}$ be a Chevalley basis for $\frak{g}$ as in \eqref{chevalley}. As 
$\frak{g}_{\bar{a}}^\nu$ are finite dimensional, $\frak{g}_{\bar{a}}^\nu$ are direct sums of irreducible $\frak{g}_0^\nu$-modules. First 
we show that the module $\frak{g}_{\bar{1}}^\nu$ is irreducible. \\
Recall that $E_0 \ (\neq 0) \in \frak{g}_{(\alpha_0,1)} \subset \frak{g}_{\bar{1}}^\nu$ and dim$(\frak{g}_{(\alpha_0,1)})= 1$. By 
\eqref{7}, any weight space of the $\frak{g}_0^\nu$-module $\frak{g}_{\bar{1}}^\nu$ corresponding to a non-zero weight is generated by $E_0$. 
So if $V$ is the irreducible submodule of $\frak{g}_{\bar{1}}^\nu$ containing $\frak{g}_{(\alpha_0,1)}$, then any 
 weight space of $\frak{g}_{\bar{1}}^\nu$ corresponding to a non-zero weight is contained in $V$. Now we show that $V$ also contains  
the weight space corresponding to the zero weight. If not, then there is a non-zero vector $H_0$ corresponding to the zero weight such that 
$[H_0, \bar{E}_\phi] = 0$ for all $\phi \in \Phi$. Here recall that $\bar{E}_\phi= \sum\limits_{i =0} ^ {k-1} E_{\bar{\nu}^i (\phi)}$ is a 
root vector corresponding to a simple root of $\frak{g}_0^\nu$. 
\\ 
Assume that $k=2$. The weight space of $\frak{g}_{\bar{1}}^\nu$ corresponding to the zero weight is given by 
$\frak{h}_{\bar{1}}^\nu = \sum\limits_{\substack{\phi \in \Phi \\ \phi \neq \bar{\nu}(\phi)}} \mathbb{C} (H_\phi^* - H_{\bar{\nu}(\phi)}^*)$. Now 
\[ [H_\phi^* - H_{\bar{\nu}(\phi)}^* , \ \bar{E}_\psi ] = [H_\phi^* - H_{\bar{\nu}(\phi)}^* , \ E_\psi + E_{\bar{\nu}(\psi)}] =  
(a_{\psi \phi} - a_{\psi \bar{\nu}(\phi)}) (E_\psi - E_{\bar{\nu}(\psi)}), \]
where $a_{\psi \phi}= \psi (H_\phi^*)$ for all $\phi , \psi \in \Phi$. So for $\psi \in \Phi$, if $\psi = \bar{\nu}(\psi)$, then 
$[H , \ \bar{E}_\psi ] = 0$, for all $H \in \frak{h}_{\bar{1}}^\nu$. Note that $H_\phi^* - H_{\bar{\nu}(\phi)}^* = -(H_{\bar{\nu}(\phi)}^* 
 - H_{\bar{\nu}^2(\phi)}^* ), \ (\phi \in \Phi, \ \phi \neq \bar{\nu}(\phi))$. So the vectors $H_\phi^* - H_{\bar{\nu}(\phi)}^*  
\ (\phi \in \Phi, \ \phi \neq \bar{\nu}(\phi))$ are linearly dependent. Choose a maximal linearly independent subset 
$\{ H_{\phi_i}^* - H_{\bar{\nu}(\phi_i)}^* : 1 \le i \le p\}$ in the linearly dependent set 
$\{ H_\phi^* - H_{\bar{\nu}(\phi)}^* :  \phi \in \Phi, \ \phi \neq \bar{\nu}(\phi)\}$ and define $a_{ij} = 
a_{\phi_i \phi_j},\ a_{i\bar{\nu}(j)} = a_{\phi_i \bar{\nu}(\phi_j)}$ for all $1 \le i, j \le p$. Note that $p \le n$, where 
$n = $rank$(\frak{g}_0^\nu)$. Let $H_0 = \sum\limits_{i = 1}^{p} c_i (H_{\phi_i}^* - H_{\bar{\nu}(\phi_i)}^*)$.  Now 
$[H_0 , \bar{E}_{\phi_i}] = 0$ for all $1 \le i \le p$ implies 
\[ \sum\limits_{j=1}^{p}(a_{ij} - a_{i\bar{\nu}(j)})c_j = 0 \textrm{ for all } 1 \le i \le p .\] 
So if the $(p \times p)$ matrix $A = (a_{ij} - a_{i\bar{\nu}(j)})$ is non-singular, then $H_0$ must be zero, which contradicts our assumption. 
So we show that the matrix $A$ is non-singular, via case by case consideration. \\ 
 Let $\frak{g} = \frak{a}_{2n} \ (n \ge 1)$. 

\begin{center}
\begin{tikzpicture} 

\draw (0,0) circle [radius = 0.1]; 
\draw (1,0) circle [radius = 0.1]; 
\draw (2.5,0) circle [radius = 0.1]; 
\draw (3.5,0) circle [radius = 0.1]; 
\draw (5,0) circle [radius = 0.1]; 
\draw (6,0) circle [radius = 0.1]; 
\node [left] at (-0.5,0) {$\frak{a}_{2n} :$}; 
\node [left] at (-0.5,-0.5) {$(n \ge 1)$}; 
\node [below] at (0.025,-0.025) {$\phi_1$}; 
\node [below] at (1.025,-0.025) {$\phi_2$}; 
\node [below] at (2.525,-0.025) {$\phi_n$}; 
\node [below] at (3.575,-0.025) {$\phi_{n+1}$}; 
\node [below] at (5.075,-0.025) {$\phi_{2n-1}$}; 
\node [below] at (6.1,-0.025) {$\phi_{2n}$}; 
\draw (0.1,0) -- (0.9,0); 
\draw (1.1,0) -- (1.5,0); 
\draw[dotted] (1.5,0) -- (2,0); 
\draw (2,0) -- (2.4,0); 
\draw (2.6,0) -- (3.4,0); 
\draw (3.6,0) -- (4,0); 
\draw[dotted] (4,0) -- (4.5,0); 
\draw (4.5,0) -- (4.9,0); 
\draw (5.1,0) -- (5.9,0); 

\end{tikzpicture}
\end{center} 

Here $p = n$, $\frak{h}_{\bar{1}}^\nu = \sum\limits_{i = 1}^{n} \mathbb{C} (H_{\phi_i}^* - H_{\phi_{2n-i+1}}^*)$,  
and for all $1 \le i, j \le n$, we have 
\[ a_{ij} = 
\begin{cases}
2, & \textrm{if } i = j \\
-1, & \textrm{if } |i-j| = 1 \\
0, & \textrm{otherwise};  
\end{cases}
\]
\[ a_{i\bar{\nu}(j)} = 
\begin{cases}
-1, & \textrm{if } i = j= n, \\
0, & \textrm{otherwise}.
\end{cases}
\]
Hence the matrix $A = 
\left ({\begin{array}{cccccc}
2&-1&0&\cdots &0&0 \\
-1&2&-1&\cdots &0&0 \\
\cdots &\cdots &\cdots &\cdots &\cdots &\cdots \\
\cdots &\cdots &\cdots &\cdots &\cdots &\cdots \\
0&0&\cdots &-1&2&-1 \\
0&0&\cdots &0&-1&3 \\
\end{array}
}\right )$, which is non-singular. \\ 
Let $\frak{g} = \frak{a}_{2n-1}\ (n \ge 2)$.  

\begin{center} 
\begin{tikzpicture} 

\draw (0,0) circle [radius = 0.1]; 
\draw (1,0) circle [radius = 0.1]; 
\draw (2.5,0) circle [radius = 0.1]; 
\draw (3.5,0) circle [radius = 0.1]; 
\draw (4.5,0) circle [radius = 0.1]; 
\draw (6,0) circle [radius = 0.1]; 
\draw (7,0) circle [radius = 0.1]; 
\node [left] at (-0.5,0) {$\frak{a}_{2n-1} :$}; 
\node [left] at (-0.5,-0.5) {$(n \ge 2)$}; 
\node [below] at (0.025,-0.025) {$\phi_1$}; 
\node [below] at (1.025,-0.025) {$\phi_2$}; 
\node [below] at (2.575,-0.025) {$\phi_{n-1}$}; 
\node [below] at (3.525,-0.025) {$\phi_n$}; 
\node [below] at (4.575,-0.025) {$\phi_{n+1}$}; 
\node [below] at (6.075,-0.025) {$\phi_{2n-2}$}; 
\node [below] at (7.3,-0.025) {$\phi_{2n-1}$}; 
\draw (0.1,0) -- (0.9,0); 
\draw (1.1,0) -- (1.5,0); 
\draw [dotted] (1.5,0) -- (2,0); 
\draw (2,0) -- (2.4,0); 
\draw (2.6,0) -- (3.4,0); 
\draw (3.6,0) -- (4.4,0); 
\draw (4.6,0) -- (5,0); 
\draw [dotted] (5,0) -- (5.5,0); 
\draw (5.5,0) -- (5.9,0); 
\draw (6.1,0) -- (6.9,0); 

\end{tikzpicture} 
\end{center} 

Here $p = n-1$, $\frak{h}_{\bar{1}}^\nu = \sum\limits_{i = 1}^{n-1} \mathbb{C} (H_{\phi_i}^* - H_{\phi_{2n-i}}^*)$,  
and for all $1 \le i, j \le n-1$, we have 
\[ a_{ij} = 
\begin{cases}
2, & \textrm{if } i = j \\
-1, & \textrm{if } |i-j| = 1 \\
0, & \textrm{otherwise};  
\end{cases}
\]
\[ \textrm{and }a_{i\bar{\nu}(j)} = 0 \textrm{ always}.\]
Hence the matrix $A = 
\left ({\begin{array}{cccccc}
2&-1&0&\cdots &0&0 \\
-1&2&-1&\cdots &0&0 \\
\cdots &\cdots &\cdots &\cdots &\cdots &\cdots \\
\cdots &\cdots &\cdots &\cdots &\cdots &\cdots \\
0&0&\cdots &-1&2&-1 \\
0&0&\cdots &0&-1&2 \\
\end{array}
}\right )$, which is non-singular. \\ 
Let $\frak{g} = \frak{\delta}_{n+1}\ (n \ge 3)$. 

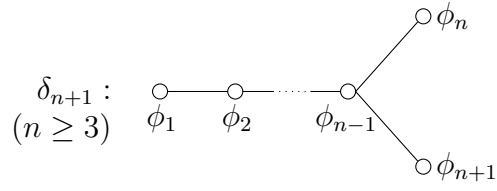
\begin{center} 
\begin{tikzpicture} 

\draw (0,0) circle [radius = 0.1]; 
\draw (1,0) circle [radius = 0.1]; 
\draw (2.5,0) circle [radius = 0.1]; 
\draw (3.5,1) circle [radius = 0.1]; 
\draw (3.5,-1) circle [radius = 0.1]; 
\node [left] at (-0.5,0) {$\frak{\delta}_{n+1} :$}; 
\node [left] at (-0.5,-0.5) {$(n \ge 3)$}; 
\node [below] at (0.025,-0.025) {$\phi_1$}; 
\node [below] at (1.025,-0.025) {$\phi_2$}; 
\node [below] at (2.475,-0.025) {$\phi_{n-1}$}; 
\node [right] at (3.525,1) {$\phi_{n}$}; 
\node [right] at (3.525,-1) {$\phi_{n+1}$}; 
\draw (0.1,0) -- (0.9,0); 
\draw (1.1,0) -- (1.5,0); 
\draw [dotted] (1.5,0) -- (2,0); 
\draw (2,0) -- (2.4,0);
\draw (2.6,0) -- (3.425,0.925); 
\draw (2.6,0) -- (3.425,-0.925); 

\end{tikzpicture} 
\end{center}

Here $p = 1$, $\frak{h}_{\bar{1}}^\nu = \mathbb{C} (H_{\phi_n}^* - H_{\phi_{n+1}}^*)$, and the matrix $A = (2)$, obviously 
non-singular. \\
Let $\frak{g} = \frak{e}_6$. 

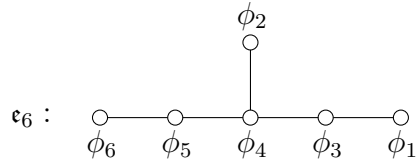
\begin{center} 
\begin{tikzpicture} 

\draw (0,0) circle [radius = 0.1]; 
\draw (1,0) circle [radius = 0.1]; 
\draw (2,0) circle [radius = 0.1]; 
\draw (2,1) circle [radius = 0.1]; 
\draw (3,0) circle [radius = 0.1]; 
\draw (4,0) circle [radius = 0.1]; 
\node [left] at (-0.5,0) {$\frak{e}_6 :$}; 
\node [below] at (0.025,-0.025) {$\phi_6$}; 
\node [below] at (1.025,-0.025) {$\phi_5$}; 
\node [below] at (2.025,-0.025) {$\phi_4$}; 
\node [above] at (2.025,1.025) {$\phi_2$}; 
\node [below] at (3.025,-0.025) {$\phi_3$}; 
\node [below] at (4.025,-0.025) {$\phi_1$}; 
\draw (0.1,0) -- (0.9,0); 
\draw (1.1,0) -- (1.9,0); 
\draw (2,0.1) -- (2,0.9); 
\draw (2.1,0) -- (2.9,0); 
\draw (3.1,0) -- (3.9,0); 

\end{tikzpicture} 
\end{center} 

Here $p = 2$, $\frak{h}_{\bar{1}}^\nu =\mathbb{C} (H_{\phi_1}^* - H_{\phi_6}^*) \oplus \mathbb{C} 
(H_{\phi_3}^* - H_{\phi_5}^*)$, and the matrix $A = \left ({
\begin{array}{cc} 
2&-1 \\
-1&2 \\
\end{array}
}\right)$, which is non-singular. \\ 
Now let $k=3$. Then $\frak{g} = \frak{\delta}_{4}$. 

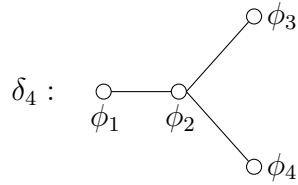
\begin{center} 
\begin{tikzpicture}

\draw (0,0) circle [radius = 0.1]; 
\draw (1,0) circle [radius = 0.1]; 
\draw (2,1) circle [radius = 0.1]; 
\draw (2,-1) circle [radius = 0.1]; 
\node [left] at (-0.5,0) {$\frak{\delta}_4 :$}; 
\node [below] at (0.025,-0.025) {$\phi_1$}; 
\node [below] at (1.025,-0.025) {$\phi_2$};  
\node [right] at (2.025,1) {$\phi_3$}; 
\node [right] at (2.025,-1) {$\phi_4$}; 
\draw (0.1,0) -- (0.9,0); 
\draw (1.1,0) -- (1.925,0.925); 
\draw (1.1,0) -- (1.925,-0.925); 

\end{tikzpicture} 
\end{center} 

The zero weight space of $\frak{g}_{\bar{1}}^\nu$ is given by $\frak{h}_{\bar{1}}^\nu = 
\mathbb{C} (H_{\phi_1}^* +\epsilon_0^2 H_{\bar{\nu}(\phi_1)}^* + \epsilon_0 H_{\bar{\nu}^2(\phi_1)}^*)$, 
$\epsilon_0 = e^{\frac{2\pi i}{3}}$. Now $\bar{E}_{\phi_1} = E_{\phi_1} + E_{\bar{\nu}(\phi_1)} + E_{\bar{\nu}^2(\phi_1)}$ is 
a non-zero root vector of $\frak{g}_0^\nu$, and 
\[ [ (H_{\phi_1}^* +\epsilon_0^2 H_{\bar{\nu}(\phi_1)}^* + \epsilon_0 H_{\bar{\nu}^2(\phi_1)}^*, 
E_{\phi_1} + E_{\bar{\nu}(\phi_1)} + E_{\bar{\nu}^2(\phi_1)}] = 2 ( E_{\phi_1} + \epsilon_0^2 E_{\bar{\nu}(\phi_1)} + 
\epsilon_0 E_{\bar{\nu}^2(\phi_1)}) \neq 0. \]
Hence $\frak{h}_{\bar{1}}^\nu$ is not invariant under $\frak{g}_0^\nu$ and so it is contained in $V$. Now we show that 
 $\frak{g}_{\bar{2}}^\nu$ for $\frak{g} = \frak{\delta_4}$, is irreducible. In this case, $\frak{g}_0^\nu = \frak{g}_2$. 

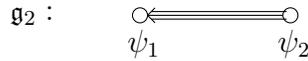
\begin{center}
\begin{tikzpicture}[scale = 2]

\node [left] at (-0.5,0) {$\frak{g}_2 :$}; 
\draw (0,0) circle [radius = 0.05]; 
\draw (1,0) circle [radius = 0.05]; 
\node [below] at (0.025,-0.025) {$\psi_1$}; 
\node [below] at (1.025,-0.025) {$\psi_2$}; 
 \draw (0.05,0) -- (0.1,0.05);
\draw (0.05,0) -- (0.1,-0.05); 
\draw (0.05,0) -- (0.95,0); 
\draw (0.075,0.025) -- (0.975,0.025); 
\draw (0.075,-0.025) -- (0.975,-0.025); 
 
\end{tikzpicture}
\end{center}

Note that $\bar{E}_{\phi_1} = E_{\phi_1} + E_{\phi_3} + E_{\phi_4}$ and $\bar{E}_{\phi_2} = 3 E_{\phi_2}$ are root vectors of 
$\frak{g}_0^\nu$ corresponding to the roots $\psi_1$ and $\psi_2$ respectively. Then $E_{\phi_1} + \epsilon_0 E_{\bar{\nu}(\phi_1)} 
+ \epsilon_0^2 E_{\bar{\nu}^2(\phi_1)}$, $E_{-\phi_1} + \epsilon_0 E_{-\bar{\nu}(\phi_1)} + 
\epsilon_0^2 E_{-\bar{\nu}^2(\phi_1)}$  are weight vectors of $\frak{g}_{\bar{2}}^\nu$ corresponding to the weights $\psi_1$, 
$-\psi_1$  respectively. Clearly, 
\[ [\bar{E}_{-\phi_1}, E_{\phi_1} + \epsilon_0 E_{\bar{\nu}(\phi_1)} + \epsilon_0^2 E_{\bar{\nu}^2(\phi_1)}] = 
-(H_{\phi_1}^* +\epsilon_0 H_{\bar{\nu}(\phi_1)}^* + \epsilon_0^2 H_{\bar{\nu}^2(\phi_1)}^*). \]
\[ [\bar{E}_{-\phi_1}, -(H_{\phi_1}^* +\epsilon_0 H_{\bar{\nu}(\phi_1)}^* + \epsilon_0^2 H_{\bar{\nu}^2(\phi_1)}^*)] = 
-2 (E_{-\phi_1} + \epsilon_0 E_{-\bar{\nu}(\phi_1)} + \epsilon_0^2 E_{-\bar{\nu}^2(\phi_1)}).\]
Also $[\bar{E}_{\phi_2},  E_{\phi_1} + \epsilon_0 E_{\bar{\nu}(\phi_1)} + \epsilon_0^2 E_{\bar{\nu}^2(\phi_1)}] \neq 0,\ 
[\bar{E}_{\phi_1}, \bar{E}_{\phi_2},  E_{\phi_1} + \epsilon_0 E_{\bar{\nu}(\phi_1)} + \epsilon_0^2 E_{\bar{\nu}^2(\phi_1)}] \neq 0, \ 
[\bar{E}_{-\phi_2},  E_{-\phi_1} + \epsilon_0 E_{-\bar{\nu}(\phi_1)} + \epsilon_0^2 E_{-\bar{\nu}^2(\phi_1)}] \neq 0,\ 
[\bar{E}_{-\phi_1}, \bar{E}_{-\phi_2},  E_{-\phi_1} + \epsilon_0 E_{-\bar{\nu}(\phi_1)} + \epsilon_0^2 E_{-\bar{\nu}^2(\phi_1)}] \neq 0$. 
These are weight vectors of $\frak{g}_{\bar{2}}^\nu$ corresponding to the weights $\psi_1 + \psi_2,\ 2\psi_1 + \psi_2,\ -\psi_1 - \psi_2,\ 
-2\psi_1 - \psi_2$ respectively. As dim$(\frak{g}_{\bar{2}}^\nu)=7$, $\frak{g}_{\bar{2}}^\nu$ is generated by 
$E_{\phi_1} + \epsilon_0 E_{\bar{\nu}(\phi_1)} + \epsilon_0^2 E_{\bar{\nu}^2(\phi_1)}$ as a $\frak{g}_0^\nu$-module. Hence 
$\frak{g}_{\bar{2}}^\nu$ is irreducible. Note that the lowest weight of $\frak{g}_{\bar{2}}^\nu$ is $-2\psi_1 - \psi_2 = \alpha_0$, as 
$\alpha_0 +2\psi_1 + \psi_2 = 0$, \cite[Tables of Diagrams $S(A)$, Ch. X]{helgason}. Hence $\frak{g}_{\bar{2}}^\nu \cong 
\frak{g}_{\bar{1}}^\nu$, as $\frak{g}_0^\nu$-modules. 

(iii) {\it Except for conjugation, $\{\sigma, \sigma \theta : \sigma \textrm{ is defined as in } \eqref{sigma} \}$ are 
all automorphisms of $\frak{g}^\mathbb{R}$ of order $m$ and leave $\frak{u}$ invariant} : 
Since $\frak{g}$ is simple, Aut$(\frak{g})$ is a subgroup of Aut$(\frak{g}^\mathbb{R})$ of index $2$. Hence 
$\textrm{Aut}(\frak{g}^\mathbb{R}) = \textrm{Aut}(\frak{g})  \cup \textrm{Aut}(\frak{g}) \theta$. So it is sufficient to prove that if 
$\sigma_1, \sigma_2$ are automorphisms of $\frak{g}$ of order $m$ such that these are conjugate in Aut$(\frak{g})$ and leave $\frak{u}$ 
invariant, then $\sigma_1, \sigma_2$  are conjugate in Aut$(\frak{u})$. To prove this, we follow the argument of \cite[Prop. 1.4, Ch. X]
{helgason}. Let $g \in \textrm{Aut}(\frak{g})$ be such that $\sigma_2 = g\sigma_1 g^{-1}$. Now $g\frak{u}$ is also a compact real form 
of $\frak{g}$. So there exists $g_0 \in \textrm{Int}(\frak{g})$ such that $g\frak{u} = g_0 \frak{u}$, where Int$(\frak{g})$ is the identity 
component of Aut$(\frak{g}^\mathbb{R})$. Hence $g_0^{-1}g \in \tilde{U}$, the normaliser of $\frak{u}$ in Aut$(\frak{g}^\mathbb{R})$. 
So we can write $g$ as $g = pu$, where $p \in \textrm{exp}(J\frak{u}),\ u \in \tilde{U}$. Thus $\sigma_2 = pu\sigma_1 u^{-1}p^{-1}\ 
(\sigma_1, \sigma_2, u \in \tilde{U})$. Let $\tilde{\theta}$ be the Lie group automorphism of Aut$(\frak{g}^\mathbb{R})$ given by 
$\tilde{\theta}(\sigma) = \theta \sigma \theta^{-1}$. Now applying $\tilde{\theta}$ on both sides of equation 
$\sigma_2 = pu\sigma_1 u^{-1}p^{-1}$, we have $\sigma_2 = p^{-1}u\sigma_1 u^{-1}p$. This implies $Ap^2 A^{-1} = p^2$, where 
$A = u \sigma_1 u^{-1}$. Let $p = \textrm{exp}(JX)\ (X \in \frak{u})$. Now $Ap^2 A^{-1} = p^2$ implies 
exp$(2J\textrm{Ad}(A)(X)) =$ exp$(2JX)$. As exp is one-to-one on $J\frak{u}$, we have Ad$(A)(X) = X$ and so $A$ commutes with 
exp$(JX) = p$. Hence $\sigma_2 = pu\sigma_1 u^{-1}p^{-1} = u\sigma_1 u^{-1}$, and $\sigma_1, \sigma_2$  are conjugate 
in Aut$(\frak{u})$. 
}
\end{remark} 

\noindent 
\subsection{ The condition Or}\label{op} 

  Let $G$ be a connected complex simple Lie group and $U$ be a maximal compact subgroup of $G$. Let Lie$(G) = \frak{g}$, Lie$(U) 
= \frak{u}$, and $\theta$ be the Cartan involution corresponding to the Cartan decomposition $\frak{g}^\mathbb{R} = 
\frak{u} \oplus J\frak{u}$, 
where $J$ denotes, as usual, the complex structure of $\frak{g}^\mathbb{R}$ corresponding to the multiplication by $i$ of $\frak{g}$. 
Let $\bar{\theta}$ denote the corresponding Cartan involution of $G$. Let $\frak{t}$ be a maximal abelian subspace of $\frak{u}$ 
and $\frak{h} = \frak{t}^\mathbb{C}$. Then $\frak{h}$ is a Cartan subalgebra of $\frak{g}$. Choose a system of positive roots 
 $\Delta^+ $ in the set of all non-zero roots $\Delta = \Delta(\frak{g}, \frak{h})$. Let $\Phi$ be the set of simple roots in $\Delta^+$. Let 
$\{ H_\phi ^* \  , E_\alpha : \phi \in \Phi , \alpha \in \Delta \}$ be a Chevalley basis for $\frak{g}$ as in  \eqref{chevalley}. 
Then 
\[ \frak{u} =  \sum_{\phi \in \Phi} \mathbb{R} (i  H_\phi ^*) \oplus \sum_{\alpha \in \Delta^+} \mathbb{R} X_\alpha \oplus 
\sum_{\alpha \in \Delta^+} \mathbb{R} Y_\alpha, \]
where $X_\alpha = E_\alpha - E_{-\alpha} , Y_\alpha = i(E_\alpha + E_{-\alpha})$ for all $\alpha \in \Delta^+$. 

Let $\bar{\sigma}$ be an involution of $G$ whose differential at identity is an automorphism $\sigma$ of $\frak{g}$ of order $2$ 
as in \eqref{sigma}. Recall that $\sigma (\frak{u}) = \frak{u}$. Let $\frak{u} = \frak{u}_0 \oplus \frak{u}_1$, $\frak{g} = 
\frak{g}_0 \oplus \frak{g}_1$ be the decompositions of $\frak{u}, \frak{g}$ into $1$ and $-1$ eigenspaces of $\sigma$ respectively. 
Note that $U$ is invariant under $\bar{\sigma}$. 
Let $G(\mu) = \{ g \in G : \mu(g) = g\}$ and $U(\mu) = \{ u \in U : \mu(u) = u\}$, where $\mu = \bar{\sigma}, \bar{\sigma}\bar{\theta}$. 
Then $G(\mu)$ is a closed reductive subgroup of $G$ and $U(\mu)$ is a maximal compact subgroup of $G(\mu)$. $X(\mu) = 
G(\mu)/U(\mu)$ is a Riemannian globally symmetric space of non-compact type. Note that $X(\bar{\sigma}\bar{\theta})$ is an irreducible 
Riemannian globally symmetric space of type III. For our purpose, it is important to know that when the canonical action of 
$G(\mu)$ on $X(\mu)$ is orientation preserving for $\mu = \bar{\sigma}, \bar{\sigma}\bar{\theta}$. We proceed as follows: 

  Note that $G(\bar{\sigma}) = U(\bar{\sigma})\textrm{exp}(J\frak{u}_0)$ and $G(\bar{\sigma}\bar{\theta}) = 
U(\bar{\sigma})\textrm{exp}(J\frak{u}_1)$. So it is sufficient to check whether the canonical action of $U(\bar{\sigma})$ on $X(\mu)$ is 
orientation preserving. If $\it o = U(\bar{\sigma})$ is the identity coset in $X(\mu)$, then $U(\bar{\sigma})(\it o) = \it o$ and the differential  
of this action is given by $\textrm{Ad} :  U(\bar{\sigma}) \longrightarrow T_{\it o}(X(\mu))$. Hence it is sufficient to check whether 
det$(\textrm{Ad}(u)|_{i\frak{u}_k})= 1$ for all $u \in  U(\bar{\sigma})$, where $k=0, 1$. 

  Let $\tilde{U}$ be the simply connected Lie group with Lie algebra $\frak{u}$ and $p : \tilde{U} \longrightarrow U$ be the covering projection 
whose differential is the identity map of $\frak{u}$. Let $\tilde{Z}$ denote the centre of $\tilde{U}, S = \textrm{ker}(p) 
\subset \tilde{Z}$, and $\tilde{\sigma}$ be the automorphism of $\tilde{U}$ with $d\tilde{\sigma} = (d\tilde{\sigma})_e = 
\sigma|_{\frak{u}}$. Then $\tilde{U}(\tilde{\sigma})$, the set of fixed points of $\tilde{\sigma}$ is connected 
\cite[Th. 8.2, Ch. VII]{helgason}. Let $L = p^{-1}(U(\bar{\sigma})) = \{ u \in \tilde{U} : \bar{\sigma}p(u) = p(u) \} = 
\{ u \in \tilde{U} : p\tilde{\sigma}(u) = p(u) \} = \{ u \in \tilde{U} : p(\tilde{\sigma}(u)u^{-1}) =  e \} = \{ u \in \tilde{U} : 
\tilde{\sigma}(u)u^{-1} \in S \}$. Then $L$ is a closed subgroup $\tilde{U}$ and $\tilde{U}(\tilde{\sigma})$ is the connected component of 
$L$. Also note that $\tilde{U}(\tilde{\sigma})S \subset L$. If $\tilde{U}(\tilde{\sigma})S = L$, then $U(\bar{\sigma}) = p(L) = 
p(\tilde{U}(\tilde{\sigma}))$, hence connected. But it may happen that $\tilde{U}(\tilde{\sigma})S \subset L$ but $\tilde{U}(\tilde{\sigma})S 
 \neq L$. Since the covering projection $p$ is orientation preserving, we need to check that det$(\textrm{Ad}(u)|_{i\frak{u}_k})= 1$ 
for all $u \in L$, where $k=0, 1$.

  Let $\frak{a}$ be a maximal abelian subspace of $\frak{u}_1$. For any $u \in \tilde{U}$, there exist $u_1, u_2 \in \tilde{U}(\tilde{\sigma})$ 
and $X \in \frak{a}$ such that $u = \textrm{exp}(\textrm{Ad}(u_1)(X))u_2$ \cite[Th. 8.6, Ch. VII]{helgason}. Now \\
$u \in L \Leftrightarrow \tilde{\sigma}(u)u^{-1} \in S \Leftrightarrow 
\textrm{exp}(\textrm{Ad}(u_1)(-X))u_2 u_2^{-1}\textrm{exp}(\textrm{Ad}(u_1)(-X)) = \textrm{exp}(\textrm{Ad}(u_1)(-2X)) \in S
\Leftrightarrow u_1 \textrm{exp}(-2X) u_1^{-1} \in S \Leftrightarrow \textrm{exp}(-2X) \in S, \textrm{as } S\subset \tilde{Z}$. \\
To check whether det$(\textrm{Ad}(u)|_{i\frak{u}_k})= 1$ for all $u \in L$, where $k=0, 1$; it is sufficient to check whether 
det$(\textrm{Ad}(\textrm{exp}(X))|_{i\frak{u}_k})= 1$ for all $X \in \frak{a}$ with $\textrm{exp}(-2X) \in S$, where $k=0, 1$. Now  
det$(\textrm{Ad}(u)|_{i\frak{u}_0})$ det$(\textrm{Ad}(u)|_{i\frak{u}_1})=$ det$(\textrm{Ad}(u)|_{i\frak{u}})= 1$ for all $u \in L$. 
So it is sufficient to check whether det$(\textrm{Ad}(\textrm{exp}(X))|_{\frak{u}_0})= 1$ for all $X \in \frak{a}$ with 
$\textrm{exp}(-2X) \in S$. 

   Let the Dynkin diagram automorphism induced by $\sigma$ be $\bar{\nu}$. 
Let $\{\gamma_1, \gamma_2, \ldots , \gamma_r\}$ be a maximal set of strongly orthogonal roots in 
$\{\alpha \in \Delta^+ : \sigma(H_\alpha ^*) = H_\alpha ^* ,\ \frak{g}_\alpha \subset \frak{g}_1 \}$. Then 
$\frak{a} = \sum\limits_{\substack{\phi \in \Phi \\ \bar{\nu}(\phi) \neq \phi }}  \mathbb{R} i(H_\phi ^* - H_{\bar{\nu}(\phi)}^*) \oplus 
\sum\limits_{j =1}^{r} \mathbb{R} Y_{\gamma_j}$ is a maximal abelian subspace of $\frak{u}_1$. Let $c = \textrm{Ad}(\textrm{exp}
(\frac{\pi}{4}\sum\limits_{j=1}^{r} X_{\gamma_j}))$. Then $c(Y_{\gamma_j}) = iH_{\gamma_j}^*$ for all $1 \le j \le r$ and 
$c(H_\phi ^* - H_{\bar{\nu}(\phi)}^*) = H_\phi ^* - H_{\bar{\nu}(\phi)}^*$ for all $\phi \in \Phi$. So $c(\frak{a}) = 
 \sum\limits_{\substack{\phi \in \Phi \\ \bar{\nu}(\phi) \neq \phi }}  \mathbb{R} i(H_\phi ^* - H_{\bar{\nu}(\phi)}^*) \oplus 
\sum\limits_{j =1}^{r} \mathbb{R}(iH_{\gamma_j}^*)$. Let $\frak{a}^\perp$ be the orthogonal complement of $c(\frak{a})$ in $\frak{t}$
with respect to the positive definite symmetric bilinear form $-B (H, H') (H, H' \in \frak{t})$. 
Since $\gamma_j (H) = 0$ for all $H \in \frak{a}^\perp$ and for all $1 \le j \le r$, we have $c(H) = H$ for all $H \in \frak{a}^\perp$. 
Hence if $\frak{t}' = \frak{a}^\perp \oplus \frak{a}$, then $c(\frak{t}') = \frak{a}^\perp \oplus c(\frak{a})= \frak{t}$. For $X \in \frak{a}, 
\textrm{exp}(-2X) \in \tilde{Z} \Leftrightarrow \alpha'(-2X) \in 2\pi i \mathbb{Z}$ for all 
$\alpha' \in \Delta (\frak{g}, \frak{t}'^{\mathbb{C}})$ \cite[Th. 6.7, Ch.VII]{helgason} $\Leftrightarrow \alpha(c(X)) \in \pi i \mathbb{Z}$ 
for all $\alpha \in \Delta$. So if $X = iH + \sum\limits_{j=1}^{r} c_j Y_{\gamma_j} (H \in \sum\limits_{\substack
{\phi \in \Phi \\ \bar{\nu}(\phi) \neq \phi }}  \mathbb{R} (H_\phi ^* - H_{\bar{\nu}(\phi)}^*),\  c_j \in \mathbb{R})$, then 
$\textrm{exp}(-2X) \in \tilde{Z} \Leftrightarrow \alpha (H) + \sum\limits_{j = 1}^{r} c_j \alpha (H_{\gamma_j}^*) \in \pi \mathbb{Z}$ for all 
$\alpha \in \Delta \Leftrightarrow \phi (H) + \sum\limits_{j = 1}^{r} c_j \phi (H_{\gamma_j}^*) \in \pi \mathbb{Z}$ for all $\phi \in \Phi$. In 
particular, we have $\textrm{exp}(-2X) \in \tilde{Z}$ implies $c_j \in \frac{\pi}{2} \mathbb{Z}$ (taking $\alpha = \gamma_j$) for all 
$1 \le j \le r$. 

   For $X \in \frak{a}$ with $\textrm{exp}(-2X) \in S$, let $\sigma_X = \textrm{exp}(X)$. Then $\sigma_X^{-2} \in \tilde{Z}$ and so 
$\textrm{Ad}(\sigma_X)|_{\frak{u}_0} : \frak{u}_0 \longrightarrow \frak{u}_0$ is an involution. Note that 
$\frak{t}_0 = \sum\limits_{\phi \in \Phi} \mathbb{R} i(H_\phi ^* + H_{\bar{\nu}(\phi)}^*)$ is a maximal abelian subspace of $\frak{u}_0$. 
Now $\frak{t}^- = \sum\limits_{j =1}^{r} \mathbb{R}(iH_{\gamma_j}^*) \subset \frak{t}_0$. Let 
$\frak{t}^+ = \{ H \in \frak{t}_0 : \gamma_j (H) =0 \textrm{ for all } 1 \le j \le r \}$. 
So Ad$(\sigma_X) (H) = H$ for all $H \in \frak{t}^+$. Now \\
$[X, iH_{\gamma_j}^*] = c_j [Y_{\gamma_j}, iH_{\gamma_j}^*] = 2c_j  X_{\gamma_j}, \  
 [X,  X_{\gamma_j}] = c_j [Y_{\gamma_j}, X_{\gamma_j}] = -2c_j iH_{\gamma_j}^*$. Hence \\
Ad$(\sigma_X)(iH_{\gamma_j}^*)= (\cos 2c_j) iH_{\gamma_j}^* + (\sin 2c_j)X_{\gamma_j} = (\cos 2c_j) iH_{\gamma_j}^*$ 
for all $1 \le j \le r$.  \\ 
So Ad$(\sigma_X)(\frak{t}_0) = \frak{t}_0$. Let $\Delta_0 = \Delta(\frak{g}_0,\frak{t}_0^{\mathbb{C}})$ and 
$\Delta_0^+$ be a system of positive roots in $\Delta_0$.  For $X \in \frak{a}$ with $\textrm{exp}(-2X) \in S$, 
let $s'_X \in W(\frak{g}_0,\frak{t}_0^{\mathbb{C}})$ be 
such that Ad$(\sigma_X) \circ s'_X (\Delta_0^+ ) = \Delta_0^+$ and $s_X = \textrm{Ad}(\sigma_X)\circ s'_X$. For $\alpha \in \Delta_0$,  
choose $\bar{E}_\alpha \in (\frak{g}_0)_{\alpha}$ such that $\frak{u}_0 = \frak{t}_0 \oplus \sum\limits_{\alpha \in \Delta_0^+} 
\mathbb{R}(\bar{E}_\alpha - \bar{E}_{-\alpha}) \oplus \sum\limits_{\alpha \in \Delta_0^+} \mathbb{R} 
i(\bar{E}_\alpha + \bar{E}_{-\alpha})$ and $s_X (\bar{E}_\alpha) = a_\alpha \bar{E}_{\alpha'} (s_X(\alpha') = \alpha ,\  a_\alpha \in 
\mathbb{C})$ with $a_\alpha a_{-\alpha} = 1$ and $|a_\alpha| = 1$ \cite[Cor. 5.2, Ch. IX]{helgason}. For $\alpha \in \Delta_0^+$, if 
$\bar{X}_\alpha =  \bar{E}_\alpha - \bar{E}_{-\alpha}$ and $\bar{Y}_\alpha = i(\bar{E}_\alpha + \bar{E}_{-\alpha})$, then 
\[ s_X ( \bar{X}_\alpha) = x_\alpha \bar{X}_{\alpha'} + y_\alpha \bar{Y}_{\alpha'} \]
\[ s_X ( \bar{Y}_\alpha) = -y_\alpha \bar{X}_{\alpha'} + x_\alpha \bar{Y}_{\alpha'}, \]
where $a_\alpha = x_\alpha +i y_\alpha,\ x_\alpha , y_\alpha \in \mathbb{R}$. Then det$(s_X\Big{|}_{\mathbb{R}\bar{X}_\alpha + 
\mathbb{R}\bar{Y}_\alpha}) = x_\alpha ^2 + y_\alpha ^2 = |a_\alpha|^2 = 1$ for all $\alpha \in \Delta_0^+$. So it is sufficient to check 
whether det$(s_X|_{\frak{t}_0^{\mathbb{C}}}) = 1$ for all $X \in \frak{a}$ with $\textrm{exp}(-2X) \in S$. 

   Recall that the Dynkin diagram automorphism $\bar{\nu}$ of $\frak{g}$ induced by $\sigma$ has order $k \ (k=1 \textrm{ or } 2)$. Let 
$\nu$ be an automorphism of $\frak{g}$ induced by $\bar{\nu}$, as in \S \ref{victor}. Then $\nu$ has order $k$ and we have an 
$\mathbb{Z}_k$-gradation of $\frak{g}$ as $\frak{g} = \mathop{\oplus}\limits_{i \in \mathbb{Z}_k} \frak{g}_i ^\nu $ into the 
eigenspaces of $\nu$. Recall that $\frak{g}_0^\nu$ is a simple Lie algebra and $\frak{h}^\nu = \frak{h} \cap \frak{g}_0^\nu$ is a Cartan 
subalgebra of $\frak{g}_0^\nu$. Also the set $\Phi$ of simple roots in $\Delta^+$ induces a basis $\Psi = \{ \psi_1, \psi_2, \ldots , \psi_n \}$ 
(determined by the map $\bar{\nu} : \Phi \longrightarrow \Phi$) of the root system $\Delta(\frak{g}_0^\nu , \frak{h}^\nu )$ 
\cite[the proof of Lemma 5.11, Ch. X]{helgason}. Let $\alpha_0$ be the lowest weight of the $\frak{g}_0^\nu$-module $\frak{g}_{\bar{1}}^\nu $. 
Note that if $k =1$ that is, if $\bar{\nu}$ is the identity map, then $\alpha_0 = - \delta$, where $\delta$ is the highest root of $\Delta^+$.  
Let $\alpha_0 + \sum\limits_{i=1}^{n} a_i \psi_i =0 \ (a_i \in \mathbb{N} \textrm{ for all } 1 \le i \le n)$. Then as in \eqref{sigma}, there are 
non-zero integers $s_0, s_1, \ldots , s_n$ without non-trivial common factor such that $2 = k (s_0 + \sum\limits_{i = 1}^{n} a_i s_i )$ and 
$\sigma$ is an involution of $\frak{g}$ of type $(s_0, s_1, \ldots , s_n; k)$. Let $i_1 , \ldots, i_t$ be all the indices with $s_{i_1} = \cdots = s_{i_t}=0$. 
Then the Lie algebra $\frak{g}_0 = \{ X \in \frak{g} : \sigma (X) = X \}$ is the direct sum of an $(n-t)$-dimensional centre and 
a semisimple Lie algebra whose Dynkin diagram is the subdiagram of the diagram $\frak{g}^{(k)}$ (given in \S \ref{victor}) consisting of the vertices 
$\psi_{i_1} , \ldots , \psi_{i_t}$. From now on we assume that $\Delta_0^+$ is the system of positive roots in $\Delta_0 = \Delta (\frak{g}_0 , 
\frak{t}_0^\mathbb{C})$ corresponding to the basis $\{\psi_{i_1}, \psi_{i_2}, \ldots , \psi_{i_t} \}$. 

\begin{remark}\label{or}
{\em 
(i) {\it We may choose $\bar{E}_\alpha \in (\frak{g}_0)_{\alpha}$ such that $\frak{u}_0 = \frak{t}_0 \oplus \sum\limits_{\alpha 
\in \Delta_0^+} \mathbb{R}(\bar{E}_\alpha - \bar{E}_{-\alpha}) \oplus \sum\limits_{\alpha \in \Delta_0^+} \mathbb{R} 
i(\bar{E}_\alpha + \bar{E}_{-\alpha})$ and $s_X (\bar{E}_\alpha) = a_\alpha \bar{E}_{\alpha'} (s_X (\alpha') = \alpha ,\  a_\alpha \in 
\mathbb{C})$ with $a_\alpha a_{-\alpha} = 1$ and $a_\alpha = \pm 1$} :  For $X \in \frak{a}$ with $\textrm{exp}(-2X) \in S$, 
$\textrm{Ad}(\sigma_X)|_{\frak{u}_0}$ is an involution and 
$s'_X \in W(\frak{g}_0, \frak{t}_0^{\mathbb{C}})$ be such that Ad$(\sigma_X) \circ s'_X (\Delta_0^+ ) = \Delta_0^+$. So $s'_X$ is also an 
involution. Hence $s_X = \textrm{Ad}(\sigma_X)\circ s'_X : \frak{g}_0 \longrightarrow \frak{g}_0$ is an involution with 
$s_X (\Delta_0^+) = \Delta_0^+$. Now the result follows from \cite[Lemma 3.5]{Borel}. 

(ii) {\it If $\frak{g} = \frak{e}_8, \frak{f}_4, \textrm{or }\frak{g}_2$, then the canonical action of $G(\mu)$ on $X(\mu)$ is 
orientation preserving for $\mu = \bar{\sigma}, \bar{\sigma}\bar{\theta}$} : In these cases, the simply connected Lie group $\tilde{U}$ 
has trivial centre, that is $\tilde{Z} = \{e\}$ \cite[Cor. 7.8, Ch. VII, Lemma 3.30 and Th. 3.32, Ch. X]{helgason}. So the result follows. 

(iii) {\it If $\frak{a} = \sum\limits_{\substack{\phi \in \Phi \\ \bar{\nu}(\phi) \neq \phi }}  \mathbb{R} i(H_\phi ^* - H_{\bar{\nu}(\phi)}^*)$ is 
a maximal abelian subspace of $\frak{u}_1$ (this happens exactly when the Dynkin diagram automorphism $\bar{\nu}$ has order $2$ and 
$\sigma = \nu$, the automorphism of $\frak{g}$ induced by $\bar{\nu}$), 
then the canonical action of $G(\mu)$ on $X(\mu)$ is orientation preserving for 
$\mu = \bar{\sigma}, \bar{\sigma}\bar{\theta}$} : For Ad$(\sigma_X)|_{\frak{t}_0}$ is the identity map, in this case. 
 
(iv) {\it If $X(\bar{\sigma}\bar{\theta})$ is a Hermitian symmetric space and is not of tube type, then the canonical action of $G(\mu)$ 
on $X(\mu)$ is orientation preserving for $\mu = \bar{\sigma}, \bar{\sigma}\bar{\theta}$} : If $X(\bar{\sigma}\bar{\theta})$ is a 
Hermitian symmetric space, then $\sigma|_{\frak{h}}$ is the identity map and 
there is a maximal set $\{\gamma_1, \gamma_2, \ldots , \gamma_r\}$ of strongly orthogonal roots in 
$\{\alpha \in \Delta^+ : \frak{g}_\alpha \subset \frak{g}_1 \}$ such that if $\psi \in \Delta_0 ^+$ is 
a simple root of $\frak{g}_0$, then $\psi \Big{|}_{(\frak{t}^-)^\mathbb{C}} = \frac{1}{2} (\gamma_{i+1} - \gamma_i)$ or $0$ or 
$-\frac{1}{2} \gamma_i$ for some $i$ \cite[Lemma.13]{hc1}. Now $X(\bar{\sigma}\bar{\theta})$ is not of tube type iff there is a 
simple root $\psi \in \Delta_0 ^+$ with $\psi \Big{|}_{(\frak{t}^-)^\mathbb{C}} = \frac{1}{2} \gamma_i$ for some $i$ 
\cite[Prop. 4.4 and its Remark]{kor-wolf}. Then for $X = \sum\limits_{j=1}^{r} c_j Y_{\gamma_j} \in \frak{a}$, 
$\sum\limits_{j = 1}^{r} c_j \phi (H_{\gamma_j}^*) \in \pi \mathbb{Z}$ for all $\phi \in \Phi$ $\Leftrightarrow 
2c_1, c_2 - c_1, \ldots , c_r - c_{r-1}, - c_i \ (\textrm{for some }i) \in \pi \mathbb{Z} \Leftrightarrow c_j \in \pi \mathbb{Z}$ for all 
$1 \le j \le r$. So  Ad$(\sigma_X)|_{\frak{t}_0}$ is the identity map, in this case. 

(v) {\it If $X(\bar{\sigma}\bar{\theta})$ is a Hermitian symmetric space and is of tube type, then the canonical action of $G(\mu)$ 
on $X(\mu)$ may not be orientation preserving for $\mu = \bar{\sigma}, \bar{\sigma}\bar{\theta}$} : 
Since $X(\bar{\sigma}\bar{\theta})$ is a Hermitian symmetric space, $\frak{g}_0$ has one dimensional centre. Since 
$X(\bar{\sigma}\bar{\theta})$ is of tube type, the element $Z = \sum\limits_{j =1}^{r} iH_{\gamma_j}^*$ lies in the centre  
of $\frak{g}_0$ \cite [Prop. 3.12]{kor-wolf}. Again if  $\psi \in \Delta_0 ^+$ is a simple root of $\frak{g}_0$, 
then $\psi \Big{|}_{(\frak{t}^-)^\mathbb{C}} = \frac{1}{2} (\gamma_{i+1} - \gamma_i)$ or $0$ for some $i$. So 
for $X = \sum\limits_{j=1}^{r} c_j Y_{\gamma_j} \in \frak{a}$, 
$\sum\limits_{j = 1}^{r} c_j \phi (H_{\gamma_j}^*) \in \pi \mathbb{Z}$ for all $\phi \in \Phi$ $\Leftrightarrow 
2c_1, c_2 - c_1, \ldots , c_r - c_{r-1} \in \pi \mathbb{Z} \Leftrightarrow \cos 2c_j = \cos 2c_1$ for all $1 \le j \le r$, where $2c_1 \in 
\pi \mathbb{Z}$. Hence $s_X (Z)  = \textrm{Ad}(\sigma_X)(Z) = \pm Z$. So if $s_X (Z) = Z$, then 
Ad$(\sigma_X|_{\frak{t}_0})$ is the identity map and hence 
det$(s_X|_{\frak{t}_0^{\mathbb{C}}}) = 1$. If $s_X (Z) = -Z$, then 
det$(s_X|_{\frak{t}_0^{\mathbb{C}}}) = 1$ if the Dynkin diagram automorphism of $[\frak{g}_0, \frak{g}_0]$ induced by $s_X$ is an 
odd permutation. 

(vi) {\it If the Riemannian globally symmetric space $X(\bar{\sigma}\bar{\theta})$ is not a Hermitian symmetric space, then 
the canonical action of $G(\mu)$ on $X(\mu)$ is orientation preserving  for $\mu = \bar{\sigma}, \bar{\sigma}\bar{\theta}$ 
iff the Dynkin diagram automorphism of $\frak{g}_0$ induced 
by $s_X$ is an even permutation for any $X \in \frak{a}$ with $\textrm{exp}(-2X) \in S$} : In this case, $\frak{g}_0$ is semisimple and 
for all $X \in \frak{a}$ with $\textrm{exp}(-2X) \in S$, det$(s_X|_{\frak{t}_0^{\mathbb{C}}}) = 1$  iff the Dynkin diagram automorphism 
of $\frak{g}_0$ induced by $s_X$ is an even permutation. 

}
\end{remark}

  Now we shall check whether det$(s_X|_{\frak{t}_0^{\mathbb{C}}}) = 1$ for all $X \in \frak{a}$ with $\textrm{exp}(-2X) \in \tilde{Z}$, via 
case by case consideration. 

1. 
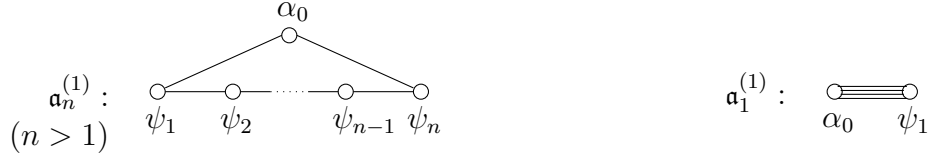
\begin{center}
\begin{tikzpicture}

%1
\draw (0,0) circle [radius = 0.1];
\draw (1,0) circle [radius = 0.1]; 
\draw (2.5,0) circle [radius = 0.1]; 
\draw (3.5,0) circle [radius = 0.1]; 
\draw (1.75,0.75) circle [radius = 0.1]; 
\node [below] at (0.05,-0.05) {$\psi_1$}; 
\node [below] at (1.05,-0.05) {$\psi_2$}; 
\node [below] at (2.75,-0.05) {$\psi_{n-1}$}; 
\node [below] at (3.55,-0.05) {$\psi_n$}; 
\node [above] at (1.80, 0.80) {$\alpha_0$}; 
\draw (0.1,0) -- (0.9,0); 
\draw (1.1,0) -- (1.5,0); 
\draw [dotted] (1.5,0) -- (2,0); 
\draw (2,0) -- (2.4,0); 
\draw (2.6,0) -- (3.4,0); 
\draw (0.1,0.05) -- (1.65,0.75); 
\draw (1.85,0.75) -- (3.4,0.05); 
\node [left] at (-0.5,0) {$\frak{a}_n^{(1)} :$};  
\node [left] at (-0.5,-0.6) {$(n >1)$};  

%2
\node [left] at (8.5,0) {$\frak{a}_{1}^{(1)} :$}; 
\draw (9,0) circle [radius = 0.1]; 
\draw (10,0) circle [radius = 0.1]; 
\node [below] at (9.05,-0.15) {$\alpha_0$}; 
\node [below] at (10.05,-0.05) {$\psi_1$}; 
\draw (9.08,0.016) -- (9.92,0.016); 
\draw (9.04,0.07) -- (9.96,0.07); 
\draw (9.08,-0.035) -- (9.92,-0.035); 
\draw (9.04,-0.09) -- (9.96,-0.09);

\end{tikzpicture}
\end{center}

  Here $\frak{g} = \frak{a}_n$ and $\alpha_0 + \psi_1 + \psi_2 + \cdots + \psi_n = 0$. Without loss of generality, we may assume that 
$\sigma$ is an involution of $\frak{g}$ of type $(s_0, 0, \ldots ,0, s_p, 0, \ldots ,0; 1) \ (1\le p \le n, \ p \le n+1-p)$ with $s_0 = 1 = s_p$. 
Then $\frak{u}_0 = \frak{su}(p) \oplus \frak{su}(n+1-p) \oplus i\mathbb{R}$ and $X(\bar{\sigma}\bar{\theta})$ is a 
Hermitian symmetric space. Let 
\[ \gamma_1 = \psi_p, \gamma_2 = \psi_{p-1} + \psi_p + \psi_{p+1}, \ldots , \gamma_p = \psi_1 + \cdots + \psi_p + \cdots \psi_{2p-1}.\]
Then $\{\gamma_1, \gamma_2 , \ldots , \gamma_p \}$ is a maximal set of strongly orthogonal roots in 
$\{\alpha \in \Delta^+ : \frak{g}_\alpha \subset \frak{g}_1 \}$. The Hermitian symmetric space $X(\bar{\sigma}\bar{\theta})$ is of 
tube type iff $n = 2p -1$ that is, $p = n+1-p$. So if $p < n+1-p$, then the canonical action of $G(\mu)$ on $X(\mu)$ is 
orientation preserving for $\mu = \bar{\sigma}, \bar{\sigma}\bar{\theta}$, by Remark \ref{or}(iv). 

  Now assume that $p = n+1-p$. Then $\psi_{p-1}\Big{|}_{(\frak{t}^-)^\mathbb{C}} = \frac{1}{2}(\gamma_2 - \gamma_1) = 
\psi_{p+1}\Big{|}_{(\frak{t}^-)^\mathbb{C}}, \psi_{p-2}\Big{|}_{(\frak{t}^-)^\mathbb{C}} = \frac{1}{2}(\gamma_3 - \gamma_2) = 
\psi_{p+2}\Big{|}_{(\frak{t}^-)^\mathbb{C}}, \ldots , \psi_1\Big{|}_{(\frak{t}^-)^\mathbb{C}} = \frac{1}{2}(\gamma_p - \gamma_{p-1})
 = \psi_{2p-1}\Big{|}_{(\frak{t}^-)^\mathbb{C}}$, where $n = 2p-1$. Let $X \in \frak{a}$ with $\textrm{exp}(-2X) \in \tilde{Z}$ and 
$\textrm{Ad}(\sigma_X)(Z) = - Z$, where $Z = \sum\limits_{j =1}^{r} iH_{\gamma_j}^*$. Then \\ 
Ad$(\sigma_X) (\psi_{p\pm j}) = \textrm{Ad}(\sigma_X) (\frac{1}{2}(\gamma_{j+1} - \gamma_j)+ \psi_{p\pm j} 
-\frac{1}{2}(\gamma_{j+1} - \gamma_j)) \\ 
=  -\frac{1}{2}(\gamma_{j+1} - \gamma_j) + \psi_{p \pm j} -\frac{1}{2}(\gamma_{j+1} - \gamma_j) = \psi_{p \pm j} - (\gamma_{j+1} - \gamma_j) \\
= - \psi_{p \mp j}$, for all $1 \le j \le p-1$. \\
Let $w_{\frak{g}_0}^0 \in W(\frak{g}_0,\frak{t}_0^{\mathbb{C}})$ be the 
longest element that is, $w_{\frak{g}_0}^0(\psi_j) = -\psi_{p-j}$ and $w_{\frak{g}_0}^0(\psi_{p+j}) = -\psi_{2p-j}$ for all $1 \le j \le p-1$. Then 
$s_X = \textrm{Ad}(\sigma_X)\circ w_{\frak{g}_0}^0$ with $s_X (\Delta_0^+) = \Delta_0^+$. Now $s_X(\psi_{p-j}) = \psi_{2p-j}$ and 
$s_X(\psi_{p+j}) = \psi_j$ for all $1 \le j \le p-1$. So  det$(s_X|_{\frak{t}_0^{\mathbb{C}}}) = - (-1)^{p-1} = (-1)^p$. So if $p$ is even, 
then the canonical action of $G(\mu)$ on $X(\mu)$ is orientation preserving for $\mu = \bar{\sigma}, \bar{\sigma}\bar{\theta}$. 

2. 
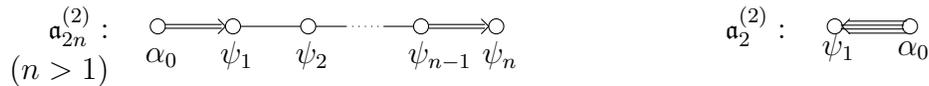
\begin{center}
\begin{tikzpicture} 

\node [left] at (-0.5,0) {$\frak{a}_{2n}^{(2)} :$}; 
\node [left] at (-0.5, -0.6) {$(n >1)$}; 
\draw (0,0) circle [radius = 0.1]; 
\draw (1,0) circle [radius = 0.1]; 
\draw (2,0) circle [radius = 0.1]; 
\draw (3.5,0) circle [radius = 0.1]; 
\draw (4.5,0) circle [radius = 0.1]; 
\node [below] at (0.05,-0.15) {$\alpha_0$}; 
\node [below] at (1.05,-0.05) {$\psi_1$}; 
\node [below] at (2.05,-0.05) {$\psi_2$}; 
\node [below] at (3.75,-0.05) {$\psi_{n-1}$}; 
\node [below] at (4.55,-0.05) {$\psi_n$}; 
\draw (0.9,0) -- (0.8,0.1); 
\draw (0.9,0) -- (0.8,-0.1); 
\draw (0.1,0.025) -- (0.85,0.025); 
\draw (0.1,-0.025) -- (0.85,-0.025); 
\draw (1.1,0) -- (1.9,0); 
\draw (2.1,0) -- (2.5,0); 
\draw [dotted] (2.5,0) -- (3,0); 
\draw (3,0) -- (3.4,0); 
\draw (4.4,0) -- (4.3,0.1); 
\draw (4.4,0) -- (4.3,-0.1); 
\draw (3.6,0.025) -- (4.35, 0.025); 
\draw (3.6,-0.025) -- (4.35,-0.025); 

\node [left] at (8.5,0) {$\frak{a}_2^{(2)} :$}; 
\draw (9,0) circle [radius = 0.1]; 
\draw (10,0) circle [radius = 0.1]; 
\node [below] at (9.05,0.05) {$\psi_1$}; 
\node [below] at (10.05,-0.05) {$\alpha_0$}; 
\draw (9.1,0) -- (9.2,0.1); 
\draw (9.1,0) -- (9.2,-0.1); 
\draw (9.08,0.016) -- (9.92,0.016); 
\draw (9.18,0.07) -- (9.96,0.07); 
\draw (9.08,-0.035) -- (9.92,-0.035); 
\draw (9.18,-0.09) -- (9.96,-0.09); 

\end{tikzpicture}
\end{center}

   Here $\frak{g} = \frak{a}_{2n} ,\  \alpha_0 + 2\psi_1 + 2\psi_2 + \cdots + 2\psi_n = 0$, and $\sigma$ is an involution of 
$\frak{g}$ of type $(1, 0, \ldots ,0; 2)$. Then $\frak{u}_0 = \frak{so}(2n+1)$ and $\frak{g}_0 = \frak{b}_n$, which does not have 
any non-trivial Dynkin diagram automorphism. So det$(s_X|_{\frak{t}_0^{\mathbb{C}}}) = 1$ for all $X \in \frak{a}$ with 
$\textrm{exp}(-2X) \in \tilde{Z}$, by Remark \ref{or}(vi). Hence the canonical action of $G(\mu)$ on $X(\mu)$ is orientation preserving for 
$\mu = \bar{\sigma}, \bar{\sigma}\bar{\theta}$. 

3. 
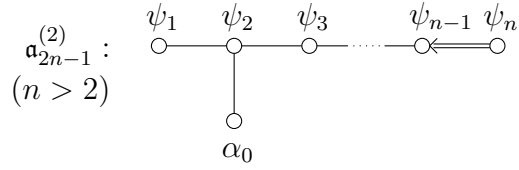
\begin{center} 
\begin{tikzpicture} 

\node [left] at (-0.5,0) {$\frak{a}_{2n-1}^{(2)} :$}; 
\node [left] at (-0.5,-0.6) {$(n > 2)$}; 
\draw (0,0) circle [radius = 0.1]; 
\draw (1,0) circle [radius = 0.1]; 
\draw (2,0) circle [radius = 0.1]; 
\draw (1,-1) circle [radius = 0.1]; 
\draw (3.5,0) circle [radius = 0.1]; 
\draw (4.5,0) circle [radius = 0.1]; 
\node [above] at (0.05,0.05) {$\psi_1$}; 
\node [above] at (1.05,0.05) {$\psi_2$}; 
\node [above] at (2.05,0.05) {$\psi_3$}; 
\node [below] at (1.05,-1.15) {$\alpha_0$}; 
\node [above] at (3.75,0.05) {$\psi_{n-1}$}; 
\node [above] at (4.55,0.05) {$\psi_n$}; 
\draw (0.1,0) -- (0.9,0); 
\draw (1.1,0) -- (1.9,0); 
\draw (1,-0.1) -- (1,-0.9); 
\draw (2.1,0) -- (2.5,0); 
\draw [dotted] (2.5,0) -- (3,0); 
\draw (3,0) -- (3.4,0); 
\draw (3.6,0) -- (3.7,0.1); 
\draw (3.6,0) -- (3.7,-0.1); 
\draw (3.65,0.025) -- (4.4,0.025); 
\draw (3.65,-0.025) -- (4.4,-0.025); 

\end{tikzpicture} 
\end{center} 

   Now $\frak{g} = \frak{a}_{2n-1} (n >2)$. In this case, $\alpha_0 + \psi_1 +2\psi_2 + \cdots + 2\psi_{n-1} + \psi_n = 0$. 

  (i) First assume that $\sigma$ is an involution of $\frak{g}$ of type $(1, 0, \ldots ,0; 2)$ (similarly for type $(0, 1, 0, \ldots ,0; 2)$). Then 
$\frak{u}_0 = \frak{sp}(n)$ and $\frak{g}_0 = \frak{c}_n (n > 2)$, which does not have 
any non-trivial Dynkin diagram automorphism. So det$(s_X|_{\frak{t}_0^{\mathbb{C}}}) = 1$ for all $X \in \frak{a}$ with 
$\textrm{exp}(-2X) \in \tilde{Z}$, by Remark \ref{or}(vi). Hence the canonical action of $G(\mu)$ on $X(\mu)$ is orientation preserving for 
$\mu = \bar{\sigma}, \bar{\sigma}\bar{\theta}$. 

 (ii) Next assume that $\sigma$ is an involution of type $(0, 0, \ldots ,0,1; 2)$. Then $\frak{u}_0 = \frak{so}(2n)$ and 
$\frak{g}_0 = \frak{\delta}_n (n > 2)$. 
The diagram $\frak{a}_{2n-1}^{(2)}$ is corresponding to the Dynkin diagram automorphism $\bar{\nu}$ 
of $\frak{a}_{2n-1}$ given by $\bar{\nu}(\phi_j) = \phi_{2n-j}$ for all $1 \le j \le 2n-1$. 

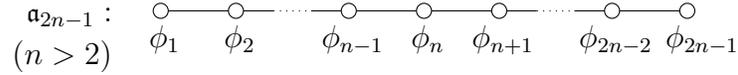
\begin{center} 
\begin{tikzpicture}

\draw (0,0) circle [radius = 0.1]; 
\draw (1,0) circle [radius = 0.1]; 
\draw (2.5,0) circle [radius = 0.1]; 
\draw (3.5,0) circle [radius = 0.1]; 
\draw (4.5,0) circle [radius = 0.1]; 
\draw (6,0) circle [radius = 0.1]; 
\draw (7,0) circle [radius = 0.1]; 
\node [left] at (-0.5,-0.1) {$\frak{a}_{2n-1} :$}; 
\node [left] at (-0.5,-0.6) {$(n > 2)$}; 
\node [below] at (0.05,-0.05) {$\phi_1$}; 
\node [below] at (1.05,-0.05) {$\phi_2$}; 
\node [below] at (2.55,-0.05) {$\phi_{n-1}$}; 
\node [below] at (3.55,-0.05) {$\phi_n$}; 
\node [below] at (4.55,-0.05) {$\phi_{n+1}$}; 
\node [below] at (6.05,-0.05) {$\phi_{2n-2}$}; 
\node [below] at (7.2,-0.05) {$\phi_{2n-1}$}; 
\draw (0.1,0) -- (0.9,0); 
\draw (1.1,0) -- (1.5,0); 
\draw [dotted] (1.5,0) -- (2,0); 
\draw (2,0) -- (2.4,0); 
\draw (2.6,0) -- (3.4,0); 
\draw (3.6,0) -- (4.4,0); 
\draw (4.6,0) -- (5,0); 
\draw [dotted] (5,0) -- (5.5,0); 
\draw (5.5,0) -- (5.9,0); 
\draw (6.1,0) -- (6.9,0); 

\end{tikzpicture} 
\end{center} 

  Let $\{ \gamma_1, \gamma_2, \ldots , \gamma_r \}$ be a maximal set of strongly orthogonal roots in 
$\{\alpha \in \Delta^+ : \sigma(H_\alpha ^*) = H_\alpha ^* ,\ \frak{g}_\alpha \subset \frak{g}_1 \}$, where $\gamma_1 = \phi_n$. 
Then $\frak{a} = \sum\limits_{j=1}^{n-1}  \mathbb{R} i(H_{\phi _j}^* - H_{\phi_{2n-j}}^*) \oplus 
\sum\limits_{j=1}^{r}\mathbb{R} Y_{\gamma_j}$ is a maximal abelian subspace of $\frak{u}_1$. Let 
$X = iH + \frac{\pi}{2} Y_{\gamma_1}$, where $H = \frac{\pi}{2n} \sum\limits_{j=1}^{n-1}j (H_{\phi _j}^* - H_{\phi_{2n-j}}^*)$. Now 
$\phi_{n+j} (H) = -\phi_{n-j} (H)$ for all $1 \le j \le n-1$, $\phi_j (H) = 0$ for $1 \le j \le n-2$ and $j = n$, and 
$\phi_{n-1} (H) = \frac{\pi}{2}$. Hence $\phi_j (H) + \frac{\pi}{2} \phi_j (H_{\gamma_1}^*) = 0$ for $1 \le j \le 2n-1, j \neq n, n+1$, 
$\phi_n (H) + \frac{\pi}{2} \phi_n (H_{\gamma_1}^*) = \pi$, and $\phi_{n+1} (H) + \frac{\pi}{2} \phi_{n+1}(H_{\gamma_1}^*) = -\pi$. 
Hence $\phi_j (H) + \frac{\pi}{2} \phi (H_{\gamma_1}^*) \in \pi \mathbb{Z}$ for all $1 \le j \le 2n-1$ that is, $X \in \frak{a}$ with 
exp$(-2X) \in \tilde{Z}$. Now Ad$(\sigma_X)(H_{\gamma_1}^*) = -H_{\gamma_1}^*$ and Ad$(\sigma_X) (H) = H$ 
for all $\{ H \in \frak{h}^\nu = \frak{t}_0^\mathbb{C} : \gamma_1 (H) =0 \}$. Hence \\
$\textrm{Ad}(\sigma_X) (\psi_j) = \psi_j \textrm{ for all } 1 \le j \le n-2, \\ 
\textrm{Ad}(\sigma_X) (\alpha_0) = \alpha_0\ (\textrm{as } n>2), \\
\textrm{Ad}(\sigma_X) (\psi_n) =- \psi_n, \textrm{ and} \\ 
\textrm{Ad}(\sigma_X) (\psi_{n-1}) = \textrm{Ad}(\sigma_X) (-\frac{1}{2}\psi_n + \psi_{n-1} + \frac{1}{2}\psi_n) = 
\frac{1}{2}\psi_n + \psi_{n-1} + \frac{1}{2}\psi_n = \psi_{n-1} + \psi_n \\
= \psi_{n-1} -(\alpha_0 + \psi_1 + 2\psi_2 + \cdots + 2\psi_{n-1}) 
= - \mu$, \\ 
where $\mu$ is the highest root in $\Delta_0^+$. \\
Therefore $\textrm{Ad}(\sigma_X) (\{\alpha_0, \psi_1, \ldots , \psi_{n-2},\psi_{n-1}\}) = \{\alpha_0, \psi_1, \ldots , \psi_{n-2}, -\mu \}$.

\begin{lemma}\label{weyl}
Let $\frak{l}_0$ be a real simple Lie algebra, $\frak{l}_0 = \frak{k}_0 \oplus \frak{e}_0$ be a Cartan decomposition of $\frak{l}_0$, and 
$\frak{k}_0$ has one dimensional centre (that is, the corresponding Riemannian globally symmetric space is Hermitian symmetric space). Let 
$\frak{b}_0$ be a maximal abelian subspace of $\frak{k}_0$, $\frak{l}= \frak{l}_0^\mathbb{C}, \frak{k}= \frak{k}_0^\mathbb{C}$, and 
$\frak{b}=\frak{b}_0^\mathbb{C}$. Then $\frak{b} \subset \frak{k}$ is a Cartan subalgebra of $\frak{l}$. Let 
$\Delta = \Delta(\frak{l}, \frak{b})$, and $\Delta_0 = \Delta(\frak{k}, \frak{b}) = $the set of all compact roots in $\Delta$. Let $\Delta^+$ 
be a system of positive roots in $\Delta$ such that the corresponding simple system contains exactly one non-compact root $\nu$ and the 
coefficient $n_\nu (\mu)$ of $\nu$ in the highest root $\mu$ when expressed as a sum of simple roots is $1$. Let 
$\Delta_0^+ = \Delta_0 \cap \Delta^+ , \Delta_{\pm{1}} = \{ \alpha \in \Delta : n_\nu (\alpha) = \pm{1}\}$, and $w_{\frak{l}}^0$ 
(respectively, $w_{\frak{k}}^0$) denote the longest element of the Weyl group $W(\frak{g}, \frak{b})$ (respectively, 
$W(\frak{k}, \frak{b})$) with respect to the positive system $\Delta^+$ (respectively, $\Delta_0^+ $). Then $\Delta^+ = \Delta_0^+ 
\cup \Delta_1$, and $w_0 (\Delta^+ ) = \Delta_0^+ \cup \Delta_{-1}$, where $w_0 = w_{\frak{k}}^0 w_{\frak{l}}^0 \in 
W(\frak{l}, \frak{b})$. If $w_{\frak{l}}^0 (\nu ) = -\nu$ (that is, the Hermitian symmetric space is of tube type), 
then $w_0 (\Delta_0^+) = \Delta_0^+$ and $w_0 (\nu) = -\mu$. 
\end{lemma} 

\begin{proof} 
Let $\frak{e} = \frak{e}_0^\mathbb{C}$. Then $\frak{k} = \frak{b} \oplus \sum\limits_{\substack{\alpha \in \Delta \\ n_\nu (\alpha)=0}} 
\frak{l}_\alpha$, $\frak{e} = \frak{e}_{+} \oplus \frak{e}_{-}$, $\frak{e}_{\pm} = \sum\limits_{\substack
{\alpha \in \Delta \\ n_\nu (\alpha)=\pm{1}}} \frak{l}_\alpha$. Also $[\frak{k}, \frak{e}_+] \subset \frak{e}_+ , 
[\frak{k}, \frak{e}_-] \subset \frak{e}_-$. Since $\frak{l}$ is simple, the $\frak{k}$-modules $\frak{e}_+, \frak{e}_-$ are 
irreducible with highest weight $\mu$, $-\nu$ respectively. Also $\Delta_1$ (respectively, $\Delta_{-1}$) is the set of all weights of the 
$\frak{k}$-module $\frak{e}_+$ (respectively, $\frak{e}_-$). Hence $w_{\frak{k}}^0 (\Delta_1) = \Delta_1 , 
w_{\frak{k}}^0 (\Delta_{-1}) = \Delta_{-1}$, and $w_{\frak{k}}^0(\mu)$ (respectively, $w_{\frak{k}}^0(-\nu)$) is the lowest weight of 
$\frak{e}_+$ (respectively, $\frak{e}_-$). Hence $w_{\frak{k}}^0(\mu) = \nu$, and $w_{\frak{k}}^0(-\nu) = -\mu$.  
 
  Now $w_0 (\Delta^+ ) =  w_{\frak{k}}^0 w_{\frak{l}}^0 ( \Delta_0^+ \cup \Delta_{1} ) = 
w_{\frak{k}}^0 ( -(\Delta_0^+) \cup \Delta_{-1} )=  \Delta_0^+ \cup \Delta_{-1}$, and if $w_{\frak{l}}^0 (\nu ) = -\nu$, 
then  $w_0 (\nu) = w_{\frak{k}}^0 w_{\frak{l}}^0 (\nu) =  w_{\frak{k}}^0 (-\nu) = -\mu$.  
\end{proof}

\begin{remark}
{\em 
The above remark will be useful to determine the Weyl group element $s'_X$ (defined in \S \ref{op}) in case by case consideration. 
}
\end{remark}

  Now returning to the case 3.(ii), let $w_{\frak{g}_0}^0$ denote the longest element of the Weyl group of $\frak{\delta}_n$ 
with respect to the simple system $\{\alpha_0, \psi_1, \ldots , \psi_{n-2},\psi_{n-1}\}$. The hypotheses of the Lemma \ref{weyl} is 
satisfied for the Hermitian symmetric space $SO_0(2,2n-2)/SO(2) \times SO(2n-2)$ and $\nu = \psi_{n-1}$. 
Let $w_0 \in  W(\frak{g}_0, \frak{t}_0^\mathbb{C})$ be as in Lemma \ref{weyl} and $s'_X = w_0$. 
Then $s'_X (\{\alpha_0, \psi_1, \ldots , \psi_{n-2},\psi_{n-1}\})= 
\{\alpha_0, \psi_1, \ldots , \psi_{n-2}, -\mu \}$. 
Since $w_{\frak{g}_0}^0(\psi_{n-1}) = -\psi_{n-1}$, we have $s'_X (\psi_{n-1}) = -\mu$ and $s'_X (\{\alpha_0, \psi_1, \ldots , \psi_{n-2}\}) 
= \{\alpha_0, \psi_1, \ldots , \psi_{n-2}\}$. Now $w_{\frak{g}_0}^0 (\alpha_0)$ is $-\alpha_0$ or $-\psi_1$ according as $n$ is even or $n$ is 
odd. In any case, we have $s'_X (\alpha_0) = \psi_1, s'_X (\psi_1) = \alpha_0, s'_X (\psi_2) = \psi_2, \ldots , s'_X (\psi_{n-2}) =  
\psi_{n-2}$. Let $s_X = \textrm{Ad}(\sigma_X) \circ s'_X$. Then $s_X (\alpha_0) = \psi_1, s_X(\psi_1) = \alpha_0, s_X(\psi_j) = \psi_j$ for 
$2 \le j \le n-1$. So det$(s_X|_{\frak{t}_0^{\mathbb{C}}}) = -1$. 

4. 
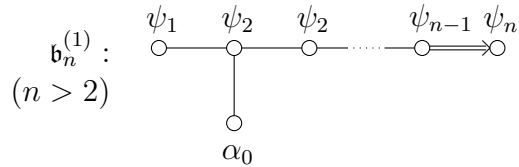
\begin{center} 
\begin{tikzpicture}

\draw (0,0) circle [radius = 0.1]; 
\draw (1,0) circle [radius = 0.1]; 
\draw (2,0) circle [radius = 0.1]; 
\draw (1,-1) circle [radius = 0.1]; 
\draw (3.5,0) circle [radius = 0.1]; 
\draw (4.5,0) circle [radius = 0.1]; 
\node [above] at (0.05,0.05) {$\psi_1$}; 
\node [above] at (1.05,0.05) {$\psi_2$}; 
\node [above] at (2.05,0.05) {$\psi_2$}; 
\node [below] at (1.05,-1.15) {$\alpha_0$}; 
\node [above] at (3.75,0.05) {$\psi_{n-1}$}; 
\node [above] at (4.55,0.05) {$\psi_n$}; 
\node [left] at (-0.5,0) {$\frak{b}_n^{(1)} :$}; 
\node[ left] at (-0.5,-0.6) {$(n > 2)$}; 
\draw (0.1,0) -- (0.9,0); 
\draw (1.1,0) -- (1.9,0);
\draw (1,-0.1) -- (1,-0.9); 
\draw (2.1,0) -- (2.5,0); 
\draw [dotted] (2.5,0) -- (3,0); 
\draw (3,0) -- (3.4,0); 
\draw (4.4,0) -- (4.3,0.1); 
\draw (4.4,0) -- (4.3,-0.1); 
\draw (3.6,0.025) -- (4.35,0.025); 
\draw (3.6,-0.025) -- (4.35,-0.025); 

\end{tikzpicture} 
\end{center} 

 Here $\frak{g} = \frak{b}_n$ and $\alpha_0 + \psi_1 + 2\psi_2 + \cdots + 2\psi_n = 0$. 

(i) First assume that 
$\sigma$ is an involution of $\frak{g}$ of type $(1,1,0, \ldots ,0; 1)$. Then $\frak{u}_0 = \frak{so}(2n-1) \oplus i\mathbb{R}$ 
and $X(\bar{\sigma}\bar{\theta})$ is a Hermitian symmetric space of tube type. Let 
\[ \gamma_1 = \psi_1, \gamma_2 = \psi_1 + 2\psi_2 + \cdots +2 \psi_n.\]
Then $\{\gamma_1, \gamma_2 \}$ is a maximal set of strongly orthogonal roots in 
$\{\alpha \in \Delta^+ : \frak{g}_\alpha \subset \frak{g}_1 \}$ and so $\frak{a} = \mathbb{R} Y_{\gamma_1} \oplus 
\mathbb{R} Y_{\gamma_2}$ is a maximal abelian subspace of $\frak{u}_1$.
Since $[\frak{g}_0, \frak{g}_0] = \frak{b}_{n-1}$ does not 
admit any non-trivial Dynkin diagram automorphism, det$(s_X|_{\frak{t}_0^{\mathbb{C}}}) = -1$ for some $X \in \frak{a}$ with 
$\textrm{exp}(-2X) \in \tilde{Z}$, by Remark \ref{or}(v). 

(ii) Next assume that $\sigma$ is an involution of $\frak{g}$ of type $(0, \ldots ,0, s_p, 0, \ldots ,0; 1) \ (2\le p < n)$ with $s_p = 1$. 
Then $\frak{u}_0 = \frak{so}(2p) \oplus \frak{so}(2n+1-2p)$. Let $q = \textrm{min}\{p, n-p\}$, and define \\
$\gamma_1 = \psi_p , \gamma'_1 = \psi_p + 2 \psi_{p+1} + \cdots + 2 \psi_n , \\ 
\gamma_2 = \psi_{p-1} + \psi_p + \psi_{p+1}, \gamma'_2 =\psi_{p-1} + \psi_p + \psi_{p+1} + 
2\psi_{p+2}+ \cdots + 2 \psi_n , \ldots , \\
\gamma_q = \psi_{p-q+1} + \cdots + \psi_p + \cdots + \psi_{p+q-1}, 
\gamma'_q = \psi_{p-q+1} + \cdots + \psi_p + \cdots + \psi_{p+q-1} + 2\psi_{p+q} + \cdots + 2\psi_n $. 
If $n-p < p$, define $\gamma_0 = \psi_{2p-n} + \cdots + \psi_n$.  Let $\Gamma = \{\gamma_1, 
\gamma'_1, \ldots , \gamma_q, \gamma'_q \}$ if $p \le n-p$, and \\ 
$\Gamma = \{\gamma_1, \gamma'_1, \ldots , \gamma_q, \gamma'_q , \gamma_0 \}$ if $p > n-p$. 
Then $\Gamma$ is a maximal set of strongly orthogonal roots in 
$\{\alpha \in \Delta^+ : \frak{g}_\alpha \subset \frak{g}_1 \}$, and 
$\frak{a} = \sum\limits_{\gamma \in \Gamma}\mathbb{R} Y_{\gamma}$ is a maximal abelian subspace of $\frak{u}_1$. Let 
$X = \frac{\pi}{2} (Y_{\gamma_1} + Y_{\gamma'_1})$. Since 
$\frac{\pi}{2} \psi_j (H_{\gamma_1}^* + H_{\gamma'_1}^* ) = 0$ for $1 \le j \le n, j \neq p-1, p$, 
$\frac{\pi}{2} \psi_{p-1} (H_{\gamma_1}^* + H_{\gamma'_1}^* ) = -\pi$, and 
$\frac{\pi}{2} \psi_{p} (H_{\gamma_1}^* + H_{\gamma'_1}^* ) = \pi$; 
hence $X \in \frak{a}$ with exp$(-2X) \in \tilde{Z}$. 
Now Ad$(\sigma_X)(H_{\gamma_1}^*) = -H_{\gamma_1}^*$, 
Ad$(\sigma_X)(H_{\gamma'_1}^*) = -H_{\gamma'_1}^*$, and 
Ad$(\sigma_X) (H) = H$ 
for all $\{ H \in \frak{t}_0^{\mathbb{C}} : \gamma_1 (H) =0= \gamma'_1 (H)\}$. Hence \\
$\textrm{Ad}(\sigma_X) (\psi_j) = \psi_j \textrm{ for all } 1 \le j \le p-2, \ p+2 \le j \le n; \\ 
\textrm{Ad}(\sigma_X) (\alpha_0) = \alpha_0\ (\textrm{if } p>2), \\
\textrm{Ad}(\sigma_X) (\psi_p) =- \psi_p, \textrm{ and} \\ 
\textrm{Ad}(\sigma_X)(\psi_{p-1}) = \textrm{Ad}(\sigma_X) (-\frac{1}{2}(\gamma_1 + \gamma'_1) 
+ \frac{1}{2} (\gamma_1 + \gamma'_1) + \psi_{p-1}) 
= \frac{1}{2}(\gamma_1 + \gamma'_1) + \frac{1}{2} (\gamma_1 + \gamma'_1) + \psi_{p-1} \\
= \psi_{p-1} + \gamma_1 + \gamma'_1 = \psi_{p-1} + 2\psi_p + \cdots + 2 \psi_n $. \\
So $\textrm{Ad}(\sigma_X) (\psi_{p-1})=-\alpha_0 , \textrm{ if } p=2 $. \\
$\textrm{Ad}(\sigma_X) (\psi_{p-1}) = -\alpha_0 -\psi_1 - \psi_2 = - \mu , \textrm{ if } p=3$; and \\
$\textrm{Ad}(\sigma_X) (\psi_{p-1}) = -\alpha_0 - \psi_1 -2\psi_2-\cdots -2\psi_{p-2} - \psi_{p-1} 
= - \mu, \textrm{ if } p >3$; \\ 
where $\mu$ is the highest root of $\frak{\delta}_p$ with respect to the basis
$\{ \alpha_0 , \psi_1, \psi_2, \ldots , \psi_{p-1} \}$ of the root system of $\frak{\delta}_p$.  \\
Similarly if $p=2$, then $\textrm{Ad}(\sigma_X) (\alpha_0) = -\psi_1$. \\ 
Let $s'_X \in W (\frak{g}_0, \frak{t}_0^{\mathbb{C}})$ be such that Ad$(\sigma_X) \circ s'_X (\Delta_0^+) 
= \Delta_0^+$, and $s_X = \textrm{Ad}(\sigma_X)\circ s'_X $. Since 
$\frak{b}_{n-p}$ does not admit any non-trivial Dynkin diagram automorphism, we have \\ 
$s_X(\alpha_0) = \psi_1, s_X(\psi_1) = \alpha_0, s_X(\psi_j) = \psi_j$  for all $2 \le j \le n, j \neq p$, 
as in case 3(ii). So det$(s_X|_{\frak{t}_0^{\mathbb{C}}}) = -1$. 

(iii) Finally assume that $\sigma$ is an involution of $\frak{g}$ of type $(0, \ldots ,0,1; 1)$. 
Then $\frak{u}_0 = \frak{so}(2n)$. Define $\gamma_1 = \psi_n$. Then $\{\gamma_1 \}$ is a maximal set of 
strongly orthogonal roots in $\{\alpha \in \Delta^+ : \frak{g}_\alpha \subset \frak{g}_1 \}$, and 
$\frak{a} = \mathbb{R} Y_{\gamma_1}$ is a maximal abelian subspace of $\frak{u}_1$. Let 
$X = \frac{\pi}{2} Y_{\gamma_1}$. Since 
$\frac{\pi}{2} \psi_j (H_{\gamma_1}^* )= 0$ for $1 \le j \le n-2$,  
$\frac{\pi}{2} \psi_{n-1} (H_{\gamma_1}^*) = -\pi$, and 
$\frac{\pi}{2} \psi_{n} (H_{\gamma_1}^*) = \pi$; 
hence $X \in \frak{a}$ with exp$(-2X) \in \tilde{Z}$. 
Now Ad$(\sigma_X)(H_{\gamma_1}^*) = -H_{\gamma_1}^*$, and 
Ad$(\sigma_X) (H) = H$ 
for all $\{ H \in \frak{t}_0^{\mathbb{C}} : \gamma_1 (H) =0\}$. Hence \\
$\textrm{Ad}(\sigma_X) (\psi_j) = \psi_j \textrm{ for all } 1 \le j \le n-2, \\ 
\textrm{Ad}(\sigma_X) (\alpha_0) = \alpha_0\ (\textrm{as } n>2), \\
\textrm{Ad}(\sigma_X) (\psi_n) =- \psi_n, \textrm{ and} \\ 
\textrm{Ad}(\sigma_X) (\psi_{n-1}) = \textrm{Ad}(\sigma_X)(-\psi_n +  \psi_{n-1}+ \psi_n) 
= \psi_n + \psi_{n-1} + \psi_{n} = \psi_{n-1} + 2\psi_n \\
= \psi_{n-1} -\alpha_0 - \psi_1 - 2\psi_2 - \cdots - 2\psi_{n-1} 
=-\mu$, 
where $\mu$ is the highest root in $\Delta_0^+$.  \\
Let $s'_X \in W (\frak{g}_0, \frak{t}_0^{\mathbb{C}})$ be such that Ad$(\sigma_X)\circ s'_X (\Delta_0^+) 
= \Delta_0^+$,  
and $s_X = \textrm{Ad}(\sigma_X)\circ s'_X $. Then 
$s_X(\alpha_0) = \psi_1, s_X(\psi_1) = \alpha_0, s_X(\psi_j) = \psi_j$  for all $2 \le j \le n-1$, 
as in case 3(ii). So det$(s_X|_{\frak{t}_0^{\mathbb{C}}}) = -1$. 

5. 
\begin{center} 
\begin{tikzpicture}
 
\draw (0,0) circle [radius = 0.1]; 
\draw (1,0) circle [radius = 0.1]; 
\draw (2.5,0) circle [radius = 0.1]; 
\draw (3.5,0) circle [radius = 0.1]; 
\node [above] at (0.05,0.05) {$\alpha_0$}; 
\node [above] at (1.05,0.05) {$\psi_1$}; 
\node [above] at (2.75,0.05) {$\psi_{n-1}$}; 
\node [above] at (3.55,0.05) {$\psi_n$}; 
\node [left] at (-0.5,0) {$\frak{c}_n^{(1)} :$}; 
\node [left] at (-0.5,-0.6) {$(n > 1)$}; 
\draw (0.9,0) -- (0.8,0.1); 
\draw (0.9,0) -- (0.8,-0.1); 
\draw (0.1,0.025) -- (0.85,0.025); 
\draw (0.1,-0.025) -- (0.85,-0.025); 
\draw (1.1,0) -- (1.5,0); 
\draw [dotted] (1.5,0) -- (2,0); 
\draw (2,0) -- (2.4,0); 
\draw (2.6,0) -- (2.7,0.1); 
\draw (2.6,0) -- (2.7,-0.1); 
\draw (2.65,0.025) -- (3.4,0.025); 
\draw (2.65,-0.025) -- (3.4,-0.025); 

\end{tikzpicture} 
\end{center} 

 Here $\frak{g} = \frak{c}_n$ and $\alpha_0 + 2\psi_1 + 2\psi_2 + \cdots + 2\psi_{n-1}+ \psi_n = 0$. 
 
(i) First assume that 
$\sigma$ is an involution of $\frak{g}$ of type $(1, 0, \ldots ,0, 1; 1)$. Then $\frak{u}_0 = 
\frak{su}(n) \oplus i\mathbb{R}$ and $X(\bar{\sigma}\bar{\theta})$ is a 
Hermitian symmetric space of tube type. Let 
\[ \gamma_1 = \psi_n, \gamma_2 = 2\psi_{n-1} + \psi_n, \ldots , \gamma_n = 2\psi_1 + \cdots + 
2\psi_{n-1} + \psi_n.\]
Then $\{\gamma_1, \gamma_2 , \ldots , \gamma_n \}$ is a maximal set of strongly orthogonal roots in 
$\{\alpha \in \Delta^+ : \frak{g}_\alpha \subset \frak{g}_1 \}$. Note that 
$\psi_{n-1} = \frac{1}{2}(\gamma_2 - \gamma_1), \psi_{n-2} = \frac{1}{2}(\gamma_3 - \gamma_2), 
\ldots , \psi_1 = \frac{1}{2}(\gamma_n - \gamma_{n-1})$. 
Let $X \in \frak{a}$ with $\textrm{exp}(-2X) \in \tilde{Z}$ and 
$\textrm{Ad}(\sigma_X)(Z) = - Z$, where $Z = \sum\limits_{j =1}^{n} iH_{\gamma_j}^*$. Then 
Ad$(\sigma_X) (\psi_{n- j} ) = -\frac{1}{2}(\gamma_{j+1} - \gamma_j) = -\psi_{n-j}$ 
for all $1 \le j \le n-1$. Let $w_{\frak{g}_0}^0 \in W(\frak{g}_0,\frak{t}_0^{\mathbb{C}})$ be the 
longest element that is, $w_{\frak{g}_0}^0(\psi_j) = -\psi_{n-j}$ for all $1 \le j \le n-1$. Then 
$s_X = \textrm{Ad}(\sigma_X)\circ w_{\frak{g}_0}^0$ with $s_X (\Delta_0^+) = \Delta_0^+$. Now $s_X(\psi_{j}) = \psi_{n-j}$ for all 
$1 \le j \le n-1$. So  det$(s_X|_{\frak{t}_0^{\mathbb{C}}}) = - (-1)^{[\frac{n-1}{2}]} = 1$ iff $n \in 4\mathbb{Z}$ or $n \in 3 + 4\mathbb{Z}$. So if 
$n \in 4\mathbb{Z}$ or $n \in 3 + 4\mathbb{Z}$, 
then the canonical action of $G(\mu)$ on $X(\mu)$ is orientation preserving for $\mu = \bar{\sigma}, \bar{\sigma}\bar{\theta}$. 

(ii) Next assume that $\sigma$ is an involution of $\frak{g}$ of type $(0, 0, \ldots ,0,s_p,0,\ldots,0; 1)\ 
(1 \le p \le n-1, \ p \neq n-p)$ with $s_p = 1$. Then 
$\frak{u}_0 = \frak{sp}(p) \oplus \frak{sp}(n-p)$ and $\frak{g}_0 = \frak{c}_p \oplus \frak{c}_{n-p} \ 
(p \neq n-p)$, which does not have 
any non-trivial Dynkin diagram automorphism. So det$(s_X|_{\frak{t}_0^{\mathbb{C}}}) = 1$ for all $X \in \frak{a}$ with 
$\textrm{exp}(-2X) \in \tilde{Z}$. Hence the canonical action of $G(\mu)$ on $X(\mu)$ is orientation preserving for 
$\mu = \bar{\sigma}, \bar{\sigma}\bar{\theta}$. 

(iii) Finally assume that $\sigma$ is an involution of $\frak{g}$ of type 
$(0, \ldots ,0, s_p, 0, \ldots ,0; 1)\ (s_p = 1)$, where $n$ is even and $p = \frac{n}{2}$.  
Then $\frak{u}_0 = \frak{sp}(p) \oplus \frak{sp}(p)$. Define \\
$\gamma_1 = \psi_p , \gamma_2 = \psi_{p-1} + \psi_p + \psi_{p+1}, \ldots , 
\gamma_p = \psi_1 + \cdots + \psi_{p-1} + \psi_p + \psi_{p+1} + \cdots + \psi_{n-1} \ 
(\textrm{as } n = 2p)$. Then $\{\gamma_1, \gamma_2, \ldots , \gamma_p \}$ 
is a maximal set of strongly orthogonal roots in 
$\{\alpha \in \Delta^+ : \frak{g}_\alpha \subset \frak{g}_1 \}$, and 
$\frak{a} = \sum\limits_{j=1}^{p} \mathbb{R} Y_{\gamma_j}$ is a maximal abelian subspace of 
$\frak{u}_1$. Let $X = \sum\limits_{j=1}^{p}  c_j Y_{\gamma_j}$. Then 
exp$(-2X) \in \tilde{Z}$ iff 
$c_p - c_{p-1}, c_{p-1} - c_{p-2}, \ldots , c_2 - c_1, 2c_1, -2c_p \in \pi \mathbb{Z}$ iff 
$\cos 2c_j = \cos 2c_1$ for all $1 \le j \le p$, and $\cos 2c_1 = \pm 1$. \\ 
Let $X = \sum\limits_{j=1}^{p}  c_j Y_{\gamma_j}$ with exp$(-2X) \in \tilde{Z}$ and $\cos 2c_1 = -1$. Then 
Ad$(\sigma_X)(H_{\gamma_j}^*) = -H_{\gamma_j}^*$, and 
Ad$(\sigma_X) (H) = H$ 
for all $\{ H \in \frak{t}_0^{\mathbb{C}} : \gamma_j (H) =0 \textrm{ for all } 1\le j \le p\}$. Recall that 
$\frak{t}^- = \sum\limits_{j =1}^{p} \mathbb{R}(iH_{\gamma_j}^*) \subset \frak{t}_0$, and 
$\frak{t}^+ = \{ H \in \frak{t}_0 : \gamma_j (H) =0 \textrm{ for all } 1 \le j \le p \}$. Now 
$\psi_1\Big{|}_{(\frak{t}^-)^{\mathbb{C}}} = \frac{1}{2}(\gamma_p - \gamma_{p-1}), 
\psi_2\Big{|}_{(\frak{t}^-)^{\mathbb{C}} }= \frac{1}{2}(\gamma_{p-1} - \gamma_{p-2}), \ldots , 
\psi_{p-1}\Big{|}_{(\frak{t}^-)^{\mathbb{C}}} = \frac{1}{2}(\gamma_2 - \gamma_1), 
\psi_{p+1}\Big{|}_{(\frak{t}^-)^{\mathbb{C}}} = \frac{1}{2}(\gamma_2 - \gamma_1), \ldots , 
\psi_{n-1}\Big{|}_{(\frak{t}^-)^{\mathbb{C}}}= \frac{1}{2}(\gamma_p - \gamma_{p-1}), 
\psi_n\Big{|}_{(\frak{t}^-)^{\mathbb{C}}} = -\gamma_p , \alpha_0\Big{|}_{(\frak{t}^-)^{\mathbb{C}}} = 
-\gamma_p$. Hence \\
$\textrm{Ad}(\sigma_X) (\psi_{p \pm j}) = \textrm{Ad}(\sigma_X) (\frac{1}{2}(\gamma_{j+1} - 
\gamma_{j}) + \psi_{p\pm j} - \frac{1}{2} (\gamma_{j+1} - \gamma_{j}) ) 
=- \frac{1}{2}(\gamma_{j+1} - \gamma_{j}) + \psi_{p\pm j} - 
\frac{1}{2} (\gamma_{j+1} - \gamma_{j}) 
= \psi_{p\pm j} - (\gamma_{j+1} - \gamma_{j}) = - \psi_{p\mp j}, \textrm{ for all } 1 \le j \le p-1; \\
\textrm{Ad}(\sigma_X) (\psi_n) = \textrm{Ad}(\sigma_X) (-\gamma_p + \psi_n + \gamma_p) 
=\gamma_p + \psi_{n} + \gamma_p = \psi_n + 2\gamma_p = - \alpha_0, \textrm{ and} \\ 
\textrm{Ad}(\sigma_X) (\alpha_0) = \textrm{Ad}(\sigma_X) (-\gamma_p + \alpha_0+ \gamma_p) 
=\gamma_p + \alpha_0 + \gamma_p = \alpha_0 + 2\gamma_p = - \psi_n$. \\
Therefore $\textrm{Ad}(\sigma_X) (\{\alpha_0, \psi_1, \ldots , \psi_{p-1},\psi_{p+1}, \ldots , 
\psi_n\}) = \{-\alpha_0, -\psi_1, \ldots , -\psi_{p-1},-\psi_{p+1}, \ldots , -\psi_n\}$. \\
Let $w_{\frak{g}_0}^0 \in W(\frak{g}_0,\frak{t}_0^{\mathbb{C}})$ be the 
longest element that is, $w_{\frak{g}_0}^0(\psi_j) = -\psi_j$ for all $1 \le j \le n,\ j \neq p$; and 
$w_{\frak{g}_0}^0(\alpha_0) = -\alpha_0$. Then 
$s_X = \textrm{Ad}(\sigma_X)\circ w_{\frak{g}_0}^0$ with $s_X (\Delta_0^+) = \Delta_0^+$. Now 
$s_X(\psi_{p \pm j}) = \psi_{p \mp j}$ for all $1 \le j \le p-1$, $s_X(\psi_n) = \alpha_0, \ 
s_X(\alpha_0) = \psi_n$. So  det$(s_X|_{\frak{t}_0^{\mathbb{C}}}) = (-1)^p$. Hence 
the canonical action of $G(\mu)$ on $X(\mu)$ is orientation preserving for $\mu = \bar{\sigma}, \bar{\sigma}\bar{\theta}$ 
if $p$ is even that is, if $n \in 4\mathbb{Z}$ and $p = \frac{n}{2}$. 

6. 
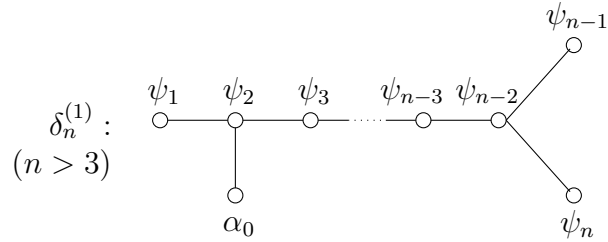
\begin{center} 
\begin{tikzpicture} 

\draw (0,0) circle [radius = 0.1]; 
\draw (1,0) circle [radius = 0.1]; 
\draw (2,0) circle [radius = 0.1]; 
\draw (1,-1) circle [radius = 0.1]; 
\draw (3.5,0) circle [radius = 0.1]; 
\draw (4.5,0) circle [radius = 0.1]; 
\draw (5.5,1) circle [radius = 0.1]; 
\draw (5.5,-1) circle [radius = 0.1]; 
\node [above] at (0.05,0.05) {$\psi_1$}; 
\node [above] at (1.05,0.05) {$\psi_2$}; 
\node [above] at (2.05,0.05) {$\psi_3$}; 
\node [below] at (1.05,-1.15) {$\alpha_0$}; 
\node [above] at (3.35,0.05) {$\psi_{n-3}$}; 
\node [above] at (4.35,0.05) {$\psi_{n-2}$}; 
\node [above] at (5.55,1.05) {$\psi_{n-1}$}; 
\node [below] at (5.55,-1.05) {$\psi_n$}; 
\node [left] at (-0.5,0) {$\frak{\delta}_n^{(1)} :$}; 
\node [left] at (-0.5,-0.6) {$(n > 3)$}; 
\draw (0.1,0) -- (0.9,0); 
\draw (1.1,0) -- (1.9,0); 
\draw (1,-0.1) -- (1,-0.9); 
\draw (2.1,0) -- (2.5,0); 
\draw [dotted] (2.5,0) -- (3,0); 
\draw (3,0) -- (3.4,0); 
\draw (3.6,0) -- (4.4,0); 
\draw (4.6,0) -- (5.45,0.95); 
\draw (4.6,0) -- (5.45,-0.95); 

\end{tikzpicture} 
\end{center} 

  Here $\frak{g} = \frak{\delta}_n$ and $\alpha_0 + \psi_1 + 2\psi_2 + \cdots + 2\psi_{n-2} + 
\psi_{n-1} + \psi_n = 0$. 
  
(i)  Assume that 
$\sigma$ is an involution of $\frak{g}$ of type $(1, 0,\ldots,0,1; 1)$ (similarly for types 
$(1,0,\ldots , 0,1,0;1), (0,1,0, \ldots , 0,1,0;1), \textrm{ or } (0,1,0,\ldots,0,1;1)$). Then $\frak{u}_0 = \frak{su}(n) 
\oplus i\mathbb{R}$ and $X(\bar{\sigma}\bar{\theta})$ is a 
Hermitian symmetric space. This Hermitian symmetric space is of tube type iff $n$ is even. So if $n$ is 
odd, then the canonical action of $G(\mu)$ on $X(\mu)$ is 
orientation preserving for $\mu = \bar{\sigma}, \bar{\sigma}\bar{\theta}$, by Remark \ref{or}(iv). 

  Now assume that $n$ is even and $r = \frac{n}{2}$. Define 
$\gamma_1 = \psi_n, \gamma_2 = \psi_{n-3} + 2\psi_{n-2}+ \psi_{n-1}+\psi_n, \\
\gamma_3 = \psi_{n-5} + 2\psi_{n-4} + 2\psi_{n-3} + 2\psi_{n-2} + \psi_{n-1} + \psi_n, \ldots , 
\gamma_r = \psi_1 + 2\psi_2 + \cdots + 2\psi_{n-2} + \psi_{n-1} + \psi_n$. That is, 
$\gamma_j = \psi_{n-2j+1} +2 \psi_{n-2j+2} + \cdots + 2\psi_{n-2} + \psi_{n-1} + \psi_n$, for all $2 \le j \le r$; and 
$\gamma_1 = \psi_n$. 
Then $\{\gamma_1, \gamma_2 , \ldots , \gamma_r \}$ is a maximal set of strongly orthogonal roots in 
$\{\alpha \in \Delta^+ : \frak{g}_\alpha \subset \frak{g}_1 \}$.  \\
Now $\psi_1\Big{|}_{(\frak{t}^-)^{\mathbb{C}}}= 0, \psi_2\Big{|}_{(\frak{t}^-)^{\mathbb{C}}} = \frac{1}{2}(\gamma_r - \gamma_{r-1}), 
\psi_3\Big{|}_{(\frak{t}^-)^{\mathbb{C}}}= 0, \psi_4\Big{|}_{(\frak{t}^-)^{\mathbb{C}}} = \frac{1}{2}(\gamma_{r-1} - \gamma_{r-2}), \ldots , 
\psi_{n-2}\Big{|}_{(\frak{t}^-)^{\mathbb{C}}} = \frac{1}{2}(\gamma_2 - \gamma_1), \psi_{n-1}\Big{|}_{(\frak{t}^-)^{\mathbb{C}}}= 0$. 
That is, $\psi_{2j} \Big{|}_{(\frak{t}^-)^{\mathbb{C}}} = \frac{1}{2}(\gamma_{r-j+1} - \gamma_{r-j})$ for all $1 \le j \le r-1$, 
$\psi_{2j-1}\Big{|}_{(\frak{t}^-)^{\mathbb{C}}}= 0$ for all $1 \le j \le r$. \\
Let $X \in \frak{a}$ with $\textrm{exp}(-2X) \in \tilde{Z}$ and 
$\textrm{Ad}(\sigma_X)(Z) = - Z$, where $Z = \sum\limits_{j =1}^{r} iH_{\gamma_j}^*$. Then \\ 
Ad$(\sigma_X) (\psi_{2j-1} ) = \psi_{2j-1}$, for all $1 \le j \le r$, and \\ 
Ad$(\sigma_X) (\psi_{2j} ) = \textrm{Ad}(\sigma_X) (\frac{1}{2}(\gamma_{r-j+1} - \gamma_{r-j}) + \psi_{2j} - \frac{1}{2}(\gamma_{r-j+1} - \gamma_{r-j})) \\ 
= -\frac{1}{2}(\gamma_{r-j+1} - \gamma_{r-j}) + \psi_{2j} - \frac{1}{2}(\gamma_{r-j+1} - \gamma_{r-j}) = \psi_{2j} - (\gamma_{r-j+1} - \gamma_{r-j}) \\ 
= -\psi_{2j-1} - \psi_{2j} - \psi_{2j+1}$, for all $1 \le j \le r-1$. So Ad$(\sigma_X) (\{ \psi_1, \psi_2, \ldots , \psi_{n-1} \}) = \{\psi_1, -\psi_1 - \psi_2 - \psi_3, 
\psi_3, -\psi_3 - \psi_4-\psi_5, \ldots , \psi_{n-3} , -\psi_{n-3} - \psi_{n-2} - \psi_{n-1}, \psi_{n-1} \}$. \\ 
Let $w_{\frak{g}_0}^0 \in W(\frak{g}_0,\frak{t}_0^{\mathbb{C}})$ be the 
longest element and $s'_X = s_{\psi_{n-1}}s_{\psi_{n-3}}\ldots s_{\psi_3}s_{\psi_1}w_{\frak{g}_0}^0$. Then 
$s'_X(\psi_{2j-1}) = \psi_{n-2j+1}$ for all $1 \le j \le r$, and $s'_X(\psi_{2j}) = -\psi_{n-2j-1}-\psi_{n-2j} - \psi_{n-2j+1}$ for all $1 \le j \le r-1$. Then 
$s_X = \textrm{Ad}(\sigma_X)\circ s'_X$ with $s_X (\Delta_0^+) = \Delta_0^+$. Now $s_X(\psi_{2j-1}) = \psi_{n-2j+1}$ for all $1 \le j \le r$, and 
$s_X(\psi_{2j}) = \psi_{n-2j}$ for all $1 \le j \le r-1$. So  det$(s_X|_{\frak{t}_0^{\mathbb{C}}}) = - (-1)^{r-1} = (-1)^r$. Hence  
the canonical action of $G(\mu)$ on $X(\mu)$ is orientation preserving for $\mu = \bar{\sigma}, \bar{\sigma}\bar{\theta}$ if $r$ is even that is, if 
$n \in 4\mathbb{Z}$.  

(ii) Assume that 
$\sigma$ is an involution of $\frak{g}$ of type $(1, 1, 0, \ldots ,0; 1)$ (similarly for type $(0,\ldots , 0,1,1;1)$). 
Then $\frak{u}_0 = \frak{so}(2n-2) \oplus i\mathbb{R}$ and $X(\bar{\sigma}\bar{\theta})$ is a Hermitian symmetric space of tube type. Let 
\[ \gamma_1 = \psi_1, \gamma_2 = \psi_1 + 2\psi_2 + \cdots +2 \psi_{n-2} + \psi_{n-1} + \psi_n.\]
Then $\{\gamma_1, \gamma_2 \}$ is a maximal set of strongly orthogonal roots in 
$\{\alpha \in \Delta^+ : \frak{g}_\alpha \subset \frak{g}_1 \}$ and so $\frak{a} = \mathbb{R} Y_{\gamma_1} \oplus 
\mathbb{R} Y_{\gamma_2}$ is a maximal abelian subspace of $\frak{u}_1$. \\ 
Now $\psi_2\Big{|}_{(\frak{t}^-)^{\mathbb{C}}}= \frac{1}{2}(\gamma_2 - \gamma_1), \psi_j\Big{|}_{(\frak{t}^-)^{\mathbb{C}}} = 0$ 
for all $3 \le j \le n$. \\ 
Let $X \in \frak{a}$ with $\textrm{exp}(-2X) \in \tilde{Z}$ and 
$\textrm{Ad}(\sigma_X)(Z) = - Z$, where $Z = i(H_{\gamma_1}^* + H_{\gamma_2}^*)$. Then \\ 
$\textrm{Ad}(\sigma_X) (\psi_j) = \psi_j \textrm{ for all } 3 \le j \le n, \textrm{ and} \\ 
\textrm{Ad}(\sigma_X) (\psi_2) = \textrm{Ad}(\sigma_X) (\frac{1}{2}(\gamma_2 - \gamma_1) + \psi_2 - \frac{1}{2}(\gamma_2 - \gamma_1)) 
= -\frac{1}{2}(\gamma_2 - \gamma_1) + \psi_2 - \frac{1}{2}(\gamma_2 - \gamma_1) \\
= \psi_2 - (\gamma_2 - \gamma_1) = \psi_2 - (2\psi_2 + \cdots + 2\psi_{n-2} + \psi_{n-1} + \psi_n) 
=-\mu$, 
where $\mu$ is the highest root in $\Delta_0^+$.  \\
Let $s'_X \in W (\frak{g}_0, \frak{t}_0^{\mathbb{C}})$ be such that Ad$(\sigma_X) \circ s'_X (\Delta_0^+) 
= \Delta_0^+$,  
and $s_X = \textrm{Ad}(\sigma_X)\circ s'_X $. Then 
$s_X(\psi_{n-1}) = \psi_n, s_X(\psi_n) = \psi_{n-1}, s_X(\psi_j) = \psi_j$  for all $2 \le j \le n-2$, 
as in case 3(ii). So det$(s_X|_{\frak{t}_0^{\mathbb{C}}}) = -1\times -1 = 1$. Hence  
the canonical action of $G(\mu)$ on $X(\mu)$ is orientation preserving for $\mu = \bar{\sigma}, \bar{\sigma}\bar{\theta}$. 

(iii) Next assume that $\sigma$ is an involution of $\frak{g}$ of type $(0, 0,\ldots ,0, s_p, 0, \ldots ,0; 1) \ (2\le p \le n-2,\ p \le n-p)$ 
with $s_p = 1$. Then $\frak{u}_0 = \frak{so}(2p) \oplus \frak{so}(2n-2p)$. Define \\
$\gamma_1 = \psi_p , \gamma'_1 = (\psi_p + \cdots + \psi_{n-2}) + (\psi_{p+1} + \cdots + \psi_n), \\ 
\gamma_2 = \psi_{p-1} + \psi_p + \psi_{p+1}, \gamma'_2 =(\psi_{p-1} + \cdots + \psi_{n-2}) + (\psi_{p+2}+ \cdots + \psi_n), \ldots , \\
\gamma_p = \psi_{1} + \cdots + \psi_p + \cdots + \psi_{2p-1}, 
\gamma'_p = (\psi_{1} + \cdots + \psi_{n-2}) + (\psi_{2p} + \cdots + \psi_n)$. 
Then $\{\gamma_1, \gamma'_1, \ldots , \gamma_p, \gamma'_p\}$ is a maximal set of strongly orthogonal roots in 
$\{\alpha \in \Delta^+ : \frak{g}_\alpha \subset \frak{g}_1 \}$, and 
$\frak{a} = \sum\limits_{j=1}^{p}(\mathbb{R} Y_{\gamma_j} + \mathbb{R}Y_{\gamma'_j})$ is a maximal abelian subspace of $\frak{u}_1$. 
Let $X = \sum\limits_{j=1}^{p} (c_j Y_{\gamma_j} + c'_j Y_{\gamma'_j}) \in \frak{a}$. Then exp$(-2X) \in \tilde{Z}$ iff 
$c_p + c'_p - c_{p-1} - c'_{p-1}, c_{p-1} + c'_{p-1} - c_{p-2} - c'_{p-2}, \ldots , c_2 + c'_2 - c_1 - c'_1, 2c_1, 
- c_1 + c'_1 + c_2 - c'_2, - c_2 + c'_2 + c_3 - c'_3, \ldots , - c_{p-1} + c'_{p-1} + c_p - c'_p \in \pi \mathbb{Z}$ and 
$- c_p + c'_p(\textrm{respectively,} - c_{p-1} + c'_{p-1} - c_p + c'_p) \in \pi \mathbb{Z}$, if $p < n-p$ (respectively, 
$p = n-p$). This is true iff $\cos 2c_j = \cos 2c'_j = \pm 1$ (respectively, $\cos 2c_j = \pm 1, \cos 2c'_j = \pm 1$) 
for all $1 \le j \le p$, if $p < n-p$ (respectively, $p=n-p$). 

  Let exp$(-2X) \in \tilde{Z}$ and $p < n-p$. Then for $1 \le j \le p$, either $\textrm{Ad}(\sigma_X) (\gamma_j ) = \gamma_j, 
\textrm{Ad}(\sigma_X) (\gamma'_j) = \gamma'_j$, or $\textrm{Ad}(\sigma_X) (\gamma_j ) = -\gamma_j, 
\textrm{Ad}(\sigma_X) (\gamma'_j) = -\gamma'_j$. Thus for $1 \le j \le p$, either $\textrm{Ad}(\sigma_X) 
(\frac{\gamma_j  + \gamma'_j}{2}) = \frac{\gamma_j + \gamma'_j}{2}, \textrm{Ad}(\sigma_X) (\frac{\gamma'_j -\gamma_j}{2}) 
= \frac{\gamma'_j-\gamma_j}{2}$, or \\ 
$\textrm{Ad}(\sigma_X) (\frac{\gamma_j + \gamma'_j}{2}) = -\frac{\gamma_j + \gamma'_j}{2}, 
\textrm{Ad}(\sigma_X) (\frac{\gamma'_j -\gamma_j}{2}) = -\frac{\gamma'_j-\gamma_j}{2}$. Now $\frak{g}_0 = \frak{\delta}_p \oplus 
\frak{\delta}_{n-p}$ and Ad$(\sigma_X) (\frak{\delta}_p) = \frak{\delta}_p, \textrm{Ad}(\sigma_X)(\frak{\delta}_{n-p}) = \frak{\delta}_{n-p}$. So 
Ad$(\sigma_X)$ is an inner automorphism of $\frak{g}_0$ iff $|\{ j : \cos 2c_j = -1\}|$ is even \cite[PLANCHE IV]{bourbaki}. 
So if $s'_X \in W (\frak{g}_0, \frak{t}_0^{\mathbb{C}})$ be such that Ad$(\sigma_X) \circ s'_X (\Delta_0^+) 
= \Delta_0^+$, and $s_X = \textrm{Ad}(\sigma_X)\circ s'_X$, then either 
$s_X(\alpha_0) = \alpha_0, s_X (\psi_j) = \psi_j$  for all $1 \le j \le n, j \neq p$; or 
$s_X(\alpha_0) = \psi_1, s_X(\psi_1) = \alpha_0, s_X(\psi_{n-1}) = \psi_n, s_X(\psi_n) = \psi_{n-1}, 
s_X (\psi_j) = \psi_j$  for all $2 \le j \le n-2,\ j \neq p$. In any case, det$(s_X|_{\frak{t}_0^{\mathbb{C}}}) = 1$. 
Hence the canonical action of $G(\mu)$ on $X(\mu)$ is orientation preserving for $\mu = \bar{\sigma}, \bar{\sigma}\bar{\theta}$. 

  Let $p = n-p$ that is $n=2p$. Then $\frak{g}_0$ is the sum of two ideals, each is isomorphic with $\frak{\delta}_p$. Let $\frak{\delta}_p^{(1)}$ 
be the ideal of $\frak{g}_0$ whose Dynkin diagram is generated by $\{\alpha_0, \psi_1, \ldots , \psi_{p-1} \}$, and 
$\frak{\delta}_p^{(2)}$ be the ideal of $\frak{g}_0$ whose Dynkin diagram is generated by $\{\psi_{p+1}, \ldots , \psi_n \}$.   
Let exp$(-2X) \in \tilde{Z}$. If $\cos 2c_j = \cos 2c'_j$ for all $1 \le j \le p$, then as before det$(s_X|_{\frak{t}_0^{\mathbb{C}}}) = 1$. 
If $\cos 2c_j \neq \cos 2c'_j$ for some $j$, then $\textrm{Ad}(\sigma_X) (\frac{\gamma_j + \gamma'_j}{2}) = \pm \frac{\gamma'_j - \gamma_j}{2}, 
\textrm{Ad}(\sigma_X) (\frac{\gamma'_j - \gamma_j}{2}) = \pm \frac{\gamma_j + \gamma'_j}{2}$. Hence Ad$(\sigma_X)(\frak{\delta}_p^{(1)}) = 
\frak{\delta}_p^{(2)}$, and so Ad$(\sigma_X)$ is not an inner automorphism of $\frak{g}_0$. Therefore $s_X$ induces a non-trivial Dynkin diagram 
automorphism of $\frak{g}_0$. Since Ad$(\sigma_X)(\frak{\delta}_p^{(1)}) = \frak{\delta}_p^{(2)}$, we have 
$s_X(\psi_j) = \psi_{n-j}$ for all $2 \le j \le n-2,\ j \neq p$, $s_X(\alpha_0) = \psi_{n-1} \textrm{ or } \psi_n,\ s_X(\psi_1) 
= \psi_n \textrm{ or } \psi_{n-1},\ s_X (\psi_{n-1}) = \alpha_0 \textrm{ or } \psi_1, s_X (\psi_n) = \psi_1 \textrm{ or } \alpha_0$. So 
det$(s_X|_{\frak{t}_0^{\mathbb{C}}}) = (-1)^p$. Hence the canonical action of $G(\mu)$ on $X(\mu)$ is orientation preserving for 
$\mu = \bar{\sigma}, \bar{\sigma}\bar{\theta}$ if $p$ is even that is, if $n \in 4\mathbb{Z}$ and $p = \frac{n}{2}$.

7. 
\begin{center} 
\begin{tikzpicture} 

\node [left] at (-0.5,0) {$\frak{\delta}_{n+1}^{(2)} :$}; 
\node [left] at (-0.4,-0.6) {$(n > 1)$}; 
\draw (0,0) circle [radius = 0.1]; 
\draw (1,0) circle [radius = 0.1]; 
\draw (2.5,0) circle [radius = 0.1]; 
\draw (3.5,0) circle [radius = 0.1]; 
\node [above] at (0.05, 0.05) {$\alpha_0$}; 
\node [above] at (1.05,0.05) {$\psi_1$}; 
\node [above] at (2.75,0.05) {$\psi_{n-1}$}; 
\node [above] at (3.55,0.05) {$\psi_n$}; 
\draw (0.1,0) -- (0.2,0.1); 
\draw (0.1,0) -- (0.2,-0.1); 
\draw (0.15,0.025) -- (0.9,0.025); 
\draw (0.15,-0.025) -- (0.9,-0.025); 
\draw (1.1,0) -- (1.5,0); 
\draw [dotted] (1.5,0) -- (2,0); 
\draw (2,0) -- (2.4,0); 
\draw (3.4,0) -- (3.3,0.1); 
\draw (3.4,0) -- (3.3,-0.1); 
\draw (2.6,0.025) -- (3.35,0.025); 
\draw (2.6,-0.025) -- (3.35,-0.025); 

\end{tikzpicture} 
\end{center} 

   Here $\frak{g} = \frak{\delta}_{n+1} (n >1)$, and $\alpha_0 + \psi_1 +\psi_2 + \cdots + \psi_{n-1} + \psi_n = 0$. 

  (i) First assume that $\sigma$ is an involution of $\frak{g}$ of type $(0, \ldots ,0, s_p , 0, \ldots , 0; 2)(0 \le p \le n,\ p \neq n-p)$ with $s_p = 1$. Then 
$\frak{u}_0 = \frak{so}(2p+1) \oplus \frak{so}(2n-2p+1)$ and $\frak{g}_0 = \frak{b}_p \oplus \frak{b}_{n-p}\ (p \neq n-p)$, which does not have 
any non-trivial Dynkin diagram automorphism. So det$(s_X|_{\frak{t}_0^{\mathbb{C}}}) = 1$ for all $X \in \frak{a}$ with 
$\textrm{exp}(-2X) \in \tilde{Z}$. Hence the canonical action of $G(\mu)$ on $X(\mu)$ is orientation preserving for 
$\mu = \bar{\sigma}, \bar{\sigma}\bar{\theta}$. 

 (ii) Next assume that $n > 2$ is even and $\sigma$ is an involution of type $(0, \ldots ,0, s_p , 0, \ldots , 0; 2)$ with $p = \frac{n}{2}$ and $s_p = 1$. 
Then $\frak{g}_0$ is the sum of two ideals, each is isomorphic with $\frak{b}_p$. Let $\frak{b}_p^{(1)}$ 
be the ideal of $\frak{g}_0$ whose Dynkin diagram is generated by $\{\alpha_0, \psi_1, \ldots , \psi_{p-1} \}$, and 
$\frak{b}_p^{(2)}$ be the ideal of $\frak{g}_0$ whose Dynkin diagram is generated by $\{\psi_{p+1}, \ldots , \psi_n \}$.   
The diagram $\frak{\delta}_{n+1}^{(2)} (n > 2)$ is corresponding to the Dynkin diagram automorphism $\bar{\nu}$ 
of $\frak{\delta}_{n+1}$ given by $\bar{\nu}(\phi_j) = \phi_{j}$ for all $1 \le j \le n-1$, $\bar{\nu}(\phi_n) = \phi_{n+1}$, 
$\bar{\nu}(\phi_{n+1}) = \phi_{n}$.   

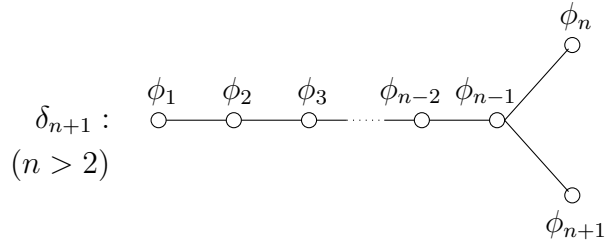
\begin{center} 
\begin{tikzpicture} 

\draw (0,0) circle [radius = 0.1]; 
\draw (1,0) circle [radius = 0.1]; 
\draw (2,0) circle [radius = 0.1]; 
\draw (3.5,0) circle [radius = 0.1]; 
\draw (4.5,0) circle [radius = 0.1]; 
\draw (5.5,1) circle [radius = 0.1]; 
\draw (5.5,-1) circle [radius = 0.1]; 
\node [above] at (0.05,0.05) {$\phi_1$}; 
\node [above] at (1.05,0.05) {$\phi_2$}; 
\node [above] at (2.05,0.05) {$\phi_3$}; 
\node [above] at (3.35,0.05) {$\phi_{n-2}$}; 
\node [above] at (4.35,0.05) {$\phi_{n-1}$}; 
\node [above] at (5.55,1.05) {$\phi_{n}$}; 
\node [below] at (5.55,-1.05) {$\phi_{n+1}$}; 
\node [left] at (-0.5,0) {$\frak{\delta}_{n+1} :$}; 
\node [left] at (-0.5,-0.6) {$(n > 2)$}; 
\draw (0.1,0) -- (0.9,0); 
\draw (1.1,0) -- (1.9,0); 
\draw (2.1,0) -- (2.5,0); 
\draw [dotted] (2.5,0) -- (3,0); 
\draw (3,0) -- (3.4,0); 
\draw (3.6,0) -- (4.4,0); 
\draw (4.6,0) -- (5.45,0.95); 
\draw (4.6,0) -- (5.45,-0.95); 

\end{tikzpicture} 
\end{center} 
  
  Now we want to determine $\{\alpha \in \Delta^+ : \sigma(H_\alpha ^*) = H_\alpha ^* ,\ \frak{g}_\alpha \subset \frak{g}_1 \}$. 
Note that  $\{\alpha \in \Delta^+ : \sigma(H_\alpha ^*) = H_\alpha ^* \} = \{ \alpha \in \Delta^+ : n_{\phi_n}(\alpha) = 
n_{\phi_{n+1}}(\alpha) \} \\ 
= \{ \phi_i + \cdots + \phi_{j-1} , \ (\phi_i + \cdots \phi_{n-1}) + (\phi_j + \cdots + \phi_{n+1}) : 1 \le i < j \le n \}$.  \\
Again since $\sigma|_{\frak{h}} = \nu|_{\frak{h}}$, $\{\alpha \in \Delta^+ : \sigma(H_\alpha ^*) = H_\alpha ^* \} 
\subset (\frak{h}^\nu)^*$. 
Let $E_j$ be a non-zero root vector corresponding to the root $\phi_j$ for all $1 \le j \le n+1$. Define \\ 
$E_{ij} = [E_i , \ldots , E_{j-1}]$ for all $1 \le i < j-1 < j \le n+1$, $E_{ij} = E_i$ for all $1 \le i = j-1 < j \le n+1$; 
and $E'_{i(n+1)} = [E_i , \ldots , E_{n-1}, E_{n+1}]$ for all $1 \le i \le n-1$, $E'_{n(n+1)} = E_{n+1}$.   \\ 
Then $E_{ij} \neq 0,\ E'_{i(n+1)} \neq 0,\ \nu (E_{ij}) = E_{ij}\ (\textrm{for all } 1 \le i < j \le n),\ 
\nu (E_{i(n+1)}) = E'_{i(n+1)},\ \nu (E'_{i(n+1)}) = E_{i(n+1)}\ (\textrm{for all } 1 \le i \le n)$. 
Define \\ 
$E^{ij} = \big{[} E_{j(n+1)}, [ E_{ij} , E'_{j(n+1)}]\big{]}$ for all $1 \le i < j \le n$. \\   
Then $E^{ij} \neq 0,\ \nu (E^{ij}) = E^{ij}\ (\textrm{for all } 1 \le i < j \le n)$. \\ 
This shows that $\frak{g}_\alpha \subset \frak{g}_0^\nu$ for all $\alpha \in \{\alpha \in \Delta^+ : \sigma(H_\alpha ^*)= H_\alpha ^* \}$. 
Thus $\{\alpha \in \Delta^+ : \sigma(H_\alpha ^*) = H_\alpha ^* \} \subset \Delta (\frak{g}_0^\nu , \frak{h}^\nu)$. Now 
$\sigma (\frak{g}_0^\nu) = \frak{g}_0^\nu$ and so $\frak{g}_0^\nu = \frak{k}^\nu \oplus \frak{p}^\nu$, where 
$\frak{k}^\nu = \frak{g}_0^\nu \cap \frak{g}_0,\ \frak{p}^\nu = \frak{g}_0^\nu \cap \frak{g}_1$. Since $\frak{h}^\nu \subset 
\frak{k}^\nu,\ [\frak{k}^\nu , \frak{k}^\nu] \subset \frak{k}^\nu,\ [\frak{k}^\nu , \frak{p}^\nu] \subset \frak{p}^\nu$, and 
$(\frak{g}_0^\nu)_\alpha$ is one-dimensional for all $\alpha \in \Delta (\frak{g}_0^\nu , \frak{h}^\nu)$; we have 
$(\frak{g}_0^\nu)_\alpha \subset \frak{k}^\nu \textrm{ or } \frak{p}^\nu$. Thus \\ 
$\{\alpha \in \Delta^+ : \sigma(H_\alpha ^*) = H_\alpha ^* ,\ \frak{g}_\alpha \subset \frak{g}_1 \} 
= \{\alpha \in \Delta^+ : \sigma(H_\alpha ^*) = H_\alpha ^* ,\ (\frak{g}_0^\nu)_\alpha \subset \frak{p}^\nu \} \\
= \{\alpha \in \Delta^+ : n_{\phi_n}(\alpha)=n_{\phi_{n+1}}(\alpha), \textrm{ and } n_{\phi_p}(\alpha) \textrm{ is odd}\}$. Let \\
$\gamma_1 = \phi_p, \gamma'_1 = (\phi_p + \cdots + \phi_{n-1}) + (\phi_{p+1} + \cdots + \phi_{n+1}), \\ 
\gamma_2 = \phi_{p-1} + \phi_p + \phi_{p+1}, \gamma'_2 = (\phi_{p-1} + \cdots + \phi_{n-1}) + (\phi_{p+2} + \cdots + \phi_{n+1}), \ldots , \\ 
\gamma_p = \phi_1 + \cdots + \phi_p + \cdots + \phi_{n-1}, \gamma'_p = (\phi_1 + \cdots + \phi_{n-1}) + (\phi_n + \phi_{n+1})$. \\
Then $\{\gamma_1, \gamma'_1, \gamma_2, \gamma'_2, \ldots , \gamma_p, \gamma'_p \}$ is a maximal set of strongly orthogonal roots in 
$\{\alpha \in \Delta^+ : \sigma(H_\alpha ^*) = H_\alpha ^* ,\ \frak{g}_\alpha \subset \frak{g}_1 \}$, and 
$\frak{a} = \mathbb{R} i(H_{\phi _n}^* - H_{\phi_{n+1}}^*) \oplus 
\sum\limits_{j=1}^{p}(\mathbb{R} Y_{\gamma_j} + \mathbb{R} Y_{\gamma'_j})$ is a maximal abelian subspace of $\frak{u}_1$. \\ 
Let $X = ic_0(H_{\phi _n}^* - H_{\phi_{n+1}}^*) + \sum\limits_{j=1}^{p}(c_jY_{\gamma_j} + c'_j Y_{\gamma'_j}) \in \frak{a}$. 
Then exp$(-2X) \in \tilde{Z}$ iff $c_p + c'_p - c_{p-1} - c'_{p-1}, c_{p-1} + c'_{p-1} - c_{p-2} - c'_{p-2}, \ldots , c_2 + c'_2 - c_1 - c'_1, 2c_1, 
- c_1 + c'_1 + c_2 - c'_2, - c_2 + c'_2 + c_3 - c'_3, \ldots , - c_{p-1} + c'_{p-1} + c_p - c'_p , 2c_0 - c_p + c'_p, -2c_0 - c_p + c'_p \in \pi \mathbb{Z}$ 
iff $2c_0 \in \pi \mathbb{Z}, \cos 2c_j  = \pm{1}, \cos 2c'_j = \pm{1}$ for all $1 \le j \le p$. \\ 
If $\cos 2c_j  = \cos 2c'_j $ for all $1 \le j \le p$, then either $\textrm{Ad}(\sigma_X) (\gamma_j ) = \gamma_j, 
\textrm{Ad}(\sigma_X)(\gamma'_j) = \gamma'_j$, or $\textrm{Ad}(\sigma_X) (\gamma_j ) = -\gamma_j, 
\textrm{Ad}(\sigma_X) (\gamma'_j) = -\gamma'_j$. Thus for $1 \le j \le p$, either $\textrm{Ad}(\sigma_X) 
(\frac{\gamma_j  + \gamma'_j}{2}) = \frac{\gamma_j + \gamma'_j}{2}, \textrm{Ad}(\sigma_X) (\frac{\gamma'_j -\gamma_j}{2}) 
= \frac{\gamma'_j-\gamma_j}{2}$, or \\ 
$\textrm{Ad}(\sigma_X) (\frac{\gamma_j + \gamma'_j}{2}) = -\frac{\gamma_j + \gamma'_j}{2}, 
\textrm{Ad}(\sigma_X) (\frac{\gamma'_j -\gamma_j}{2}) = -\frac{\gamma'_j-\gamma_j}{2}$. Now $\frak{g}_0 = \frak{b}_p^{(1)} 
\oplus \frak{b}_p^{(2)}$ and Ad$(\sigma_X) (\frak{b}_p^{(1)}) = \frak{b}_p^{(1)}, \textrm{Ad}(\sigma_X)(\frak{b}_p^{(2)}) = \frak{b}_p^{(2)}$. 
Since $\frak{b}_p^{(1)} \cong \frak{b}_p \cong \frak{b}_p^{(2)}$ does not admit any non-trivial Dynkin diagram automorphism, we have 
det$(s_X|_{\frak{t}_0^{\mathbb{C}}}) = 1$. \\  
If $\cos 2c_j \neq \cos 2c'_j$ for some $j$, then $\textrm{Ad}(\sigma_X) (\frac{\gamma_j + \gamma'_j}{2}) = \pm \frac{\gamma'_j - \gamma_j}{2}, 
\textrm{Ad}(\sigma_X) (\frac{\gamma'_j - \gamma_j}{2}) = \pm \frac{\gamma_j + \gamma'_j}{2}$. Hence Ad$(\sigma_X)(\frak{b}_p^{(1)}) = 
\frak{b}_p^{(2)}$, and so Ad$(\sigma_X)$ is not an inner automorphism of $\frak{g}_0$. Therefore $s_X$ induces a non-trivial Dynkin diagram 
automorphism of $\frak{g}_0$. Since Ad$(\sigma_X)(\frak{b}_p^{(1)}) = \frak{b}_p^{(2)}$, we have 
$s_X(\psi_j) = \psi_{n-j}$ for all $1 \le j \le n-1,\ j \neq p$, $s_X(\alpha_0) = \psi_n, s_X(\psi_n) 
= \alpha_0$. So 
det$(s_X|_{\frak{t}_0^{\mathbb{C}}}) = (-1)^p$. Hence the canonical action of $G(\mu)$ on $X(\mu)$ is orientation preserving for 
$\mu = \bar{\sigma}, \bar{\sigma}\bar{\theta}$ if $p$ is even that is, if $n \in 4\mathbb{Z}$ and $p = \frac{n}{2}$. 

(iii) 
\begin{center} 
\begin{tikzpicture}

\node [left] at (-0.5,0) {$\frak{\delta}_3^{(2)} :$}; 
\draw (0,0) circle [radius = 0.1]; 
\draw (1,0) circle [radius = 0.1]; 
\draw (2,0) circle [radius = 0.1]; 
 \node [above] at (0.05, 0.05) {$\alpha_0$}; 
\node [above] at (1.05,0.05) {$\psi_1$}; 
\node [above] at (2.05,0.05) {$\psi_2$}; 
 \draw (0.1,0) -- (0.2,0.1); 
\draw (0.1,0) -- (0.2,-0.1); 
\draw (0.15,0.025) -- (0.9,0.025); 
\draw (0.15,-0.025) -- (0.9,-0.025); 
\draw (1.9,0) -- (1.8,0.1); 
\draw (1.9,0) -- (1.8,-0.1); 
\draw (1.1,0.025) -- (1.85,0.025); 
\draw (1.1,-0.025) -- (1.85,-0.025); 

\end{tikzpicture} 
\end{center} 

  Finally assume that $n = 2$ and $\sigma$ is an involution of type $(0,1,0; 2)$. 
Then $\frak{g}_0$ is the sum of two ideals, each is isomorphic with $\frak{a}_1$. Let $\frak{a}_1^{(1)}$ 
be the ideal of $\frak{g}_0$ whose Dynkin diagram is generated by $\{\alpha_0\}$, and 
$\frak{a}_1^{(2)}$ be the ideal of $\frak{g}_0$ whose Dynkin diagram is generated by $\{\psi_2 \}$.   
The diagram $\frak{\delta}_{3}^{(2)}$ is corresponding to the Dynkin diagram automorphism $\bar{\nu}$ 
of $\frak{\delta}_{3}$ given by $\bar{\nu}(\phi_j) = \phi_{4-j}$ for all $1 \le j \le 3$.

\begin{center} 
\begin{tikzpicture} 

\draw (0,0) circle [radius = 0.1]; 
\draw (1,0) circle [radius = 0.1]; 
\draw (2,0) circle [radius = 0.1]; 
\node [left] at (-0.5,-0.1) {$\frak{\delta}_{3} :$}; 
\node [below] at (0.05,-0.05) {$\phi_1$}; 
\node [below] at (1.05,-0.05) {$\phi_2$}; 
\node [below] at (2.05,-0.05) {$\phi_3$}; 
\draw (0.1,0) -- (0.9,0); 
\draw (1.1,0) -- (1.9,0); 

\end{tikzpicture} 
\end{center} 

  Now $\{\alpha \in \Delta^+ : \sigma(H_\alpha ^*) = H_\alpha ^* ,\ \frak{g}_\alpha \subset \frak{g}_1 \} = 
\{\alpha \in \Delta^+ : n_{\phi_1}(\alpha)=n_{\phi_{3}}(\alpha), \textrm{ and } n_{\phi_2}(\alpha) \textrm{ is odd}\}$ 
(as in the case 7(ii)), since $\psi_1 = \phi_2|_{\frak{h}^\nu}$. Let \\
$\gamma_1 = \phi_2, \gamma_2 = \phi_1 + \phi_2 + \phi_3$.  \\
Then $\{\gamma_1, \gamma_2 \}$ is a maximal set of strongly orthogonal roots in 
$\{\alpha \in \Delta^+ : \sigma(H_\alpha ^*) = H_\alpha ^* ,\ \frak{g}_\alpha \subset \frak{g}_1 \}$, and 
$\frak{a} = \mathbb{R} i(H_{\phi _1}^* - H_{\phi_{3}}^*) \oplus 
\mathbb{R} Y_{\gamma_1} + \mathbb{R} Y_{\gamma_2}$ is a maximal abelian subspace of $\frak{u}_1$. \\ 
Let $X = ic_0(H_{\phi _1}^* - H_{\phi_{3}}^*) + c_1Y_{\gamma_1} + c_2 Y_{\gamma_2} \in \frak{a}$. 
Then exp$(-2X) \in \tilde{Z}$ iff $2c_1, 2c_0 - c_1 + c_2, -2c_0 - c_1 + c_2 \in \pi \mathbb{Z}$ 
iff $2c_0 \in \pi \mathbb{Z}, \cos 2c_1 = \pm{1}, \cos 2c_2 = \pm{1}$. \\ 
Let $\cos 2c_1 = -1, \textrm{ and } \cos 2c_2 = 1$. Then \\ 
$\textrm{Ad}(\sigma_X) (\psi_1) = -\psi_1, \\ 
\textrm{Ad}(\sigma_X) (\alpha_0) = \textrm{Ad}(\sigma_X) (-\frac{\psi_1}{2} + \alpha_0 + \frac{\psi_1}{2}) 
= \frac{\psi_1}{2} + \alpha_0 + \frac{\psi_1}{2} = \alpha_0 + \psi_1 = -\psi_2, \textrm{and similarly} \\ 
\textrm{Ad}(\sigma_X) (\psi_2) = -\alpha_0$. \\ 
Let $w_{\frak{g}_0}^0 \in W(\frak{g}_0,\frak{t}_0^{\mathbb{C}})$ be the 
longest element that is, $w_{\frak{g}_0}^0(\alpha_0) = -\alpha_0$ and 
$w_{\frak{g}_0}^0(\psi_2) = -\psi_2$. Then 
$s_X = \textrm{Ad}(\sigma_X)\circ w_{\frak{g}_0}^0$ with $s_X (\Delta_0^+) = \Delta_0^+$. Now  
$s_X(\alpha_0) = \psi_2$ and $s_X(\psi_2) = \alpha_0$. So 
det$(s_X|_{\frak{t}_0^{\mathbb{C}}}) = -1$. 

8.   
\begin{center} 
\begin{tikzpicture}  

\draw (0,0) circle [radius = 0.1]; 
\draw (1,0) circle [radius = 0.1]; 
\draw (2,0) circle [radius = 0.1]; 
\draw (2,1) circle [radius = 0.1]; 
\draw (2,2) circle [radius = 0.1]; 
\draw (3,0) circle [radius = 0.1]; 
\draw (4,0) circle [radius = 0.1]; 
\node [below] at (0.05,-0.05) {$\psi_6$}; 
\node [below] at (1.05,-0.05) {$\psi_5$}; 
\node [below] at (2.05,-0.05) {$\psi_4$}; 
\node [right] at (2.05,1) {$\psi_2$}; 
\node [right] at (2.05,2) {$\alpha_0$}; 
\node [below] at (3.05,-0.05) {$\psi_3$}; 
\node [below] at (4.05,-0.05) {$\psi_1$}; 
\node [left] at (-0.5,0) {$\frak{e}_6^{(1)} :$}; 
\draw (0.1,0) -- (0.9,0); 
\draw (1.1,0) -- (1.9,0); 
\draw (2,0.1) -- (2,0.9); 
\draw (2,1.1) -- (2,1.9); 
\draw (2.1,0) -- (2.9,0); 
\draw (3.1,0) -- (3.9,0); 

\end{tikzpicture} 
\end{center} 

Here $\frak{g} = \frak{e}_6$ and $\alpha_0 + \psi_1+ 2\psi_2 + 2\psi_3 + 3\psi_4 + 2\psi_5 + \psi_6 = 0$. 

(i) First assume that 
$\sigma$ is an involution of $\frak{g}$ of type $(1, 1, 0,0,0,0,0; 1)$ (similarly for types 
$(1,0,0,0,0, 0,1;1) \textrm{ or }(0,1,0,0,0, 0,1;1)$). Then $\frak{u}_0 = \frak{so}(10) 
\oplus i\mathbb{R}$ and $X(\bar{\sigma}\bar{\theta})$ is a 
Hermitian symmetric space. This Hermitian symmetric space is not of tube type. 
So the canonical action of $G(\mu)$ on $X(\mu)$ is 
orientation preserving for $\mu = \bar{\sigma}, \bar{\sigma}\bar{\theta}$, by Remark \ref{or}(iv). 

(ii) Next assume that $\sigma$ is an involution of $\frak{g}$ of type $(0,0,1,0,0,0 , 0; 1)$ (similarly for types 
$(0,0,0,1,0,0,0;1) \textrm{ or } (0,0,0,0, 0,1,0;1)$). Then 
$\frak{u}_0 = \frak{su}(2) \oplus \frak{su}(6)$ and $\frak{g}_0 = \frak{a}_1 \oplus \frak{a}_5$, which has only one 
non-trivial Dynkin diagram automorphism namely, $\alpha_0 \mapsto \alpha_0, \psi_1 \mapsto \psi_6, 
\psi_3 \mapsto \psi_5, \psi_4 \mapsto \psi_4, \psi_5 \mapsto \psi_3, \psi_6 \mapsto \psi_1$; and this is an even 
permutation. So det$(s_X|_{\frak{t}_0^{\mathbb{C}}}) = 1$ for all $X \in \frak{a}$ with 
$\textrm{exp}(-2X) \in \tilde{Z}$. Hence the canonical action of $G(\mu)$ on $X(\mu)$ is orientation preserving for 
$\mu = \bar{\sigma}, \bar{\sigma}\bar{\theta}$. 

9. 
\begin{center} 
\begin{tikzpicture} 

\node [left] at (-0.5,0) {$\frak{e}_6^{(2)} :$}; 
\draw (0,0) circle [radius = 0.1]; 
\draw (1,0) circle [radius = 0.1]; 
\draw (2,0) circle [radius = 0.1]; 
\draw (3,0) circle [radius = 0.1]; 
\draw (4,0) circle [radius = 0.1]; 
\node [above] at (0.05,0.05) {$\psi_1$}; 
\node [above] at (1.05,0.05) {$\psi_2$}; 
\node [above] at (2.05,0.05) {$\psi_3$}; 
\node [above] at (3.05,0.05) {$\psi_4$}; 
\node [above] at (4.05,0.05) {$\alpha_0$}; 
\draw (0.1,0) -- (0.9,0); 
\draw (1.9,0) -- (1.8,0.1); 
\draw (1.9,0) -- (1.8,-0.1); 
\draw (1.1,0.025) -- (1.85,0.025); 
\draw (1.1,-0.025) -- (1.85,-0.025); 
\draw (2.1,0) -- (2.9,0); 
\draw (3.1,0) -- (3.9,0); 

\end{tikzpicture} 
\end{center} 

   Here $\frak{g} = \frak{e}_6$, and $\alpha_0 + \psi_1 +2\psi_2 + 3\psi_3 + 2\psi_4 = 0$. 

  (i) First assume that $\sigma$ is an involution of $\frak{g}$ of type $(1,0,0,0,0; 2)$. Then 
$\frak{g}_0 = \frak{f}_4$, which does not have 
any non-trivial Dynkin diagram automorphism. So det$(s_X|_{\frak{t}_0^{\mathbb{C}}}) = 1$ for all $X \in \frak{a}$ with 
$\textrm{exp}(-2X) \in \tilde{Z}$. Hence the canonical action of $G(\mu)$ on $X(\mu)$ is orientation preserving for 
$\mu = \bar{\sigma}, \bar{\sigma}\bar{\theta}$. 

 (ii) Next assume that $\sigma$ is an involution of $\frak{g}$ of type $(0, 1, 0, 0, 0; 2)$. Then 
$\frak{g}_0 = \frak{c}_4$, which does not have 
any non-trivial Dynkin diagram automorphism. So det$(s_X|_{\frak{t}_0^{\mathbb{C}}}) = 1$ for all $X \in \frak{a}$ with 
$\textrm{exp}(-2X) \in \tilde{Z}$. Hence the canonical action of $G(\mu)$ on $X(\mu)$ is orientation preserving for 
$\mu = \bar{\sigma}, \bar{\sigma}\bar{\theta}$.

10. 
\begin{center} 
\begin{tikzpicture} 

\draw (0,0) circle [radius = 0.1]; 
\draw (1,0) circle [radius = 0.1]; 
\draw (2,0) circle [radius = 0.1]; 
\draw (3,0) circle [radius = 0.1]; 
\draw (3,1) circle [radius = 0.1]; 
\draw (4,0) circle [radius = 0.1]; 
\draw (5,0) circle [radius = 0.1]; 
\draw (6,0) circle [radius = 0.1]; 
\node [below] at (0.05,-0.05) {$\psi_7$}; 
\node [below] at (1.05,-0.05) {$\psi_6$}; 
\node [below] at (2.05,-0.05) {$\psi_5$}; 
\node [below] at (3.05,-0.05) {$\psi_4$}; 
\node [right] at (3.05,1) {$\psi_2$}; 
\node [below] at (4.05,-0.05) {$\psi_3$}; 
\node [below] at (5.05,-0.05) {$\psi_1$}; 
\node [below] at (6.05, -0.15) {$\alpha_0$}; 
\node [left] at (-0.5,0) {$\frak{e}_7^{(1)} :$}; 
\draw (0.1,0) -- (0.9,0); 
\draw (1.1,0) -- (1.9,0); 
\draw (2.1,0) -- (2.9,0); 
\draw (3,0.1) -- (3,0.9); 
\draw (3.1,0) -- (3.9,0); 
\draw (4.1,0) -- (4.9,0); 
\draw (5.1,0) -- (5.9,0); 

\end{tikzpicture} 
\end{center} 

  Here $\frak{g} = \frak{e}_7$ and $\alpha_0 + 2\psi_1 + 2\psi_2 + 3\psi_3 + 4\psi_4 + 3\psi_5 + 2\psi_6 + \psi_7 = 0$. 

(i) First assume that $\sigma$ is an involution of $\frak{g}$ of type $(1, 0,0,0,0,0, 0,1; 1)$. 
Then $\frak{g}_0 = \frak{e}_6 \oplus \mathbb{C}$ and $X(\bar{\sigma}\bar{\theta})$ is a 
Hermitian symmetric space of tube type. Now $[\frak{g}_0, \frak{g}_0] = \frak{e}_6$, which has only one 
non-trivial Dynkin diagram automorphism namely, $\psi_1 \mapsto \psi_6, \psi_2 \mapsto \psi_2, 
\psi_3 \mapsto \psi_5, \psi_4 \mapsto \psi_4, \psi_5 \mapsto \psi_3, \psi_6 \mapsto \psi_1$; and this is an even 
permutation. Let 
\[ \gamma_1 = \psi_7, \gamma_2 = \psi_2 + \psi_3 + 2\psi_4 + 2\psi_5 + 2\psi_6 + \psi_7, 
\gamma_3 = 2\psi_1 + 2\psi_2 + 3\psi_3 + 4\psi_4 + 3\psi_5 + 2\psi_6 + \psi_7.\]
Then $\{\gamma_1, \gamma_2 ,\gamma_3 \}$ is a maximal set of strongly orthogonal roots in 
$\{\alpha \in \Delta^+ : \frak{g}_\alpha \subset \frak{g}_1 \}$.  Let $X \in \frak{a}$ with $\textrm{exp}(-2X) \in \tilde{Z}$ and 
$\textrm{Ad}(\sigma_X)(Z) = - Z$, where $Z = \sum\limits_{j =1}^{3} iH_{\gamma_j}^*$. Then 
det$(s_X|_{\frak{t}_0^{\mathbb{C}}}) = - 1$, even if $s_X$ induces the non-trivial Dynkin diagram automorphism of 
$[\frak{g}_0 , \frak{g}_0]$. 

(ii) Next assume that $\sigma$ is an involution of $\frak{g}$ of type $(0, 0, 1, 0,0,0,0, 0;1)$. Then $\frak{u}_0 = \frak{su}(8)$. Define \\
$\gamma_1 = \psi_2 , \gamma_2 = \psi_2 + \psi_3 + 2 \psi_4 + \psi_5 , \gamma_3 = \psi_1 + \psi_2 + \psi_3 + 2\psi_4 + \psi_5 + \psi_6, \\
\gamma_4 =\psi_1 + \psi_2+ 2\psi_3 + 2\psi_4 + 2 \psi_5 + \psi_6 , \gamma_5 = \psi_2 + \psi_3 + 2\psi_4 + 2\psi_5 + 2\psi_6 + \psi_7, \\ 
\gamma_6 = \psi_1 + \psi_2+ \psi_3 + 2\psi_4 + 2 \psi_5 + \psi_6 + \psi_7, \gamma_7 = \psi_1 + \psi_2+ 2\psi_3 + 2\psi_4 + \psi_5 + \psi_6 + \psi_7$. \\ 
Then $\{\gamma_1, \gamma_2, \ldots , \gamma_7 \}$ is a maximal set of strongly orthogonal roots in 
$\{\alpha \in \Delta^+ : \frak{g}_\alpha \subset \frak{g}_1 \}$, and 
$\frak{a} = \sum\limits_{j=1}^{7}\mathbb{R} Y_{\gamma_j}$ is a maximal abelian subspace of $\frak{u}_1$. Also we have \\ 
$\psi_1 = \frac{1}{2}(-\gamma_2 + \gamma_3 - \gamma_5 + \gamma_6) , \psi_3 = \frac{1}{2}(-\gamma_3 + \gamma_4 - \gamma_6 + \gamma_7),  
\psi_4 = \frac{1}{2}(-\gamma_1 + \gamma_2 + \gamma_3 - \gamma_4), \psi_5 = \frac{1}{2}(-\gamma_3 + \gamma_4 + \gamma_6 - \gamma_7), \\ 
\psi_6 = \frac{1}{2}(-\gamma_2 + \gamma_3 + \gamma_5 - \gamma_6),  \psi_7 = \frac{1}{2}(-\gamma_3 - \gamma_4 + \gamma_6 + \gamma_7), 
\alpha_0 = \frac{1}{2}(-\gamma_3 - \gamma_4 - \gamma_6 - \gamma_7) $. \\ 
Let $X = \frac{\pi}{2} (Y_{\gamma_5} + Y_{\gamma_6} + Y_{\gamma_7})$. Then  
$\frac{\pi}{2} \psi_j (H_{\gamma_5}^* + H_{\gamma_6}^* + H_{\gamma_7}^* ) = 0$ for $1 \le j \le 6$, and 
$\frac{\pi}{2} \psi_7 (H_{\gamma_5}^* + H_{\gamma_6}^* + H_{\gamma_7}^* ) = \pi$; 
hence $X \in \frak{a}$ with exp$(-2X) \in \tilde{Z}$. \\ 
Now Ad$(\sigma_X)(H_{\gamma_j}^*) = H_{\gamma_j}^*$ for all $1 \le j \le 4$, and 
Ad$(\sigma_X)(H_{\gamma_j}^*) = -H_{\gamma_j}^*$ for all $5 \le j \le 7$. Thus \\ 
$\textrm{Ad}(\sigma_X) (\alpha_0) = \psi_7,   \textrm{Ad}(\sigma_X) (\psi_1) = \psi_6, \textrm{Ad}(\sigma_X) (\psi_3) = \psi_5, \\ 
\textrm{Ad}(\sigma_X) (\psi_4) = \psi_4, \textrm{Ad}(\sigma_X)(\psi_5) = \psi_3, \textrm{Ad}(\sigma_X) (\psi_6) = \psi_1, 
\textrm{Ad}(\sigma_X) (\psi_7) = \alpha_0$. \\ 
Therefore Ad$(\sigma_X) (\Delta_0^+) = \Delta_0^+$, and $s_X = \textrm{Ad}(\sigma_X)$. So det$(s_X|_{\frak{t}_0^{\mathbb{C}}}) = -1$. 

(iii)  Assume that 
$\sigma$ is an involution of $\frak{g}$ of type $(0, 1, 0,0,0,0,0,0; 1)$ (similarly for type  
$(0,0,0,0,0, 0,1,0;1)$). Then $\frak{u}_0 = \frak{su}(2) \oplus \frak{so}(12)$. Define \\ 
$\gamma_1 = \psi_1, \gamma_2 = \psi_1 + \psi_2 + 2\psi_3 + 2\psi_4 + \psi_5, 
\gamma_3 = \psi_1 + \psi_2 + 2\psi_3 + 2\psi_4 + 2\psi_5 + 2\psi_6 + \psi_7, \gamma_4 = \psi_1 + 2\psi_2 + 2\psi_3 + 4\psi_4 + 3\psi_5 + 2\psi_6 + \psi_7$. 
Then $\{\gamma_1, \gamma_2 ,\gamma_3, \gamma_4 \}$ is a maximal set of strongly orthogonal roots in 
$\{\alpha \in \Delta^+ : \frak{g}_\alpha \subset \frak{g}_1 \}$, and 
$\frak{a} = \sum\limits_{j=1}^{4}\mathbb{R} Y_{\gamma_j}$ is a maximal abelian subspace of $\frak{u}_1$. \\ 
Now $\alpha_0\Big{|}_{(\frak{t}^-)^{\mathbb{C}}}= -\frac{1}{2}(\gamma_1 + \gamma_2 + \gamma_3 + \gamma_4), 
\psi_2\Big{|}_{(\frak{t}^-)^{\mathbb{C}}}, \psi_5\Big{|}_{(\frak{t}^-)^{\mathbb{C}}}, \psi_7\Big{|}_{(\frak{t}^-)^{\mathbb{C}}} = 0, 
\psi_3\Big{|}_{(\frak{t}^-)^{\mathbb{C}}}= \frac{1}{2}(-\gamma_1 + \gamma_2 + \gamma_3 - \gamma_4), 
\psi_4\Big{|}_{(\frak{t}^-)^{\mathbb{C}}} = \frac{1}{2}(-\gamma_3 + \gamma_4), \psi_6\Big{|}_{(\frak{t}^-)^{\mathbb{C}}} = \frac{1}{2}(-\gamma_2 + \gamma_3)$. \\ 
Let $X = \sum\limits_{j=1}^{4} c_j Y_{\gamma_j} \in \frak{a}$. Then exp$(-2X) \in \tilde{Z}$ iff 
$2c_1 , -c_1 + c_2 + c_3 - c_4, -c_3 + c_4 , -c_2 + c_3 \in \pi \mathbb{Z}$ iff  $\cos 2c_j = \cos 2c_1 = \pm 1$ for all $1 \le j \le 4$. 
Assume that $X = \sum\limits_{j=1}^{4} c_j Y_{\gamma_j}$ with exp$(-2X) \in \tilde{Z}$ and $\cos 2c_1 = -1$. Then 
Ad$(\sigma_X)(H_{\gamma_j}^*) = -H_{\gamma_j}^*$, and 
Ad$(\sigma_X) (H) = H$ 
for all $\{ H \in \frak{t}_0^{\mathbb{C}} : \gamma_j (H) =0 \textrm{ for all } 1\le j \le 4\}$. Thus \\ 
Ad$(\sigma_X) (\alpha_0 ) = -\alpha_0,\ \textrm{Ad}(\sigma_X) (\psi_j ) = \psi_j$, for $j = 2,5,7$;  \\ 
Ad$(\sigma_X) (\psi_3) = \textrm{Ad}(\sigma_X) (\frac{1}{2}(-\gamma_1 + \gamma_2 + \gamma_3 - \gamma_4) + \psi_3 - 
\frac{1}{2}(-\gamma_1 + \gamma_2 + \gamma_3 - \gamma_4)) \\ 
= -\frac{1}{2}( -\gamma_1 + \gamma_2 + \gamma_3 - \gamma_4) + \psi_3 - \frac{1}{2}(-\gamma_1 + \gamma_2 + \gamma_3 - \gamma_4) 
= \psi_3 -  (-\gamma_1 + \gamma_2 + \gamma_3 - \gamma_4) = -\psi_3; \\ 
\textrm{Ad}(\sigma_X) (\psi_4) = \textrm{Ad}(\sigma_X) (\frac{1}{2}(-\gamma_3 + \gamma_4) + \psi_4 - 
\frac{1}{2}(- \gamma_3 + \gamma_4)) \\ 
= -\frac{1}{2}( -\gamma_3 + \gamma_4) + \psi_4 - \frac{1}{2}(-\gamma_3 + \gamma_4) 
= \psi_4 -  (-\gamma_3 + \gamma_4) = -\psi_2 - \psi_4 - \psi_5; \textrm{ and} \\  
\textrm{Ad}(\sigma_X) (\psi_6) = \textrm{Ad}(\sigma_X) (\frac{1}{2}(-\gamma_2 + \gamma_3) + \psi_6 - 
\frac{1}{2}(- \gamma_2 + \gamma_3)) \\ 
= -\frac{1}{2}( -\gamma_2 + \gamma_3) + \psi_6 - \frac{1}{2}(-\gamma_2 + \gamma_3) 
= \psi_6 -  (-\gamma_2 + \gamma_3) = -\psi_5 - \psi_6 - \psi_7$. \\ 
So Ad$(\sigma_X) (\{ \alpha_0, \psi_2, \psi_3, \psi_4, \psi_5, \psi_6, \psi_7 \}) = \{-\alpha_0, \psi_2, -\psi_3, -\psi_2 - \psi_4 - \psi_5, 
\psi_5, -\psi_5 - \psi_6-\psi_7, \psi_7 \}$. \\ 
Let $w_{\frak{g}_0}^0 \in W(\frak{g}_0,\frak{t}_0^{\mathbb{C}})$ be the 
longest element that is, $w_{\frak{g}_0}^0 (\alpha_0) = -\alpha_0,\ w_{\frak{g}_0}^0 (\psi_j) = -\psi_j$ for all $2 \le j \le 7$; and 
 $s'_X = s_{\psi_7}s_{\psi_5}s_{\psi_2}w_{\frak{g}_0}^0$. Then 
$s'_X(\alpha_0) = -\alpha_0,\ s'_X(\psi_j) = \psi_j$ for $j = 2, 5, 7$, $s'_X(\psi_3) = -\psi_3,\ s'_X(\psi_4) = -\psi_2 - \psi_4 - \psi_5$, and 
$s'_X(\psi_6) = -\psi_5 - \psi_6 - \psi_7$. Thus if  
$s_X = \textrm{Ad}(\sigma_X)\circ s'_X$, then $s_X (\Delta_0^+) = \Delta_0^+$. Clearly $s_X(\psi_j) = \psi_j$ for all $2 \le j \le 7$, and 
$s_X(\alpha_0) = \alpha_0$. So  det$(s_X|_{\frak{t}_0^{\mathbb{C}}}) = 1$. Hence  
the canonical action of $G(\mu)$ on $X(\mu)$ is orientation preserving for $\mu = \bar{\sigma}, \bar{\sigma}\bar{\theta}$.

\noindent 
\subsection{Table for the {\it condition Or} of a connected complex simple Lie group of adjoint type}
   Let $\bar{G} = \textrm{Int}(\frak{g})$, the connected component of Aut$(\frak{g})$. Then $\bar{G}$ is a connected 
complex simple Lie group of adjoint type, Lie$(\bar{G}) = \frak{g}$, and $\bar{G} \cong \tilde{G}/\tilde{Z}$. 
The {\it condition Or} for $\bar{G},\bar{\sigma}, \bar{\sigma}\bar{\theta}$; in each case, is given in the first table. 
If the {\it condition Or} for $\bar{G},\bar{\sigma}, \bar{\sigma}\bar{\theta}$ is satisfied, then the dimensions of 
$X(\bar{\sigma})$ and $X(\bar{\sigma}\bar{\theta})$ are given in the second table. Here \\ 
$S(GL(p,\mathbb{C}) \times GL(q,\mathbb{C})) = \Big{\{}
\begin{pmatrix} 
A_1 & 0 \\
0 & A_2 
\end{pmatrix} 
: A_1 \in GL(p,\mathbb{C}), A_2 \in GL(q,\mathbb{C}), \textrm{ and det}A_1\textrm{ det}A_2 =1 \Big{\}}$, and \\
 $S(U(p) \times U(q)) = \Big{\{}
\begin{pmatrix} 
A_1 & 0 \\
0 & A_2 
\end{pmatrix} 
: A_1 \in U(p), A_2 \in U(q), \textrm{ and det}A_1\textrm{ det}A_2 =1 \Big{\}}$. \\ 
We follow \cite{helgason} for other notations. 

\begin{landscape} 
\begin{table}
\caption{Table for the {\it condition Or} of a connected complex simple Lie group of adjoint type}\label{ortable} 
\begin{tabular}{|c|c|c|c|c|}
\toprule
\addlinespace[5pt]
$\frak{g}$ & 
 type of $\sigma$ & $X(\bar{\sigma})$ & $X(\bar{\sigma}\bar{\theta})$ & 
\begin{tabular}{c} is {\it condition}\\
{\it Or} for $\bar{G},\sigma,$\\ 
$\sigma\theta$ satisfied? \end{tabular} 
\\ 
\midrule
\addlinespace[5pt]
\begin{tabular}{c}$\frak{a}_{n-1}$ \\ $(n>1, n\in 4\mathbb{Z})$ \end{tabular} & 
\begin{tabular}{c}$(s_0,0, \ldots , 0,s_p,0,\ldots , 0;1)$ \\ $(1 \le p \le n/2)$ with $s_0 = 1 = s_p.$ \end{tabular} & 
$\frac{S(GL(p,\mathbb{C}) \times GL(n-p,\mathbb{C}))}{S(U(p) \times U(n-p))}$ & 
$\frac{SU(p,n-p)}{S(U(p) \times U(n-p))}$ & 
yes 
\\ 
\addlinespace[5pt] 
\hline 
\addlinespace[5pt]
\begin{tabular}{c}$\frak{a}_{n-1}$ \\ $(n>4, n\in 4\mathbb{Z})$ \end{tabular} & 
\begin{tabular}{c}$(1,0, \ldots ,0;2)$\\ \\$(0, \ldots ,0,1;2)$ \end{tabular} & 
\begin{tabular}{c}$\frac{Sp(\frac{n}{2},\mathbb{C})}{Sp(\frac{n}{2})}$\\ \\$\frac{SO(n,\mathbb{C})}{SO(n)}$ \end{tabular} & 
\begin{tabular}{c}$\frac{SU^*(n)}{Sp(\frac{n}{2})}$\\ \\$\frac{SL(n,\mathbb{R})}{SO(n)}$ \end{tabular} & 
\begin{tabular}{c} yes \\ \\ no \end{tabular} 
\\ 
\addlinespace[5pt] 
\hline 
\addlinespace[5pt] 
$\frak{\delta}_3$ & 
\begin{tabular}{c}$(1,0,0;2)$\\ \\$(0,1,0;2)$ \end{tabular} & 
\begin{tabular}{c}$\frac{SO(5,\mathbb{C})}{SO(5)}$\\ \\$\frac{SO(3,\mathbb{C})\times SO(3,\mathbb{C})}{SO(3)\times SO(3)}$  \end{tabular} & 
\begin{tabular}{c}$\frac{SO_0(1,5)}{SO(5)}$\\ \\$\frac{SO_0(3,3)}{SO(3)\times SO(3)}$ \end{tabular} & 
\begin{tabular}{c} yes \\ \\ no \end{tabular} 
\\ 
\addlinespace[5pt] 
\hline 
\addlinespace[5pt]
\begin{tabular}{c}$\frak{a}_{n-1}$\\($n>1, n \in 2 + 4\mathbb{Z}$) \end{tabular} & 
\begin{tabular}{c}$(s_0,0, \ldots , 0,s_p,0,\ldots , 0;1)$\\$(1 \le p < n/2)$ with $s_0 = 1 = s_p$\\ \\ 
$(s_0,0, \ldots , 0,s_{\frac{n}{2}},0,\ldots , 0;1)$\\with $s_0 = 1 = s_{\frac{n}{2}}$ \end{tabular} & 
\begin{tabular}{c}$\frac{S(GL(p,\mathbb{C}) \times GL(n-p,\mathbb{C}))}{S(U(p) \times U(n-p))}$\\ \\ 
$\frac{S(GL(\frac{n}{2},\mathbb{C}) \times GL(\frac{n}{2},\mathbb{C}))}{S(U(\frac{n}{2}) \times U(\frac{n}{2}))}$ \end{tabular} & 
\begin{tabular}{c}$\frac{SU(p,n-p)}{S(U(p) \times U(n-p))}$\\ \\$\frac{SU(\frac{n}{2},\frac{n}{2})}
{S(U(\frac{n}{2}) \times U(\frac{n}{2}))}$ \end{tabular} & 
\begin{tabular}{c} yes \\ \\ no \end{tabular} 
\\ 
\addlinespace[5pt]
\hline 
\addlinespace[5pt]
\begin{tabular}{c}$\frak{a}_{n-1}$\\$(n>2, n\in 2+4\mathbb{Z})$ \end{tabular} & 
\begin{tabular}{c}$(1,0, \ldots ,0;2)$\\ \\$(0, \ldots ,0,1;2)$ \end{tabular} & 
\begin{tabular}{c}$\frac{Sp(\frac{n}{2},\mathbb{C})}{Sp(\frac{n}{2})}$\\ \\$\frac{SO(n,\mathbb{C})}{SO(n)}$ \end{tabular} & 
\begin{tabular}{c}$\frac{SU^*(n)}{Sp(\frac{n}{2})}$\\ \\$\frac{SL(n,\mathbb{R})}{SO(n)}$ \end{tabular} & 
\begin{tabular}{c} yes \\ \\ no \end{tabular} 
\\ 
\addlinespace[5pt] 
\hline 
\addlinespace[5pt]
\begin{tabular}{c}$\frak{a}_{n-1}$\\($n>1, n \in 1 + 2\mathbb{Z}$) \end{tabular} & 
\begin{tabular}{c}$(s_0,0, \ldots , 0,s_p,0,\ldots , 0;1)$\\$(1 \le p \le \frac{n-1}{2})$ with $s_0 = 1 = s_p$\\ \\ 
$(0,\ldots,0,1;2)$ \end{tabular} & 
\begin{tabular}{c}$\frac{S(GL(p,\mathbb{C}) \times GL(n-p,\mathbb{C}))}{S(U(p) \times U(n-p))}$\\ \\ 
$\frac{SO(n,\mathbb{C})}{SO(n)}$ \end{tabular} & 
\begin{tabular}{c}$\frac{SU(p,n-p)}{S(U(p) \times U(n-p))}$\\ \\$\frac{SL(n,\mathbb{R})}{SO(n)}$ \end{tabular} & 
\begin{tabular}{c} yes \\ \\ yes \end{tabular} 
\\  
\addlinespace[5pt] 
\bottomrule
\end{tabular} 
\end{table} 
\end{landscape}
 
\begin{landscape} 
\begin{table} 
\begin{tabular}{|c|c|c|c|c|}
\toprule
\addlinespace[5pt]
$\frak{g}$ & 
 type of $\sigma$ & $X(\bar{\sigma})$ & $X(\bar{\sigma}\bar{\theta})$ & 
\begin{tabular}{c} is {\it condition}\\
{\it Or} for $\bar{G},\sigma,$\\ 
$\sigma\theta$ satisfied? \end{tabular} 
\\
\midrule 
\addlinespace[5pt] 
\begin{tabular}{c}$\frak{c}_n$\\$(n\ge 2, n\in 4\mathbb{Z}$,\\or $n \in 3+4\mathbb{Z})$ \end{tabular} & 
\begin{tabular}{c}$(1,0,\ldots,0,1;1)$\\ \\$(0,\ldots,0,s_p,0,\ldots,0;1)$\\$(1\le p \le n-1)$ with $s_p=1$ \end{tabular} & 
\begin{tabular}{c}$\frac{GL(n,\mathbb{C})}{U(n)}$\\ \\ 
$\frac{Sp(p,\mathbb{C})\times Sp(n-p,\mathbb{C})}{Sp(p)\times Sp(n-p)}$ \end{tabular} & 
\begin{tabular}{c}$\frac{Sp(n,\mathbb{R})}{U(n)}$\\ \\
$\frac{Sp(p,n-p)}{Sp(p)\times Sp(n-p)}$ \end{tabular} & 
\begin{tabular}{c} yes \\ \\ yes \end{tabular} 
\\ 
\addlinespace[5pt] 
\hline 
\addlinespace[5pt] 
\begin{tabular}{c}$\frak{c}_n$\\$(n\ge 2, n\in 1+4\mathbb{Z}$,\\or $n \in 2+4\mathbb{Z})$ \end{tabular} & 
\begin{tabular}{c}$(1,0,\ldots,0,1;1)$\\ \\$(0,\ldots,0,s_p,0,\ldots,0;1)$\\$(1\le p \le n-1, p\neq \frac{n}{2}
)$ with \\ $s_p=1$ \end{tabular} & 
\begin{tabular}{c}$\frac{GL(n,\mathbb{C})}{U(n)}$\\ \\ 
$\frac{Sp(p,\mathbb{C})\times Sp(n-p,\mathbb{C})}{Sp(p)\times Sp(n-p)}$ \end{tabular} & 
\begin{tabular}{c}$\frac{Sp(n,\mathbb{R})}{U(n)}$\\ \\
$\frac{Sp(p,n-p)}{Sp(p)\times Sp(n-p)}$ \end{tabular} & 
\begin{tabular}{c} no \\ \\ yes \end{tabular} 
\\ 
\addlinespace[5pt] 
\hline 
\addlinespace[5pt] 
\begin{tabular}{c}$\frak{c}_n$\\$(n\ge 2, n \in 2+4\mathbb{Z})$ \end{tabular} & 
\begin{tabular}{c}$(0,\ldots,0,s_{\frac{n}{2}},0,\ldots,0;1)$\\ with $s_{\frac{n}{2}}=1$ \end{tabular} & 
$\frac{Sp(\frac{n}{2},\mathbb{C})\times Sp(\frac{n}{2},\mathbb{C})}{Sp(\frac{n}{2})\times 
Sp(\frac{n}{2})}$ & 
$\frac{Sp(\frac{n}{2},\frac{n}{2})}{Sp(\frac{n}{2})\times Sp(\frac{n}{2})}$ & 
no 
\\ 
\addlinespace[5pt] 
\hline 
\addlinespace[5pt] 
\begin{tabular}{c}$\frak{b}_n$\\$(n \ge 3)$ \end{tabular} & 
\begin{tabular}{c}$(1,1,0,\ldots,0;1)$\\ \\$(0,\ldots,0,s_p,0,\ldots,0;1)$\\$(2\le p \le n)$ with $s_p=1$ \end{tabular} & 
\begin{tabular}{c}$\frac{SO(2n-1,\mathbb{C})\times SO(2,\mathbb{C})}{SO(2n-1)\times SO(2)}$\\ \\ 
$\frac{SO(2p,\mathbb{C})\times SO(2n-2p+1,\mathbb{C})}{SO(2p)\times SO(2n-2p+1)}$ \end{tabular} & 
\begin{tabular}{c}$\frac{SO_0(2n-1,2)}{SO(2n-1)\times SO(2)}$\\ \\
$\frac{SO_0(2p,2n-2p+1)}{SO(2p)\times SO(2n-2p+1)}$ \end{tabular} & 
\begin{tabular}{c} no \\ \\ no \end{tabular} 
\\ 
\addlinespace[5pt] 
\hline 
\addlinespace[5pt] 
\begin{tabular}{c}$\frak{\delta}_n$\\$(n\ge 4, n \in 4\mathbb{Z}$,\\or $n \in 1 + 4\mathbb{Z})$ \end{tabular} & 
\begin{tabular}{c}$(1,0,\ldots,0,1;1)$\\ \\$(1,1,0,\ldots,0;1)$\\ \\
$(0,\ldots,0,s_p,0,\ldots,0;1)$\\$(2\le p \le n-2, 2p \le n)$\\ with $s_p=1$\\ \\$(0,\ldots,0,s_p,0,\ldots,0;2)$\\$(0\le p \le n-1)$ with $s_p=1$ 
\end{tabular} & 
\begin{tabular}{c}$\frac{GL(n,\mathbb{C})}{U(n)}$\\ \\
$\frac{SO(2n-2,\mathbb{C})\times SO(2,\mathbb{C})}{SO(2n-2)\times SO(2)}$\\ \\ 
$\frac{SO(2p,\mathbb{C})\times SO(2n-2p,\mathbb{C})}{SO(2p)\times SO(2n-2p)}$\\ \\
$\frac{SO(2p+1,\mathbb{C})\times SO(2n-2p-1,\mathbb{C})}{SO(2p+1)\times SO(2n-2p-1)}$ \end{tabular} & 
\begin{tabular}{c}$\frac{SO^*(2n)}{U(n)}$\\ \\
$\frac{SO_0(2n-2,2)}{SO(2n-2)\times SO(2)}$\\ \\
$\frac{SO_0(2p,2n-2p)}{SO(2p)\times SO(2n-2p)}$\\ \\$\frac{SO_0(2p+1,2n-2p-1)}{SO(2p+1)\times SO(2n-2p-1)}$ \end{tabular} & 
\begin{tabular}{c} yes \\ \\ yes \\ \\ yes \\ \\ yes \end{tabular} 
\\ 
\addlinespace[5pt] 
\bottomrule
\end{tabular} 
\end{table} 
\end{landscape} 

\begin{landscape} 
\begin{table} 
\begin{tabular}{|c|c|c|c|c|}
\toprule
\addlinespace[5pt] 
$\frak{g}$ & 
 type of $\sigma$ & $X(\bar{\sigma})$ & $X(\bar{\sigma}\bar{\theta})$ & 
\begin{tabular}{c} is {\it condition}\\
{\it Or} for $\bar{G},\sigma,$\\ 
$\sigma\theta$ satisfied? \end{tabular} 
\\ 
\midrule
\addlinespace[5pt] 
\begin{tabular}{c}$\frak{\delta}_n$\\$(n\ge 4, n \in 2 + 4\mathbb{Z})$ \end{tabular} & 
\begin{tabular}{c}$(1,0,\ldots,0,1;1)$\\ \\$(1,1,0,\ldots,0;1)$\\ \\
$(0,\ldots,0,s_p,0,\ldots,0;1)$\\$(2\le p \le n-2, 2p < n)$\\ with $s_p=1$\\ \\$(0,\ldots,0,s_{\frac{n}{2}},0,\ldots,0;1)$\\ with 
$s_{\frac{n}{2}}=1$ \\ \\$(0,\ldots,0,s_p,0,\ldots,0;2)$\\$(0\le p \le n-1)$ with $s_p=1$ 
\end{tabular} & 
\begin{tabular}{c}$\frac{GL(n,\mathbb{C})}{U(n)}$\\ \\
$\frac{SO(2n-2,\mathbb{C})\times SO(2,\mathbb{C})}{SO(2n-2)\times SO(2)}$\\ \\ 
$\frac{SO(2p,\mathbb{C})\times SO(2n-2p,\mathbb{C})}{SO(2p)\times SO(2n-2p)}$\\ \\ 
$\frac{SO(n,\mathbb{C})\times SO(n,\mathbb{C})}{SO(n)\times SO(n)}$\\ \\ 
$\frac{SO(2p+1,\mathbb{C})\times SO(2n-2p-1,\mathbb{C})}{SO(2p+1)\times SO(2n-2p-1)}$ \end{tabular} & 
\begin{tabular}{c}$\frac{SO^*(2n)}{U(n)}$\\ \\
$\frac{SO_0(2n-2,2)}{SO(2n-2)\times SO(2)}$\\ \\
$\frac{SO_0(2p,2n-2p)}{SO(2p)\times SO(2n-2p)}$\\ \\ 
$\frac{SO_0(n,n)}{SO(n)\times SO(n)}$\\ \\$\frac{SO_0(2p+1,2n-2p-1)}{SO(2p+1)\times SO(2n-2p-1)}$ \end{tabular} & 
\begin{tabular}{c} no \\ \\ yes \\ \\ yes \\ \\ no \\ \\ yes \end{tabular} 
\\ 
\addlinespace[5pt] 
\hline 
\addlinespace[5pt] 
 \begin{tabular}{c}$\frak{\delta}_n$\\$(n\ge 4, n \in 3 + 4\mathbb{Z})$ \end{tabular} & 
\begin{tabular}{c}$(1,0,\ldots,0,1;1)$\\ \\$(1,1,0,\ldots,0;1)$\\ \\
$(0,\ldots,0,s_p,0,\ldots,0;1)$\\$(2\le p \le n-2, 2p \le n)$\\ with $s_p=1$\\ \\$(0,\ldots,0,s_p,0,\ldots,0;2)$\\$(0\le p \le n-1, p \neq 
\frac{n-1}{2})$\\ with $s_p=1$\\ \\$(0,\ldots,0,s_{\frac{n-1}{2}},0,\ldots,0;2)$\\ with $s_p=1$ \end{tabular} & 
\begin{tabular}{c}$\frac{GL(n,\mathbb{C})}{U(n)}$\\ \\
$\frac{SO(2n-2,\mathbb{C})\times SO(2,\mathbb{C})}{SO(2n-2)\times SO(2)}$\\ \\ 
$\frac{SO(2p,\mathbb{C})\times SO(2n-2p,\mathbb{C})}{SO(2p)\times SO(2n-2p)}$\\ \\
$\frac{SO(2p+1,\mathbb{C})\times SO(2n-2p-1,\mathbb{C})}{SO(2p+1)\times SO(2n-2p-1)}$\\ \\  
$\frac{SO(n,\mathbb{C})\times SO(n,\mathbb{C})}{SO(n)\times SO(n)}$\end{tabular} & 
\begin{tabular}{c}$\frac{SO^*(2n)}{U(n)}$\\ \\
$\frac{SO_0(2n-2,2)}{SO(2n-2)\times SO(2)}$\\ \\
$\frac{SO_0(2p,2n-2p)}{SO(2p)\times SO(2n-2p)}$\\ \\$\frac{SO_0(2p+1,2n-2p-1)}{SO(2p+1)\times SO(2n-2p-1)}$\\ \\ 
$\frac{SO_0(n,n)}{SO(n)\times SO(n)}$ \end{tabular} & 
\begin{tabular}{c} yes \\ \\ yes \\ \\ yes \\ \\ yes \\ \\ no \end{tabular} 
\\ 
\addlinespace[5pt] 
\bottomrule
\end{tabular} 
\end{table} 
\end{landscape} 

 \begin{landscape} 
\begin{table} 
\begin{tabular}{|c|c|c|c|c|}
\toprule
\addlinespace[5pt] 
$\frak{g}$ & 
 type of $\sigma$ & $X(\bar{\sigma})$ & $X(\bar{\sigma}\bar{\theta})$ & 
\begin{tabular}{c} is {\it condition}\\
{\it Or} for $\bar{G},\sigma,$\\ 
$\sigma\theta$ satisfied? \end{tabular} 
\\
\midrule
\addlinespace[5pt] 
 $\frak{e}_6$ & \begin{tabular}{c}$(1,1,0,0,0,0,0;1)$\\ \\$(0,0,1,0,0,0,0;1)$\\ \\$(1,0,0,0,0;2)$\\ \\$(0,1,0,0,0;2)$ \end{tabular} & 
\begin{tabular}{c}$\frac{SO(10,\mathbb{C})\times SO(2,\mathbb{C})}{SO(10)\times SO(2)}$\\ \\ 
$\frac{SL(2,\mathbb{C})\times SL(6,\mathbb{C})}{SU(2)\times SU(6)}$\\ \\ 
$\frac{F_4^\mathbb{C}}{F_4}$\\ \\ 
$\frac{Sp(4,\mathbb{C})}{Sp(4)}$ \end{tabular} & 
\begin{tabular}{c}$(\frak{e}_{6(-14)},\frak{so}(10) + \mathbb{R})$\\ \\ 
$(\frak{e}_{6(2)}, \frak{su}(2) + \frak{su}(6))$\\ \\ 
$(\frak{e}_{6(-26)}, \frak{f}_4)$\\ \\ 
$(\frak{e}_{6(6)}, \frak{sp}(4))$ \end{tabular} & 
\begin{tabular}{c} yes \\ \\ yes \\ \\ yes \\ \\ yes \end{tabular} 
\\ 
\addlinespace[5pt] 
\hline 
\addlinespace[5pt] 
$\frak{e}_7$ & \begin{tabular}{c}$(1,0,0,0,0,0,0,1;1)$\\ \\$(0,0,1,0,0,0,0,0;1)$\\ \\$(0,1,0,0,0,0,0,0;1)$ \end{tabular} & 
\begin{tabular}{c}$\frac{E_6^\mathbb{C}\times SO(2,\mathbb{C})}{E_6\times SO(2)}$\\ \\ 
$\frac{SL(8,\mathbb{C})}{SU(8)}$\\ \\ 
$\frac{SL(2,\mathbb{C})\times SO(12,\mathbb{C})}{SU(2)\times SO(12)}$ \end{tabular} & 
\begin{tabular}{c}$(\frak{e}_{7(-25)},\frak{e}_6 + \mathbb{R})$\\ \\ 
$(\frak{e}_{7(7)}, \frak{su}(8))$\\ \\ 
$(\frak{e}_{7(-5)}, \frak{su}(2) + \frak{so}(12))$ \end{tabular} & 
\begin{tabular}{c} no \\ \\ no \\ \\ yes \end{tabular} 
\\
\addlinespace[5pt] 
\hline 
\addlinespace[5pt]
$\frak{e}_8$ & \begin{tabular}{c}$(0,1,0,0,0,0,0,0,0;1)$\\ \\$(0,0,0,0,0,0,0,0,1;1)$ \end{tabular} & 
\begin{tabular}{c} $\frac{SO(16,\mathbb{C})}{SO(16)}$\\ \\ 
$\frac{SL(2,\mathbb{C})\times E_7^\mathbb{C}}{SU(2)\times E_7}$ \end{tabular} & 
\begin{tabular}{c}$(\frak{e}_{8(8)},\frak{so}(16))$\\ \\ 
$(\frak{e}_{8(-24)}, \frak{su}(2) + \frak{e}_7)$ \end{tabular} & 
\begin{tabular}{c} yes \\ \\ yes \end{tabular} 
\\ 
\addlinespace[5pt] 
\hline 
\addlinespace[5pt]
$\frak{f}_4$ & \begin{tabular}{c}$(0,1,0,0,0;1)$\\ \\$(0,0,0,0,1;1)$ \end{tabular} & 
\begin{tabular}{c}$\frac{SL(2,\mathbb{C})\times Sp(3,\mathbb{C})}{SU(2)\times Sp(3)}$\\ \\ 
$\frac{SO(9,\mathbb{C})}{SO(9)}$\end{tabular} & 
\begin{tabular}{c}$(\frak{f}_{4(4)}, \frak{su}(2) + \frak{sp}(3))$\\ \\ 
$(\frak{f}_{4(-20)},\frak{so}(9))$ \end{tabular} & 
\begin{tabular}{c} yes \\ \\ yes \end{tabular} 
\\ 
\addlinespace[5pt]
\hline 
\addlinespace[5pt]
$\frak{g}_2$ & $(0,0,1;1)$ & $\frac{SL(2,\mathbb{C}) \times SL(2, \mathbb{C})}{SU(2) \times SU(2)}$ & 
$(\frak{g}_{2(2)}, \frak{su}(2) + \frak{su}(2))$ & yes 
\\
\addlinespace[5pt] 
\bottomrule
\end{tabular} 
\end{table}
\end{landscape}

\begin{landscape}
\begin{table} 
\caption{Table for the dimensions of $X(\bar{\sigma})$ and $X(\bar{\sigma}\bar{\theta})$ when the 
{\it condition Or} has been satisfied for $\bar{G},\ \bar{\sigma},\ \bar{\sigma}\bar{\theta}$}\label{resulttable} 
\begin{tabular}{|c|c|c|c|c|}
\toprule
\addlinespace[5pt]
$\frak{g}$ & $X(\bar{\sigma})$ & $X(\bar{\sigma}\bar{\theta})$ & dim$(X(\bar{\sigma}))$ & dim$(X(\bar{\sigma}\bar{\theta}))$ 
\\ 
\midrule
\addlinespace[5pt]
\begin{tabular}{c}$\frak{a}_{n-1}$ \\ $(n>1, n\in 4\mathbb{Z})$ \end{tabular} & 
\begin{tabular}{c}$\frac{S(GL(p,\mathbb{C}) \times GL(n-p,\mathbb{C}))}{S(U(p) \times U(n-p))}$\\$(1\le p \le \frac{n}{2})$\\  \\ 
$\frac{Sp(\frac{n}{2},\mathbb{C})}{Sp(\frac{n}{2})}$ \end{tabular} & 
\begin{tabular}{c}$\frac{SU(p,n-p)}{S(U(p) \times U(n-p))}$\\$(1\le p \le \frac{n}{2})$\\ \\   
$\frac{SU^*(n)}{Sp(\frac{n}{2})}$ \end{tabular} & 
\begin{tabular}{c}$p^2 + (n-p)^2 -1$\\ \\$\frac{n(n+1)}{2}$ \end{tabular} & 
\begin{tabular}{c}$2p(n-p)$\\ \\$\frac{(n-2)(n+1)}{2}$ \end{tabular} 
\\ 
\addlinespace[5pt] 
\hline 
\addlinespace[5pt]
\begin{tabular}{c}$\frak{a}_{n-1}$\\($n>1, n \in 2 + 4\mathbb{Z}$) \end{tabular} & 
\begin{tabular}{c}$\frac{S(GL(p,\mathbb{C}) \times GL(n-p,\mathbb{C}))}{S(U(p) \times U(n-p))}$\\$(1\le p < \frac{n}{2})$\\ \\ 
$\frac{Sp(\frac{n}{2},\mathbb{C})}{Sp(\frac{n}{2})}$\\$(n > 2)$ \end{tabular} & 
\begin{tabular}{c}$\frac{SU(p,n-p)}{S(U(p) \times U(n-p))}$\\$(1\le p < \frac{n}{2})$\\ \\
$\frac{SU^*(n)}{Sp(\frac{n}{2})}$\\$(n > 2)$ \end{tabular} & 
\begin{tabular}{c}$p^2 + (n-p)^2-1$\\ \\$\frac{n(n+1)}{2}$ \end{tabular} & 
\begin{tabular}{c}$2p(n-p)$\\ \\$\frac{(n-2)(n+1)}{2}$ \end{tabular} 
\\ 
\addlinespace[5pt]
\hline 
\addlinespace[5pt]
\begin{tabular}{c}$\frak{a}_{n-1}$\\($n>1, n \in 1 + 2\mathbb{Z}$) \end{tabular} & 
\begin{tabular}{c}$\frac{S(GL(p,\mathbb{C}) \times GL(n-p,\mathbb{C}))}{S(U(p) \times U(n-p))}$\\$(1\le p \le \frac{n-1}{2})$\\ \\ 
$\frac{SO(n,\mathbb{C})}{SO(n)}$ \end{tabular} & 
\begin{tabular}{c}$\frac{SU(p,n-p)}{S(U(p) \times U(n-p))}$\\$(1\le p \le \frac{n-1}{2})$\\ \\
$\frac{SL(n,\mathbb{R})}{SO(n)}$ \end{tabular} & 
\begin{tabular}{c}$p^2 + (n-p)^2-1$\\ \\$\frac{n(n-1)}{2}$ \end{tabular} & 
\begin{tabular}{c}$2p(n-p)$\\ \\$\frac{(n-1)(n+2)}{2}$ \end{tabular} 
\\  
\addlinespace[5pt]
\hline 
\addlinespace[5pt] 
\begin{tabular}{c}$\frak{\delta}_n$\\$(n\ge 4)$ \end{tabular} & 
\begin{tabular}{c}$\frac{GL(n,\mathbb{C})}{U(n)}$\\$(n \not \in 2 + 4\mathbb{Z})$\\ \\
$\frac{SO(p,\mathbb{C})\times SO(2n-p,\mathbb{C})}{SO(p)\times SO(2n-p)}$\\$(1 \le p < n)$\\ \\ 
$\frac{SO(n,\mathbb{C})\times SO(n,\mathbb{C})}{SO(n)\times SO(n)}$\\$(n \in 4\mathbb{Z}, \textrm{or } 1 + 4\mathbb{Z})$ \end{tabular} & 
\begin{tabular}{c}$\frac{SO^*(2n)}{U(n)}$\\$(n \not \in 2 + 4\mathbb{Z})$\\ \\
$\frac{SO_0(p,2n-p)}{SO(p)\times SO(2n-p)}$\\$(1 \le p < n)$\\ \\ 
$\frac{SO_0(n,n)}{SO(n)\times SO(n)}$\\$(n \in 4\mathbb{Z}, \textrm{or } 1 + 4\mathbb{Z})$ \end{tabular} & 
\begin{tabular}{c}$n^2$\\ \\$\frac{p(p-1)+(2n-p)(2n-p-1)}{2}$\\ \\$n(n-1)$ \end{tabular} & 
\begin{tabular}{c}$n(n-1)$\\ \\$p(2n-p)$ \\ \\$n^2$ \end{tabular} 
\\ 
\addlinespace[5pt] 
\bottomrule
\end{tabular} 
\end{table}
\end{landscape} 

\begin{landscape} 
\begin{table} 
\begin{tabular}{|c|c|c|c|c|}
\toprule
\addlinespace[5pt] 
$\frak{g}$ & $X(\bar{\sigma})$ & $X(\bar{\sigma}\bar{\theta})$ & 
dim$(X(\bar{\sigma}))$ & dim$(X(\bar{\sigma}\bar{\theta}))$ 
\\
\midrule 
\addlinespace[5pt] 
\begin{tabular}{c}$\frak{c}_n$\\$(n\ge 2, n\in 4\mathbb{Z}$,\\or $n \in 3+4\mathbb{Z})$ \end{tabular} & 
\begin{tabular}{c}$\frac{GL(n,\mathbb{C})}{U(n)}$\\ \\ 
$\frac{Sp(p,\mathbb{C})\times Sp(n-p,\mathbb{C})}{Sp(p)\times Sp(n-p)}$\\$(1\le p \le n-1)$ \end{tabular} & 
\begin{tabular}{c}$\frac{Sp(n,\mathbb{R})}{U(n)}$\\ \\
$\frac{Sp(p,n-p)}{Sp(p)\times Sp(n-p)}$\\$(1\le p \le n-1)$  \end{tabular} & 
\begin{tabular}{c}$n^2$\\ \\$p(2p+1)+$\\$(n-p)(2n-2p+1)$ \end{tabular} & 
\begin{tabular}{c}$n(n+1)$\\ \\$4p(n-p)$ \end{tabular} 
\\ 
\addlinespace[5pt] 
\hline 
\addlinespace[5pt] 
\begin{tabular}{c}$\frak{c}_n$\\$(n\ge 2, n\in 1+4\mathbb{Z}$,\\or $n \in 2+4\mathbb{Z})$ \end{tabular} & 
\begin{tabular}{c}$\frac{Sp(p,\mathbb{C})\times Sp(n-p,\mathbb{C})}{Sp(p)\times Sp(n-p)}$\\ 
$(1\le p \le n-1, p\neq \frac{n}{2})$ \end{tabular} & 
\begin{tabular}{c}$\frac{Sp(p,n-p)}{Sp(p)\times Sp(n-p)}$\\ 
$(1\le p \le n-1, p\neq \frac{n}{2})$ \end{tabular} & 
\begin{tabular}{c}$p(2p+1)+$\\$(n-p)(2n-2p+1)$ \end{tabular} & $4p(n-p)$ 
\\ 
\addlinespace[5pt] 
\hline 
\addlinespace[5pt] 
$\frak{e}_6$ & 
\begin{tabular}{c}$\frac{SO(10,\mathbb{C})\times SO(2,\mathbb{C})}{SO(10)\times SO(2)}$\\ \\ 
$\frac{SL(2,\mathbb{C})\times SL(6,\mathbb{C})}{SU(2)\times SU(6)}$\\ \\ 
$\frac{F_4^\mathbb{C}}{F_4}$\\ \\ 
$\frac{Sp(4,\mathbb{C})}{Sp(4)}$ \end{tabular} & 
\begin{tabular}{c}$(\frak{e}_{6(-14)},\frak{so}(10) + \mathbb{R})$\\ \\ 
$(\frak{e}_{6(2)}, \frak{su}(2) + \frak{su}(6))$\\ \\ 
$(\frak{e}_{6(-26)}, \frak{f}_4)$\\ \\ 
$(\frak{e}_{6(6)}, \frak{sp}(4))$ \end{tabular} & 
\begin{tabular}{c}$46$\\ \\$38$\\ \\$52$\\ \\$36$ \end{tabular} & 
\begin{tabular}{c}$32$\\ \\$40$\\ \\$26$\\ \\$42$ \end{tabular}
\\ 
\addlinespace[5pt] 
\hline 
\addlinespace[5pt] 
$\frak{e}_7$ & 
$\frac{SL(2,\mathbb{C})\times SO(12,\mathbb{C})}{SU(2)\times SO(12)}$ & 
$(\frak{e}_{7(-5)}, \frak{su}(2) + \frak{so}(12))$ & 
$69$ & $64$ 
\\
\addlinespace[5pt] 
\hline 
\addlinespace[5pt]
$\frak{e}_8$ & 
\begin{tabular}{c} $\frac{SO(16,\mathbb{C})}{SO(16)}$\\ \\ 
$\frac{SL(2,\mathbb{C})\times E_7^\mathbb{C}}{SU(2)\times E_7}$ \end{tabular} & 
\begin{tabular}{c}$(\frak{e}_{8(8)},\frak{so}(16))$\\ \\ 
$(\frak{e}_{8(-24)}, \frak{su}(2) + \frak{e}_7)$ \end{tabular} & 
\begin{tabular}{c}$120$\\ \\$136$ \end{tabular} & 
\begin{tabular}{c}$128$\\ \\$112$ \end{tabular}
\\ 
\addlinespace[5pt] 
\hline 
\addlinespace[5pt]
$\frak{f}_4$ & 
\begin{tabular}{c}$\frac{SL(2,\mathbb{C})\times Sp(3,\mathbb{C})}{SU(2)\times Sp(3)}$\\ \\ 
$\frac{SO(9,\mathbb{C})}{SO(9)}$\end{tabular} & 
\begin{tabular}{c}$(\frak{f}_{4(4)}, \frak{su}(2) + \frak{sp}(3))$\\ \\ 
$(\frak{f}_{4(-20)},\frak{so}(9))$ \end{tabular} & 
\begin{tabular}{c}$24$\\ \\$36$ \end{tabular} & 
\begin{tabular}{c}$28$\\ \\$16$ \end{tabular} 
\\ 
\addlinespace[5pt]
\hline 
\addlinespace[5pt]
$\frak{g}_2$ &  
$\frac{SL(2,\mathbb{C}) \times SL(2, \mathbb{C})}{SU(2) \times SU(2)}$ & 
$(\frak{g}_{2(2)}, \frak{su}(2) + \frak{su}(2))$ & $6$ & $8$ 
\\
\addlinespace[5pt] 
\bottomrule
\end{tabular} 
\end{table} 
\end{landscape}

\noindent 
\subsection{Proof of Theorem \ref{th1}}\label{pfth1}  

  Note that $X = G/U$ is a Riemannian globally symmetric space of type IV.  Let $\bar{G} = \textrm{Ad}(G)$ be the 
adjoint group of $G$, and $\frak{g}$ be the Lie algebra 
of $G$. Let $\frak{u}$ be a compact real form of $\frak{g}$, and $\theta$ be the Cartan involution of $\frak{g}^\mathbb{R}$ 
corresponding to the Cartan decomposition $\frak{g}^\mathbb{R} = \frak{u} \oplus i\frak{u}$. Let $\bar{\theta}$ denote the 
corresponding Cartan involution of $\bar{G}$. Let $\bar{U} = \{ g \in \bar{G} : \bar{\theta}(g) = g \}$. Then $X = \bar{G}/\bar{U}$. 
Let $\frak{t}$ be a maximal abelian subspace of $\frak{u}$, 
and $\frak{h} = \frak{t}^\mathbb{C}$. Then $\frak{h}$ is a Cartan subalgebra of $\frak{g}$. Choose a system of positive roots 
$\Delta^+ $ in the set of all non-zero roots $\Delta = \Delta(\frak{g}, \frak{h})$. Let $\Phi$ be the set of simple roots in $\Delta^+$. Let 
$\{ H_\phi ^* \  , E_\alpha : \phi \in \Phi , \alpha \in \Delta \}$ be a Chevalley basis for $\frak{g}$ as in  \eqref{chevalley}. 
Then 
\[ \frak{u} =  \sum_{\phi \in \Phi} \mathbb{R} (i  H_\phi ^*) \oplus \sum_{\alpha \in \Delta^+} \mathbb{R} X_\alpha \oplus 
\sum_{\alpha \in \Delta^+} \mathbb{R} Y_\alpha, \]
where $X_\alpha = E_\alpha - E_{-\alpha} , Y_\alpha = i(E_\alpha + E_{-\alpha})$ for all $\alpha \in \Delta^+$. 

  Let $\sigma$ of an involution of $\frak{g}$ as in \eqref{sigma} and $\bar{\sigma} : \bar{G} \longrightarrow \bar{G}$ be the involution 
with $d\bar{\sigma} = \sigma$. Then $\sigma \theta = \theta \sigma$. 
Let $\bar{\nu}$ be the Dynkin diagram automorphism induced by $\sigma$ and $\nu$ be the linear 
extension of $\bar{\nu}$ on the dual space of $i\frak{t}$. Recall that \\ 
$\sigma (iH_\phi ^*) = iH_{\nu(\phi)}^*$ for all $\phi \in \Phi$, and \\ 
$\sigma (X_\alpha) = q_\alpha X_{\nu(\alpha)},\ \sigma (Y_\alpha) = q_\alpha Y_{\nu(\alpha)}$ ($q_\alpha = \pm 1$); 
for all $\alpha \in \Delta^+$. Note that $q_\alpha= q_{\nu(\alpha)}$ for all $\alpha \in \Delta^+$. 

  Let $\frak{u} = \frak{u}_0 \oplus \frak{u}_1$ be the decomposition of $\frak{u}$ in to $1$ and $-1$ eigenspaces of $\sigma$. Then 
$\frak{g}^\sigma = \frak{u}_0 \oplus i\frak{u}_1$ is a non-compact real form of $\frak{g}$, and $\sigma \Big{|}_{\frak{g}^\sigma}$ 
is a Cartan involution of $\frak{g}^\sigma$. Note that \\ 
$\frak{u}_0 = \sum\limits_{\phi \in \Phi} \mathbb{R}i(H_\phi ^* + H_{\nu(\phi)}^* ) \oplus \sum\limits_{\substack
{\alpha \in \Delta^+ \\ q_\alpha = 1}} (\mathbb{R}(X_\alpha + X_{\nu(\alpha)}) \oplus \mathbb{R}(Y_\alpha + Y_{\nu(\alpha)})) 
\oplus \sum\limits_{\substack
{\alpha \in \Delta^+ \\ q_\alpha = -1}} (\mathbb{R}(X_\alpha - X_{\nu(\alpha)}) \oplus \mathbb{R}(Y_\alpha - Y_{\nu(\alpha)}))$, and \\ 
$i\frak{u}_1 = \sum\limits_{\phi \in \Phi} \mathbb{R}(H_\phi ^* - H_{\nu(\phi)}^* ) \oplus \sum\limits_{\substack
{\alpha \in \Delta^+ \\ q_\alpha = 1}} (\mathbb{R}i(X_\alpha - X_{\nu(\alpha)}) \oplus \mathbb{R}i(Y_\alpha - Y_{\nu(\alpha)})) 
\oplus \sum\limits_{\substack
{\alpha \in \Delta^+ \\ q_\alpha = -1}} (\mathbb{R}i(X_\alpha + X_{\nu(\alpha)}) \oplus \mathbb{R}i(Y_\alpha + Y_{\nu(\alpha)}))$. 
 
  Let $B' \subset \{i(H_\phi ^* + H_{\nu(\phi)}^* ), \ (H_\phi ^* - H_{\nu(\phi)}^*) : \phi \in \Phi \} \cup 
\{X_\alpha + X_{\nu(\alpha)},\ Y_\alpha + Y_{\nu(\alpha)},\ i(X_\alpha - X_{\nu(\alpha)}),\ i(Y_\alpha - Y_{\nu(\alpha)}) : 
\alpha \in \Delta^+ ,\ q_\alpha = 1\} \cup 
\{X_\alpha - X_{\nu(\alpha)},\ Y_\alpha - Y_{\nu(\alpha)},\ i(X_\alpha + X_{\nu(\alpha)}),\ i(Y_\alpha + Y_{\nu(\alpha)}) : 
\alpha \in \Delta^+ ,\ q_\alpha = -1\}$ be a basis of $\frak{g}^\sigma$. Then $B'$ is a basis of $\frak{g}^\sigma$ consisting of 
eigenvectors of the Cartan involution $\sigma \Big{|}_{\frak{g}^\sigma}$, with respect to which the structural constants are 
all integers. Let $\Gamma$ be an arithmetic uniform lattice of Aut$(\frak{g})$ of type $3$ with respect to the non-compact real form 
$\frak{g}^\sigma$ and the basis $B'$ of $\frak{g}^\sigma$. Then $\sigma \in \Gamma$. Also $\sigma \in \Gamma$ for any 
arithmetic uniform lattice of Aut$(\frak{g}^\mathbb{R})$ of type $i$, $i = 1$, or $2$. 

  Now assume that $\Gamma'$ be an arithmetic uniform lattice of Aut$(\frak{g}^\mathbb{R})$ of type $i (i =1,2,$ or $3)$, and $F$ be 
the corresponding algebraic number field with ring of integers $\mathcal{O}$. Arithmetic uniform lattices of 
Aut$(\frak{g}^\mathbb{R})$ of type $3$ considered here are defined with respect to the non-compact real form 
$\frak{g}^\sigma$ and the basis $B'$ of $\frak{g}^\sigma$. Let $\Gamma$ be 
the set of all torsion-free elements of $\Gamma' \cap \bar{G}$. Then $\bar{G}$ is defined over $F$, $\bar{\theta},\ \bar{\sigma}$ are 
defined over $F$, and $\Gamma \subset \bar{G}_{\mathcal{O}}$ is a torsion-free, $\langle \bar{\sigma} , \bar{\theta} \rangle$-stable, 
arithmetic uniform lattice of $\bar{G}$. Then if the {\it condition Or} is satisfied for 
$\bar{G}, \bar{\sigma}, \bar{\sigma}\bar{\theta}$; there exists a $\langle \bar{\sigma} , \bar{\theta} \rangle$-stable subgroup 
$\Gamma''$ of $\Gamma$ of finite index such that the cohomology 
classes defined by $[C(\bar{\sigma} , \Gamma'')], [C(\bar{\sigma}\bar{\theta}, \Gamma'')]$ via Poincar\'e duality are non-zero and 
are not represented by $\bar{G}$-invariant differential forms on $X$, by Theorem \ref{mira}. Since $G$ is a covering group 
of $\bar{G}$, the cohomology classes defined by $[C(\bar{\sigma} , \Gamma'')], [C(\bar{\sigma}\bar{\theta}, \Gamma'')]$ via 
Poincar\'e duality are also not represented by $G$-invariant differential forms on $X$. This completes the proof.

\noindent
\section{Automorphic representations of a connected complex simple Lie group} 

Let $G$ be a non-compact semisimple Lie group with finite centre and $\Gamma \subset G$ be a lattice. 
Consider the Hilbert space $L^2 (\Gamma \backslash G)$ of square integrable functions on $\Gamma \backslash G$ with respect to 
a finite $G$-invariant measure. The group $G$ acts unitarily on the Hilbert space $L^2 (\Gamma \backslash G)$ via the right translation 
action of $G$ on $\Gamma \backslash G$. 

  When $\Gamma$ is a uniform lattice, we have 
\[L^2 (\Gamma \backslash G) \cong \widehat{\bigoplus}_{\pi \in \hat{G}} m(\pi , \Gamma ) H_\pi,\]
due to Gelfand and Pyatetskii-Shapiro \cite{ggp}, \cite{gp}; 
where $\hat{G}$ denotes the unitary dual of $G$; $H_\pi$ is the representation space of $\pi \in \hat{G}$; 
and $m(\pi , \Gamma ) \in \mathbb{N} \cup \{0\}$, 
the multiplicity of $\pi$ in $L^2 (\Gamma \backslash G)$. 
If $(\tau, \mathbb{C})$ is the trivial representation of $G$, then $m(\tau , \Gamma) = 1$. 

  A unitary representation $\pi \in \hat{G}$ such that $m(\pi , \Gamma ) > 0$ for some uniform lattice $\Gamma$, is called an
automorphic representation with respect to $\Gamma$. 
The connection between geometric cycles and automorphic representations has been made 
by the Matsushima's isomorphism. 

  Let $G$ be a connected semisimple Lie group with finite centre and $K$ be a maximal compact subgroup of $G$ with 
Cartan involution $\theta$. Let $X = G \slash K$ be the associated Riemannian globally symmetric space, 
$\frak{g}$ be the Lie algebra of $G$ and $\frak{g}^\mathbb{C}$ be the complexification of $\frak{g}$. 
If $\pi$ be an admissible unitary representation of $G$ on a Hilbert space $H_\pi$, 
we denote by $H_{\pi,K}$ the space of all $K$-finite vectors of $H_\pi$. 
The space $H_{\pi, K}$ is the associated $(\frak{g}^\mathbb{C}, K)$-module.  
 
  Let $\Gamma \subset G$ be a torsion-free uniform lattice. 
Then the isomorphism $L^2 (\Gamma \backslash G) \cong \widehat{\bigoplus}_{\pi \in \hat{G}} m(\pi, \Gamma) H_\pi$ implies 
\[\bigoplus_{\pi \in \hat{G}} m_\pi H_{\pi,K} \xhookrightarrow { } C^\infty (\Gamma \backslash G)_K.\] 
Matsushima's formula \cite{matsushima} says that the above inclusion induces an isomorphism 
\[ \bigoplus_{\pi \in \hat{G}} m_\pi H^p (\frak{g}^\mathbb{C}, K; H_{\pi,K}) \cong 
H^p (\frak{g}^\mathbb{C}, K; C^\infty (\Gamma \backslash G)_K). \] 
Also we have the well-known isomorphism 
\[H^p (\frak{g}^\mathbb{C}, K; C^\infty (\Gamma \backslash G)_K) \cong 
H^p (\Gamma \backslash X; \mathbb{C}). \] See \cite[Cor. 2.7, Ch. VII]{borel-wallach}. Hence 
\[H^p (\Gamma \backslash X; \mathbb{C}) \cong  \bigoplus_{\pi \in \hat{G}} m_\pi H^p (\frak{g}^\mathbb{C}, K; H_{\pi,K}). \]
 
  Hence a non-vanishing (in the cohomology level) geometric cycle will contribute to the LHS and it may help to detect occurrence of 
some $\pi \in \hat{G}$ with non-zero $(\frak{g}^\mathbb{C}, K)$-cohomology. If $X_u$ denotes the compact dual of $X$, 
then the image of the Matsushima map $k_\Gamma : H^*(X_u ; \mathbb{C}) \longrightarrow H^* (\Gamma \backslash X ; \mathbb{C})$ 
corresponds to the trivial representation $(\tau, \mathbb{C})$ of $G$. So if the cohomology class of a geometric cycle 
does not lie in the image $k_\Gamma (H^*(X_u ; \mathbb{C}))$, then it may help to detect occurrence of 
some non-trivial $\pi \in \hat{G}$ with non-zero $(\frak{g}^\mathbb{C}, K)$-cohomology. For this purpose, it is important 
to know the irreducible unitary representations of $G$ with non-zero $(\frak{g}^\mathbb{C}, K)$-cohomology. 
The details are given in the following subsections.

\noindent
\subsection{Irreducible unitary representations with non-zero $(\frak{g}^\mathbb{C}, K)$-cohomology}\label{general} 

    Let $G$ be a connected semisimple Lie group with finite centre and $\frak{g}$ be the Lie algebra of $G$. Let $\frak{g} = 
\frak{k} \oplus \frak{p}$ be a Cartan decomposition and $\theta$ be the corresponding Cartan involution. Let $K$ be the connected Lie 
subgroup of $G$ with Lie$(K) = \frak{k}$. Then $K$ is a maximal compact subgroup of $G$. 
Let $\frak{g}^{\mathbb{C}}$ be the complexification of $\frak{g}$ and $\frak{k}^\mathbb{C},\ \frak{p}^\mathbb{C} \subset 
\frak{g}^\mathbb{C}$ be the complexifications of $\frak{k},\ \frak{p}$ respectively. The complex linear extension of $\theta$ to 
$\frak{g}^\mathbb{C}$ is denoted by the same notation $\theta$. If $\pi$ be an admissible unitary representation of $G$ on a 
Hilbert space $H_\pi$, recall that $H_{\pi,K}$ is the space of all $K$-finite vectors of $H_\pi$. By a theorem of D. Wigner, 
if $\pi \in \hat{G}$, then $H^*( \frak{g}^\mathbb{C}, K; H_{\pi , K}) \neq 0$ implies the infinitesimal character $\chi_\pi$ of $\pi$ is trivial that is, 
$\chi_\pi = \chi_0$, the infinitesimal character of the trivial representation of $G$.  
Hence there are only finitely irreducible unitary representations 
with non-zero $(\frak{g}^\mathbb{C}, K)$-cohomology. 
In fact, the irreducible unitary representations with non-zero relative Lie algebra 
cohomology have been classified in terms of the $\theta$-stable parabolic subalgebras 
$\frak{q} \subset \frak{g}^\mathbb{C}$ of $\frak{g}$. 
  
  A $\theta$-stable parabolic subalgebra of $\frak{g}$ is by definition, a parabolic subalgebra $\frak{q}$ of $\frak{g}^\mathbb{C}$ 
such that 
(a) $\theta(\frak{q}) = \frak{q}$, and (b) $\bar{\frak{q}} \cap \frak{q}= \frak{l}^\mathbb{C}$ is a Levi subalgebra of $\frak{q}$, 
where $\bar{\ }$ denotes the conjugation of $\frak{g}^\mathbb{C}$ with respect to $\frak{g}$. By (b), $\frak{l}^\mathbb{C}$ is the 
complexification of a real subalgebra $\frak{l}$ of $\frak{g}$. Also $\theta(\frak{l}) = \frak{l}$ and and $\frak{l}$ contains a maximal 
abelian subalgebra $\frak{t}$ of $\frak{k}$. Then $\frak{h} = \frak{z}_\frak{g} (\frak{t})$ is a $\theta$-stable Cartan 
subalgebra of $\frak{g}$, $\frak{h}^\mathbb{C}$ is a Cartan subalgebra of $\frak{g}^\mathbb{C}$ and $\frak{h}^\mathbb{C} \subset 
\frak{q}$. Let $\frak{u}_\frak{q}$ be the nilradical of $\frak{q}$ so that $\frak{q} = \frak{l}^\mathbb{C} \oplus \frak{u}_\frak{q}$. Then 
$\frak{u}_\frak{q}$ is $\theta$-stable and so $\frak{u}_\frak{q} = (\frak{u}_\frak{q} \cap \frak{k}^\mathbb{C}) \oplus (\frak{u}_\frak{q} 
\cap \frak{p}^\mathbb{C})$. 

   If $V$ is finite dimensional complex $L$-module, where $L$ is an abelian Lie algebra; we denote by $\Delta (V)$ ( or by $\Delta (V , L)$) , 
the set of all non-zero weights of $V$ and by $\delta (V)$ (or by $\delta (V , L)$), $1/2$ of the sum of elements in $\Delta (V)$ counted 
with their respective multiplicities. 

  Fix systems of positive roots $\Delta ^+ ((\frak{l} \cap \frak{k})^\mathbb{C} , \frak{t}^\mathbb{C})$ and 
$\Delta ^+ (\frak{l}^\mathbb{C} , \frak{h}^\mathbb{C})$, compatible with 
$\Delta ^+ ((\frak{l} \cap \frak{k})^\mathbb{C} , \frak{t}^\mathbb{C})$. Then $\Delta ^+ _{\frak{k}} = \Delta ^+ ((\frak{l} \cap \frak{k})^\mathbb{C} , 
\frak{t}^\mathbb{C}) \cup \Delta (\frak{u}_\frak{q} \cap \frak{k}^\mathbb{C} , \frak{t}^\mathbb{C})$ and 
$\Delta ^+  = \Delta ^+ (\frak{l}^\mathbb{C} , \frak{h}^\mathbb{C}) \cup \Delta (\frak{u}_\frak{q} , \frak{h}^\mathbb{C})$ are system 
of positive roots in $\Delta (\frak{k}^\mathbb{C}, \frak{t}^\mathbb{C})$ and 
$\Delta = \Delta (\frak{g}^\mathbb{C} , \frak{h}^\mathbb{C})$ respectively. 

  Now associated with a $\theta$-stable parabolic subalgebra $\frak{q}$, we have an irreducible unitary representation 
$\mathcal{R}^S _\frak{q} (\mathbb{C}) = A_\frak{q}$ of $G$ with trivial infinitesimal character, where $S = \textrm{dim}
(\frak{u}_\frak{q} \cap \frak{k}^\mathbb{C})$. The associated $(\frak{g}^\mathbb{C}, K)$-module $A_{\frak{q}, K}$ contains an 
irreducible $K$-submodule $V$ of highest weight (with respect to $\Delta ^+ _{\frak{k}}$) 
$2 \delta (\frak{u}_\frak{q} \cap \frak{p}^\mathbb{C}, \frak{t}^\mathbb{C}) = 
\sum\limits_{\alpha \in \Delta (\frak{u}_\frak{q} \cap \frak{p}^\mathbb{C} , \frak{t}^\mathbb{C})} \alpha $ and 
it occurs with multiplicity one in $A_{\frak{q}, K}$. Any other irreducible $K$-module that occurs in $A_{\frak{q}, K}$ has highest weight 
of the form $2 \delta (\frak{u}_\frak{q} \cap \frak{p}^\mathbb{C}, \frak{t}^\mathbb{C}) + 
\sum\limits_{\gamma \in \Delta (\frak{u}_\frak{q} \cap \frak{p}^\mathbb{C} , \frak{t}^\mathbb{C})} n_\gamma \gamma $,
with $n_\gamma$ a non-negative integer \cite[Th. 2.5]{voganz}.  

  If $\frak{q}$ is a $\theta$-stable parabolic subalgebra, then so is Ad$(k)(\frak{q})$ ($k \in K$); and
$A_\frak{q}$, $A_{\textrm{Ad}(k)(\frak{q})}$ are unitarily equivalent. So it is sufficient to consider $\theta$-stable parabolic 
 subalgebras of $\frak{g}$ which contain $\frak{t}$, and $\Delta^+_{\frak{k}}$ is contained in the corresponding system of positive roots $\Delta^+$. 
It is known that \cite[Prop. 4.5]{riba}, 
for two such parabolic subalgebras $\frak{q}$ and $\frak{q}'$, $A_\frak{q}$ is unitarily equivalent to $A_{\frak{q}'}$ if and only if 
$\frak{u}_\frak{q} \cap \frak{p}^\mathbb{C} = \frak{u}_{\frak{q}'} \cap \frak{p}^\mathbb{C}$.  Actually we have 
$A_\frak{q}$ is unitarily equivalent to $A_{\frak{q}'}$ if and only if  
$\delta (\frak{u}_\frak{q} \cap \frak{p}^\mathbb{C}, \frak{t}^\mathbb{C}) =  
\delta (\frak{u}_{\frak{q}'} \cap \frak{p}^\mathbb{C}, \frak{t}^\mathbb{C})$. The proof can be deduced from 
\cite[Lemma 4.6 and Lemma 4.8]{riba} just noting the fact that if $\frak{q}, \tilde{\frak{q}}$ are two $\theta$-stable parabolic subalgebras 
with $\frak{q} \subset \tilde{\frak{q}}$, then as they contain the same Borel subalgebra of $\frak{g}^\mathbb{C}$, we have 
$\frak{u}_\frak{q} \cap \frak{p}^\mathbb{C}= \frak{u}_{\tilde{\frak{q}}} \cap \frak{p}^\mathbb{C}$ if and only if 
$\delta (\frak{u}_\frak{q} \cap \frak{p}^\mathbb{C}, \frak{t}^\mathbb{C}) = 
\delta (\frak{u}_{\tilde{\frak{q}}} \cap \frak{p}^\mathbb{C}, \frak{t}^\mathbb{C})$. 

   If $\frak{q}$ is a $\theta$-stable parabolic subalgebra of $\frak{g}$, then the Levi subgroup $L = \{g \in G : \textrm{Ad}(g) (\frak{q}) = 
\frak{q} \}$ is a connected reductive Lie subgroup of $G$ with Lie algebra $\frak{l}$. As $\theta(\frak{l}) = \frak{l}$, $L \cap K$ is a maximal 
compact subgroup of $L$. One has 
\[ H^p (\frak{g}^\mathbb{C}, K; A_{\frak{q}, K}) \cong H^{p-R(\frak{q})} (\frak{l}^\mathbb{C}, L\cap K ; \mathbb{C}), \]
where $R(\frak{q}) := \textrm{dim}(\frak{u}_\frak{q} \cap \frak{p}^\mathbb{C})$. Let $Y_\frak{q}$ denote the compact dual of the 
Riemannian globally symmetric space $L/{L\cap K}$. Then $H^p (\frak{l}^\mathbb{C}, L\cap K ; \mathbb{C}) \cong 
H^p (Y_\frak{q} ; \mathbb{C})$. And hence 
\[  H^p (\frak{g}^\mathbb{C}, K; A_{\frak{q}, K}) \cong H^{p-R(\frak{q})} (Y_\frak{q} ; \mathbb{C}).\]
If $P(\frak{q}, t)$ denotes the Poincar\'{e} polynomial of $ H^* (\frak{g}^\mathbb{C}, K; A_{\frak{q}, K})$. Then by the above result, we 
have 
\[ P(\frak{q}, t) = t^{R(\frak{q})} P(Y_\frak{q} , t). \]

  If rank$(G) =$ rank$(K)$ and $\frak{q}$ is a $\theta$-stable Borel subalgebra that is, $\frak{q}$ is a Borel subalgebra of $\frak{g}^\mathbb{C}$ 
containing a Cartan subalgebra of $\frak{k}^\mathbb{C}$, then $A_\frak{q}$ is a 
discrete series representation of $G$ with trivial infinitesimal character. In this case, $R(\frak{q})= \frac{1}{2} \textrm{ dim}(G/K)$,  
$L$ is a maximal torus in $K$ and hence 
\[H^p (\frak{g}^\mathbb{C}, K; A_{\frak{q}, K}) =
\begin{cases}
0 & \textrm{if } p \neq R(\frak{q}), \\
\mathbb{C} & \textrm{if } p = R(\frak{q}). 
\end{cases}
\]
If we take $\frak{q} = \frak{g}^\mathbb{C}$, then $L=G$ and $A_\frak{q} = \mathbb{C}$, the trivial representation of $G$. 

  Conversely, if $\pi  \in \hat{G}$ with $H^*( \frak{g}^\mathbb{C}, K; H_{\pi , K}) \neq 0$, then $H_\pi $ is unitarily equivalent 
to $A_\frak{q}$ for some $\theta$-stable parabolic subalgebra $\frak{q}$ of $\frak{g}$ \cite[Th. 4.1]{voganz}. 

  The $(\frak{g}^\mathbb{C} , K)$-modules $A_{\frak{q} , K}$ were first constructed, in general, by Parthasarathy \cite{parthasarathy1}. 
Delorme \cite{delorme} and Enright \cite{enright} gave a construction of those for complex Lie groups. 
Vogan and Zuckerman \cite{voganz} gave a construction of 
the $(\frak{g}^\mathbb{C} , K)$-modules $A_{\frak{q} , K}$ via cohomological induction and Vogan \cite{vogan} proved that these are 
unitarizable. See \cite{vogan97} for a beautiful description of Matsushima isomorphism and the theory of 
$(\frak{g}^\mathbb{C} , K)$-modules $A_{\frak{q} , K}$.

\noindent
\subsection{Irreducible unitary representations with non-zero $(\frak{g}^\mathbb{C}, K)$-cohomology of a connected complex 
semisimple Lie group}\label{particular} 

  Now assume that $\frak{g}$ is a complex semisimple Lie algebra and $\theta$ is a Cartan involution on $\frak{g}^\mathbb{R}$. Let 
$\frak{g}^\mathbb{R} = \frak{u} \oplus i \frak{u}$ be the corresponding Cartan decomposition, for some compact real form $\frak{u}$ of 
$\frak{g}$. Let $G$ be a connected Lie group with Lie algebra $\frak{g}$ and $U$ be a Lie subgroup of $G$ corresponding to the 
subalgebra $\frak{u}$ of $\frak{g}^\mathbb{R}$. Then $U$ is a maximal compact subgroup of $G$. Recall that we shall identify $\frak{g}$ with the 
subalgebra $\{(X, X) + i(Y, -Y) : X, Y \in \frak{u} \}$ of $(\frak{g}^\mathbb{R})^\mathbb{C} \cong \frak{g} \times \frak{g}$ and via this 
identification the complex linear extension of $\theta$ (denoted by the same notation) on $\frak{g} \times \frak{g}$ is given by 
$(Z_1, Z_2) \mapsto (Z_2, Z_1)$, where $Z_1 , Z_2 \in \frak{g}$. Then $\frak{k} = \{ (Z, Z) : Z \in \frak{g} \}$ and 
$\frak{p} = \{ (Z, -Z) : Z \in \frak{g} \}$ are the eigenspaces of $\theta$ corresponding to the eigenvalues $1$ and $-1$ respectively. 

  A parabolic subalgebra of $(\frak{g}^\mathbb{R})^\mathbb{C} \cong \frak{g} \times \frak{g}$ is of the form 
$\frak{q}_1 \times \frak{q}_2$, for some parabolic subalgebras $\frak{q}_1 , \frak{q}_2$ of $\frak{g}$. Hence $\theta (\frak{q}_1 \times 
\frak{q}_2 ) = \frak{q}_1 \times \frak{q}_2$ if and only if $\frak{q}_1 = \frak{q}_2$. If $\frak{q} \times \frak{q}$ is $\theta$-stable, then 
$\frak{q}$ contains a $\theta$-stable Cartan subalgebra of $\frak{g}$ (see \S \ref{general}).  
Let $\frak{t}$ be a maximal abelian subalgebra of $\frak{u}$. Then $\frak{h} = \frak{t}^\mathbb{C}$ is $\theta$-stable 
Cartan subalgebra of $\frak{g}$. Let $\Delta = \Delta (\frak{g}, \frak{h})$. Since 
$\frak{h}$ is $\theta$-stable, define $\theta (\alpha) (H) = \overline{\alpha(\theta H)}$ for all $H \in \frak{h}$, where $\alpha \in \frak{h}^*$. 
Note that $\theta (\alpha) = - \alpha$ for all $\alpha \in \Delta$. So if $\frak{q}$ is a parabolic subalgebra of $\frak{g}$ containing the 
Cartan subalgebra $\frak{h}$, then $\frak{q} \cap \theta(\frak{q}) = \frak{l}$, the Levi factor of $\frak{q}$ relative to the 
Cartan subalgebra $\frak{h}$. Let $\bar{\ }$ denote the conjugation of $\frak{g} \times \frak{g}$ with respect to the real form 
$\frak{g} \cong \{(X, X) + i(Y, -Y) : X, Y \in \frak{u} \}$. The map $\bar{\ } : \frak{g} \times \frak{g} \longrightarrow \frak{g} \times 
\frak{g}$ is given by $(Z_1, Z_2) \mapsto (\theta(Z_2), \theta(Z_1))$. 
Hence $\overline{\frak{q} \times \frak{q}} = \theta(\frak{q}) \times \theta(\frak{q})$  
and so $\overline{(\frak{q} \times \frak{q})} \cap (\frak{q} \times \frak{q}) = 
(\theta(\frak{q}) \cap \frak{q}) \times (\theta(\frak{q}) \cap \frak{q}) = \frak{l} \times \frak{l}$, the Levi factor of the parabolic subalgebra 
$\frak{q} \times \frak{q}$ of $\frak{g} \times \frak{g}$ relative to the Cartan subalgebra $\frak{h} \times \frak{h}$. Consequently, we have 
a parabolic subalgebra of $\frak{g} \times \frak{g}$ is $\theta$-stable if and only if it is of the form $\frak{q} \times \frak{q}$, for some 
parabolic subalgebra $\frak{q}$ of $\frak{g}$ containing a $\theta$-stable Cartan subalgebra of $\frak{g}$. 

   Fix a maximal abelian subalgebra $\frak{t}$ of $\frak{u}$ and a system of positive roots $\Delta^+$ in $\Delta = 
\Delta (\frak{g}, \frak{h})$, where $\frak{h} = \frak{t}^\mathbb{C}$, a $\theta$-stable Cartan subalgebra of $\frak{g}$. By \S \ref{general}, 
it is sufficient to consider $\theta$-stable parabolic subalgebras of $\frak{g}$ which contain $\frak{t}$, and $\Delta^+$ is contained 
in the corresponding system of positive roots in $\Delta(\frak{g} \times \frak{g}, \frak{h} \times \frak{h})$. Now the 
$\theta$-stable parabolic subalgebras of $\frak{g} \times \frak{g}$,  which contain $\frak{t}$ and $\Delta^+$ is contained in the 
corresponding system of positive roots in $\Delta(\frak{g} \times \frak{g}, \frak{h} \times \frak{h})$, 
are of the form $\frak{q} \times \frak{q}$, where $\frak{q}$ is a parabolic subalgebra of $\frak{g}$ 
containing the Borel subalgebra $ \frak{b} = \frak{h} \oplus \sum\limits_{\alpha \in \Delta^+} \frak{g}_\alpha$. Let $\Phi $ be the set 
of simple roots in $\Delta^+$. 
The parabolic subalgebras of $\frak{g}$ containing the Borel subalgebra $\frak{b}$ are in one-one 
correspondence with $p(\Phi)$, the power set of $\Phi$. Namely, for $\Phi' \subset \Phi$, the parabolic subalgebra $\frak{q}_{\Phi'}$ 
corresponding to $\Phi'$ is given by 
\[\frak{q}_{\Phi'} = \frak{l}_{\Phi'} \oplus \frak{u}_{\Phi'} ,\]
where $\frak{l}_{\Phi'} = \frak{h} \oplus \sum\limits_{\substack{n_\psi (\alpha) =0 \\ \forall \psi \in \Phi'}} \frak{g}_\alpha$,  
$\frak{u}_{\Phi'} = \sum\limits_{\substack{n_\psi (\alpha) >0 \\ \textrm{for some } \psi \in \Phi'}} \frak{g}_\alpha$; and 
$\alpha =\sum\limits_{\psi \in \Phi}  n_\psi (\alpha) \psi \in \Delta$. So the $\theta$-stable parabolic subalgebras of 
$\frak{g} \times \frak{g}$,  which contain $\frak{t}$ and $\Delta^+$ is contained in the corresponding system of positive roots 
in $\Delta(\frak{g} \times \frak{g}, \frak{h} \times \frak{h})$, are in 
one-one correspondence with $p(\Phi)$. The one corresponding to $\Phi' \subset \Phi$ is given by 
$\frak{q}_{\Phi'} \times \frak{q}_{\Phi'}$, where $\frak{q}_{\Phi'}$ is given above. 

   Note that the root space decomposition of $\frak{g} \times \frak{g}$ with respect to the Cartan subalgebra $\frak{h} \times \frak{h}$ 
is given by 
\[ \frak{g} \times \frak{g} = \frak{h} \times \frak{h} \oplus \sum\limits_{\alpha \in \Delta} \frak{g}_{(\alpha , 0)} \oplus 
\frak{g}_{(0 , \alpha)}, \]
where $ \frak{g}_{(\alpha , 0)} = \{(Z, 0) : Z \in \frak{g}_\alpha \}, \frak{g}_{(0 , \alpha)} = \{(0, Z) : Z \in \frak{g}_\alpha \}$ for all 
$\alpha \in \Delta$. So for $\Phi' \subset \Phi$, if $\tilde{\frak{u}}_{\Phi'}$ denotes the nilradical of the $\theta$-stable parabolic 
subalgebra $\frak{q}_{\Phi'} \times \frak{q}_{\Phi'}$, then   
\[ \tilde{\frak{u}}_{\Phi'} = \sum\limits_{\substack{\alpha \in \Delta \\ n_\psi (\alpha) >0 \\ \textrm{for some } \psi \in \Phi'}} 
\frak{g}_{(\alpha , 0)} \oplus \frak{g}_{(0 , \alpha)}. \]
Again $\theta(\tilde{\frak{u}}_{\Phi'}) = \tilde{\frak{u}}_{\Phi'}$ implies $\tilde{\frak{u}}_{\Phi'} = (\tilde{\frak{u}}_{\Phi'}  
\cap \frak{k}) \oplus (\tilde{\frak{u}}_{\Phi'} \cap \frak{p})$. Hence $\tilde{\frak{u}}_{\Phi'} \cap \frak{p} =
\sum\limits_{\substack{\alpha \in \Delta \\ n_\psi (\alpha) >0 \\ \textrm{for some } \psi \in \Phi'}}\{ (Z, -Z) : Z \in \frak{g}_\alpha \}$, 
and so dim$((\tilde{\frak{u}}_{\Phi'} \cap \frak{p})=$ dim$(\frak{u}_{\Phi'})$. 

  The Levi subgroup $L = \{ g \in G : \textrm{Ad}(g)(\frak{q}_{\Phi'} \times \frak{q}_{\Phi'}) = 
\frak{q}_{\Phi'} \times \frak{q}_{\Phi'} \}$ is a connected reductive Lie subgroup 
of $G$ with Lie algebra $\frak{l}_{\Phi'}$. As $\theta(\frak{l}_{\Phi'}) = \frak{l}_{\Phi'}$, $\frak{l}_{\Phi'} \cap \frak{u}$ is compact 
real form of $\frak{l}_{\Phi'}$ and $L \cap U$ is a maximal compact subgroup of $L$. Also the centre of the reductive Lie algebra 
$\frak{l}_{\Phi'}$ is $|\Phi'|$-dimensional, where $|\Phi'|$ denotes the cardinality of the set $\Phi'$. Let $Y_{\Phi'}$ 
denote the compact dual of the Riemannian globally symmetric space $L /L \cap U$. Then $Y_{\Phi'} = L \cap U$, a connected compact Lie group. 
Hence 
\[ H^p (Y_{\Phi'} ; \mathbb{C}) = H^p ((\frak{l}_{\Phi'} \cap \frak{u})^\mathbb{C} ; \mathbb{C} ) = 
H^p (\frak{l}_{\Phi'} ; \mathbb{C}). \]
If $\frak{s}$ is a finite dimensional complex Lie algebra, we denote by $P(\frak{s}, t)$, the Poincar\'{e} polynomial of 
$H^* (\frak{s}; \mathbb{C})$. So 
\[ P(Y_{\Phi'} , t ) = (1 + t)^{|\Phi' |} P(\frak{l}_1 , t) P(\frak{l}_2, t) \cdots P(\frak{l}_k , t) , \textrm{ by Ku\"{n}neth formula}\  ;  \]
where $\frak{l}_1 , \frak{l}_2 , \ldots \frak{l}_k $ are the simple factors of the semisimple part $[\frak{l}_{\Phi'} , \frak{l}_{\Phi'} ]$ of 
$\frak{l}_{\Phi'} $. If $\frak{s}$ is a finite dimensional complex simple Lie algebra, the Poincar\'{e} polynomial $P(\frak{s}, t)$ is given by 
\[ P(\frak{s}, t) = (1 + t^{2d_1 + 1} )(1 + t ^{2d_2 +1 }) \cdots (1 + t^{2d_l + 1} )\  \cdots \cdots \cdots \cdots \cdots  (*), \]
where $l = \textrm{rank}(\frak{s})$ and $d_1 , d_2 , \ldots d_l$ are the exponents of $\frak{s}$ (see the table given below). 
If $A_{\Phi'}$ is the irreducible unitary representation of $G$ associated with the $\theta$-stable parabolic subalgebra $\frak{q}_{\Phi'} 
\times \frak{q}_{\Phi'} $, then the Poincar\'{e} polynomial of $H^* (\frak{g} \times \frak{g} , U; A_{\Phi' , U})$ is given by 
\[ P( \Phi' , t ) = t^{\textrm{dim}(\frak{u}_{\Phi'})}(1 + t)^{|\Phi'|} P(\frak{l}_1 , t) P(\frak{l}_2, t) \cdots P(\frak{l}_k , t), \]
where each $ P(\frak{l}_i , t)$ is given by the formula $(*)$. Also for $\Phi',\ \Phi'' \subset \Phi$, $A_{\Phi'}$ is unitarily equivalent to 
$A_{\Phi''}$ if and only if $\tilde{\frak{u}}_{\Phi'} \cap \frak{p} =  \tilde{\frak{u}}_{\Phi''} \cap \frak{p}$ if and only if 
$\frak{u}_{\Phi'} = \frak{u}_{\Phi''}$ if and only if $\Phi' = \Phi''$. 

  The exponents of complex simple Lie algebras are given below : 

\begin{center} 
\begin{table} 
\caption{Table for the exponents of a complex simple Lie algebra $\frak{s}$ of rank $l$}\label{exponents}
\begin{tabular}{|c|c|} 
\toprule 
\addlinespace[5pt] 
$\frak{s}$ & $d_1,d_2,\ldots ,d_l$ \\ 
\midrule 
\addlinespace[5pt] 
$\frak{a}_l$ & $1,2,\ldots ,l$ \\ 
\hline 
\addlinespace[5pt] 
$\frak{b}_l$ & $1,3,\ldots ,2l-1$ \\ 
\hline 
\addlinespace[5pt] 
$\frak{c}_l$ & $1,3,\ldots , 2l-1$ \\ 
\hline 
\addlinespace[5pt] 
$\frak{\delta}_l$ & $1,3,\ldots ,2l-3,l-1$ \\ 
\hline 
\addlinespace[5pt] 
$\frak{e}_6$ & $1,4,5,7,8,11$ \\ 
\hline 
\addlinespace[5pt] 
$\frak{e}_7$ & $1,5,7,9,11,13,17$ \\ 
\hline 
\addlinespace[5pt] 
$\frak{e}_8$ & $1,7,11,13,17,19,23,29$ \\ 
\hline 
\addlinespace[5pt] 
$\frak{f}_4$ & $1,5,7,11$ \\ 
\hline 
\addlinespace[5pt] 
$\frak{g}_2$ & $1,5$ \\ 
\bottomrule 
\end{tabular} 
\end{table} 
\end{center} 

\noindent 
\subsection{Proof of Theorem \ref{th2}}\label{pfth2}  

  Let $G$ be a connected complex simple Lie group with Lie$(G) = \frak{g}$. Let $\frak{u}$ be a compact real form of $\frak{g}$, 
$\theta$ be the corresponding Cartan involution of $\frak{g}^\mathbb{R}$, and $U$ be the connected Lie subgroup of $G$ with 
Lie algebra $\frak{u}$. Let $\frak{h}$ be a $\theta$-stable Cartan subalgebra of $\frak{g}$. 
Choose a system of positive roots $\Delta^+$ in the set of all non-zero roots $\Delta = \Delta(\frak{g}, \frak{h})$. 
Let $\Phi$ be the set of all simple roots in $\Delta^+$. We have seen that up to unitary equivalence, 
the irreducible unitary representations of $G$ with non-zero $(\frak{g} \times \frak{g} , U)$-cohomology 
are in one-one correspondence with $p(\Phi)$, the power set of $\Phi$. The one corresponding to $\Phi' \subset \Phi$ is 
$A_{\Phi'}$ which is the irreducible unitary representation of $G$ associated with the $\theta$-stable parabolic subalgebra 
$\frak{q}_{\Phi'} \times \frak{q}_{\Phi'}$ of $\frak{g}$, where 
$\frak{q}_{\Phi'} = \frak{l}_{\Phi'} \oplus \frak{u}_{\Phi'}$, 
$\frak{l}_{\Phi'} = \frak{h} \oplus \sum\limits_{\substack{n_\psi (\alpha) =0 \\ \forall \psi \in \Phi'}} \frak{g}_\alpha$,  
$\frak{u}_{\Phi'} = \sum\limits_{\substack{n_\psi (\alpha) >0 \\ \textrm{for some } \psi \in \Phi'}} \frak{g}_\alpha$, and 
$\alpha =\sum\limits_{\psi \in \Phi}  n_\psi (\alpha) \psi \in \Delta$. 
The Poincar\'{e} polynomial of $H^* (\frak{g} \times \frak{g} , U; A_{\Phi' , U})$ is given by 
$P( \Phi' , t ) = t^{\textrm{dim}(\frak{u}_{\Phi'})}(1 + t)^{|\Phi'|} P(\frak{l}_1 , t) P(\frak{l}_2, t) \cdots P(\frak{l}_k , t)$, 
where $\frak{l}_1 , \frak{l}_2 , \ldots \frak{l}_k $ are the simple factors of the semisimple part $[\frak{l}_{\Phi'} , \frak{l}_{\Phi'} ]$ of 
$\frak{l}_{\Phi'} $. We begin with the following results regarding $P( \Phi' , t )$. 

\begin{lemma}\label{lemma1} 
If $\bar{\nu}$ is an automorphism of the Dynkin diagram of $\frak{g}$, then $P( \Phi' , t ) = P( \bar{\nu}(\Phi') , t )$ for any 
$\Phi' \subset \Phi$. 
\end{lemma} 

\begin{proof} 
Note that $\{ \alpha \in \Delta : n_\psi (\alpha) > 0 \textrm{ for some } \psi \in \bar{\nu}(\Phi' )\} = 
\bar{\nu}(\{\alpha \in \Delta : n_\psi (\alpha) > 0 \textrm{ for some } \psi \in \Phi' \})$, and so 
dim$(\frak{u}_{\Phi'}) =$ dim$(\frak{u}_{\bar{\nu}(\Phi')})$. Also the Dynkin diagram of $[\frak{l}_{\Phi'} , \frak{l}_{\Phi'} ]$ 
is the subdiagram of the Dynkin diagram of $\frak{g}$ consisting of the vertices in $\Phi \setminus \Phi'$. Now 
$\Phi \setminus \bar{\nu}(\Phi') = \bar{\nu}(\Phi \setminus \Phi')$ and $\bar{\nu}$ maps the subdiagram consisting of the vertices in 
$\Phi \setminus \Phi'$ onto the subdiagram consisting of the vertices in $\bar{\nu}(\Phi \setminus \Phi')$. 
So $[\frak{l}_{\Phi'} , \frak{l}_{\Phi'} ] \cong [\frak{l}_{\bar{\nu}(\Phi'}), \frak{l}_{\bar{\nu}(\Phi')} ]$. Hence the proof is complete.  
\end{proof} 

\begin{lemma}\label{lemma2} 
If $\Phi',\ \Phi'' \subset \Phi$, then the degree of $P(\Phi' , t)$ - the degree of $P(\Phi'',t) =$ 
dim$(\frak{u}_{\Phi''}) - \textrm{dim}(\frak{u}_{\Phi'})$. 
\end{lemma} 

\begin{proof} 
Recall that $P( \Phi' , t ) = t^{\textrm{dim}(\frak{u}_{\Phi'})}P(Y_{\Phi'} , t )$, where $Y_{\Phi'}$ is the connected Lie subgroup of $U$ 
with Lie algebra $\frak{l}_{\Phi'} \cap \frak{u}$, which is a compact real form of $\frak{l}_{\Phi'}$. Hence the degree 
of $P(Y_{\Phi'} , t )$ is the dimension of $Y_{\Phi'}$, which is equal to dim$(\frak{l}_{\Phi'}) = \textrm{dim}(\frak{g}) - 
2\textrm{dim}(\frak{u}_{\Phi'})$. Thus the degree of $P( \Phi' , t ) = \textrm{dim}(\frak{g}) - \textrm{dim}(\frak{u}_{\Phi'})$. 
Similarly the degree of $P( \Phi'' , t ) = \textrm{dim}(\frak{g}) - \textrm{dim}(\frak{u}_{\Phi''})$. So 
the degree of $P(\Phi' , t)$ - the degree of $P(\Phi'',t) = \textrm{dim}(\frak{g}) - \textrm{dim}(\frak{u}_{\Phi'}) - 
\textrm{dim}(\frak{g}) + \textrm{dim}(\frak{u}_{\Phi''}) = \textrm{dim}(\frak{u}_{\Phi''}) - \textrm{dim}(\frak{u}_{\Phi'})$. 
\end{proof} 

\begin{remark} 
{\em 
If $\Phi',\ \Phi'' \subset \Phi$ with $\Phi' \subset \Phi''$, then 
dim$(\frak{u}_{\Phi'}) \le $ dim$(\frak{u}_{\Phi''})$. Thus 
dim$(\frak{u}_{\Phi'}) \le $ dim$(\frak{u}_{\Phi''}) <$ the degree of $P(\Phi'' , t) \le $ 
the degree of $P(\Phi' , t)$. In particular,  
dim$(\frak{u}_{\{\alpha \}}) \le $ dim$(\frak{u}_{\Phi}) <$ the degree of $P(\Phi , t) \le $ the degree of 
$P(\{\alpha\}, t)$, for any $\alpha \in \Phi$. 
} 
\end{remark} 

\begin{lemma}\label{lemma3} 
If $f(t) = (1 +t)(1+t^3)(1+t^5)\cdots (1+t^{2l+1})\ (l \in \mathbb{N})$, then the coefficients of $t^2$ and 
$t^{(l+1)^2 - 2}$ in $f(t)$ are zero and the coefficients of $t^n \ (0 \le n \le (l+1)^2,\ n \neq 2,\ (l+1)^2 - 2)$ 
in $f(t)$ are non-zero. 
\end{lemma} 

\begin{proof} 
We shall prove this by induction on $l$. For $l=1$, the result is obviously true. Assume that $l >1$ and 
the result is true for $l-1$. So the coefficients of $t^2,\ t^{l^2 - 2}$ in 
$(1 +t)(1+t^3)(1+t^5)\cdots (1+t^{2l-1})$ are zero and the coefficients of $t^n 
\ (0 \le n \le l^2,\ n \neq 2,\ l^2 - 2)$ are non-zero. Now the degree of $f(t) = (l+1)^2$, and so the coefficient of $t^n$ is 
non-zero iff the coefficient of $t^{(l+1)^2 - n}$ is non-zero. 
Since $(l+1)^2 -(l^2 -2) = 2l+3 = 1 + 3 + (2l-1),\ (l+1)^2-(l^2 + 2l) = 1$, 
and $(l+1)^2 - (l^2 + 2m-1) = 2l-2m +2 = 1 + (2(l-m)+1),\ (l+1)^2 - (l^2 + 2m) = 2(l-m) +1$ for all $1 \le m \le l-1$, we have 
the coefficients of $t^{(l+1)^2 -(l^2 -2)},\ t^{(l+1)^2 -(l^2 + 2l)},\ t^{(l+1)^2 - (l^2 + 2m-1)},\ t^{(l+1)^2 - (l^2 + 2m)}$ in 
$f(t)$ are non-zero for all $1 \le m \le l-1$, and thus the coefficients of $t^{(l^2 -2)},\ t^{(l^2 + 2l)},\ t^{(l^2 + 2m-1)},\ 
t^{(l^2 + 2m)}$ in $f(t)$ are non-zero for all $1 \le m \le l-1$. Clearly the coefficient of $t^2$ and 
so the coefficient of $t^{(l+1)^2 - 2}$ are zero. Hence the result. 
\end{proof} 

  Now we shall return to the proof of Theorem \ref{th2}. Note that $X := G/U$ is a Riemannian globally 
symmetric space of type IV, and 
\[H^k (\Gamma \backslash X; \mathbb{C}) \cong  \bigoplus_{\Phi' \subset \Phi} m_{\Phi'} 
H^k (\frak{g}\times \frak{g}, U; A_{\Phi' , U}) \] 
for all $k$ and for any uniform lattice $\Gamma$ of $G$, where $m_{\Phi'}$ is the multiplicity of $A_{\Phi'}$ in 
$L^2(\Gamma \backslash G)$. For the empty subset $\phi,\  m_\phi = 1$. Now by Theorem \ref{th1}, for each 
$i = 1,2,$ or $3$, there exists $\Gamma \in \mathcal{L}_i (G)$ such that 
$H^k (\Gamma \backslash X ; \mathbb{C})$ contains a non-zero cohomology class which has no 
$H^k (\frak{g}\times \frak{g}, U; A_{\phi , U})$-component, for some $k$ (which depends only on $\frak{g}$ if 
$\Gamma \in \mathcal{L}_1 (G)$ or $\mathcal{L}_2 (G)$, and it depends on $\frak{g}$ and $\Gamma$ if 
$\Gamma \in \mathcal{L}_3 (G)$) given as 
dim$(X(\bar{\sigma}))$ and dim$(X(\bar{\sigma}\bar{\theta}))$ in the Table \ref{resulttable}. 
Now we shall determine possible $\Phi' \subset \Phi$ such that the non-zero cohomology class in 
$H^k (\Gamma \backslash X ; \mathbb{C})$ 
has a $H^k (\frak{g}\times \frak{g}, U; A_{\Phi', U})$-component, via case by case consideration. 
Note that $\textrm{dim}(X(\bar{\sigma})) + \textrm{dim}(X(\bar{\sigma}\bar{\theta})) = 
\textrm{dim}_\mathbb{C}(\frak{g}) = m$(say). Let $q = \textrm{min}\{\textrm{dim}(X(\bar{\sigma})),\ 
\textrm{dim}(X(\bar{\sigma}\bar{\theta}))\}$. Then for $\Phi' \subset \Phi$, 
the degree of $P(\Phi' , t) - (m-q) = \textrm{dim}(\frak{u}_{\Phi'}) + \textrm{dim}(\frak{l}_{\Phi'}) - m +q = 
q -  \textrm{dim}(\frak{u}_{\Phi'})$. Thus dim$(\frak{u}_{\Phi'}) \le q \le$ the degree of $P(\Phi' , t)$ if and only if 
dim$(\frak{u}_{\Phi'}) \le m-q \le$ the degree of $P(\Phi' , t)$, and in this case 
$q -  \textrm{dim}(\frak{u}_{\Phi'}) =$ the degree of $P(\Phi' , t) - (m-q)$. So the coefficient of $t^q$ 
in $P(\Phi' , t)$ is non-zero if and only if the coefficient of $t^{m-q}$ in $P(\Phi' , t)$ is non-zero. Hence it is 
sufficient to determine possible $\Phi' \subset \Phi$ such that the non-zero cohomology class in 
$H^k (\Gamma \backslash X ; \mathbb{C})$ (where $k = \textrm{min}\{\textrm{dim}(X(\bar{\sigma})),\ 
\textrm{dim}(X(\bar{\sigma}\bar{\theta}))\}$)  
has a $H^k (\frak{g}\times \frak{g}, U; A_{\Phi', U})$-component. 

1. $\frak{g} = \frak{a}_{n-1}$, $n > 1$. Then $k = 2p(n-p),\ p^2 + (n-p)^2 -1,\ \frac{(n-2)(n+1)}{2},\ 
\frac{n(n+1)}{2}\  (1\le p \le \frac{n}{2})$, if $n\in 4\mathbb{Z}$; 
$k = 2p(n-p),\ p^2 + (n-p)^2 -1,\ \frac{(n-2)(n+1)}{2},\ 
\frac{n(n+1)}{2}\  (1\le p < \frac{n}{2})$, if $n\in 2 + 4\mathbb{Z},\ n >2$; and 
$k = 2p(n-p),\ p^2 + (n-p)^2 -1,\ \frac{n(n-1)}{2},\ \frac{(n-1)(n+2)}{2},\ (1\le p \le \frac{n-1}{2})$, if 
$n\in 1 +2\mathbb{Z}$. 

\begin{center}
\begin{tikzpicture}

\draw (0,0) circle [radius = 0.1];
\draw (1,0) circle [radius = 0.1]; 
\draw (2.5,0) circle [radius = 0.1]; 
\draw (3.5,0) circle [radius = 0.1]; 
\node [below] at (0.05,-0.05) {$\psi_1$}; 
\node [below] at (1.05,-0.05) {$\psi_2$}; 
\node [below] at (2.50,-0.05) {$\psi_{n-2}$}; 
\node [below] at (3.80,-0.05) {$\psi_{n-1}$}; 
\draw (0.1,0) -- (0.9,0); 
\draw (1.1,0) -- (1.5,0); 
\draw [dotted] (1.5,0) -- (2,0); 
\draw (2,0) -- (2.4,0); 
\draw (2.6,0) -- (3.4,0); 
\node [left] at (-0.5,0) {$\frak{a}_{n-1} :$};  
\node [left] at (-0.5,-0.6) {$(n > 1)$};  

\end{tikzpicture} 
\end{center} 

  Here $\Phi = \{\psi_1, \psi_2, \ldots , \psi_{n-1} \}$, and dim$(\frak{u}_{\{\psi_j\}}) = j(n-j)$ 
for all $1 \le j \le n-1$.  \\ 
$P(\{\psi_1\}, t) = t^{n-1}(1+t)(1 + t^3)(1+t^5) \cdots (1+t^{2n - 3}) = P(\{\psi_{n-1}\}, t)$, and \\ 
$P(\{\psi_j \}, t) = t^{j(n-j)}(1+t)(1+t^3)(1+t^5) \cdots (1+t^{2j-1})(1+t^3)(1+t^5) \cdots (1 +t^{2n-2j-1})$ for all 
$2 \le p \le n-2$. \\  
Also $P(\Phi, t) = t^{\frac{n(n-1)}{2}}(1+t)^{n-1}$. \\ 
Now dim$(\frak{u}_{\{\psi_j\}})-$ dim$(\frak{u}_{\{\psi_{j-1}\}}) = j(n-j) - (j-1)(n-j+1) = -j + n-j+1 = n-2j + 1$. 
Thus dim$(\frak{u}_{\{\psi_j\}})-$ dim$(\frak{u}_{\{\psi_{j-1}\}}) 
>0$ for all $1 < j < \frac{n+1}{2}$. That is,  \\ 
dim$(\frak{u}_{\{\psi_1\}})<$  dim$(\frak{u}_{\{\psi_2\}})< \cdots <$  dim$(\frak{u}_{\{\psi_c\}})$, \\ 
where $c = \frac{n-1}{2}$, if $n$ is odd, and $c = \frac{n}{2}$, if $n$ is even. \\ 
Clearly dim$(\frak{u}_{\{\psi_{n-j}\}}) =$dim$(\frak{u}_{\{\psi_{j}\}})$ for all $1 \le j \le c$. 

  First consider $k = 2(n-1)$. Note that $2(n-1) < 3(n-3) =$ dim$(\frak{u}_{\{\psi_3\}})$ iff $n > 7$. 
So if $n > 7,\ 2(n-1) <$ dim$(\frak{u}_{\Phi'})$ for any $\Phi' \subset \Phi$ with 
$\psi_j \in \Phi'$ for some $3 \le j \le n-3$, since dim$(\frak{u}_{\Phi'}) \ge $ dim$(\frak{u}_{\{\alpha\}})$ for any 
$\alpha \in \Phi'$. 
So the coefficient of $t^{2(n-1)}$ in $P(\Phi', t)$ is zero for any $\Phi' \subset \Phi$ with 
$\psi_j \in \Phi'$ for some $3 \le j \le n-3$, if $n > 7$. Also the coefficient of $t^{2(n-1)}$ in 
$P(\{\psi_1\},t)$ is non-zero iff $n \neq 3$ (by Lemma \ref{lemma3}). 
Since $2(n-1) = 2(n-2) +2$, the coefficient of $t^{2(n-1)}$ in $P(\{\psi_2 \},t)$ is always 
zero (by Lemma \ref{lemma3}).
Now dim$(\frak{u}_{\{\psi_1, \psi_2\}}) = 2n-3=$ 
dim$(\frak{u}_{\{\psi_1, \psi_{n-1}\}})$, dim$(\frak{u}_{\{\psi_1, \psi_{n-2}\}}) = 3n - 7$, 
dim$(\frak{u}_{\{\psi_2, \psi_{n-2}\}}) = 4n - 12$. 
So the coefficients of $t^{2(n-1)}$ in 
$P(\{\psi_1, \psi_2 \},t),\ P(\{\psi_1, \psi_{n-1} \},t),\ P(\{\psi_{n-1}, \psi_{n-2} \},t)$ are non-zero. Thus we do not get 
any significant result. The other values of $k$ also do not give any significant result. 

2. $\frak{g} = \frak{b}_n\ (n \ge 2)$. 

\begin{center} 
\begin{tikzpicture}

\draw (0,0) circle [radius = 0.1]; 
\draw (1,0) circle [radius = 0.1]; 
\draw (2,0) circle [radius = 0.1]; 
\draw (3.5,0) circle [radius = 0.1]; 
\draw (4.5,0) circle [radius = 0.1]; 
\node [above] at (0.05,0.05) {$\psi_1$}; 
\node [above] at (1.05,0.05) {$\psi_2$}; 
\node [above] at (2.05,0.05) {$\psi_3$}; 
 \node [above] at (3.75,0.05) {$\psi_{n-1}$}; 
\node [above] at (4.55,0.05) {$\psi_n$}; 
\node [left] at (-0.5,0) {$\frak{b}_n :$}; 
\node[ left] at (-0.5,-0.6) {$(n \ge 2)$}; 
\draw (0.1,0) -- (0.9,0); 
\draw (1.1,0) -- (1.9,0);
\draw (2.1,0) -- (2.5,0); 
\draw [dotted] (2.5,0) -- (3,0); 
\draw (3,0) -- (3.4,0); 
\draw (4.4,0) -- (4.3,0.1); 
\draw (4.4,0) -- (4.3,-0.1); 
\draw (3.6,0.025) -- (4.35,0.025); 
\draw (3.6,-0.025) -- (4.35,-0.025); 

\end{tikzpicture} 
\end{center} 

 Here $\Phi = \{\psi_1, \psi_2, \ldots , \psi_n \}$, and dim$(\frak{u}_{\{\psi_j\}}) = 2j(n-j)+\frac{j(j+1)}{2}$ 
for all $1 \le j \le n$.  \\  
$P(\{\psi_j \}, t) = t^{2j(n-j)+\frac{j(j+1)}{2}}(1+t)(1+t^3)(1+t^5) \cdots (1+t^{2j-1})(1+t^3)(1+t^7) \cdots (1 +t^{4n-4j-1})$ 
for all $1 \le j \le n-1$, and  \\   
$P(\{\psi_n \}, t) = t^{\frac{n(n+1)}{2}}(1+t)(1+t^3)(1+t^5) \cdots (1+t^{2n-1})$. \\ 
Also $P(\Phi, t) = t^{n^2}(1+t)^{n}$. 

 In this case, we do not have Theorem \ref{th1}. See Table \ref{resulttable}. 

3. $\frak{g} = \frak{c}_{n}$, $n \ge 3$. Then $k = 4p(n-p),\ p(2p+1) + (n-p)(2n-2p+1),\ n^2,\ 
n(n+1)\  (1\le p \le n-1)$, if $n\in 4\mathbb{Z}$ or $n \in 3 + 4\mathbb{Z}$; 
$k = 4p(n-p),\ p(2p+1) + (n-p)(2n-2p+1),\ (1\le p \le n-1,\ p \neq \frac{n}{2})$, if $n\in 1 +4\mathbb{Z}$ or 
$n\in 2 +4\mathbb{Z}$.  

\begin{center} 
\begin{tikzpicture}

\draw (0,0) circle [radius = 0.1];  
\draw (1,0) circle [radius = 0.1]; 
\draw (2.5,0) circle [radius = 0.1]; 
\draw (3.5,0) circle [radius = 0.1]; 
\node [above] at (0.05,0.05) {$\psi_1$}; 
\node [above] at (1.05,0.05) {$\psi_2$}; 
\node [above] at (2.75,0.05) {$\psi_{n-1}$}; 
\node [above] at (3.55,0.05) {$\psi_n$}; 
\node [left] at (-0.5,0) {$\frak{c}_n :$}; 
\node [left] at (-0.5,-0.6) {$(n \ge 3)$}; 
\draw (0.1,0) -- (0.9,0); 
\draw (1.1,0) -- (1.5,0); 
\draw [dotted] (1.5,0) -- (2,0); 
\draw (2,0) -- (2.4,0); 
\draw (2.6,0) -- (2.7,0.1); 
\draw (2.6,0) -- (2.7,-0.1); 
\draw (2.65,0.025) -- (3.4,0.025); 
\draw (2.65,-0.025) -- (3.4,-0.025); 

\end{tikzpicture} 
\end{center} 

  Here $\Phi = \{\psi_1, \psi_2, \ldots , \psi_{n-1}, \psi_n \}$, and dim$(\frak{u}_{\{\psi_j\}}) = 
2j(n-j) + \frac{j(j+1)}{2}$. \\  
$P(\{\psi_j \}, t) = t^{2j(n-j)+\frac{j(j+1)}{2}}(1+t)(1+t^3)(1+t^5) \cdots (1+t^{2j-1})(1+t^3)(1+t^7) \cdots (1 +t^{4n-4j-1})$ 
for all $1 \le j \le n-1$, and  \\   
$P(\{\psi_n \}, t) = t^{\frac{n(n+1)}{2}}(1+t)(1+t^3)(1+t^5) \cdots (1+t^{2n-1})$. \\ 
Also $P(\Phi, t) = t^{n^2}(1+t)^n$. \\ 
Now dim$(\frak{u}_{\{\psi_j\}})-$ dim$(\frak{u}_{\{\psi_{j-1}\}}) = 2j(n-j) + \frac{j(j+1)}{2} - 2(j-1)(n-j+1) - 
\frac{(j-1)j}{2} = 2(n-j+1) - 2j + j = 2n + 2 - 3j$. 
So dim$(\frak{u}_{\{\psi_j\}})-$ dim$(\frak{u}_{\{\psi_{j-1}\}}) 
>0$ for all $1 < j < \frac{2n+2}{3}$. Thus \\ 
dim$(\frak{u}_{\{\psi_1\}})<$  dim$(\frak{u}_{\{\psi_2\}})< \cdots <$  dim$(\frak{u}_{\{\psi_c\}}) \ge $ 
dim$(\frak{u}_{\{\psi_{c+1}\}})>$ dim$(\frak{u}_{\{\psi_{c+2}\}})> \cdots >$ dim$(\frak{u}_{\{\psi_n\}})$, \\ 
where $c = \frac{2n}{3}$, if $n \in 3\mathbb{Z}$, $c = \frac{2n+1}{3}$, if $n \in 1 + 3\mathbb{Z}$, and 
$c = \frac{2n-1}{3}$, if $n \in 2 + 3\mathbb{Z}$. 

  First consider $k = 4(n-1)$. Note that $4(n-1) <$ dim$(\frak{u}_{\{\psi_3\}})$, and 
$4(n-1)< \frac{n(n+1)}{2} =$ dim$(\frak{u}_{\{\psi_n\}})$ iff $n \ge 6$. 
So if $n \ge 6,\ 4(n-1) <$ dim$(\frak{u}_{\Phi'})$ for any $\Phi' \subset \Phi$ with 
$\psi_j \in \Phi'$ for some $3 \le j \le n$, since dim$(\frak{u}_{\Phi'}) \ge $ dim$(\frak{u}_{\{\alpha\}})$ for any 
$\alpha \in \Phi'$. 
So the coefficient of $t^{4(n-1)}$ in $P(\Phi', t)$ is zero for any $\Phi' \subset \Phi,\ \Phi' \neq \phi,\ \{\psi_1\},\ \{\psi_2\},\ 
\{\psi_1, \psi_2\}$, if $n \ge 6$.  
Since $4(n-1)=$ dim$(\frak{u}_{\{\psi_2\}}) + 1 =$ dim$(\frak{u}_{\{\psi_1,\ \psi_2\}})$, 
the coefficients of $t^{4(n-1)}$ in $P(\{\psi_2\}, t)$ and $P(\{\psi_1,\ \psi_2\}, t)$ are non-zero. Now 
$4(n-1) - $ dim$(\frak{u}_{\{\psi_1\}}) = 2n - 3$, and 
$P(\{\psi_1 \}, t) = t^{2n-1}(1+t)(1+t^3)(1+t^7) \cdots (1 +t^{4(n-1)-1})$. Note that 
$2n - 3 = \frac{4(n-1)}{2} - 1$ if $n$ is odd, and $2n-3 = 3 + 7 + (\frac{4(n-6)}{2} -1)$ if $n$ is even. 
So if $n$ is odd, then the coefficient of $t^{4(n-1)}$ in $P(\{\psi_1\}, t)$ is non-zero. And if $n$ is even with 
$n \ge 12$, then the coefficient of $t^{4(n-1)}$ in $P(\{\psi_1\}, t)$ is non-zero. Also for $n = 6,\ 8$, or $10$, 
the coefficient of $t^{4(n-1)}$ in $P(\{\psi_1\}, t)$ is zero. 
Thus if $n \ge 6$, the non-zero cohomology class in $H^{4(n-1)} (\Gamma \backslash X ; \mathbb{C})$ has a 
$H^{4(n-1)} (\frak{g}\times \frak{g}, U; A_{\Phi', U})$-component, where $\Phi' = \{\psi_1\}$, or $\{\psi_2\}$, or $\{\psi_1,\ \psi_2\}$. 
If $n = 6, 8$, or $10$, we can discard $\{\psi_1\}$ among these. 
This implies in particular that 
if Lie$(G) = \frak{c}_n (n \ge 6)$, 
then for each $i = 1,2,$ or $3$, there is a uniform lattice $\Gamma \in \mathcal{L}_i(G)$, such that 
$L^2(\Gamma \backslash G)$ has an irreducible $A_{\Phi'}$-component for at least one $\Phi'$ given above. 
The other values of $k$ do not give any significant result. 

4. $\frak{g} = \frak{\delta}_{n}$, $n \ge 4$. Then $k = p(2n-p),\ \frac{p(p-1) + (2n-p)(2n-p-1)}{2},\ n(n-1), 
\ n^2 \ (1\le p \le n-1)$, if $n\not \in 2 + 4\mathbb{Z}$; $k = p(2n-p),\ \frac{p(p-1) + (2n-p)(2n-p-1)}{2}\ 
(1\le p \le n-1)$, if $n \in 2 + 4\mathbb{Z}$.  

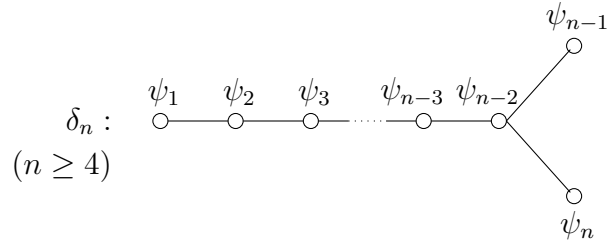
\begin{center} 
\begin{tikzpicture} 

\draw (0,0) circle [radius = 0.1]; 
\draw (1,0) circle [radius = 0.1]; 
\draw (2,0) circle [radius = 0.1]; 
\draw (3.5,0) circle [radius = 0.1]; 
\draw (4.5,0) circle [radius = 0.1]; 
\draw (5.5,1) circle [radius = 0.1]; 
\draw (5.5,-1) circle [radius = 0.1]; 
\node [above] at (0.05,0.05) {$\psi_1$}; 
\node [above] at (1.05,0.05) {$\psi_2$}; 
\node [above] at (2.05,0.05) {$\psi_3$}; 
\node [above] at (3.35,0.05) {$\psi_{n-3}$}; 
\node [above] at (4.35,0.05) {$\psi_{n-2}$}; 
\node [above] at (5.55,1.05) {$\psi_{n-1}$}; 
\node [below] at (5.55,-1.05) {$\psi_n$}; 
\node [left] at (-0.5,0) {$\frak{\delta}_n :$}; 
\node [left] at (-0.5,-0.6) {$(n \ge 4)$}; 
\draw (0.1,0) -- (0.9,0); 
\draw (1.1,0) -- (1.9,0); 
\draw (2.1,0) -- (2.5,0); 
\draw [dotted] (2.5,0) -- (3,0); 
\draw (3,0) -- (3.4,0); 
\draw (3.6,0) -- (4.4,0); 
\draw (4.6,0) -- (5.45,0.95); 
\draw (4.6,0) -- (5.45,-0.95); 

\end{tikzpicture} 
\end{center} 

 Here $\Phi = \{\psi_1, \psi_2, \ldots , \psi_{n-2}, \psi_{n-1}, \psi_n \}$; dim$(\frak{u}_{\{\psi_j\}}) = 
2j(n-j) + \frac{j(j-1)}{2}$ for all $1 \le j \le n, j \neq n-1$, and dim$(\frak{u}_{\{\psi_{n-1}\}}) =$ dim$(\frak{u}_{\{\psi_n\}})$.  \\ 
$P(\{\psi_j \}, t) = t^{2j(n-j)+\frac{j(j-1)}{2}}(1+t)(1+t^3)(1+t^5) \cdots (1+t^{2j-1})(1+t^3)(1+t^7) \cdots (1 +t^{4n-4j-5})(1+t^{2n-2j-1})$ 
for all $1 \le j \le n-2$, and  \\   
$P(\{\psi_n \}, t) = t^{\frac{n(n-1)}{2}}(1+t)(1+t^3)(1+t^5) \cdots (1+t^{2n-1}) = P(\{\psi_{n -1}\}, t)$. \\ 
Also $P(\Phi, t) = t^{n(n-1)}(1+t)^n$. \\ 
Now dim$(\frak{u}_{\{\psi_j\}})-$ dim$(\frak{u}_{\{\psi_{j-1}\}}) = 2j(n-j) + \frac{j(j-1)}{2} - 2(j-1)(n-j+1) - 
\frac{(j-1)(j-2)}{2} = 2(n-j+1) - 2j + j - 1 = 2n + 1 - 3j$ for all $1 < j \le n-2$.   
Thus,  \\ 
dim$(\frak{u}_{\{\psi_1\}})<$  dim$(\frak{u}_{\{\psi_2\}})< \cdots <$  dim$(\frak{u}_{\{\psi_c\}}) \ge $ 
dim$(\frak{u}_{\{\psi_{c+1}\}})>$ dim$(\frak{u}_{\{\psi_{c+2}\}})> \cdots >$ dim$(\frak{u}_{\{\psi_{n-2}\}})>$ 
dim$(\frak{u}_{\{\psi_{n-1}\}})=$ dim$(\frak{u}_{\{\psi_n\}})$,  \\ 
where $c = \frac{2n}{3}$, if $n \in 3\mathbb{Z}$, $c = \frac{2n-2}{3}$, if $n \in 1 + 3\mathbb{Z}$, and 
$c = \frac{2n-1}{3}$, if $n \in 2 + 3\mathbb{Z}$. 

  First consider $k = (2n-1)$. Note that $2n-1 < 2j(n-j) < 2j(n-j) + \frac{j(j-1)}{2} =$ dim$(\frak{u}_{\{\psi_j\}})$ for all 
$2 \le j \le n-2$, and $2n-1< \frac{n(n-1)}{2} =$ dim$(\frak{u}_{\{\psi_{n-1}\}})=$ dim$(\frak{u}_{\{\psi_n\}})$ iff $n >4$. 
So if $n > 4,\ 2n-1 <$ dim$(\frak{u}_{\Phi'})$ for any $\Phi' \subset \Phi$ with $|\Phi'| \ge 2$, 
as $\psi_j \in \Phi'$ for some $2 \le j \le n$, and dim$(\frak{u}_{\Phi'}) \ge $ dim$(\frak{u}_{\{\alpha\}})$ for any 
$\alpha \in \Phi'$. 
So the coefficient of $t^{2n-1}$ in $P(\Phi', t)$ is zero for any $\Phi' \subset \Phi,\ \Phi' \neq \phi,\ \{\psi_1\}$, if $n >4$.  
Since $2n-1 =$ dim$(\frak{u}_{\{\psi_1\}}) + 1$, the coefficient of $t^{2n-1}$ in $P(\{\psi_1\}, t)$ is non-zero. 
Thus the non-zero cohomology class in $H^{2n-1} (\Gamma \backslash X ; \mathbb{C})$ has only 
$H^{2n-1} (\frak{g}\times \frak{g}, U; A_{\{\psi_1\}, U})$-component. This implies in particular that 
if Lie$(G) = \frak{\delta}_n (n > 4)$, 
then for each $i = 1,2,$ or $3$, there is a uniform lattice $\Gamma \in \mathcal{L}_i(G)$, such that 
$L^2(\Gamma \backslash G)$ has an irreducible $A_{\{\psi_1\}}$-component. If $n = 4$, then 
corresponding to the value $k= 2n-1 = 7$, we can say that at least one 
$A_{\{\psi_j\}}\ (j=1,3,\textrm{or }4)$ will occur in $L^2(\Gamma \backslash G)$. 
The other values of $k$ do not give any significant result. 

5. $\frak{g} = \frak{e}_6$. Then $k = 26,\ 52,\ 32,\ 46,\ 36,\ 42,\ 38,\ 40$. 

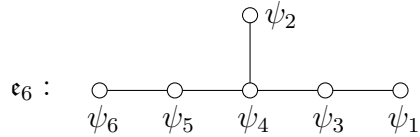
\begin{center} 
\begin{tikzpicture}  

\draw (0,0) circle [radius = 0.1]; 
\draw (1,0) circle [radius = 0.1]; 
\draw (2,0) circle [radius = 0.1]; 
\draw (2,1) circle [radius = 0.1]; 
\draw (3,0) circle [radius = 0.1]; 
\draw (4,0) circle [radius = 0.1]; 
\node [below] at (0.05,-0.05) {$\psi_6$}; 
\node [below] at (1.05,-0.05) {$\psi_5$}; 
\node [below] at (2.05,-0.05) {$\psi_4$}; 
\node [right] at (2.05,1) {$\psi_2$}; 
\node [below] at (3.05,-0.05) {$\psi_3$}; 
\node [below] at (4.05,-0.05) {$\psi_1$}; 
\node [left] at (-0.5,0) {$\frak{e}_6 :$}; 
\draw (0.1,0) -- (0.9,0); 
\draw (1.1,0) -- (1.9,0); 
\draw (2,0.1) -- (2,0.9); 
\draw (2.1,0) -- (2.9,0); 
\draw (3.1,0) -- (3.9,0); 

\end{tikzpicture} 
\end{center} 

  Here $\Phi = \{\psi_1, \psi_2, \ldots , \psi_6 \}$, and \\ 
$P(\{\psi_1\}, t) = t^{16}(1 + t)(1 + t^3)(1+t^7)(1 + t^{11})(1+t^{15})(1+t^9) = P(\{\psi_6\}, t)$, \\ 
$P(\{\psi_2\}, t) = t^{21}(1 + t)(1 + t^3)(1+t^5)(1 + t^7)(1+t^9)(1+t^{11})$, \\ 
$P(\{\psi_3\}, t) = t^{25}(1 + t)(1 + t^3)^2(1+t^5)(1 + t^7)(1+t^9) = P(\{\psi_5\}, t)$, and \\ 
$P(\{\psi_4\}, t) = t^{29}(1 + t)(1 + t^3)^3(1+t^5)^2$.  \\ 
Also $P(\Phi, t) = t^{36}(1+t)^6$. 

  Clearly dim$(\frak{u}_{\{\psi_j\}})< 26,\ 52,\ 32,\ 46,$  
$36,\ 42,\ 38,\ 40 <$ the degree of $P(\{\psi_j\} , t)$ for all $j$ except $j =4$, and 
dim$(\frak{u}_{\{\psi_4\}})< 32,\ 46,\ 36,\ 42,\ 38,\ 40 <$ the degree of $P(\{\psi_4\} , t)$. Not only these, but also 
the coefficients of $t^k$ in $P(\{\psi_j\}, t)$ are non-zero for all $1 \le j \le 6$, $k = 32,\ 46,\ 36,\ 42,\ 38,\ 40$; and 
the coefficients of $t^k$ in $P(\{\psi_j\}, t)$ are non-zero for all $1 \le j \le 6,\ j \neq 4$, $k = 26,\ 52$. Thus we do not get any 
significant result. 

6. $\frak{g} = \frak{e}_7$. Then $k = 64,\ 69$. 

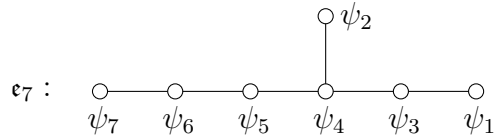
\begin{center} 
\begin{tikzpicture} 

\draw (0,0) circle [radius = 0.1]; 
\draw (1,0) circle [radius = 0.1]; 
\draw (2,0) circle [radius = 0.1]; 
\draw (3,0) circle [radius = 0.1]; 
\draw (3,1) circle [radius = 0.1]; 
\draw (4,0) circle [radius = 0.1]; 
\draw (5,0) circle [radius = 0.1]; 
\node [below] at (0.05,-0.05) {$\psi_7$}; 
\node [below] at (1.05,-0.05) {$\psi_6$}; 
\node [below] at (2.05,-0.05) {$\psi_5$}; 
\node [below] at (3.05,-0.05) {$\psi_4$}; 
\node [right] at (3.05,1) {$\psi_2$}; 
\node [below] at (4.05,-0.05) {$\psi_3$}; 
\node [below] at (5.05,-0.05) {$\psi_1$}; 
\node [left] at (-0.5,0) {$\frak{e}_7 :$}; 
\draw (0.1,0) -- (0.9,0); 
\draw (1.1,0) -- (1.9,0); 
\draw (2.1,0) -- (2.9,0); 
\draw (3,0.1) -- (3,0.9); 
\draw (3.1,0) -- (3.9,0); 
\draw (4.1,0) -- (4.9,0); 

\end{tikzpicture} 
\end{center} 

  Here $\Phi = \{\psi_1, \psi_2, \ldots , \psi_7 \}$, and \\ 
$P(\{\psi_1\}, t) = t^{33}(1 + t)(1 + t^3)(1+t^7)(1 + t^{11})(1+t^{15})(1+t^{19})(1+t^{11})$, \\ 
$P(\{\psi_2\}, t) = t^{42}(1 + t)(1 + t^3)(1+t^5)(1 + t^7)(1+t^9)(1+t^{11})(1+t^{13})$, \\ 
$P(\{\psi_3\}, t) = t^{47}(1 + t)(1 + t^3)^2(1+t^5)(1 + t^7)(1+t^9)(1+t^{11})$, \\   
$P(\{\psi_4\}, t) = t^{53}(1 + t)(1 + t^3)^3(1+t^5)^2(1+t^7)$,  \\   
$P(\{\psi_5\}, t) = t^{50}(1 + t)(1 + t^3)^2(1+t^5)^2(1 + t^7)(1+t^9)$, \\  
$P(\{\psi_6\}, t) = t^{42}(1 + t)(1 + t^3)^2(1+t^7)(1 + t^{11})(1+t^{15})(1+t^9)$, and \\   
$P(\{\psi_7\} , t) = t^{27}(1+t)(1 + t^3)(1+t^9)(1 + t^{11})(1+t^{15})(1+t^{17})(1+t^{23})$. \\ 
Also $P(\Phi, t) = t^{63}(1+t)^7$. 

  Clearly dim$(\frak{u}_{\{\psi_j\}})<$  dim$(\frak{u}_{\Phi})<64,\ 69<$ 
the degree of $P(\Phi , t)<$ the degree of $P(\{\psi_j\} , t)$ for all $j$. Not only these, but also 
the coefficients of $t^k$ in $P(\{\psi_j\}, t)$ are non-zero for all $1 \le j \le 7$, $k = 64,\ 69$. Thus we do not get any 
significant result. 

7. $\frak{g} = \frak{e}_8$. Then $k = 112,\ 136,\ 120,\ 128$. 

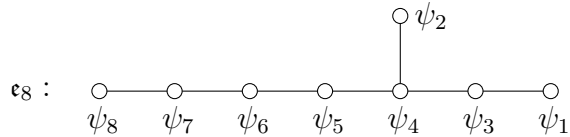
\begin{center} 
\begin{tikzpicture} 

\draw (0,0) circle [radius = 0.1]; 
\draw (1,0) circle [radius = 0.1]; 
\draw (2,0) circle [radius = 0.1]; 
\draw (3,0) circle [radius = 0.1]; 
\draw (4,0) circle [radius = 0.1]; 
\draw (4,1) circle [radius = 0.1]; 
\draw (5,0) circle [radius = 0.1]; 
\draw (6,0) circle [radius = 0.1]; 
\node [below] at (0.05,-0.05) {$\psi_8$}; 
\node [below] at (1.05,-0.05) {$\psi_7$}; 
\node [below] at (2.05,-0.05) {$\psi_6$}; 
\node [below] at (3.05,-0.05) {$\psi_5$}; 
\node [below] at (4.05,-0.05) {$\psi_4$}; 
\node [right] at (4.05,1) {$\psi_2$}; 
\node [below] at (5.05,-0.05) {$\psi_3$}; 
\node [below] at (6.05,-0.05) {$\psi_1$}; 
\node [left] at (-0.5,0) {$\frak{e}_8 :$}; 
\draw (0.1,0) -- (0.9,0); 
\draw (1.1,0) -- (1.9,0); 
\draw (2.1,0) -- (2.9,0); 
\draw (3.1,0) -- (3.9,0); 
\draw (4.1,0) -- (4.9,0); 
\draw (4,0.1) -- (4,0.9); 
\draw (5.1,0) -- (5.9,0); 

\end{tikzpicture} 
\end{center} 

  Here $\Phi = \{ \psi_1, \psi_2, \ldots , \psi_8 \}$, and \\ 
$P(\{\psi_1\}, t) = t^{78}(1 + t)(1 + t^3)(1+t^7)(1 + t^{11})(1+t^{15})(1+t^{19})(1+t^{23})(1+t^{13})$, \\ 
$P(\{\psi_2\}, t) = t^{92}(1 + t)(1 + t^3)(1+t^5)(1 + t^7)(1+t^9)(1+t^{11})(1+t^{13})(1+t^{15})$, \\ 
$P(\{\psi_3\}, t) = t^{98}(1 + t)(1 + t^3)^2(1+t^5)(1 + t^7)(1+t^9)(1+t^{11})(1+t^{13})$, \\   
$P(\{\psi_4\}, t) = t^{106}(1 + t)(1 + t^3)^3(1+t^5)^2(1+t^7)(1+t^9)$,  \\   
$P(\{\psi_5\}, t) = t^{104}(1 + t)(1 + t^3)^2(1+t^5)^2(1 + t^7)^2(1+t^9)$, \\  
$P(\{\psi_6\}, t) = t^{97}(1 + t)(1 + t^3)^2(1+t^5)(1+t^7)(1 + t^{11})(1+t^{15})(1+t^9)$, \\   
$P(\{\psi_7\} , t) = t^{83}(1+t)(1 + t^3)^2(1+t^9)(1 + t^{11})(1+t^{15})(1+t^{17})(1+t^{23})$, and \\ 
$P(\{\psi_8 \} , t) = t^{57}(1+t)(1 + t^3)(1+t^{11})(1 + t^{15})(1+t^{19})(1+t^{23})(1+t^{27})(1+t^{35})$. \\  
Also $P(\Phi, t) = t^{120}(1+t)^8$. 

  Clearly dim$(\frak{u}_{\{\psi_j\}})< 112,\ 136,\ 120,\ 128 <$ 
the degree of $P(\{\psi_j\} , t)$ for all $j$. Not only these, but also 
the coefficients of $t^k$ in $P(\{\psi_j\}, t)$ are non-zero for all $1 \le j \le 8$, $k = 112,\ 136,\ 120,\ 128$. Thus we do not get any 
significant result. 

8. $\frak{g} = \frak{f}_4$. Then $k = 16,\ 36,\ 24,\ 28$. 

\begin{center} 
\begin{tikzpicture} 
 
\draw (0,0) circle [radius = 0.1]; 
\draw (1,0) circle [radius = 0.1]; 
\draw (2,0) circle [radius = 0.1]; 
\draw (3,0) circle [radius = 0.1]; 
\node [below] at (0.05,-0.05) {$\psi_1$}; 
\node [below] at (1.05,-0.05) {$\psi_2$}; 
\node [below] at (2.05,-0.05) {$\psi_3$}; 
\node [below] at (3.05,-0.05) {$\psi_4$}; 
\node [left] at (-0.5,0) {$\frak{f}_4 :$}; 
\draw (0.1,0) -- (0.9,0); 
\draw (1.9,0) -- (1.8,0.1); 
\draw (1.9,0) -- (1.8,-0.1); 
\draw (1.1,0.025) -- (1.85,0.025); 
\draw (1.1,-0.025) -- (1.85,-0.025); 
\draw (2.1,0) -- (2.9,0); 

\end{tikzpicture} 
\end{center} 

  Here $\Phi = \{\psi_1, \psi_2, \psi_3, \psi_4 \}$, and  \\  
$P(\{\psi_1\}, t) = t^{15}(1 + t)(1 + t^3)(1+t^7)(1 + t^{11})$, \\ 
$P(\{\psi_2\}, t) = t^{20}(1 + t)(1 + t^3)^2(1+t^5)$, \\ 
$P(\{\psi_3\}, t) = t^{20}(1 + t)(1 + t^3)^2(1+t^5)$, and \\ 
$P(\{\psi_4\}, t) = t^{15}(1 + t)(1 + t^3)(1+t^7)(1+t^{11})$.  \\ 
Also $P(\Phi, t) = t^{24}(1+t)^4$. 

  Clearly dim$(\frak{u}_{\{\psi_j\}})<$  dim$(\frak{u}_{\Phi}) = 24 < 28=$ 
the degree of $P(\Phi , t)<$ the degree of $P(\{\psi_j\} , t)$ for all $j$. But 
the coefficients of $t^k$ in $P(\{\psi_j\}, t)$ are non-zero only for $j = 2,\ 3$; $k = 24,\ 28$. Yet these values do not give any 
significant result. Now consider the values $k = 16,\ 36$. Clearly the coefficient of $t^k\ (k = 16,\ 36)$ in $P(\Phi', t)$ is zero 
for any $\Phi' \subset \Phi$ such that $\Phi'$ contains $\psi_2$ or $\psi_3$. Also $16<20 =$ dim$(\frak{u}_{\{\psi_1,\ \psi_4\}}) < 32=$ 
the degree of $P(\{\psi_1, \psi_4\} , t) < 36$, and the coefficients of $t^k\ (k = 16,\ 36)$ in $P(\{\psi_1\}, t),\ P(\{\psi_4\}, t)$ are 
non-zero. Thus the non-zero cohomology class in $H^k (\Gamma \backslash X ; \mathbb{C})\ (k=16,\ 36)$ has a 
$H^k (\frak{g}\times \frak{g}, U; A_{\Phi', U})$-component, where \\ 
$\Phi' = \{\psi_1\}$, or $\{\psi_4\}$. \\  
This implies in particular that 
if Lie$(G) = \frak{f}_4$, 
then for each $i = 1,2,$ or $3$, there is a uniform lattice $\Gamma \in \mathcal{L}_i(G)$, such that 
$L^2(\Gamma \backslash G)$ has an irreducible $A_{\Phi'}$-component for at least one $\Phi'$ given above. 

9. $\frak{g} = \frak{g}_2$. Then $k = 6,\ 8$. 

\begin{center} 
\begin{tikzpicture} 
 
\draw (0,0) circle [radius = 0.1]; 
\draw (1,0) circle [radius = 0.1]; 
\node [below] at (0.05,-0.05) {$\psi_2$}; 
\node [below] at (1.05,-0.05) {$\psi_1$}; 
\node [left] at (-0.5,0) {$\frak{g}_2 :$}; 
\draw (0.9,0) -- (0.8,0.1); 
\draw (0.9,0) -- (0.8,-0.1); 
\draw (0.1,0) -- (0.9,0); 
\draw (0.1,0.075) -- (0.8,0.075); 
\draw (0.1,-0.075) -- (0.8,-0.075); 

\end{tikzpicture}
\end{center} 

  Here $\Phi = \{\psi_1, \psi_2 \}$, and  \\  
$P(\{\psi_1\}, t) = t^5(1 + t)(1 + t^3)$, and \\ 
$P(\{\psi_2\}, t) = t^5(1 + t)(1 + t^3)$. \\ 
Also $P(\Phi, t) = t^6(1+t)^2$. 

 Clearly the coefficients of $t^k$ in $P(\{\psi_1\}, t),\ P(\{\psi_2\}, t),\ P(\{\psi_1,\ \psi_2\}, t)$ are non-zero, $k = 6,\ 8$. 
Thus we do not get any significant result. 

  Thus the proof of Theorem \ref{th2} is complete. 

\begin{remark}\label{conclusion}  
{\em 
(i) If $\frak{g} = \frak{a}_2$, then $k = 3,\ 4,\ 5$. Also the coefficients of $t^4$ in $P(\{\psi_1\}, t)$, 
$P(\{\psi_2\}, t)$ are zero, and 
the coefficient of $t^4$ in $P(\{\psi_1,\ \psi_2\}, t)$ is non-zero. This shows that if Lie$(G) = \frak{a}_2$, 
then for each $i = 1,2,$ or $3$, there is a uniform lattice $\Gamma \in \mathcal{L}_i(G)$, such that 
$L^2(\Gamma \backslash G)$ has an irreducible $A_{\{\psi_1, \psi_2\}}$-component. That is, 
$A_{\{\psi_1, \psi_2\}}$ is an automorphic representation of $G$. See \cite[Cor. 7.7]{schimpf} for $n =3$. 

(ii) Let $G$ be a connected non-compact semisimple Lie group with finite centre, $K$ be a maximal compact 
subgroup of $G$ with $\theta$, the corresponding Cartan involution, and $X = G/K$. Let $\sigma$ be an 
involutive automorphism of $G$ such that $\sigma \theta = \theta \sigma$, $G(\sigma) = 
\{g \in G : \sigma (g) = g\}$, $K(\sigma) = K \cap G(\sigma)$, and $X(\sigma) = G(\sigma)/K(\sigma)$. 
Let $\frak{g} =$ Lie$(G)$, and $\frak{g}(\sigma)=$ Lie$(G(\sigma))$. 
Let $\Gamma$ be a torsion-free $\sigma$-stable uniform lattice in $G$ such that $\Gamma(\sigma) \backslash X(\sigma)$ 
is embedded inside $\Gamma \backslash X$, where $\Gamma (\sigma) = \Gamma \cap G(\sigma)$. Let $C(\sigma, \Gamma)$ 
be the image of $\Gamma(\sigma) \backslash X(\sigma)$ in $\Gamma \backslash X$, and $\mathcal{P}(C(\sigma, \Gamma))$ 
be the Poincar\'{e} dual of the fundamental class $[C(\sigma, \Gamma)]$. Let $A_\frak{q}$ be the irreducible unitary 
representation with trivial infinitesimal character associated with the $\theta$-parabolic subalgebra $\frak{q}$ of $\frak{g}$. 
Then we have \\ 
{\it If $G$ is simple, $A_\frak{q}$ is not the trivial representation of $G$, and $A_\frak{q}$ is discretely decomposable as a 
$(\frak{g}(\sigma), K(\sigma))$-module, then $\mathcal{P}(C(\sigma, \Gamma))$ does not have a $A_\frak{q}$-component} 
\cite[Cor. 4.2]{mondal1}. \\ 
This is a modification of \cite[Th. 4.3]{koboda}, and is corollary of a more general result in 
\cite[Th. 4.1]{mondal1}, \cite[Th. 1.2]{kobpreprint}. Now there is a classification of all pairs $(\frak{g}, \frak{g}(\sigma))$ and 
the modules $A_\frak{q}$, such that $A_\frak{q}$ is discretely decomposable as a $(\frak{g}(\sigma), K(\sigma))$-module. See 
\cite{kobosh}. According to this classification, if $\frak{g} = \frak{a}_{2n},\ \frak{b}_n,\ \frak{c}_n,\ \frak{e}_6,\ \frak{e}_7, 
\ \frak{e}_8,\ \frak{f}_4,\ \frak{g}_2$, then there is no involutive automorphism $\sigma$ and $\theta$-stable parabolic 
subalgebra $\frak{q} (\neq \frak{g})$ such that $A_\frak{q}$ is discretely decomposable as a $(\frak{g}(\sigma), K(\sigma))$-module 
(\cite[Th. 4.12]{kobosh}). If $\frak{g} = \frak{a}_{2n-1}$, and $\frak{g}(\sigma) = \frak{sp}(n,\mathbb{C})$, or 
$\frak{su}^*(2n)$, then $A_\frak{q}$ is discretely decomposable as a $(\frak{g}(\sigma), K(\sigma))$-module iff 
$\frak{q} = \frak{q}_{\{\psi_1\}}$, or $\frak{q}_{\{\psi_{2n-1}\}}$, or $\frak{q} = \frak{q}_{\phi}$. 
For any other $\frak{g}(\sigma)$, no non-trivial representation $A_\frak{q}$ is discretely decomposable as a $(\frak{g}(\sigma), K(\sigma))$-module. 
 If $\frak{g} = \frak{\delta}_{n}$, and $\frak{g}(\sigma) = \frak{so}(2n-1,\mathbb{C})$, or 
$\frak{so}(2n-1,1)$, then $A_\frak{q}$ is discretely decomposable as a $(\frak{g}(\sigma), K(\sigma))$-module iff 
$\frak{q} = \frak{q}_{\{\psi_n\}}$, or $\frak{q} = \frak{q}_{\phi}$. 
For any other $\frak{g}(\sigma)$, no non-trivial representation $A_\frak{q}$ is discretely decomposable as a $(\frak{g}(\sigma), K(\sigma))$-module. See 
\cite[Table C.4, C.5]{kobosh}. We know from the proof of Theorem \ref{th2} that if $\frak{g}= \delta_4$, at least one 
$A_{\{\psi_j\}}\ (j=1,3,\textrm{or }4)$ is an automorphic representation. 
Now if we apply the above results for $\frak{g} = \frak{\delta}_4$, we see that either $A_{\{\psi_1\}}$, or 
$A_{\{\psi_3\}}$ is an automorphic representation. 
}
\end{remark}

\noindent 
{\bf Acknowledgements.} The author thank Parameswaran Sankaran for his continuous encouragement, support and 
for many helpful comments. The proof of Lemma \ref{nf} is due to him. 
He also shared his articles \cite{mondal-sankaran1}, \cite{mondal-sankaran2}, which 
are primary motivation for this paper.


\begin{thebibliography}{99} 
\bibitem{anderson} Anderson, G. Theta functions and holomorphic differential forms on compact quotients of bounded
Symmetrie domains, Duke Math. J. {\bf 50} (1983), 1137--1170. 
\bibitem{Borel}Borel, Armand Compact Clifford-Klein forms of symmetric spaces. Topology ({\bf 2}) (1963) 111--122. 
\bibitem{borel-wallach} Borel, Armand; Wallach, Nolan {\it Continuous cohomology, discrete groups, and representations of 
reductive groups.} Ann. Math. Stud. {\bf 94}, Princeton Univ. Press, 1980. Second Ed.  
\bibitem{bourbaki} Bourbaki, N. {\it Groupes et alg\'ebres de Lie. Chapitre IV-VII.\/} Hermann, Paris 1968. 
\bibitem{clozel} Clozel, L. On limit multiplicities of discrete series representations in spaces of automorphic forms, Invent.
Math. {\bf 83} (1986), 265--284.
\bibitem{degeorge-wallach} DeGeorge, D.; Wallach, N. Limit formulas for multiplicities in $L^2(\Gamma\backslash G)$, Ann. Math. {\bf 107} (1978), 133--150. 
\bibitem{delorme} Delorme, P. "Sur la cohomologie continue des repr\'{e}sentations unitaires irr\'{e}ductibles des groupes de Lie 
semi-simples complexes", pp. 179--191 in {\it Analyse harmonique sur les groupes de Lie, II} (Nancy-Strasbourg, 1976--1978), 
edited by P. Eymard et al., Lecture Notes in Mathematics {\bf 739}, Springer, Berlin, 1979. 
\bibitem{enright} Enright, T. J. Relative Lie algebra cohomology and unitary representations of complex Lie groups. Duke Math. J. {\bf 46} 
(1979), no. 3, 513--525. 
\bibitem{ggp} Gelfand, I. M.; Graev, M. I.; Pyatetskii-Shapiro, I. I.  {\it Representation theory and automorphic functions.} W. A. Saunders Co. Philadelphia, 1966.  
\bibitem{gp} Gelfand, I. M.; Pyatetskii-Shapiro, I. I.  Theory of representations and theory of automorphic functions. Amer. Math. Soc. Transl. (2) {\bf 26} (1963) 173--200.
\bibitem{hc1} Harish-Chandra Representations of semisimple Lie groups. VI. Integrable and square-integrable representations. Amer. J. Math. 78 (1956), 564--628.
\bibitem{helgason} Helgason, Sigurdur {\it Differential geometry, Lie groups, and symmetric spaces.} Corrected reprint of the 1978 original. Graduate Studies in Mathematics, 
{\bf 34}. American Mathematical Society, Providence, RI, 2001.
\bibitem{jacobson} Jacobson, Nathan {\it Basic Algebra I.} Second Edition. W. H. Freeman and Company, New York, 1985. 
\bibitem{kac} Ka\v{c}, V. G. Automorphisms of finite order of semisimple Lie algebras. {\it Funkcional. Anal. i Prilo\v{z}en.} {\bf 3} (1969), 
no. 3, 94--96. 
\bibitem{knappvogan} Knapp, Anthony W.; Vogan, David A., Jr. {\it Cohomological Induction and Unitary Representations.} Princeton 
University Press, Princeton, NJ. 1995. 
\bibitem{kobpreprint} Kobayashi, T. A vanishing theorem for modular symbols on locally symmetric spaces II. preprint 2009. 
\bibitem{koboda} Kobayashi, T.; Oda, T. A vanishing theorem for modular symbols on locally symmetric spaces. Comment. Math. Helv. 
{\bf 73} (1998), 45--70. 
\bibitem{kobosh} Kobayashi, T.; Oshima, Y. Classification of discretely decomposable $A_\frak{q} (\lambda)$ with respect to reductive 
symmetric pairs. Adv. Math. {\bf 231} (2012) 2013--2047. 
\bibitem{kor-wolf} Kor\`{a}nyi, A.; Wolf, J. A. Realization of Hermitian Symmetric Spaces as Generalized Half-planes. {\it Ann. of Math.} 
{\bf 81} (1965), 265--288. 
\bibitem{li}  Li, J.-S.  Non-vanishing theorem for cohomology of certain arithmetic quotients. Jour. reine angew. Math. {\bf 428} (1992) 177-217. 
\bibitem{matsushima} Matsushima, Y.   A formula for the Betti numbers of compact locally symmetric Riemannian 
manifolds.  J. Diff. Geom. {\bf 1} (1967) 99--109.           
\bibitem{mira}  Millson, John J.; Raghunathan, M. S.  Geometric construction of cohomology for arithmetic groups. I. 
Proc. Indian Acad. Sci. (Math. Sci.), {\bf 90} (1981) 103--123. 
\bibitem{mondal1}  Mondal, Arghya $A_\frak{q}$-components of geometric cycles in compact Hermitian locally symmetric 
spaces. Arxiv:1812.03730. 
\bibitem{mondal-sankaran1} Mondal, Arghya; Sankaran, Parameswaran Non-vanishing cohomology classes in uniform lattices in $SO(n,\mathbb{H})$ 
and automorphic representations. Arxiv:1610.01368. 
\bibitem{mondal-sankaran2} Mondal, Arghya; Sankaran, Parameswaran Geometric cycles in compact locally Hermitian symmetric spaces and automorphic representations. 
Arxiv:1703.03206. 
\bibitem{parthasarathy1} Parthasarathy, R. A generalization of Enright-Varadarajan modules. Compositio Math. {\bf 36} (1978), 53--73. 
\bibitem{parthasarathy2} Parthasarathy, R. Criterion for unitarizability of certain highest weight modules. Proc. Ind. Acad. Sci. {\bf 89} 
(1980) 1--25. 
\bibitem{riba} Salamanca-Riba, S. On the unitary dual of some classical Lie groups. Comp. Math. {\bf 68} (1988) 251--303. 
\bibitem {rs} Rohlfs, J\"urgen; Schwermer, Joachim Intersection numbers of special cycles. J. Amer. Math. Soc. {\bf 6} (1993), 
no. 3, 755--778. 
\bibitem{schimpf} Schimpf, Susanne On the geometric construction of cohomology classes for cocompact discrete subgroups of $SL_n(\mathbb{R})$  and $SL_n(\mathbb{C})$. 
Pacific J. Math. {\bf 282} (2016), 445--477. 
\bibitem{sw} Schwermer, Joachim; Waldner, Christoph On the cohomology of uniform arithmetically defined subgroups in $SU^*(2n).$ 
Math. Proc. Cambridge Philos. Soc. {\bf 151} (2011), no. 3, 421--440. 
\bibitem{vogan} Vogan, David Unitarizability of certain series of representations. Ann. Math. {\bf 120} (1984) 141--187. 
\bibitem{vogan97} Vogan, David  Cohomology and group representations, Proc. Symp. Pur. Math. {\bf 61} (1997) 219--234, Amer. Math. Soc., Providence, RI. 
\bibitem{voganz} Vogan, David A., Jr.; Zuckerman, Gregg J. Unitary representations with nonzero cohomology. Compositio Math. {\bf 53} 
(1984), no. 1, 51--90. 
\bibitem{waldner} Waldner, Christoph Geometric cycles and the cohomology of arithmetic subgroups of the exceptional group $G_2$. J. Topol. {\bf 3} (2010), no. 1, 81--109. 
\end{thebibliography}
\end{document}